\documentclass{article}
\usepackage[utf8]{inputenc} 

\usepackage{geometry} 
\usepackage{amsfonts, mathrsfs, bbm, amssymb, upgreek,bm}
\usepackage{graphicx} 
\usepackage{amsmath,amsthm}
\usepackage[parfill]{parskip} 

\usepackage{subeqnarray}
\usepackage{epstopdf}
\usepackage{array} 
\usepackage{paralist} 
\usepackage{textcomp}
\usepackage{gensymb}
\usepackage{verbatim} 
\usepackage{xcolor}
\usepackage{subfig} 
\usepackage{multirow}
\usepackage{marginnote}
\usepackage{latexsym}
\usepackage{color}
\usepackage [normalem]{ulem}
\usepackage {soul}
\usepackage{cancel}
\usepackage{bbm}
\usepackage{sectsty}
\usepackage{mathtools}

\allsectionsfont{\sffamily\mdseries\upshape} 
\usepackage{hyperref}\hypersetup{
    colorlinks=true,
    citecolor=blue,
    filecolor=magenta,      
    urlcolor=cyan,
}
\numberwithin{equation}{section}

\theoremstyle{plain}

\newcommand{\aj}[1]{\textcolor{black}{#1}}    
\newcommand{\ch}[1]{\textcolor{black}{#1}}  
\newcommand{\norm}[2]{\left\|#1\right\|_{#2}}
  
\newcommand{\defeq}[0]{\mathrel{\mathop:}=}
\newtheorem{definition}{Definition}[section]
\newtheorem{theorem}{Theorem}
\newtheorem{proposition}{Proposition}[section]
\newtheorem{remark}{Remark}[section]
\newtheorem{lemma}{Lemma} [section]

\newcommand\numberthis{\addtocounter{equation}{1}\tag{\theequation}}

\usepackage{scalerel,stackengine}
\stackMath
\newcommand\reallywidehat[1]{
	\savestack{\tmpbox}{\stretchto{
			\scaleto{
				\scalerel*[\widthof{\ensuremath{#1}}]{\kern-.6pt\bigwedge\kern-.6pt}
				{\rule[-\textheight/2]{1ex}{\textheight}}
			}{\textheight}
		}{0.5ex}}
	\stackon[1pt]{#1}{\tmpbox}
}
\usepackage[backend=bibtex,style=numeric-comp]{biblatex}
\bibliography{ref}

\title{Boundary layer formation in the quasigeostrophic model near nonperiodic rough coasts}
\author{Gabriela L\'opez Ruiz}
\date{}
\begin{document}
	\maketitle
	\begin{abstract}

	\begin{small}
We study the so-called homogeneous model of wind-driven ocean circulation or the single-layer quasigeostrophic model. Our attention focuses on performing a complete asymptotic analysis that highlights boundary layer formation along the coastal line. We assume rough coasts without any particular structure, resulting in the study of a  nonlinear PDE system for the western boundary layer in an infinite domain. As a consequence, we look for the solution in nonlocalized Sobolev spaces. Under this hypothesis,  the eastern boundary layer exhibits a singular behavior at low frequencies far from the rough boundary, leading to issues with convergence. The problem is tackled by imposing ergodicity properties. We establish the well-posedness of the governing boundary layer equations and the approximate solution.  Our results generalize the ones of the paper by Bresch and Gérard-Varet (\aj{Commun. Math. Phys.  (1) 253: (2005), 81-119}) in the context of periodic irregularities.
		
		\bigskip
	\noindent{\bf Keywords:} Boundary layers, nonperiodic roughness, single layer quasi-geostrophic model, 
	\\approximate solution.
	
	\medskip
	\noindent{\bf MSC:} 35B25, 35B40, 35C20, 35Q35.
\end{small}
	\end{abstract}
	\section{Introduction}

This paper addresses roughness-induced effects on geophysical fluid motion in a context where small irregularities have very little structure. In geophysics, this phenomenon is usual when looking at the indentations on the bottom of the ocean and shores. The analysis will be conducted on the homogeneous model of wind-driven ocean circulation, also known as the 2D quasigeostrophic model.  In this case, the input is the planetary wind-stress field over the ocean, while the output is the transport that takes place into the mid-depth layer and is forced by the Ekman pumping, due to wind stress, above this layer. Steady circulation is then maintained by bottom friction and lateral diffusion of relative vorticity. 

The mathematical description of the model is as follows: let $\Psi = \Psi(t, \mathbf{x}) \in \mathbb{R}$ be the stream function associated to the two-dimensional velocity field $u = (u_1(t, \mathbf{x}),  u_2(t, \mathbf{x}) )^{t}$. In a simply connected domain $\Omega^\varepsilon\subset\mathbb{R}^{2}$ to be described later on,  the system reads 

\begin{equation}\label{Bresch-GV_1}
\left\{\begin{array}{rcl}
\left(\partial_t+\nabla^{\perp}\Psi\cdot\nabla\right)\left(\Delta\Psi+\beta y-\mathrm{Fr}\,\Psi+\eta_B\right)+r\Delta\Psi&=&\beta\,\mathrm{curl}\,\tau+\textrm{Re}^{-1}\Delta^{2}\Psi,\\
{\displaystyle \Psi|_{\partial\Omega}}&=&\dfrac{\partial\Psi}{\partial n}=0,\\
\Psi|_{t=0}&=&\Psi_{ini},
\end{array}\right.
\end{equation}
where
\begin{itemize}
	\item $\nabla^{\perp}\Psi\cdot\nabla$ is the transport operator by the two-dimensional flow;
	\item $\Delta\Psi$ is the vorticity, and $r\Delta\Psi$ , $r > 0$, is the Ekman pumping term due to bottom friction;
	\item $\mathrm{Fr}$   is  the Froude number due to the free surface ;
	\item  $\beta>0$ is a parameter characterizing the beta-plane approximation which results from linearizing the Coriolis factor around a given latitude;
	\item $\eta_B$ describes the variations of the bottom topography;
	\item $\beta\, \mathrm{curl}\, \tau$ is the Ekman pumping term due to wind stress at the surface, where $\tau$ is a given stress tensor; and,
	\item  $\mathrm{Re}$ denotes the Reynolds number.
\end{itemize}
Here, $\Psi_{ini}$ is chosen such that data is ``well-prepared'', i.e., we need $\Psi_{ini}$ to converge at least in $H^1(\Omega^\varepsilon)$ to a function to be specified later on.
For a formal derivation of the model (\ref{Bresch-GV_1}) from the Navier–Stokes equations, we
refer the reader to \cite{Colin1996} and \cite{Pedlosky2013}. If the basin is closed and if we assume that there is no water outflow, then the flux corresponding to the horizontal velocity has to vanish at the boundary, from which we have the homogeneous boundary condition $\Psi|_{\partial\Omega}=0$. Moreover, the presence of the diffusion term $\mathrm{Re}^{-1}\Delta^{2}\Psi$ requires the no-slip boundary condition $\frac{\partial\Psi}{\partial n}=0$.

Under certain hypotheses (fast
rotation, thin layer domain, small vertical viscosity) and proper scaling, B. Desjardins and E. Grenier proved this model describes asymptotically a 2D fluid \cite{Desjardins1999}. The authors performed a complete boundary layer analysis of the model in the domain
\[\Omega=\left\{(x,y)\in\mathbb{R}^{2}:\medspace\chi_w(y)\leq  x\leq \chi_e(y),\, y_{min} \leq  y\leq  y_{max}\right\},\]
where $\chi_w$ and $\chi_e$ are smooth functions defined for $y \in [y_{min},y_{max}]$. The forcing term $\tau$ was assumed to be identically zero when $y$ is in a neighborhood of $y_{min}$ and $y_{max}$. This assumption is crucial to avoid the strong singularities near the northern (max) and southern (min) ends of the domain, known as geostrophic degeneracy \cite{Dalibard2009}.

In the context of ocean currents, the Rossby parameter $\beta$ and Reynolds number are very high. A first approximation of the solution confirms \eqref{Bresch-GV_1} is a singular perturbation problem and there is boundary layer formation. These problems are often tackled by a multi-scale approach. Therefore, it is natural to look for an approximate solution of \eqref{Bresch-GV_1} of the form
\begin{equation*}
    \Psi^\varepsilon\sim \Psi_{int}(t,x,y)+\Psi_{bl}(t,y,X(\varepsilon),Y(\varepsilon)),
\end{equation*}
where $X(\varepsilon),Y(\varepsilon)$ are the boundary layer or fast variables.

Following this reasoning, Desjardins and Grenier derived the so-called Munk layers, responsible for the western intensification of boundary currents. D. Bresch and D. Gérard-Varet \cite{Bresch2005} later generalized these results  for the case of rough shores with periodic roughness.  In this case, the ocean basin was described by a domain 
\begin{equation}\label{pdt_3}
\Omega^{\varepsilon} = \left\{(x, y)\in \mathbb{R}^2:\medspace\chi_w(y)-\varepsilon\gamma_w(\varepsilon^{-1}y) < x< \chi_e(y)+\varepsilon\gamma_e(\varepsilon^{-1}y),\quad y_{min} \leq  y\leq  y_{max}\right\},
\end{equation}
where $\varepsilon$ is a small positive parameter and $\gamma_w$, $\gamma_e$ are regular and periodic functions describing the roughness of the West and East coasts, respectively.  Taking into account rough coastlines leads naturally to additional mathematical difficulties. For example, the usual Munk system of ordinary differential equations is replaced by elliptic quasilinear partial differential equations. Nevertheless, the periodicity assumption on the structural properties of $\gamma_e$, $\gamma_w$ simplifies the analysis of the existence and uniqueness of the solutions. 

 A natural extension of this work would be dropping the periodicity assumption, since the geometry of the boundary is not meant to follow a particular spatial pattern. 
 Our goal is to study the asymptotic behavior of problem (\ref{Bresch-GV_1}) when functions $\gamma_e$, $\gamma_w$ are arbitrary, therefore generalizing the results of \cite{Bresch2005}.
 
We are able to show three main results:
 \begin{enumerate}
     \item The western boundary profiles are well-defined and decay exponentially far from the boundary (cf. Theorem \ref{theorem:existence}).
     \item In the eastern boundary layer profiles, three components with different asymptotic behavior far from the boundary can be identified: one decaying exponentially, another converging to zero at a polynomial rate, and a third one whose convergence is guaranteed adding ergodic properties. Their well-posedness is completed by adding some constraints to the interior profile to ensure the validity of the far-field condition (see Theorem \ref{prop:eastern}.
     \item Finally, we have $\Psi^\varepsilon-\Psi^\varepsilon_{app}$ is $o(\varepsilon)$ in a norm that will be specified later (cf. Theorem \ref{theorem:convergence})
 \end{enumerate}

When the periodic roughness is no longer considered, the analysis, although possible, is much more involved, as shown in \cite{Basson2008,Gerard-Varet2010,Dalibard2014,Dalibard2017} in other contexts. 
We seek the solution of the boundary layer problem in a space of infinite energy. In particular, we will be considering Kato spaces $H^{s}_{\mathrm{uloc}}$ (a definition is provided in (\ref{kato_def})). The use of such function spaces to mathematically describe fluid systems traces back to \cite{Lemarie-Rieusset1999,Lemarie-Rieusset2002}, in which existence is proven for weak solutions of the Navier-Stokes equations in $\mathbb{R}^{3}$ with initial data in $L^{2}_{\mathrm{uloc}}$. For other relevant works, see \cite{Dalibard2014} and the references therein.

New difficulties arise in this context. First, due to the unboundedness of the boundary layer domains, we deal here with only locally integrable functions, leading to a completely different treatment of the energy estimates and mathematical tools, such as Poincaré inequality are needed but are no longer valid.  Second,  being typical for fourth-order problems, the equation lacks a maximum
principle. Third, as the
roughness is nonperiodic, the boundary layer system is more complex. Indeed, the
absence of compactness both in the tangential and transverse variables and the presence of singularities at low frequencies for the eastern boundary layer functions make proving convergence in a deterministic setting extremely difficult. We, therefore, use the ergodic theorem to specify the behavior of
the solution of the eastern boundary layer far from the boundary and, later, to find the energy estimates in the analysis of the quality of the approximation. 

The rest of the paper is structured as follows. The following section contains a precise description
of the domain, some simplifying assumptions and the statements of the main mathematical results. Section \ref{s:asymptotic} contains the assumptions made for constructing the profiles of the approximate solution and the detailed derivation of the functions in the main order. In Section \ref{methodo}, we outline the general methodology to solve linear and nonlinear/linearized problems characterizing the boundary layer functions. The existence, uniqueness and asymptotic behavior of the first profile of the western boundary layer are discussed in Section \ref{western_l} for the linear case and Section \ref{western_nl}, for the nonlinear and linearized systems. Section \ref{eastern} focuses on the boundary layer analysis in the east of region. Finally, Section \ref{s:convergence} contains the construction of the approximate solution and the study of its convergence to the solution of the original problem.
    \section{Preliminaries and main results}\label{statement}
Before stating the main results, let us first state some hypotheses on the dimensionless problem (\ref{Bresch-GV_1}).  We will assume there is no stratification, therefore,  $\mathrm{Fr}=0$. Our study is solely focused on the effect of rough shores on flow behavior, so the bottom topography parameter $\eta_B$ is considered to be nil. Since for a basin of $1000 \times 1000$ km at a central latitude $\theta_0 = 45^\circ N$ \cite{Dijkstra2005}, we have that
$\eta_B+\mu\sim 1.0 - 100, \quad \beta\sim 10^{3},\;\;\textrm{and}\;\; \mathrm{Re} \sim 1.6 - 160$,
it is possible to consider  
$r = 1$ and $\mathrm{Re} = 1$ to simplify the computations. Up to minor changes, equivalent results can be obtained for arbitrary values of the Reynolds number and $r$.

Let $\varepsilon$ be the natural size of the boundary layers arising in this study,  we consider $\beta = \varepsilon^{-3}$. This choice of scaling preserves the problem's physical accuracy. Moreover, the size of the irregularities is also assumed to be equal to $\varepsilon$. The last hypothesis is mainly of mathematical significance, since it allows for a richer analysis due to the interaction of the linear and non-linear terms of the equation at the main order for the boundary layer problem.  Then, system (\ref{Bresch-GV_1}) becomes
\begin{equation}\label{Bresch-GV_3}
\left\{\begin{array}{rcll}
\left(\partial_t+\nabla^{\perp}\Psi\cdot\nabla\right)\left(\Delta\Psi+\varepsilon^{-3} y\right)+\Delta\Psi&=&\varepsilon^{-3}\,\mathrm{curl}\,\tau+\Delta^{2}\Psi,&\textrm{in }\Omega^{\varepsilon}\\
\Psi|_{\partial\Omega^{\varepsilon}}&=&\frac{\partial\Psi}{\partial n}|_{\partial\Omega^{\varepsilon}}=0,\\
\Psi|_{t=0}&=&\Psi_{ini}.
\end{array}
\right.
\end{equation}

Here, we adopt the notation and terminology in \cite{Bresch2005}. The domain of problem (\ref{Bresch-GV_3}) is defined as follows
$$\Omega^{\varepsilon}=\Omega^{\varepsilon}_w \cup\Sigma_w \cup\Omega\cup\Sigma_e \cup\Omega^{\varepsilon}_e.$$
\vspace*{-.3in}
\begin{figure}[h!]
	\centering
	\includegraphics[width=0.5\linewidth]{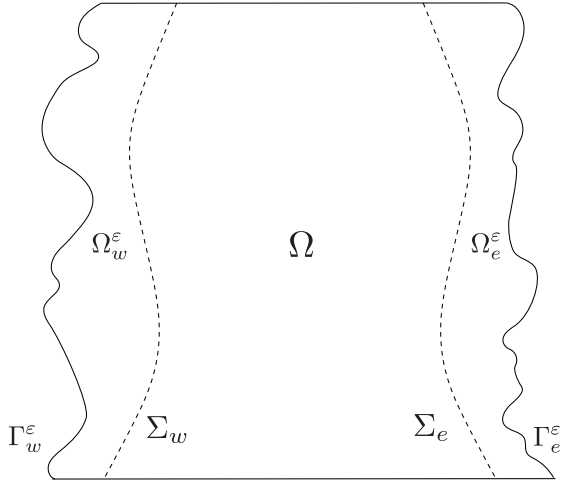}
	\caption{Domain $\Omega^{\varepsilon}$}
	\label{fig:domaine03-03}
\end{figure}
\begin{itemize}
	\item [$\bullet$] The ``interior domain” is given by
	\[\Omega=\left\{\chi_w(y) \leq  x \leq \chi_e(y),\: y \in [y_{min},y_{max}]\right\},
	\]
	where $\chi_w$ and $\chi_e$ are smooth functions defined for $y \in [y_{min},y_{max}]$.
	
	\item [$\bullet$] $\Sigma_w =\left\{(\chi_w(y),y),\: y \in (y_{min},y_{max})\right\} $ and  $\Sigma_e = \left\{(\chi_e(y),y),\ y \in (y_{min},y_{max})\right\}$ are interfaces separating the interior domain from the ``rough shores''.
	\item [$\bullet$] $\Omega^{\varepsilon}_w$ and  $\Omega^{\varepsilon}_e$ are the rough domains. The positive smooth functions $\gamma_w = \gamma_w(Y)$ and $\gamma_e = \gamma_e(Y)$ describe the irregularities. We set
		\begin{eqnarray}
		\Omega^{\varepsilon}_w &=& \left\{(x, y), 0 > x-\chi_w(y) > -\varepsilon\gamma_w(\varepsilon^{-1}y)\right\},\\
		\Omega^{\varepsilon}_e &=& \left\{(x, y),\: 0 < x - \chi_e(y) < \varepsilon\gamma_e(\varepsilon^{-1}y)\right\}.
		\end{eqnarray}
	The lateral boundaries are
	\begin{eqnarray}
	\Gamma_w^{\varepsilon} &=& \{ \left(\chi_w(y)-\varepsilon\gamma_w(\varepsilon^{-1}y),\;y\right) ,y\in(y_{min},y_{max}),\}\nonumber\\  \Gamma_e^{\varepsilon} &=& \{ \left(\chi_e(y)+\varepsilon\gamma_e(\varepsilon^{-1}y),\;y\right) ,\;y\in(y_{min},y_{max})\}.\nonumber
	\end{eqnarray}
	Let us introduce the notation $n_w$ and $n_e$ for the exterior unit normal vectors to the roughness curves $\gamma_w$ and $\gamma_e$.
\end{itemize}

 Let $T > 0$, we assume that  $\tau(t, x, y)\in L^\infty((0;T); H^s)$, for $s$ large enough. We are actually studying well-prepared data, as seen in \cite{Desjardins1999}. In order to avoid steep singularities due to advection of vorticity when approaching the northern and southern extremal points, we assume additionally the irrotational part of the wind vanishes in their vicinity, see \cite{Dalibard2009}. More precisely, we suppose that there exists $\lambda > 0$ such that 
\begin{equation}\label{eq_BGV:11}
\mathrm{curl}\,\tau = 0\quad \textrm{for}\; y \in \left[y_{max} -\lambda,y_{max}\right]\cup\left[y_{min},y_{min} +\lambda\right].
\end{equation} 

As pointed out in \cite{Desjardins1999}, problem (\ref{Bresch-GV_3}) has a unique smooth solution $\Psi^{\varepsilon}$ for all $\varepsilon > 0$.

The approximate solution is sought in form of series in powers of the small parameter $\varepsilon$ with coefficients depending on the global variables $t,x,y$, and the microscopic variables $Y = Y(y,\varepsilon)$, $X=X(x,y,\varepsilon)$
\begin{equation}\label{asymp_exp_int}
\Psi^{\varepsilon}_{app}(t,x,y)\sim \sum_{k=0}^{\infty}\varepsilon^{k}\left(\Psi^{k}_{int}(t,x,y)+\Psi^{k}_w\left(t,y,X_w,Y\right)+\Psi^{k}_e\left(t,y,X_e,Y \right)\right),
\end{equation}
where $\Psi^{k}_{int}(t,x,y)$ correspond to the interior terms, while $\Psi^{k}_w$ and $\Psi^{k}_e$ refer to the corrector terms in the western and eastern boundary layer, respectively. Such a series is substituted in the original problem and  a system of equations is obtained for each one of the profiles by equating to zero all coefficients associated to powers of $\varepsilon$. Here, $X$ and $Y$ are the fast or microscopic variables which depend on the small parameter. They are defined as follows:
\begin{equation*}
    Y=\varepsilon^{-1}y,\quad X_w=\varepsilon^{-1}(x-\chi_w(y)),\quad X_e=\varepsilon^{-1}(\chi_e(y)-x),
\end{equation*}
where $\omega_w$ and $\omega_e$ are respectively the western and eastern boundary layer domains. The former is of the form $\omega_w=\omega_w^{+}\cup\sigma_w\cup\omega^{-}_w$, where
\begin{eqnarray}\label{west_dom}
\omega_w^{+}&=&\left\{X_w>0,\quad Y\in\mathbb{R}\right\},\qquad\sigma_w=\left\{X_w=0,\quad Y\in\mathbb{R}\right\}\nonumber\\
\omega_w^{-}&=&\left\{-\gamma_w(Y)<X_w<0,\quad Y\in\mathbb{R}\right\}.
\end{eqnarray}
The domain $\omega_e$ can be defined in a similar manner.

In a first approximation of the solution, we are confronted with coastal asymmetry: it is impossible to obtain a solution in the eastern boundary layer domain satisfying all boundary conditions and decaying at infinity due to the lack of enough roots with positive real part. A usual choice under these circumstances is to consider  that $\Psi_{int}$ is tangent to the boundary $\Sigma_e$, which results in \begin{equation}\label{int_sol}
\Psi_{int}^0(t,x,y)=-\int^{\chi_e(y)}_{x}\mathrm{curl}\,\tau(t,x',y)dx',
\end{equation} 
and $\Psi_e^0\equiv 0$. Then, the key element in the construction of the approximate solution will be to determine $\Psi_w^0$ which formally solves the problem
\begin{eqnarray}\label{eq_BGV:15}
Q_w(\Psi_w^{0},\Psi_w^{0})+\partial_{X_w}\Psi_w^{0}-\Delta_w^{2}\Psi_w^{0}&=&0,\quad\textrm{in}\quad\omega_w^-\cup \omega_w^+\nonumber\\
\left[\Psi_w^{0}\right]\big|_{\sigma_w}&=&\phi,\\
\left[\partial_X^{k}\Psi_w^{0}\right]\big|_{\sigma_w}&=&0,\;k=1,2,3,\nonumber\\ \Psi_w^{0}\big|_{X=-\gamma_w(Y)}&=&\frac{\partial\Psi_w^{0}}{\partial n_w}\big|_{X=-\gamma_w(Y)}=0.\nonumber
\end{eqnarray}
Here, $[\cdot]\big|_{X=X'}$ denotes the jump operator of the function $f$ at $X=X'$ and is defined as $[f]|_{X=X'} := f(X'^{+},\cdot)- f(X'^{-},\cdot)$. The jump of the function at the western boundary of the interior domain is given by  \begin{equation}\label{def_phi}
    \phi(t,y)=\int_{\chi_w(y)}^{\chi_e(y)}\mathrm{curl}\,\tau(t,x',y)dx'.
\end{equation}
Moreover, for $\alpha_w=\chi_w' (y)$, the differential operators  are given by
\[\nabla_w = (\partial_{X_w}, \partial_Y - \alpha_w\partial_{X_w})^{t} ,\quad \nabla_w^{\bot}(y) = (\alpha_w\partial_{X_w} - \partial_Y , \partial_{X_w})^{t} ,\]
consequently,  $\Delta_w$ and $Q_w$ are defined  as follows
\begin{eqnarray}
\Delta_w&=&\partial_{X_w}^{2} + \left(\alpha\partial_{X_w}-\partial_Y\right)^{2},\nonumber\\
Q_w(\Psi,\widetilde{\Psi}  ) &=& \nabla_w^{\bot} \cdot   \left(\left(\nabla_w^{\bot}\Psi   \cdot \nabla_w\right) \nabla_w^{\bot}\widetilde{\Psi} \right).\nonumber
\end{eqnarray} 
Note $\Delta_w$ and  $\Delta_w^{2}$ are elliptic operators with respect to the variables $Y$ et $X_w$. At the level of the boundary layer, $t$ and $y$ behave as parameters.

Our first result is the existence and uniqueness of the solution for the boundary layer system (\ref{eq_BGV:15}). As usual in the steady Navier–Stokes equations theory, the well-posedness is obtained under a smallness hypothesis. The problem is defined in an unbounded set; therefore, we seek the solution in spaces of uniformly locally integrable functions, also know in the literature as Kato spaces \cite{Kato1975b}. They include a richer spectrum of functions, allowing for some singular behavior or non-decaying functions. Let us briefly recall the definition:

Let $\theta\in C^\infty_0(\mathbb{R}^d)$ be such that $\mathrm{Supp}\theta\subset [-1, 1]^d$, $\theta\equiv 1$ on $[-1/4, 1/4]^d$, and
\begin{equation}\label{DP:1-4}
\sum_{k\in\mathbb{Z}^d}\tau_k\theta(x) = 1,\quad\textrm{for all }x \in\mathbb{R}^d,
\end{equation}
where $\tau_k$ denotes the translation operator defined by $\tau_k f (x) = f (x - k)$. Then, for $s \geq 0$, $p \in  [1, +\infty)$,
\begin{equation}\label{kato_def}
\begin{split}
L^p_{\mathrm{uloc}}(\mathbb{R}^d)&=\left\{u\in L^p_{\mathrm{loc}}(\mathbb{R}^d):\quad\underset{k\in\mathbb{Z}^d}{\sup}\|(\tau_k\theta)u\|_{L^p(\mathbb{R}^d)}\right\},\\
H^s_{\mathrm{uloc}}(\mathbb{R}^d)&=\left\{u\in H^s_{\mathrm{loc}}(\mathbb{R}^d):\quad\underset{k\in\mathbb{Z}^d}{\sup}\|(\tau_k\theta)u\|_{H^s(\mathbb{R}^d)}\right\}
\end{split}
\end{equation}

We show the following:

\begin{theorem}\label{theorem:existence}
	Let $\gamma_w$ be a positive $W^{2,\infty}(\mathbb{R})$ function and $\omega_w$ be defined as before. There exists a constant $\delta_0 > 0$ such that  if $\|\phi\|_{\infty} < \delta_0$, problem (\ref{eq_BGV:15}) has a unique solution in $H^{2}_{\mathrm{uloc}}(\omega_w)$ denoted by $\Psi_w$. Moreover, for a certain constant  $\delta>0$, it satisfies the estimate
	\begin{equation}\label{eq_BGV:16}
	\|e^{\delta X_w}\Psi_w\|_{H^{2}_{\mathrm{uloc}}(\omega_w)} \leq C \|\phi\|_{\infty}.
	\end{equation}
\end{theorem}
This theorem generalizes the result of \cite{Bresch2005} for to the case of  nonperiodic roughness. A remarkable feature of this result is that exponential decay to zero persists, despite the arbitrary roughness, and without any additional assumption on the function describing the irregular boundary.

Following the ideas in Masmoudi and G\'erard-Varet \cite{Gerard-Varet2010}, we look for the solution of \eqref{eq_BGV:15} by introducing a transparent boundary which divides the domain in two: a half-space and a bounded rough channel. Then, the problem is solved in each of the subdomains, and a pseudo-differential operator of Poincaré-Steklov is used to relate the behavior of the solutions at both sides of the interface. When \eqref{eq_BGV:15} is considered linear, this last step can be done directly; otherwise, applying the implicit function theorem is needed to join the solutions at the artificial boundary. 

Once we have shown the above result on the western boundary layer, we construct the approximate solution and analyze its closeness to the original problem. The error is computed by calculating the following profiles in a systemic scheme (see Sections \ref{s:asymptotic} and \ref{s:convergence})  and is pretty straightforward.  

At order $\varepsilon^n$ the interior profile satisfies $\Psi^n_{int}=C^n(t,y)-\int_{x}^{\chi_e(y)}F_n\;dx$, where $F_n$ depends on $\Psi^j_{int}$, $j\leq n-1$. The value of $C^n(t,y)$ will be specified later. Then, the $n$-th eastern boundary layer profile meets the conditions 
\begin{equation}\label{eastern_bl_p1}
    \begin{split}
        -\partial_{X_e}\Psi^n_e-\Delta_e^{2}\Psi_e^n&=0,\quad\textrm{in}\quad\omega_e^-\cup\omega_{e}^+\\
        \left[\Psi_e^n\right]\big|_{\sigma_e}&=-C^n(t,y),\\
\left[\partial_{X_e}^k\Psi_e^n\right]\big|_{\sigma_e}&=\tilde{g}_k,\;k=1,\ldots,3,\\ \Psi_e^n\big|_{X_e=-\gamma_e(Y)}&=\frac{\partial\Psi_e^n}{\partial n_e}\big|_{X_e=-\gamma_e(Y)}=0,
    \end{split}
\end{equation}
where $\tilde{g}_k\in L^\infty$, the domain is given by
\begin{eqnarray}\label{east_dom}
	\omega_e^{+}&=&\left\{X_e>0,\quad Y\in\mathbb{R}\right\},\qquad\sigma_e=\left\{X_e=0,\quad Y\in\mathbb{R}\right\}\nonumber\\
	\omega_e^{-}&=&\left\{-\gamma_e(Y)<X_e<0,\quad Y\in\mathbb{R}\right\},
\end{eqnarray}
and the differential operators are defined as follows for $\alpha_e=\chi_e'(y)$:
\begin{eqnarray}
\nabla_e&=&\left(\partial_{X_e},- \alpha_e\partial_{X_e}-\partial_{Y}\right),\quad\nabla_e^\perp=\left(\partial_{Y}+\alpha_e\partial_{X_e},\partial_{X_e}\right) ,\\
\Delta_e&=&\partial_{X_e}^{2} + \left(\alpha\partial_{X_e}+\partial_Y\right)^{2},\quad
Q_e(\Psi,\widetilde{\Psi}  ) = \nabla_e^{\bot} \cdot   \left(\left(\nabla_e^{\bot}\Psi   \cdot \nabla_e\right) \nabla_e^{\bot}\widetilde{\Psi} \right).\nonumber
\end{eqnarray} 

Note that the main equation in \eqref{eastern_bl_p1} is elliptic with respect to $X_e$ and $Y$; $t$ and $y$ are considered parameters. Although the analysis of the well-posedness of  \eqref{eastern_bl_p1} is similar to one described for the western boundary layer, additional issues arise concerning the convergence of  $\Psi_e^n$ when $\varepsilon\rightarrow 0$. Indeed, the analysis of the problem \eqref{eastern_bl_p1} in the half-space reveals the lack of spectral gap, which prevents the decay far from the boundary, see Section \ref{linear_east}.  To guarantee the convergence of the eastern boundary layer profile when $X_e\rightarrow \infty$ some probabilistic assumptions are needed (ergodic properties).

Let $\varepsilon > 0$ and $(P, \Pi, \mu)$ be a probability space. For instance, $P$ could be considered as the set of $K$-Lipschitz functions, with $K>0$; $\Pi$ the borelian $\sigma$-algebra of $P$, seen as a subset of $C_b(\mathbb{R}^2;\mathbb{R})$ and $\mu$ a the probability measure preserved by the translation group $(\tau_{Y} )$ acting
on $P$. Then, the eastern boundary layer domain can be described as follows for $m\in P$: 
\begin{equation*}\label{Bass_22}
	\begin{split}
		\omega_e(m) &=\left\{(X_e, Y)\in\mathbb{R}^2:\:\: X_e>-\gamma_e(m, Y)\right\},
	\end{split}
\end{equation*}
where $\omega_{e}(m)=\omega_{e}^+(m)\cup\sigma_{e}\cup\omega_{e}^-(m)$ and $\omega_{e}^{\pm}(m)=\omega_{e}(m)\cap\{\pm X>0\}$. Here $\gamma_{e}$ are homogeneous and measure-preserving random process.

In this context, we are able to distinguish three components of $\Psi^n_e$ with different asymptotic behavior far from the boundary for which we have obtained the following result:

\begin{theorem}\label{prop:eastern}
Let $\omega_e$ a domain defined as before  and $\gamma_{e}$ an ergodic stationary random process, $K$-Lipschitz almost
surely, for some $K > 0$. Let $\tilde{g}_{k}\in L^{\infty}(\mathbb{R}_{+}\times[y_{\min},y_{\max}]\times\mathbb{R})$, $k=1,2,3$, then there exist a unique measurable map $C^n(t,y)$ such that problem \eqref{eastern_bl_p1} has a unique solution $\Psi_e^n=\Psi_{\mathrm{exp}}^n+\Psi_{\mathrm{alg}}^n+\Psi_\mathrm{erg}^n$ where
\begin{enumerate}
	\item 
	$\left\|\Psi^n_{\mathrm{erg}}\right\|_{L^q(\omega_e^+)}\underset{X_e\rightarrow+\infty}{\xrightarrow{\hspace*{1.2cm}}}0$,
	locally uniformly in $Y$, almost surely and in $L^q(P)$ for all finite $q$,
    \item there exist constants $\delta, C>0$ such that
    	\begin{equation}
	\|e^{\delta X_e}\Psi^n_{\mathrm{exp}}\|_{L^\infty(\omega_e^+)}\leq C\left(\sum_{k=1}^3\|\tilde{g}_k\|_{L^\infty}+\|C^n\|_{\infty}\right),
	\end{equation}
	\item there exists a constant $C>0$ such that
	\begin{equation}
	\|(1+X_e)^{1/4}\Psi^n_{\mathrm{alg}}\|_{L^\infty(\omega_e^+)}\leq C\left(\sum_{k=1}^3\|\tilde{g}_k\|_{L^\infty}+\|C^{n}\|_{\infty}\right).
	\end{equation}
\end{enumerate}
Moreover, $\Psi_e$ satisfies
\begin{equation}\label{eastern_results}
\begin{split}
\|\Psi_e^n\|_{H^2_{\mathrm{uloc}}(\omega_e)}<+\infty,\quad\textrm{almost surely}.
\end{split}
\end{equation}
\end{theorem}
The proof of the above result also relies on the use of wall laws and follows the same ideas of Theorem \ref{theorem:existence} for the western boundary layer. First, we apply Fourier analysis to problem \eqref{eastern_bl_p1} in the half-space, which hints directly to the singular behavior at low frequencies far from the boundary. We show that the properties involving $\Psi^n_{\mathrm{exp}}$ and $\Psi^n_{\mathrm{alg}}$ in Theorem \ref{prop:eastern} remain true in a deterministic setting by following the same ideas used for the western boundary layer. Only the convergence of $\Psi^n_{\mathrm{erg}}$ is shown using the ergodic theorem. Then, we define the associated Poincaré-Steklov operator for boundary data in $H^{3/2}_{\mathrm{uloc}}(\mathbb{R})\times H^{1/2}_{\mathrm{uloc}}(\mathbb{R})$. Finally, we look for the solution of a problem equivalent to \eqref{eastern_bl_p1} defined in a domain in which a transparent boundary condition is prescribed. For the equivalent system, we derive energy estimates in $H^2_{uloc}$ which are then used to prove existence and uniqueness of the solution. 

The eastern boundary layer not converging to zero or not doing it fast enough at infinity poses an issue when solving the problem at $\Psi^n_w$. In particular, the terms $\Psi^n_{\mathrm{alg}}$  and $\Psi^n_{\mathrm{erg}}$ influence the western boundary layer mainly through the nonlinear term. Adding ad hoc correctors allows us to show the following results for $\Psi^\varepsilon-\Psi_{app}^\varepsilon$.
\newline
\begin{theorem}\label{theorem:convergence}
	Let $\Psi^{\varepsilon}$ be the solution of problem (\ref{Bresch-GV_3}) and $\Psi_{app}^\varepsilon$ defined as in \eqref{asymp_exp_int}. Moreover, let $\Psi^\varepsilon_{ini}$ be such that $\Psi^\varepsilon_{ini}|_{\Omega^\varepsilon}=\partial_n\Psi^\varepsilon_{ini}=0$ and $\|\Psi^\varepsilon_{ini}-\Psi_{app}^\varepsilon|_{t=0}\|_{H^1(\Omega^\varepsilon)}\rightarrow 0$. There exists $C_{\infty}$, such that if $\|\mathrm{curl}\;\tau\|_{\infty} < C_{\infty}$, then
	\begin{equation}
	    \begin{split}
	  \|\Psi^\varepsilon-\Psi_{app}^\varepsilon\|_{L^2(0,T;H^2(\Omega^\varepsilon)}+\|\Psi^\varepsilon-\Psi_{app}^\varepsilon\|_{L^\infty(0,T;H^1(\Omega^\varepsilon)}&\rightarrow 0\quad\textrm{as }\varepsilon\rightarrow 0,\quad\textrm{almost surely}.
	    \end{split}
	\end{equation}
\end{theorem}

In the periodic setting, for which the boundary layer profiles decay
exponentially when $\varepsilon$ goes to zero, the bound becomes $O(\varepsilon^{1/2})$ for the $H^1$ estimate (see \cite{Bresch2005}). In this general setting, the convergence rate is limited by the behavior of eastern boundary layer profiles. Indeed, the lack of spectral gap requires the use of average information to guarantee the convergence of this profile far for the eastern boundary. Consequently, the convergence of $\Psi_{\text{erg}}$ and, therefore,  of $\Psi_{e}$, can be arbitrarily slow and influence the asymptotics of $\Psi_{app}^\varepsilon$. 

The convergence result in Theorem \ref{s:convergence} is obtained by computing energy estimates on  $\Psi^\varepsilon-\Psi_{app}^\varepsilon$. The accuracy of the estimates depends greatly on each element in the approximate solution, their interactions and contributions. Each component needs to be smooth enough, with proper controls on the corresponding derivatives. 

\paragraph{Plan of the paper.} The article is organized as follows. The following section is devoted to the formal construction of the approximate solution. In particular, we present modeling assumptions and discuss in detail the computation of the first profiles. Then, in Section \ref{methodo}, we briefly sketch the methodology of proof of existence and uniqueness of linear, nonlinear, and quasi-linear problems describing the behavior of the boundary layer in a rough domain. The main focus of Section \ref{western_l} is establishing the well-posedness of the linear problem describing the behavior of the western boundary layer. These results are essential to the subsequent proof of Theorem \ref{theorem:existence}. In Section \ref{western_nl}, the existence and uniqueness of the solution of the boundary layer system \eqref{eq_BGV:15} are shown for the case when $\phi$ is a given boundary data, with no decay tangentially to the boundary (Theorem \ref{theorem:existence}). To complete the analysis at the western boundary layer domain, we examine the problem driving the subsequent profiles of the western boundary layer in  Section \ref{s:linearized}. Then, Section \ref{eastern} focuses on the analysis of the singular behavior of the eastern boundary layer \eqref{eastern_bl_p1} which is summarized in Theorem \ref{prop:eastern}.  The final section aims to assess the quality of our approximation by estimating the different contributions to the error term of each one of the elements in the approximate solution. The latter is then used to prove the convergence result in Theorem \ref{theorem:convergence}.
     \section{Formal asymptotic expansion and first profiles}\label{s:asymptotic}
In this section, we construct the approximate solution $\Psi^{\varepsilon}_{app}$ for the singularly perturbed problem \eqref{Bresch-GV_3} employing a matched asymptotic expansion. Inner and outer expansions (boundary layers) are determined in the interior and the rough shores domain. Then, matching conditions at the interface are imposed to obtain an approximate global solution.

Let $X$ and $Y$ define the local variables obtained after scaling: $Y = y/\varepsilon$ while $X_w=\frac{x - \chi_w (y)}{\varepsilon}$, $X_e=\frac{\chi_e (y)-x}{\varepsilon}$. We seek for an approximate solution of (\ref{Bresch-GV_3}) of the form
\begin{equation}\label{ansatz}
\Psi^{\varepsilon}_{\mathrm{app}}(t,x,y)=\sum_{i=0}^{N}\varepsilon^{i}\left(\Psi^{i}_{int}(t,x,y)+\Psi^{i}_w\left(t,y,X_w,Y\right)+\Psi^{i}_e\left(t,y,X_e,Y     \right)\right)+O(\varepsilon^{N+1}),
\end{equation}
where $\Psi^{i}_{int}(t,x,y)$ correspond to the interior terms, while $\Psi^{i}_w$ and
$\Psi^{i}_e$ denote the western and eastern boundary layer profiles. Without loss of generality, we assume that the interior terms are zero outside $\bar{\Omega}$. 

Since the boundary layer terms are expected not to have an effect far from the boundaries, we assume
\begin{equation}\label{bc_far}
\Psi_e^{i}\underset{X_e\rightarrow\infty}{\longrightarrow} 0,\quad\Psi_w^{i}\underset{X_w\rightarrow\infty}{\longrightarrow} 0.
\end{equation}
The approximate solution must additionally satisfy the boundary condition $\Psi^{\varepsilon}_{app}=0$ at $\partial\Omega^{\varepsilon}$. Thus, we have
\begin{equation}\label{bc_dirichlet}
\Psi_w^{i}\Big|_{X_w=-\gamma_w(Y)}=0,\qquad\Psi_e^{i}\Big|_{X_e=-\gamma_e(Y)}=0
\end{equation}

From the homogeneous Neumann condition, we obtain the following conditions on the boundary layer profile: 
\begin{equation}\label{bl_neumman}
\dfrac{\partial\Psi_{w}^{i}}{\partial n_w}\Big|_{X=-\gamma_w(Y)}=0,\qquad\dfrac{\partial\Psi_{e}^{i}}{\partial n_e}\Big|_{X=\gamma_e(Y)}=0.
\end{equation} 
\aj{There is no loss of generality in assuming $\Psi_{w}^{i}|_{\Omega^\varepsilon_e}=0$ and $\Psi_{e}^{i}|_{\Omega^\varepsilon_w}=0$. This condition directly gives \eqref{bc_dirichlet} and \eqref{bl_neumman}.}

Additional conditions are needed at the interfaces separating the interior domain and the boundary layer domains to guarantee the existence of the derivatives in the weak sense over the whole domain. Since the interior
terms are zero outside $\overline{\Omega}$, they create discontinuities at the interfaces $\Sigma_w$ and $\Sigma_e$. Then, boundary layer terms are added to cancel such discontinuities; see for instance \cite{Jaeger2001,Gerard-Varet2003}. To guarantee the approximation is regular enough, we impose the condition:
\begin{equation}
\left[\partial_{X}^{k}\Psi^{\varepsilon}_{app}\right]\big|_{\Sigma_w\cup\Sigma_e}=0,\quad k=0,\ldots,3.
\end{equation}

We have the following jump conditions on the boundary layer
terms:
\begin{equation}\label{eq_BGV:26}
\left[\Psi^{i}_w(\cdot)\right]\big|_{\sigma_w}=-\Psi^{i}_{int}(\cdot)\big|_{x=\chi_w(y)}-\left[\Psi^{i}_e(\cdot)\right]\big|_{\sigma_w},\quad\left[\Psi^{i}_e(\cdot)\right]\big|_{\sigma_e}=-\Psi^{i}_{int}(\cdot)\big|_{x=\chi_e(y)}-\left[\Psi^{i}_w(\cdot)\right]\big|_{\sigma_e},
\end{equation}
and
\begin{equation}\label{eq_BGV:27}
\left[\partial^{k}_X\Psi^{i}_w(\cdot)\right]\big|_{\sigma_w}=f^{i,k}_w,\quad\left[\partial^{k}_X\Psi^{i}_e(\cdot)\right]\big|_{\sigma_e}=f^{i,k}_e,
\end{equation}
where the $f^{i,k}_w$, $k=1,2,3$, depends on the $\Psi^{j}_{int}$ and $\Psi^{e}_{j}$, $j \leq i$. Here, $f^{i,k}_e$ is chosen to be independent of $\Psi^{w}_{j}$, while still relying on the behavior of the interior profiles.

Plugging \eqref{ansatz} into (\ref{Bresch-GV_3}), and equating all terms of the same order in
powers of $\varepsilon$ provide a family of mathematical systems establishing the behavior of each one of the profiles in the ansatz. 

To facilitate the comprehension, we compute some terms of the approximation $\Psi_{app}^{\varepsilon}$. We are particularly interested in the ones corresponding to $i\in\{0,1\}$. 

When $i=0$, we obtain in the interior of the domain the so-called Sverdrup relation:
\begin{equation}\label{eq_int0}
\begin{split}
   \partial_x\Psi_{int}^0 &=\mathrm{curl}\,\tau,\\
\end{split}
\end{equation}
for which only one boundary condition can be prescribed, either on $\Sigma_e$ or on the $\Sigma_w$. 

\begin{remark}\label{remark_4.1} 
	In the non-rough case, the boundary layer problems are described by linear ODEs. Mainly, we have
	\begin{equation}\label{eq_BGV:35}
	-\partial_X\Psi_{e}-(1+\alpha_e^{2})^{2}\partial_X^{4}\Psi_{e}=0,
	\end{equation} 
	and
	\begin{equation*}
	\partial_X\Psi_{w}-(1+\alpha_w^{2})^{2}\partial_X^{4}\Psi_{w}=0.
	\end{equation*}
	Notice that there is an asymmetry between the coasts (see \cite{Desjardins1999,Pedlosky2013}). Indeed, 
	only one boundary condition can be lifted on the east boundary since there is only one root with non-negative real part, whereas the space of admissible (localized) boundary corrections is of dimension two on the western boundary. Consequently, $\Psi_e$ must vanish at first order on the East coast, leaving the solution on the boundary layer at $\Gamma_e$ to correct the trace of $\partial_n\Psi_e$. This phenomenon is still present in the rough case.
\end{remark}

Since the eastern cannot bear a large boundary layer, see Remark \ref{remark_4.1}, it is frequent in the literature to choose $\Psi_{int}^0$ tangent to the boundary $\Sigma_e$, see for example \cite{Desjardins1999,Bresch2005}. Hence, we take
\begin{equation}\label{form_int0}
\Psi_{int}(t,x,y)=\left\{\begin{array}{ccl}
-\int_{x}^{\chi_e(y)}\mathrm{curl}\,\tau(t,x',y)dx'&\textrm{in}&\Omega,\\
0&\textrm{in}&\Omega^{\varepsilon}\setminus\Omega.\end{array}\right.
\end{equation}
and consequently, at order $\varepsilon^{-4}$, the eastern boundary layer profile is $\Psi_{e}^0 \equiv 0$. At the West, we have the following system 
\begin{subeqnarray}
\label{western_bl_0}
    Q_w(\Psi_w^0,\Psi_w^0)+\partial_{X_w}\Psi_w^0-\Delta_w^{2}\Psi_w^0&=&0,\quad\textrm{in}\quad\omega_w^{+} \cup \omega_w^{-}\slabel{western_bl_0:0}\\
\left[\Psi_w^0\right]\big|_{\sigma_w}&=&-\left[\Psi_{int}^0\right]\big|_{\Sigma_w},\slabel{western_bl_0:1}\\ \left[\partial_{X_w}^{k}\Psi_w^0\right]\big|_{\sigma_w}&=&0,\quad k=1,\ldots,3,\\
\Psi_w^0\big|_{X_w=-\gamma_w(Y)}=0,&& \dfrac{\partial\Psi_w^0}{\partial n_w}\big|_{X_w=-\gamma_w(Y)}=0,\\
\Psi_w^0\longrightarrow 0&&\textrm{when}\quad X_w\rightarrow +\infty,
\end{subeqnarray}

Henceforth, the  jump condition\eqref{western_bl_0} is described by a function $\phi$ defined as follows
\begin{equation}\label{new_conds}
    \phi=\int_{\chi_w(y)}^{\chi_e(y)}\mathrm{curl}\,\tau(t,x',y)dx',
\end{equation}
which is a direct result of  \eqref{form_int0}.

It remains to prove the well-posedness of the nonlinear problem \eqref{western_bl_0}. Since it is quite technical, we address the matter later in Section \ref{western_nl}. This step concludes the computations in the main order.

Now, let us compute the next step in the asymptotic expansion. Similarly to the first interior profile, $\Psi^1_{int}$ follows the equation
\begin{equation}
    \partial_x\Psi^{1}_{int}=0,
\end{equation}
hence, $\Psi^{1}_{int}(t,x,y)=C^1(t,y)$. The lack of source term is related to the factor $\varepsilon^{-3}$ multiplying $\partial_x\Psi$ in \eqref{Bresch-GV_3}. Accordingly, the equation driving the behavior of the interior profile becomes nonhomogeneous when $i\geq 3$. 

At order $\varepsilon^{-3}$, the eastern boundary layer function is described by the equations
\begin{equation}\label{eastern_bl_1}
\begin{split}
-\partial_{X_e}\Psi_{e}^1-\Delta_e^{2}\Psi_{e}^1&=0,\quad\textrm{in}\quad\omega_e^{+} \cup \omega_e^{-},\\
\left[\Psi_{e}^1\right]\big|_{\sigma_e}&=-\left[\Psi_{int}^1\right]\big|_{\Sigma_e},\\ \left[\partial_{X_e}^{k}\Psi_{e}^1\right]\big|_{\sigma_e}&=0,\quad k=1,\ldots,3,\\
\Psi_{e}^1\big|_{X_e=-\gamma_e(Y)}=0,&\quad \dfrac{\partial\Psi_{e}^1}{\partial n_e}\big|_{X_e=-\gamma_e(Y)}=0,\\
\Psi_e^1\longrightarrow 0&\quad\textrm{when}\quad X_e\rightarrow +\infty.
\end{split}
\end{equation}

The space of admissible boundary corrections at the rough eastern domain remains insufficient to satisfy simultaneously the boundary conditions and the one at infinity. Beyond imposing conditions on $\Psi^1_{int}$, ergodicity assumptions will be needed to guarantee the existence of a solution $\Psi^1_e$ of \eqref{eastern_bl_1}. This question is the main focus of Section \ref{eastern}.

In the western boundary layer domain, $\Psi^1_w$ satisfies the following system
\begin{equation}
\begin{split}
\label{western_bl_1}
    Q_w(\Psi_w^0,\Psi_w^1)+Q_w(\Psi_w^1,\Psi_w^0)+\partial_{X_w}\Psi_w^1-\Delta_w^{2}\Psi_w^1&=F^1,\quad\textrm{in}\quad\omega_w^{+} \cup \omega_w^{-},\\
\left[\Psi_w^1\right]\big|_{\sigma_w}&=-\left[\Psi_{int}^1\right]\big|_{\Sigma_w}-\left[\Psi_{e}^1\right]\big|_{\sigma_w},\\
\left[\partial_{X_w}\Psi_w^1\right]\big|_{\sigma_w}&=-\left[\partial_x\Psi_{int}^1\right]\big|_{\Sigma_w},\\
\left[\partial_{X_w}^{k}\Psi_w^1\right]\big|_{\sigma_w}&=0,\quad k=2,3,\\
\Psi_w^1\big|_{X_w=-\gamma_w(Y)}=0,&\quad \dfrac{\partial\Psi_w^1}{\partial n_w}\big|_{X_w=-\gamma_w(Y)}=0,\\
\Psi_w^1\longrightarrow 0&\quad\textrm{when}\quad X_w\rightarrow +\infty,
\end{split}
\end{equation}
where $F^1=-(\nabla\Psi^0_{int}\cdot\nabla_w)\Delta\Psi^0_w$. The existence of a solution of problem \eqref{western_bl_1} can be shown by following  the same reasoning used for $\Psi^0_w$ (see Section \ref{s:linearized}). 

In the next section, we provide a formal method of proof of well-posedness for the problems previously mentioned with its core ideas and some general computations.

     \section[Existence and uniqueness of the solutions in a rough domain: methodology]{Existence and uniqueness of the solution of an elliptical problem in a rough domain: methodology}\label{methodo}
In hopes of facilitating the comprehension of this work, we describe a general method to prove the existence and uniqueness of the solution of a general problem encompassing all the possible behaviors within the boundary layer; in particular, the nonlinear and linearized western boundary layer systems and the linear eastern boundary layer equations. For $\alpha\in\mathbb{R}$, we start with elliptic differential systems of the form
\begin{equation}\label{eq:2.1}
\begin{split}
\mathcal{L}_\alpha(\Psi) +\mathcal{Q}_{\alpha}(\Psi, \tilde{\Psi})&=F,
\end{split}
\end{equation}
where $F$ is a regular enough source term with sufficient decay at infinity, $\mathcal{L}_\alpha$ is a fourth order elliptic linear differential operator and $\mathcal{Q}_{\alpha}$ is the nonlinear/quasilinear part of the equation. Let us consider for a fixed $\alpha\in\mathbb{R}$
\begin{eqnarray*}
\mathcal{L}_\alpha(\Psi)&=&\pm\partial_X\Psi-\Delta_\alpha^2\Psi\\
\mathcal{Q}_{\alpha}(\bar{\Psi},\Psi )&=&\frac{j}{2}\left((\nabla_\alpha^\perp\bar{\Psi}\cdot\nabla_\alpha)\Delta_\alpha\Psi+(\nabla_\alpha^\perp\Psi\cdot\nabla_\alpha)\Delta_\alpha\bar{\Psi}\right)
\end{eqnarray*}
for $j\in\{0,1,2\}$, and
\begin{eqnarray*}
\nabla_\alpha&=&(\pm\partial_{X},\partial_Y\mp\alpha\partial_{X} ),\quad\nabla^\perp_\alpha=(-\partial_Y\pm\alpha\partial_{X},\pm\partial_{X} )\\
\Delta_\alpha&=&\nabla_\alpha\cdot\nabla_\alpha=\partial_{X}^2+(\partial_Y\mp\alpha\partial_X)^2,\quad\Delta^2_\alpha=\Delta_\alpha\Delta_\alpha=\left(\partial_{X}^2+(\partial_Y\mp\alpha\partial_X)^2\right)^2.
\end{eqnarray*}
In the definition of $\mathcal{L}_{\alpha}$, the factor multiplying $\partial_X$ is linked to the definition of the local variables provided in Section \ref{s:asymptotic}: positive for the western boundary layer and negative in the eastern boundary layer domain. Moreover, $\alpha$ corresponds to the derivative of the function describing the interface between the interior and rough domain (namely $\chi'(y)$), and, therefore, different on each side. 

Let us suppose that equation \eqref{eq:2.1} holds in a domain $\omega=\omega^{+}\cup\sigma\cup\omega^{-}$, where
\begin{eqnarray*}
	\omega^{+}&=&\left\{X>0,\quad Y\in\mathbb{R}\right\},\qquad\sigma=\left\{X=0,\quad Y\in\mathbb{R}\right\}\\
	\omega^{-}&=&\left\{-\gamma(Y)<X<0,\quad Y\in\mathbb{R}\right\},
\end{eqnarray*}
and $\gamma$ is a positive Lipschitz function such that $\inf\gamma>0$. Problem \eqref{eq:2.1} is supplemented with the following jump and boundary conditions:
\begin{equation}\label{eq:2.2}
\begin{split}
\left[\partial_X^k\Psi\right]\big|_{\sigma}&=g_k,\quad k=0,\ldots,3,\\
\Psi\big|_{X=-\gamma(Y)}&=\dfrac{\partial\Psi}{\partial n}\big|_{X=-\gamma(Y)}=0.
\end{split}
\end{equation} 
Here, $n$ denotes the unit outward normal vector of $\gamma$ and $g_k$ are smooth functions.

Let us point out some difficulties related to the proof of existence and uniqueness of the solution of problem \eqref{eq:2.1}-\eqref{eq:2.2}. First, we consider a domain $\omega$ that is not bounded in the tangential direction. Moreover, functions $g_k$ do not decay as $Y$ goes to infinity, so that standard energy estimates are inefficient. As a consequence, only locally integrable functions are considered, which leads to a completely different treatment of the energy estimates. 

If the problem was set in $\omega^-$, with Dirichlet boundary conditions at $\{X=0\}$, one could build a solution $\Psi$ adapting ideas of Lady\v{z}enskaya and Solonnikov for the case of Navier-Stokes flows in tubes \cite{Ladyzenskaja1980}. The existence of the solution in \cite{Ladyzenskaja1980} is proven using an a priori differential inequality on local energies. Unfortunately, this method relies heavily on the bounded direction hypotheses to make possible the application of the Poincaré inequality. Hence, this reasoning is not applicable in our setting.

Moreover, contrary to what happens for the Laplace equation, one cannot rely on maximum principles to get an $L^{\infty}$ bound since we are dealing with a fourth-degree operator.

This problem has been overcome in the literature for the  Stokes boundary layer flow in \cite{Gerard-Varet2010} and, recently, for highly rotating fluids in \cite{Dalibard2017}. The main idea is to impose a so-called transparent boundary condition when the variable in the normal direction is equal to a certain value $M>0$, see Figure \ref{fig2}. 

\begin{figure}[ht!]
	\centering
	\includegraphics[width=0.7\linewidth]{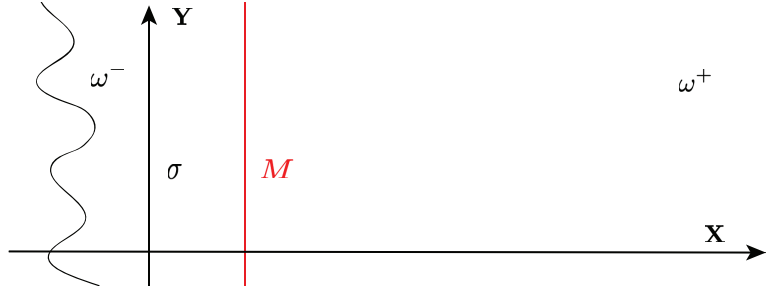}
	\caption{Boundary layer domain $\omega$ with an artificial boundary at $X=M$.}
	\label{fig2}
\end{figure}

This transparent condition separates the original domain in two: a half-plane $\{X>M\}$ and a bumped region bounded in the tangential direction. The Dirichlet problem on the half-space $\{X>M\}$ is solved by means of Fourier analysis and pseudo-differential tools in Kato spaces. The problem on the bumped sub-domain is not suitable for a similar treatment due to the nature of boundary and the fact that Kato spaces are defined through truncations in space\footnote{Similar difficulties arise in \cite{Alazard2016} when studying water waves equations in locally uniform spaces.}. Nevertheless, it is now suitable for the application of the Lady\v{z}enskaya and Solonnikov method \cite{Ladyzenskaja1980}. The remaining step consists of connecting both solutions on the artificial boundary. 

\subsection{Linear case.}\label{l_method}

If $j= 0$, the main steps of the proof are as follows:
\begin{itemize}
	\item [(L1)]	Prove existence and uniqueness of a solution of the linear system in a half-space with boundary data in $H^{3/2}(\mathbb{R})\times H^{1/2}(\mathbb{R})$.
	\begin{equation}\label{pb:halfspace_l}
	\left\{\begin{array}{rcl}
	\mathcal{L}_\alpha(\Psi)&=&F,\quad\textrm{in}\;X>M,\\
	\Psi\big|_{\sigma^M}=\psi_0,&&\partial_X\Psi\big|_{\sigma^M}=\psi_1,
	\end{array}\right.
	\end{equation}
	where $\sigma^M=\left\{X=M,\quad Y\in\mathbb{R}\right\}$. The solution is constructed by means of an integral representation using Fourier analysis. Indeed, we take the Fourier transform with respect to the tangential variable $Y$ and do a thorough analysis of the characteristic equation of the resulting problem. If $F\neq 0$, we compute the fundamental solution using the Green function.    
	\item [(L2)]	Extend this well-posedness result to boundary data in $(\psi_0,\psi_1)\in H^{3/2}_{\mathrm{uloc}}(\mathbb{R})\times H^{1/2}_{\mathrm{uloc}}(\mathbb{R})$ using ideas in \cite{Dalibard2017}. A priori estimates on a solution  of (\ref{pb:halfspace_l}) are established in this scenario. 
	\item [(L3)]	Define the Poincaré-Steklov type operator for functions in $H^{3/2}_{\mathrm{uloc}}(\mathbb{R})\times H^{1/2}_{\mathrm{uloc}}(\mathbb{R})$ using the information recovered from the problem in the half-space $\{X>M\}$ 	and extend the result to the case when the boundary data belongs to a space of uniformly locally integrable functions. The Poincaré-Steklov operator associated to $\mathcal{L}_\alpha(\Psi)$ is a positive  non-local boundary differential operator of the form
\begin{eqnarray}\label{PS}
PS_\alpha:H^{3/2}_{\mathrm{uloc}}(\mathbb{R})\times H^{1/2}_{\mathrm{uloc}}(\mathbb{R})&\rightarrow& H^{-3/2}_{\mathrm{uloc}}(\mathbb{R})\times H^{-1/2}_{\mathrm{uloc}}(\mathbb{R})\\
\begin{pmatrix}
\psi_0\\\psi_1
\end{pmatrix}&\mapsto&\begin{pmatrix}
(1+\alpha^2)\Delta_\alpha\Psi\big|_{\sigma^M}\\
-\left[(1+\alpha^2)\partial_X\mp 2\alpha\partial_Y\right]\Delta_\alpha\Psi\pm\dfrac{\Psi}{2}\Big|_{\sigma^M}\\
\end{pmatrix}=\begin{pmatrix}
\mathcal{A}^{\alpha}_2\left(\psi_0,\psi_1,F\right)\\\mathcal{A}^{\alpha}_3\left(\psi_0,\psi_1,F\right)
\end{pmatrix}\nonumber
\end{eqnarray}
where the form of the differential operators $\mathcal{A}^\alpha_i$, $i=2,3,$ depend greatly on the solution determined in (L1) and is, therefore, particular to each case.
	\item [(L4)] Define an equivalent problem in a domain with transparent boundary condition $\omega^M=\omega\cup \{X=M\}$ and then, solve the problem
	\begin{equation}\label{eq_left_l}
	\begin{split}
	    \mathcal{L}_\alpha(\Psi^-)&=F,\quad\text{in}\:\:\omega^b\setminus\sigma^{M}, \\
	\left[\partial_X^k\Psi^-\right]\big|_{\sigma}&=g_k,\;k=0,\ldots,3,\\
	(1+\alpha^2)\Delta_\alpha\Psi^-\big|_{\sigma^M}&=\mathcal{A}^{\alpha}_2\left(\Psi^-\big|_{\sigma^M},\partial_X\Psi^-\big|_{\sigma^M},F\right),\\
	-\left[(1+\alpha^2)\partial_X\mp2\alpha\partial_Y\right]\Delta_\alpha\Psi^-\pm\dfrac{\Psi^-}{2}\Big|_{\sigma^M}&=\mathcal{A}^{\alpha}_3\left(\Psi^-\big|_{\sigma^M},\partial_X\Psi^-\big|_{\sigma^M},F\right),\\
	\Psi^{-}\big|_{X=-\gamma(Y)}&=\dfrac{\partial\Psi^-}{\partial n}\Big|_{X=-\gamma(Y)}=0,
	\end{split}
	\end{equation}
	where $\omega^{b}$ refers to the rough ``tubular" domain given by $\omega^{b}=\omega^M\setminus(\left\{X>M\right\}\times\mathbb{R})$ and $\mathcal{A}^\alpha_i$, $i=2,3$ are the ones in \eqref{PS}. Note that for $M=0$, $\omega^b=\omega^-$.
	
	\begin{proposition}
		Let $\gamma \in W^{2,\infty}(\mathbb{R})$ and $g_k \in L^\infty(\mathbb{R})$, for $k=1,\dots,3$. Assume that there exists $\delta_0>0$ such that $\|g_k\|_{\infty}<\delta_0$, for all $k=0,\dots,3$. 
		\begin{itemize}
			\item Let $\Psi$ be a solution of \eqref{eq:2.1}-\eqref{eq:2.2} in $\omega$ such that $\Psi\in H^2_{\mathrm{uloc}}(\omega)$. Then $\Psi|_{\omega^M}$ is a solution of \eqref{eq_left_l}, and for $X > M$, $\Psi$ solves \eqref{pb:halfspace_l}, with $\psi_0 := \Psi|_{X=M} \in H^{3/2}_{\mathrm{uloc}}(\mathbb{R})$ and $\psi_1 := \partial_X\Psi|_{X=M} \in H^{1/2}_{\mathrm{uloc}}(\mathbb{R})$.
			\item Conversely, let $\Psi^- \in H^2_{\mathrm{uloc}}(\omega^M)$ be a solution of \eqref{eq_left_l}. Consider the solution $\Psi^+ \in H^2_{\mathrm{uloc}}(\mathbb{R}^2_+)$ of \eqref{pb:halfspace_l}. Setting
			\begin{equation*}
				\Psi(X,\cdot):=\left\{\begin{array}{ccc}
					\Psi^-(X,\cdot)&\mathrm{for}&-\gamma(\cdot)<X<M , \\
					\Psi^+(X,\cdot)& \mathrm{for}&X>M,
				\end{array}\right.
			\end{equation*}
			the function $\Psi \in H^2_{\mathrm{loc}}(\omega)$ is a solution of the problem \eqref{eq:2.1}-\eqref{eq:2.2}.
		\end{itemize}
	\end{proposition}
	\item [(L5)] Consequently, we focus our attention on the existence and uniqueness of solutions of the equivalent problem \eqref{eq_left_l}. To simplify the presentation, we replace in this paragraph the functions $A^\alpha_i\left(\Psi^-\big|_{\sigma^M},-\partial_X\Psi^-\big|_{\sigma^M},F\right)$ by $\rho_i\in H^{3/2-i}_{\mathrm{uloc}}(\mathbb{R})$, $i=2,3$. In fact, the Poincaré-Steklov operator is not local which hinders the application of the ideas in Lady\v{z}enskaya and Solonnikov \cite{Ladyzenskaja1980}, as seen in \cite{Gerard-Varet2010} and \cite{Dalibard2014}. We will address in this difficulty in Section \ref{s:rough_linear_west}. Showing the operators $\mathcal{A}^\alpha_i$ are well-defined depends on the Fourier representation of the solutions in the half-space, and consequently, on the definition of $\nabla_\alpha$ and the domain. We leave the detailed discussion of each case for later. System \eqref{eq_left_l} becomes
	\begin{equation}\label{method_bumped_problem}
	\begin{split}
		\pm\partial_{X}\Psi^--\Delta_{\alpha}^{2}\Psi^-&=F,\;\;\text{in}\;\;\omega^b\setminus\sigma^{M},\\[8pt]
		\left[\partial_X^{k}\Psi^-\right]\big|_{\sigma}&=g_k,\;\;k=0,\ldots,3 ,\\
		(1+\alpha^2)\Delta_{\alpha}\Psi^-\Big|_{\sigma^M}&=\rho_2,\\
				-\left[(1+\alpha^2)\partial_X\mp2\alpha\partial_Y\right]\Delta_\alpha\Psi^-\mp\dfrac{\Psi^-}{2}\Big|_{\sigma^M}&=\rho_3,\\
		\Psi^-\big|_{X=-\gamma(Y)}&=\frac{\partial\Psi^-}{\partial n}\Big|_{X=-\gamma(Y)}=0.
		\end{split}
	\end{equation}
To facilitate the computations, we lift the conditions at the interface $\{X=0\}$ by introducing the function
\begin{equation}\label{lift_lin}
    \Psi^{L}(X,Y)\defeq\chi(X)\sum_{k=0}^3g_k(Y)\dfrac{X^k}{k!},
\end{equation}
where $\chi\in\mathcal{C}_0^{\infty}(\mathbb{R})$ such that $\chi\equiv 1$ near $\sigma$, and $\mathrm{Supp}\chi\subset [0,\frac{M}{2}]$. Thus, $\Psi^{L}\equiv 0$ in $\omega^{-}$ and in $\omega^{+}$ close to $X=M$. Additionally, it satisfies the jump conditions
\begin{equation*}
\left[\partial_X^k\Psi^{L}\right]\Big|_{\sigma}=g_k,\quad k=0,\ldots
,3.
\end{equation*}
For $\tilde{\Psi}=\Psi^{-}-\Psi^{L}$, we have
\begin{equation}
\begin{split}\label{method_lift_problem}
	\pm\partial_{X}\tilde{\Psi}-\Delta_{\alpha}^{2}\tilde{\Psi}&=F^L,\;\;\text{in}\;\;\omega^b\setminus\sigma^M,\\[8pt]
		(1+\alpha^2)\Delta_{\alpha}\tilde{\Psi}\Big|_{\sigma^M}&=\rho_2,\\
				-\left[(1+\alpha^2)\partial_X\mp2\alpha\partial_Y\right]\Delta_\alpha\tilde{\Psi}\mp\dfrac{\tilde{\Psi}}{2}\Big|_{\sigma^M}&=\rho_3,\\
		\tilde{\Psi}\big|_{X=-\gamma(Y)}&=\frac{\partial\tilde{\Psi}}{\partial n}\Big|_{X=-\gamma(Y)}=0,
\end{split}
\end{equation}
where the source term $F^L$ depends also on $g_k$, $k=0,\ldots,3$.

Since a priori estimates are needed, it is useful to write the weak formulation of (\ref{method_lift_problem}). 

\begin{definition}\label{p:weak_formulation_meth}
	Let $\mathcal{V}$ be the space of functions $\varphi\in H^2(\overline{\omega^b})$ such that $\tilde{\Psi}\big|_{\Gamma}=\partial_{\mathrm{n}}\tilde{\Psi}\big|_{\Gamma}=0$ and $\mathrm{Supp}\varphi$ is bounded. A function $\tilde{\Psi}\in H^2_{\mathrm{uloc}}(\omega^b)$ is a solution of \eqref{method_lift_problem}  if it satisfies the homogeneous conditions at the rough boundary, and if, for all $\varphi\in \mathcal{V}$,
	\begin{eqnarray}\label{weak_formulation_meth}
	&&\mp\int_{\omega^b}\partial_{X}\tilde{\Psi}\varphi-\int_{\omega^b}\Delta_\alpha\tilde{\Psi}\Delta_\alpha\varphi\\
	&&\hspace*{1in}=\int_{\omega^b}F^L\varphi-\int_\mathbb{R}\left(\rho_3\pm\dfrac{\tilde{\Psi}}{2}\right)\varphi\big|_{X=M}\;dY-\int_\mathbb{R}\rho_2\partial_{X}\varphi\big|_{X=M}\;dY.\nonumber
	\end{eqnarray}
\end{definition}

Throughout  this step, we will frequently be using the following technical lemma:
\begin{lemma}\label{lemme_tech}
	Let $U$ be a regular open set bounded at least in one direction. Then, for  $f\in H^2(U)$ there exists a constant $C>0$ such that
	\begin{equation}\label{lemme_tech_in}
	\|f\|_{H^2(U)}\leq C\left(\|f\|_{L^2(U)}+\|\Delta_\alpha f\|_{L^2(U)}\right).
	\end{equation} 
		If the function satisfies additionally that $f=\partial_{\mathrm{n}}f=0$ on some part of the boundary $\partial U$, the first term on the right-hand side of \eqref{lemme_tech_in} is not longer needed for the inequality to hold.
	\end{lemma}
We refer to Appendix \ref{a:lemm_tech} for a proof. Note that a direct result from Lemma \ref{lemme_tech} is that controlling the $L^2$-norm of $\Delta_{\alpha} f$  immediately provides $f\in H^2(U)$. This property is easily generalized to Kato spaces.
	
\begin{itemize}	\item\textit{Energy estimates for \eqref{method_lift_problem}}. We introduce, for all $n\in\mathbb{N}$, $k\in\mathbb{N}$
	\begin{equation}\label{dom_n}
	\begin{split}
		\omega_{n}&\defeq\omega^{b}\cap \left\{(X,Y)\in\mathbb{R}^{2}:\;|Y|<n\right\},\quad \omega_{k+1,k}=\omega_{k+1}\setminus\omega_{k},\\
		\sigma_{n}^{M}&\defeq\left\{X=M,\;\textrm{and}\; |Y|<n, \;Y\in\mathbb{R}\right\},\quad \sigma^M_{k+1,k}=\sigma_{k+1}^M\setminus\sigma_{k}^M,\\
		\Gamma_n&=\{X=-\gamma(Y),\:\:|Y|< n\}.
		\end{split}
	\end{equation}

We consider the system (\ref{method_lift_problem}) in $\omega_{n}$
\begin{eqnarray*}\label{method_lift_problem_n}
	\pm\partial_{X}\tilde{\Psi}_n-\Delta_{\alpha}^{2}\tilde{\Psi}_n&=&F^L\\[8pt]
	(1+\alpha^2)\Delta_{\alpha}\tilde{\Psi}_n\Big|_{\sigma^M}&=&\rho_2\\
	-\left[(1+\alpha^2)\partial_X\mp2\partial_Y\right]\Delta_\alpha\tilde{\Psi}_n\mp\dfrac{\tilde{\Psi}_n}{2}\bigg|_{\sigma^M}&=&\rho_3\\
	\tilde{\Psi}_n\big|_{\Gamma_n}&=&\partial_{\mathrm{n}}\tilde{\Psi}\big|_{\Gamma_n}=0.
\end{eqnarray*}

	In order to prove the existence of the solution of (\ref{eq_left_l}), we derive $H^{2}_{\mathrm{uloc}}$ estimates on $\Psi_n$, uniform with respect to $n$. Then, passing to the limit when $n\rightarrow +\infty$, we achieve our goal. Indeed, taking $\tilde{\Psi}_n$ as a test function in (\ref{weak_formulation_meth}), we obtain
	\begin{eqnarray}\label{Prange_326_m}
	\|\Delta_\alpha\tilde{\Psi}_n\|^{2}_{L^{2}(\omega_n)}&=&-\int_{\omega_{n}}F^L\tilde{\Psi}_{n}+\int_{\sigma_n^{M}}\rho_3\tilde{\Psi}_n+\int_{\sigma_n^{M}}\rho_2\partial_X\tilde{\Psi}_n\nonumber\\
	&\leq&C_1\sqrt{n}\left(\|F\|_{L^2_{\mathrm{uloc}(\omega^b)}}+\sum_{k=0}^3\|g_k\|_{L^\infty(\mathbb{R})}\right)\|\tilde{\Psi}\|_{H^2(\omega_n)}\\
	&&+\,C_2 \sqrt{n}\left(\|\rho_3\|_{L^{2}_{\mathrm{uloc}}(\mathbb{R})}\|\tilde{\Psi}_n\big|_{\sigma^M}\|_{L^{2}([-n,n])}\right.\\
	&&\left.+\|\rho_2\|_{L^{2}(\omega_n)}\|\partial_X\tilde{\Psi}_n\big|_{\sigma^M}\|_{L^{2}(([-n,n])}\right)\nonumber,
	\end{eqnarray}
using the Cauchy-Schwarz and Poincaré inequalities over $\omega_n$. Thus, 
\begin{equation}\label{C0_dep}
\begin{split}
\|\Delta_\alpha\tilde{\Psi}_n\|^{2}_{L^{2}(\omega_n)}&\leq \left(\delta+\|F\|^2_{L^2_{\mathrm{uloc}}(\omega^b)}\right)\left( \|\tilde{\Psi}_n|_{X=M}\|^2_{H^1([-n,n])}+\|\Psi\|^2_{L^2(\omega_n)}\right)\\
&+C_\delta n\left(\sum_{k=0}^3\|g_k\|^2_{L^\infty(\mathbb{R})}+\|\rho_2\|^2_{L^2_{\mathrm{uloc}}(\mathbb{R})}+\|\rho_3\|^2_{L^2_{\mathrm{uloc}}(\mathbb{R})}\right)
\end{split}
\end{equation}

 Notice that the first term on the r.h.s of the previous inequality can be absorbed by the one on the l.h.s for $\delta$ and $F$ small enough. Then, using Poincar\'e inequality over the whole channel yields
\begin{equation}\label{DP:2-37_m}
E_n:=\int_{\omega^{b}}|\Delta_\alpha\tilde{\Psi}_n|^{2}\leq\int_{\omega_n}|\Delta_\alpha\tilde{\Psi}_n|^{2}\leq C_0(g_0,\ldots,g_3,\rho_2,\rho_3)n.
\end{equation}
where constant $C_0$ depends on $\alpha$ and the size of the jumps and the values of the differential operators at the artificial boundary, as seen in \eqref{C0_dep}. 
The existence of $\tilde{\Psi}_n$ in $H^2(\omega_n)$ follows. 
	
Therefore, we resort to performing energy estimates on the system (\ref{method_lift_problem}), following the strategy of Gérard-Varet and
	Masmoudi \cite{Gerard-Varet2010}. The idea is to use the quantity
	\[
	E_k^{n}:=\int_{\omega_k}|\Delta_\alpha\tilde{\Psi}_n|^{2},\]
	to derive an induction inequality on $(E_k^{n})_{k\in\mathbb{N}}$, for all $n\in\mathbb{N}$.
	 Hence, we consider $\varphi=\chi_k\Psi$, where $\chi_k\in\mathcal{C}^{\infty}_0 (\mathbb{R})$ is a cut-off function in the tangential variable such that $\mathrm{Supp}\chi_k \subset [-k-1,k+1]$ and $\chi_n\equiv 1$ on $[-k,k]$ for $k\in\mathbb{N}$. Since the problem is defined in a two-dimensional domain, the support of $\nabla^{j}\chi_k$, $j=1,\ldots,4$, is included in the reunion of two intervals of size $1$.

	Let us explain the overall strategy. We shall first derive the following inequality for all $k\in\left\{1,\ldots, n\right\}$ 
	\begin{equation}\label{inequality1_m}
		E_k^{n}\leq C_1\left((E_{k+1}^{n}-E_k^{n})+\left(\sum_{k=0}^3\|g_k\|^2_{L^\infty(\mathbb{R})}+\|\rho_2\|^2_{H^{m-1/2}_{\mathrm{uloc}}(\mathbb{R})}+\|\rho_3\|^2_{H^{m-3/2}_{\mathrm{uloc}}(\mathbb{R})}\right)(k+1)
		\right).
		\end{equation}
Here,  $C_1$ is a constant depending only on the characteristics of the domain.

	Then, by backward induction on $k$, we deduce that
	\begin{equation*}
		E_k^{n}\leq Ck,\quad\forall k\in\{k_0,\ldots,n\},
	\end{equation*}
	where $k_0\in\mathbb{N}$ is a large, but fixed integer (independent of $n$)  and $E_{k_0}^{n}$ is bounded uniformly in $n$ for a constant $C$ depending on $\omega^{b}$, $g_k$, $k=0,\ldots,3$ and $F$. This provides the uniform boundness for a maximal energy of size $k_0$. Since the derivation of energy estimates is invariant by translation on the tangential variable, we claim that
	\begin{equation}
	\underset{a\in\mathcal{I}_{k_0}}{\sup}\int_{\{{(-1,M)}\times a\}\cap\;\omega^{b}}|\Delta_\alpha\tilde{\Psi}_n|^{2}\leq C.
	\end{equation}
	
	The set $\mathcal{I}_{k_0}$ contains all the intervals of length $2k_0$ in $[-n,n]$ with extremities in $\mathbb{Z}$.
	Consequently, the uniform $H^{2}_{\mathrm{uloc}}$ bound on $\Psi^{n}$ is proved and an exact solution can be found by compactness. Indeed, by a diagonal argument, we can
	extract a subsequence $(\tilde{\Psi}_{r(n)})_{n\in\mathbb{N}}$ such that
	\[\begin{array}{rclcl}
	\tilde{\Psi}_{r(n)}&\rightharpoonup&\Psi&\textrm{weakly in}&H^{2}(\omega_{k}),\\
	\tilde{\Psi}_{r(n)}\Big|_{\sigma^M}&\rightharpoonup&\Psi\Big|_{\sigma^M}&\textrm{weakly in}&H^{3/2}(\sigma^M_{k}),\\
	\partial_{X}\tilde{\Psi}_{r(n)}\Big|_{\sigma^M}&\rightharpoonup&\partial_{X}\Psi\Big|_{\sigma^M}&\textrm{weakly in}&H^{1/2}(\sigma^M_{k}),
	\end{array}\]
	for all $k\in\mathbb{N}$. Of course, $\tilde{\Psi}$ is a solution of \eqref{method_lift_problem}, and, consequently, $\Psi^-\in H^{2}_{\mathrm{uloc}}(\omega^b)$ is solution of system \eqref{eq_left_l}.
	
	To lighten notations in the subsequent proof, we shall denote $E_k$ instead of $E^{n}_k$.

	\item\textit{Deriving the inequality.}  This part contains the proof of (\ref{inequality1_m}). Taking $\chi_k\tilde{\Psi}$ as test function in \eqref{weak_formulation_meth} provides the following expression for the l.h.s.
	\begin{equation}\label{bt_1}
	\begin{split}
	    \mp\int_{\omega^b}\partial_X\tilde{\Psi}\chi_k\tilde{\Psi}-\int_{\omega^b}\Delta_\alpha\tilde{\Psi}\Delta_\alpha(\chi_k\tilde{\Psi})&=\pm\dfrac{1}{2}\int_{\mathbb{R}}\chi_k\Psi^2\big|_{X=M}-E_k\\
	    &-2\int_{\omega^b} \Delta_\alpha\tilde{\Psi}\nabla_\alpha\chi_k\cdot\nabla_\alpha\tilde{\Psi}-\int_{\omega^b}\tilde{\Psi}\Delta_\alpha\tilde{\Psi}\partial_Y^{2}\chi_k\\
	    &-\int_{\omega_{k+1}\setminus\omega_k}\chi_k|\Delta_\alpha\tilde{\Psi}|^2.
	\end{split}
	\end{equation}
	For the third term, we simply use the Cauchy–Schwarz and Poincar\'e inequalities:
	\begin{equation}\label{commut_bound}
	\left|\int_{\omega^b} \Delta_\alpha\tilde{\Psi}\nabla_\alpha\chi_k\cdot\nabla_\alpha\tilde{\Psi}\right|\leq C\left(\int_{\omega_{k+1,k}}|\Delta_\alpha\tilde{\Psi}|^{2}\right)^{1/2}\left(\int_{\omega_{k+1,k}}|\nabla_\alpha\tilde{\Psi}|^{2}\right)^{1/2}\leq C(E_{k+1}-E_k).
	\end{equation}
	In the same fashion, we find that $\int_{\omega^b}\tilde{\Psi}\Delta_\alpha\tilde{\Psi}\partial_Y^{2}\chi_k$ and $\int_{\omega_{k+1}\setminus\omega_k}\chi_k|\Delta_\alpha\tilde{\Psi}|^2$ are also bounded by $C(E_{k+1}-E_k)$.
	
	Gathering all boundary terms stemming from the biharmonic operator and the first term in the r.h.s. of \eqref{bt_1} yields
	\begin{equation*}
	-\int_{\mathbb{R}}\chi_k\left(\rho_3\tilde{\Psi}\big|_{X=M}+\rho_2\partial_X\tilde{\Psi}\Big|_{X=M}\right).
	\end{equation*}
	
	The term above is bounded by
	\[ C\left(\|\rho_2\|_{L^{2}_{\mathrm{uloc}}}^{2}+\|\rho_3\|_{L^{2}_{\mathrm{uloc}}}^{2}\right)\left(E_{k+1}+(k+1)\right)
	\]
	where $C$ depends only on $M$, $\alpha$ and on $\|\gamma\|_{W^{2,\infty}}$. The computation of this bound relies on the trace theorem and Young's inequality. 
	
	We are left with
	\begin{eqnarray*}
		\left|\int_{\omega^b}\chi_k F^L\tilde{\Psi}\right|
		&\leq&C\left(\sum_{k=0}^3\|g_k\|_{L^\infty(\mathbb{R})}\right)E_{k+1}^{1/2}\sqrt{k+1}+\|F\|_{L^2_{\mathrm{uloc}(\omega^b)}}E_{k+1}^{1/2},\\
		&\leq&C_{\nu}\left(\sum_{k=0}^3\|g_k\|^2_{L^\infty(\mathbb{R})}\right)(k+1)+(\nu+\|F\|^2_{L^2_{\mathrm{uloc}}(\omega^b)})E_{k+1}.
	\end{eqnarray*}
	Lastly, combining all the estimates and taking $\nu$ and $\|F\|^2_{L^2_{\mathrm{uloc}}(\omega^b)}$ small enough give
	\begin{eqnarray*}
		E_k&\leq& C_1\left((E_{k+1}-E_k)+C_2(k+1)\right),
	\end{eqnarray*}
where $C_1$ is a constant independent of $k$ and 
\begin{equation}
C_2:=C_2(g_0,\ldots,g_3,\rho_2,\rho_3)=\left(\sum_{k=0}^3\|g_k\|^2_{L^\infty(\mathbb{R})}+\|\rho_2\|_{H^{m-1/2}_{\mathrm{uloc}}}^{2}+\|\rho_3\|_{H^{m-3/2}_{\mathrm{uloc}}}^{2}\right).
\end{equation}

	\item \textit{Induction.}
	Our goal is to show from (\ref{inequality1_m}) that there exists $k_0\in\mathbb{N}\setminus\{0\}$, $C>0$ such that, for all $n\in\mathbb{N}$
	\begin{equation}\label{DP:3-2_m}
	\int_{\omega_{k_0}}|\Delta_\alpha\tilde{\Psi}_n|^{2}\leq C.
	\end{equation}
	From (\ref{inequality1_m}), we claim that induction on $n-k$ indicates there exists a positive constant $C_3$ depending only on $C_0$, $C_1$ and $C_2$ appearing respectively in (\ref{DP:2-37_m}) and (\ref{DP:3-2_m}), such that, for all $k> k_0$,
	\begin{equation}\label{DP:3-3_m}
	E_k\leq C_3C_2(k+1).
	\end{equation}
	Let us insist on the fact that $C_3$ is independent of $n$, $k$ and will be adjusted in the course of the induction argument.
	
	First, notice that thanks to (\ref{DP:2-37_m}), (\ref{DP:3-3_m}) is true for $k= n$ once $C_3> C_0C_2^{-1},$ recalling that $\tilde{\Psi}_n=0$ on $\omega^{b}\setminus\omega_n$. We then assume that (\ref{DP:3-3_m}) holds for $n, n-1,\ldots,k+1$, where $k$ is a positive integer.
	
	We prove (\ref{DP:3-3_m}) at rank $k$ by contradiction. Assume that (\ref{DP:3-3_m}) does not hold at the rank $k$. Then, the induction implies
	\[
	E_{k+1}-E_k<C_3C_2.
	\]
	Since $C_0, C_1> 0$ are fixed and depend on $\alpha$ and $\|\gamma\|_{W^{2,\infty}}$ (see (\ref{DP:2-37_m}) for the definition of $C_0$), substituting the above inequality in (\ref{inequality1_m}) yields
	\begin{eqnarray}\label{DP:3-6_m}
	C_3C_2(k+1)<&E_k&\leq C_1C_2C_3+C_1C_2(k+1).
	\end{eqnarray}
	 Taking $C_3\geq 2C_1$ and plugging it in (\ref{DP:3-6_m}) results in a contradiction for $k>k_0$, where $k_0=2C_1+1$. Therefore, (\ref{DP:3-3_m}) is true at the rank $k>k_0$. Moreover, since $E_k$ is an increasing functional with respect to the value of $k$, we obtain that $E_k$ is also bounded for $k\leq k_0$.
	It follows from \eqref{DP:3-3_m}, choosing $k=2$, that there exists a constant $C>0$, depending only on $C_0,C_1,C_2,C_3$, and therefore only on $\alpha$, $\|\gamma\|_{L^\infty(\mathbb{R})}$ and on the norms on $g_k$, $k=0,\ldots,3$ and $\rho_i$, $i=2,3$, such that,
	\begin{equation}\label{DP:3-9_m}
	E_{ k_0}\leq E_{k_0+1}\leq C(k_0+1).
	\end{equation}
	Let us now consider the set $\mathcal{I}_{k_0}$ of all segments contained in $\{(M,Y):\quad |Y|\leq n\}$ of length $2k_0$. As $\mathcal{I}_{k_0}$ is finite, there exists an interval $a$ in $\mathcal{I}_{k_0}$ which maximizes
	\begin{equation*}
	\left\{\|\tilde{\Psi}_n\|_{H^2(\omega_a)}:\quad a\in\mathcal{I}_{k_0}\right\},
	\end{equation*}
	where $\omega_a =\{x\in\omega^b:\quad Y\in a \}$. We then shift $\tilde{\Psi}_n$ in such a manner that $a$ is centered at $0$. We call $\bar{\Psi}_n$ the shifted function. It is still compactly supported, but in $\omega_{2n}$ instead of $\omega_n$:
	\begin{equation*}
	\int_{\omega_{2n}}|\Delta_{\alpha}\bar{\Psi}_n|^2=\int_{\omega_{n}}|\Delta_{\alpha}\tilde{\Psi}_n|^2\quad\textrm{and}\quad \int_{\omega_{k_0}}|\Delta_{\alpha}\bar{\Psi}_n|^2=\int_{\omega_{a}}|\Delta_{\alpha}\tilde{\Psi}_n|^2.
	\end{equation*}
	Analogously to $E_k$, we define $\bar{E}_k$. The arguments leading to the derivation of energy estimates are
	invariant by horizontal translation, and all constants depend only on the parameter $\alpha$ and the norms on $g_k$, $\rho_i$, $i=2,3$, $F$ and $\gamma$, so \eqref{DP:3-9_m} still holds when $E_k$ is replaced by $\bar{E}_k$. On the other hand, $\bar{E}_{k_0}$ maximizes $\|\bar{\Psi}_n\|^2_{H^2(\omega_a)}$ on the set of intervals of length $2k_0$. This estimate being uniform, we can take $k_0$ large enough and obtain
	\begin{equation*}
	\underset{a\in I_{k_0}}{\sup}\|\tilde{\Psi}_n\|_{H^2(((0,-1)\times a)\cap\omega^b)} <\infty,
	\end{equation*}
	which means that $\tilde{\Psi}_n$ is uniformly bounded in $H^2(\omega^b)$.
		\item \textit{Uniqueness.}  To establish uniqueness, we consider $\Psi=\Psi_1-\Psi_2$, where $\Psi_i$, $i=1,2$, are solutions of  the original problem. The goal is to show that the solution $\tilde{\Psi}$ of the following problem is identically zero. 
\begin{eqnarray}\label{pb:uniqueness_m}
\pm\partial_{X}\Psi-\Delta_\alpha^{2}\Psi&=&0\quad\textrm{in}\;\omega^b,\nonumber\\
\Psi\big|_{X=-\gamma(Y)}=\partial _{n}\Psi\big|_{X=-\gamma(Y)}&=&0,\nonumber\\
\end{eqnarray}
We proceed similarly as in the ``existence part'' by multiplying the equation in (\ref{pb:uniqueness_m}) by $\Psi_k=\chi_k\Psi$ and integrating over $\omega^{b}$. The resulting induction relation is
\begin{equation*}
E_k\leq C(E_{k+1}-E_k).
\end{equation*}	
Since $E_{k+1}-E_k$ is uniformly bounded in $k$, we obtain $E_k\leq C$ uniformly in $k$, meaning that the difference between two solutions belongs to $H^{2}$. Hence, we can multiply the  equation on $\Psi$ by $\Psi$ itself and integrate by parts, disregarding $\chi_k$. This leads to
\begin{equation*}
(1-C\delta_0)\int_{\omega^{b}}|\Delta_\alpha\Psi|^{2}\leq 0,
\end{equation*}
which provides the uniqueness result when $\delta_0 <C^{-1}$.
\end{itemize}
The values of $\rho_2$ and $\rho_3$ are later replaced by the corresponding non-local operators.
\end{itemize}

\subsection{Nonlinear/ linearized problem.}\label{nl_linearized_method}

If $j\in\{1,2\}$, we proceed as follows:
\begin{itemize}
	\item [(NL1)] The well-posedness of the system  on the half-space
	\begin{equation}\label{pb:halfspace_nl}
	\left\{\begin{array}{rcl}
	\mathcal{L}_\alpha(\Psi)+\mathcal{Q}_\alpha(\bar{\Psi},\Psi)&=&F,\quad\textrm{in}\;X>M,\\
	\Psi\big|_{\sigma^M}=\psi_0,\quad\partial_X\Psi\big|_{\sigma^M}&=&\psi_1,
	\end{array}\right.
	\end{equation}
	for small enough but non-decaying boundary data $\psi_0$ and $\psi_1$ and source term $F$ is shown by combining estimates of the linear problem for a certain source function $\tilde{F}$ (steps (L1) and (L2)) with a fixed point argument. 
	The problem is clearly nonlinear when $\bar{\Psi}=\Psi$. This corresponds to the case when $j=1$ and the solution is obtained under a smallness assumption by applying a fixed point theorem in a space of exponentially decaying functions.
	
	The linearized problem ($j=2$), the solution $\Psi$ is sought in a similar manner, with the particularity of only assuming $\bar{\Psi}$ is small enough. 
	\item [(NL2)] For any $(\rho_2,\rho_3)\in H^{-1/2}_{\mathrm{uloc}}(\mathbb{R})\times H^{-3/2}_{\mathrm{uloc}}(\mathbb{R})$ small enough, we introduce the function $\Psi^{-}$ satisfying the following problem in the rough domain $\omega^{b}=\omega\setminus(\left\{X>M\right\}\times\mathbb{R})$
	\begin{equation}\label{eq_left_nl}
	\left\{\begin{array}{rcl}
	\mathcal{L}_\alpha(\Psi^{-})+\mathcal{Q}_\alpha(\bar{\Psi}^{-},\Psi^{-})&=&F,\quad \textrm{in}\quad\omega^{b}\setminus\sigma^M \\
	\left[\partial_X^k\Psi^{-}\right]\big|_{\sigma}&=&g_k,\;k=0,\ldots,3,\\
	\Psi^{-}\big|_{X=-\gamma(Y)}&=&\dfrac{\partial\Psi^{-}}{\partial n}\big|_{X=-\gamma(Y)}=0,\\ 
	\mathcal{A}_i^-\left(\Psi^{-}\big|_{\sigma^M},\partial_X\Psi^{-}\big|_{\sigma^M}\right)&=&\rho_i,\;i=2,3.
	\end{array}\right.
	\end{equation}
	Here, $\mathcal{A}_2$ and $\mathcal{A}_3$ are the second and third-degree components of the Poincaré-Steklov operator, defined at the transparent boundary in the rough channel. The nonlinear/quasilinear nature of $\mathcal{A}_3$ depends on the choice of the function $\bar{\Psi}$ since it contains the boundary terms stemming from $\mathcal{Q}_\alpha$.
	
	The proof of existence and uniqueness of the solution of \eqref{eq_left_nl} follows the same ideas of (L5). The goal is to obtain uniform estimates on the quantity $E_n^k$ by means of backward induction and then apply it to a translated channel to get a uniform local bound. The first obvious difference resides naturally in the induction relation. Here, the inequality is
	\begin{equation}\label{inequality_nl}
	E_k^{n}\leq \tilde{C}_1(E_{k+1}^{n}-E_k^{n})^{3/2}+C_1(E_{k+1}^{n}-E_k^{n})+C_2(k+1),\quad \forall k\in\left\{1,\ldots, n\right\},
	\end{equation}
	where $\tilde{C}_1$ and $C_1$ are constants depending only on the domain, while $C_2$ is determined by the norms of $g_k$, $k=0,\ldots,3$ and $\rho_i$, $i=2,3$. Relation \eqref{inequality_nl} is obtained using a truncation over $\omega_k$ and energy estimates. The smallness assumption on the boundary data (resp. on $\bar{\Psi}$) is essential in the nonlinear (resp. linearized) case since it guarantees for the terms derived from $\mathcal{Q}_\alpha$ to be absorbed by the truncated energy on the r.h.s. In particular, for the linearized case, we have that $\tilde{C}_1=0$ in \eqref{inequality_nl}.  
	\item [(NL3)] Then, we will introduce the solution $\Psi^{+}$ of (\ref{pb:halfspace_nl}) with $\psi_0=\Psi^{-}\big|_{X=M}$ and\\ \hbox{$\psi_1=\partial_X\Psi^{-}\big|_{X=M}$} and connect the solutions $\Psi^-$ and $\Psi^+$  at the transparent boundary. The strategy is to apply the implicit function theorem to a certain map $$\mathcal{F}:=\mathcal{F}(g_0,\ldots,g_3, \rho_2,\rho_3),$$ to find a solution of $\mathcal{F}=0$ in a neighborhood of zero. To do so, we first prove that $\mathcal{F}$ is a $\mathcal{C}^{1}$ mapping in a neighborhood of the transparent boundary, which means, in turn, that higher regularity of the solution is needed. 
	
	\item [(NL4)] Once the regularity estimates have been computed, we define the mapping $\mathcal{F}=(\mathcal{F}_1,\mathcal{F}_2)$, where
	\begin{equation}
	\begin{split}
	&\mathcal{F}_1(g_0,\ldots,g_3,\rho_2, \rho_3)=\mathcal{A}_2\left(\Psi^{+}\big|_{X=M},\partial_X\Psi^{+}\big|_{X=M}, F\right) -\rho_2,\\
	&\mathcal{F}_2(g_0,\ldots,g_3,\rho_2, \rho_3)=\mathcal{A}_3\left(\Psi^{+}\big|_{X=M},\partial_X\Psi^{+},\big|_{X=M}, F\right)-\rho_3.
	\end{split}
	\end{equation} The point will be to establish that for small enough $g_k$, $k=0,\ldots,3$ the system $$\left\{\begin{array}{l}
	\mathcal{F}_1(g_0,\ldots,g_3,\rho_2, \rho_3)=0,\\\mathcal{F}_2(g_0,\ldots,g_3,\rho_2, \rho_3)=0,
	\end{array}\right.$$ has a unique solution , provided that $\mathcal{F}_i(0,\ldots,0)=0$, for $i=1,2$. This result will be obtained via the implicit function theorem. When verifying that $d\mathcal{F}(0,\ldots,0)$ is an isomorphism of $H^{m-1/2}(\mathbb{R})\times H^{m-3/2}(\mathbb{R})$, we need that the only solution of the linear problem  \eqref{eq:2.1}-\eqref{eq:2.2}, when $g_k\equiv 0$ for all $k$ is $\Psi=0$. This shows once again how intrinsically connected the linear and nonlinear/linearized problems are. 
	
	Therefore, the field $\Psi$ defined by $\Psi^{\pm}$ on  each side of the transparent condition will be a solution of (\ref{eq:2.1})-(\ref{eq:2.2}). The definitions of $\Psi^{+}$ and tensors $\mathcal{A}_2^\pm\left(\cdot,\cdot,\cdot\right)$ and $\mathcal{A}_3^\pm\left(\cdot,\cdot,\cdot\right)$ provide that $\left[\partial_{X}^{k}\Psi\right]\big|_{X=M}=0$, for $k=0,\ldots,3$.
\end{itemize}

This section is a blueprint for the proofs in the remainder of the paper, and we will refer to it profusely. Especially in the derivation of energy estimates, where only the terms different from the ones discussed above will be presented.
\section{Western boundary layer: the linear case}\label{western_l}
     This section is devoted to showing the well-posedness of the western boundary layer problems in a general regime. The western boundary layer plays a fundamental role in basin-scale wind-driven ocean circulation, and it has been long studied in several theoretical works, e.g., \cite{Stommel1948,Munk1950}.
In idealized ocean models with a flat bottom, this layer is required not only to balance
the interior Sverdrup transport to close the gyre circulation, but also to dissipate the vorticity imposed by
the wind-stress curl \cite{Yang2003}.

Note that while the boundary layer functions depend on $(t,y)$, these variables behave as parameters at a microscopic scale. On that account, they will be omitted from the boundary layer functions to lighten the notation when no confusion can arise.

We start by studying the linear problem
\begin{equation}\label{e:linear_app}
\begin{array}{rcl}
\partial_{X_w}\Psi_w-\Delta_w^{2}\Psi_w&=&0,\quad\textrm{in}\quad \omega_w^+\cup\omega^-\\
\left[\partial_{X_w}^k\Psi_w\right]|_{X_w=0}&=&g_k,\quad k=0,\ldots,3,\\
\Psi_w\big|_{X=-\gamma(Y)}=0,&&\dfrac{\partial\Psi_w}{\partial n_w}\big|_{X_w=-\gamma_w(Y)}=0,
\end{array}
\end{equation}
where $g_k\in L^{\infty}(\mathbb{R})$, for all $k=0,\ldots,3$.

\begin{theorem}\label{theorem:existence_linear_all}
	Let $\gamma_w$ be a positive $W^{2,\infty}(\mathbb{R})$ function and $\omega_w$ be defined as before. Let $g_k\in L^\infty(\mathbb{R})$, for all $k=0,\ldots,3$. Then, problem (\ref{e:linear_app}) has a unique solution $\Psi_w$ in $H^{2}_{\mathrm{uloc}}(\omega_w\setminus\sigma_w)$ and there exists positive constants $C,\delta > 0$ such that
	\begin{equation}\label{eq_BGV:16_linear}
	\|e^{\delta X_w}\Psi_w\|_{H^{2}_{\mathrm{uloc}}(\omega_w)} \leq C\sum\limits_{k=0}^3 \|g_k\|_{L^{\infty}(\mathbb{R})}.
	\end{equation}
\end{theorem}

The proof of well-posedness of Theorem \ref{theorem:existence_linear_all} relies on the formulation of an equivalent system in a domain where transparent boundary conditions have been added at $X_w=M$, $M>0$. We will be following the steps listed in Section \ref{l_method} for the linear case.

First, we show some preliminary results on a problem in the half-space. Then, we define the pseudo-differential operators of the Poincaré-Steklov type relating the solution in the half-space with the one in the rough domain at the ``transparent'' interface. Finally, we restrict ourselves to the domain $\omega^b_w=\omega_w\cap \{X\leq M\}$ and solve an equivalent problem following the Lady\v{z}enskaya and Solonnikov method \cite{Ladyzenskaja1980}.  

Throughout this section, we write $X$ instead of $X_w$ since the analysis is only focused on the western boundary layer; hence no confusion can arise.
     \subsection{The linear problem on the half-space}\label{s:linear_west}
The main focus of this section is the analysis of the system
\begin{equation}\label{p:linear_nonhomgeneous0}
\begin{split}
   	\partial_X\Psi_{w}-\Delta_w^{2}\Psi_{w}&=F,\quad\textrm{in}\:\:\mathbb{R}_+^2\\
	\Psi_{w}\big|_{X=0}=\psi_0,&\quad\partial_X\Psi_{w}\big|_{X=0}=\psi_1.
\end{split}
\end{equation}
Here, $F$ is a function decaying exponentially as $X_w$ goes to infinity, and we have considered $M=0$ to facilitate the computations. The problem with a source term $F$ is necessary for the subsequent study of the nonlinear problem describing the western boundary layer.

Note that if $\Psi_w$ is a solution of \eqref{p:linear_nonhomgeneous0},  $\Psi_w(X-M, Y)$ is solution of the problem defined on $\{X>M\}$ with $M>0$ as a consequence of the equation being invariant with respect to translations on $X$.
Functional spaces of  $\psi_0$ and $\psi_1$ are provided in the following theorem, which summarizes the main result of the section. 

\begin{theorem}\label{Theorem2_DGV2017}
	Let $m\in\mathbb{N}$ such that $m\gg 1$. Let $\psi_{0} \in H^{m+3/2}_{\mathrm{uloc}}(\mathbb{R})$ and $\psi_{1} \in H^{m+1/2}_{\mathrm{uloc}}(\mathbb{R})$. Let $F$ be such that $e^{\bar{\delta} X} F \in H^{m-2}_{\mathrm{uloc}}(\mathbb{R}_+^2 )$, for $\bar{\delta}\in\mathbb{R}_+^*$. Then, there exists a unique solution $\Psi_w$ of system (\ref{p:linear_nonhomgeneous0}) satisfying 
	\begin{equation}\label{DGV2017_31}
	\|e^{\delta X}\Psi_w\|_{H^{m+2}_{\mathrm{uloc}}(\mathbb{R}_+^2 ))}\leq C\left(\|\psi_{0} \|_{H^{m+3/2}_{\mathrm{uloc}}(\mathbb{R}))}+\|\psi_{1} \|_{H^{m+1/2}_{\mathrm{uloc}}(\mathbb{R}))}+\|e^{\bar{\delta} X}F\|_{H^{m-2}_{\mathrm{uloc}}(\mathbb{R}_+^2 ))}\right),
	\end{equation}
	for a constant $C$ depending on $\alpha$, $\delta<\bar{\delta}$.
	\end{theorem}

Note that uniqueness consists of showing that if $F = 0$, $\psi_{0} = 0$ and $\psi_1 = 0$, the only solution $\Psi_w$ of (\ref{p:linear_nonhomgeneous0}) is $\Psi_w\equiv0$. The proof of the result is rather easy and will be sketched in paragraph \ref{ss:linear_case}. Consequently, the primary result will be the existence of a solution satisfying estimate (\ref{DGV2017_31}). 
Similarly to \cite{Dalibard2017}, the existence results can be obtained by compactness arguments.  

As the main equation is linear, we use a superposition principle to prove the desired result, meaning a solution of \eqref{p:linear_nonhomgeneous0} is sought of the form $$\Psi_{w}=\underline{\Psi}_{w}+\Psi_{w}^{F},$$ where $\underline{\Psi}_{w}$ is the solution of a homogeneous linear problem
\begin{equation}\label{pb:halfspace_linear_homogeneous}
\left\{\begin{array}{rcl}
\partial_X\underline{\Psi}_w-\Delta_w^{2}\underline{\Psi}_w&=&0,\quad\textrm{in}\quad\mathbb{R}_+^2 \\
\underline{\Psi}_w\big|_{X=0}=\psi^{*}_0,&&\partial_X\underline{\Psi}_w\big|_{X=0}=\psi^{*}_1,
\end{array}\right.
\end{equation}
while, the function $\Psi_w^{F}$ solves the equation
\begin{equation}\label{pb:halfspace_linear_nonhomogeneous}
\partial_X\Psi^{F}_w-\Delta_w^{2}\Psi^{F}_w=F,\quad\textrm{in}\quad\mathbb{R}_+^2 .
\end{equation}
Note that the boundary terms $\psi^{*}_0$ and $\psi^{*}_1$ are different from $\psi_0$ and $\psi_1$.  Indeed, it is convenient to construct the solution of \eqref{pb:halfspace_linear_nonhomogeneous} which does not satisfy homogeneous boundary conditions, and then lift the non-zero traces of $\Psi^F_w$ and $\partial_X\Psi^F_w$ thanks to $\underline{\Psi}_w$.

First, we apply Fourier analysis when looking for the solution of homogeneous problem (\ref{pb:halfspace_linear_homogeneous}) with boundary conditions $\psi_{0}^{*}=\psi_{0}-\Psi^{F}_w\big|_{X=0}$ and $\psi_{1}^{*}=\psi_{1}-\partial_X\Psi^{F}_w\big|_{X=0}$. Then, we tackle the sub-problem regarding function $\Psi_{w}^{F}$. In this case, we disregard temporarily about boundary conditions and focus on the equation $(\ref{pb:halfspace_linear_nonhomogeneous})$. Our goal in this step is to construct a solution by means of an integral representation involving the Green function.
\subsubsection{Homogeneous linear problem}\label{ss:linear_case}
In order to prove the existence and uniqueness of the solution of problem \eqref{pb:halfspace_linear_homogeneous} in $H^2_{\mathrm{uloc}}(\mathbb{R}_+^2)$, we first analyze the problem when the boundary data belongs to usual Sobolev spaces, see Proposition \ref{proposition:DP}, and then, extend the result to Kato spaces. 
\begin{proposition}\label{proposition:DP}
	Let $\psi^{*}_0 \in H^{m+3/2}(\mathbb{R})$ and $\psi^{*}_1 \in H^{m+1/2}(\mathbb{R})$. Then, the system
	\begin{equation}\label{linear_homogeneous}
\left\{\begin{array}{rcl}
\partial_X\underline{\Psi}_w-\Delta_w^{2}\underline{\Psi}_w&=&0,\quad\textrm{in}\quad\mathbb{R}_+^2 \\
\underline{\Psi}_w\big|_{X=0}=\psi^{*}_0,&&\partial_X\underline{\Psi}_w\big|_{X=0}=\psi^{*}_1,
\end{array}\right.
\end{equation}
admits a unique solution $\underline{\Psi}_{w}\in H^{m+2} (\mathbb{R}_+^2)$.
\end{proposition}

\begin{proof} 
	\textit{Existence.} Let us illustrate the proof for $m=0$. Given $\psi^{*}_0 \in H^{3/2}(\mathbb{R}_+^2)$ and $\psi^{*}_1 \in H^{1/2}(\mathbb{R}_+^2)$, we proceed with the construction of the fundamental solution by means of the Fourier transform. Applying the Fourier transform with respect to $Y$ results in the  following ODE problem
	\begin{equation}\label{eq_edo}
	\begin{split}
	    \partial_X\widehat{\underline{\Psi}}_w-(\partial_X^{2}+(-\alpha\partial_X+i\xi)^{2})^{2}\widehat{\underline{\Psi}}_w&=0,\quad\mathrm{in}\:\:\mathbb{R}_+^2\\
	\widehat{\underline{\Psi}}_w\big|_{X=0}=\widehat{\psi}^{*}_0,\quad\partial_X\widehat{\underline{\Psi}}_w\big|_{X=0}&=\widehat{\psi}^{*}_1,
	\end{split}
	\end{equation}
	where $\xi\in\mathbb{R}$ is the Fourier variable and $\widehat{\psi}^{*}_i$ is the Fourier transform of $\psi^{*}_i$, $i=0,1$. The corresponding characteristic polynomial is 
	\begin{equation}\label{Dalibard2014_eq24}
	P(\lambda)=-\lambda-(\lambda^{2}+(\alpha\lambda+i\xi)^{2})^{2}.
	\end{equation}
	
	We are now interested in identifying possible  degenerate cases using the relations between the coefficients and the roots of its characteristic equation.
	
	\begin{lemma}\label{p:roots}
		Let $P(\lambda)$ be the characteristic polynomial associated to the problem  \eqref{pb:halfspace_linear_homogeneous}. Then, $P(\lambda)$ has four distinct complex roots $\lambda_i^\pm$, $i=1,2$. Moreover, when $\xi\neq 0$, $\Re(\lambda_i^+)>0$ and $\Re(\lambda_i^-)<0$.
	\end{lemma}
	
	We refer the reader to Appendix \ref{s:roots} for a detailed proof of Lemma \ref{p:roots}. 
	
	The solutions of the problem resulting of applying the Fourier transform  are linear combinations of $\exp(-\lambda^+_k X)$ (with coefficients depending
	on $\xi$), where $(\lambda^{+}_k)_{k=1,2}$ are the complex-valued solutions of the characteristic polynomial satisfying $\Re(\lambda^{+}_k)>0$. There exist $A_k^{+}:\mathbb{R}\rightarrow\mathbb{C}$ such that
\begin{equation}\label{linear_fourier}
	\widehat{\underline{\Psi}}_w(X,\xi)=
	\sum_{k=1}^{2}A_k^{+}(\xi) \exp(-\lambda^{+}_k(\xi)X),\quad\textrm{in}\quad\mathbb{R}_+^2.
	\end{equation}
	Combining (\ref{linear_fourier}) with boundary conditions in \eqref{pb:halfspace_linear_homogeneous}, we have that coefficients $A_1^{+}$, $A_2^{+}$ satisfy
	\begin{equation}\label{inverse_et}
	\Lambda\left(\begin{array}{c}
	A_1^{+}(\xi)\\
	A_2^{+}(\xi)
	\end{array}\right)=\left(\begin{array}{c}
	\widehat{\psi}_{0}^{*}\\
-	\widehat{\psi}_{1}^{*}
	\end{array}\right),\quad\textrm{where}\quad
	\Lambda\defeq\left(\begin{array}{cc}
	1&1\\
	\lambda^{+}_1&\lambda^{+}_2
	\end{array}\right).
	\end{equation}

	Note that the coefficients $A_1^{+}$ and $A_2^{+}$ are well-defined since $\Lambda$ is invertible as a direct consequence of all the roots $\lambda^{+}_k$, $k=1,2$ of (\ref{Dalibard2014_eq24}) being simple.
	
It remains to check that the corresponding solution is sufficiently integrable, namely, $\|\underline{\Psi}_w\|_{H^{m+2}(\mathbb{R}_+^2)}<+\infty$. This
    assertion is equivalent to showing that for $0\leq k\leq2$
    \begin{equation}\label{relations}
\int_{\mathbb{R}^{+}\times\mathbb{R}}|\partial_{X}^{k}\widehat{\underline{\Psi}}_w|^{2}d\xi dX+|\xi|^{2k}|\widehat{\underline{\Psi}}_w|^{2}d\xi dX<+\infty.
    \end{equation}
    To that end, we need to investigate the behavior of $\lambda_k^{+}$, $A_k^{+}$ for $\xi$ close to zero and when $|\xi| \rightarrow \infty$. We gather the results in the following lemma, whose proof is postponed to Appendix \ref{ss:appendix_B}:
	\begin{lemma}\label{Dalibard2014_lemma24}
		\begin{itemize}
		\begin{itemize}
			\item As $\xi \rightarrow 0$, 
			\begin{eqnarray*}
				\lambda_{j}^{+}(\xi)&=&\bar{\lambda}_j^+O\left(|\xi|\right),\quad \Re(\bar{\lambda}^+_j)>0,\quad j=1,2,\\
				A_{j}^{+}(\xi)&=&a_{0,j}\widehat{\psi}_{0}^{*}+a_{1,j}\widehat{\psi}_{1}^{*}+O\left(|\widehat{\psi}^{*}_0||\xi|+|\widehat{\psi}^{*}_1||\xi|^2\right),
				\end{eqnarray*}
where $\bar{\lambda}_j^+$, $a_{0,j}$ and $a_{1,j}$ depend continuously on $\alpha$.
		
		\item[] As $|\xi|\rightarrow+\infty$, we have the following asymptotic behavior for $j=1,2$ when 
\begin{eqnarray}
\lambda_{j}^{+}(\xi)&=&p_1|\xi|+(-1)^{j}p_0|\xi|^{-\frac{1}{2}}+O( |\xi|^{-2})\nonumber\\
A_{j}^{+} (\xi)&=&(-1)^{j}m_2|\xi|^{3/2}+(-1)^{j}m_1|\xi|^{1/2}+O\left((|\widehat{\psi}^{*}_0|+|\widehat{\psi}^{*}_1|)|\xi|^{-1}\right),\nonumber
\end{eqnarray}
where $p_1=\zeta^{2}\mathbbm{1}_{\xi>0}+\bar{\zeta}^{2}\mathbbm{1}_{\xi<0}$, $p_0=\dfrac{1}{2}\left(i\zeta\mathbbm{1}_{\xi>0}+\bar{\zeta}\mathbbm{1}_{\xi<0}\right)$,  $m_1=\dfrac{\widehat{\psi}^{*}_1}{|\zeta|^{2}}\left(\zeta\mathbbm{1}_{\xi<0}-i\bar{\zeta}\mathbbm{1}_{\xi>0}\right)$ and $m_2=\widehat{\psi}^{*}_0\left(i\zeta\mathbbm{1}_{\xi>0}-\bar{\zeta}\mathbbm{1}_{\xi<0}\right)$. Here, $\zeta$ refers to the complex quantity $\sqrt{\frac{1-i\alpha}{\alpha^{2} +1}}$ satisfying $\Re(\zeta)>0$.
\end{itemize}
\end{itemize}
	\end{lemma}

Lemma \ref{Dalibard2014_lemma24} is used to show \eqref{relations} is in fact true for even larger values of $k$ if $\psi_0$ and $\psi_0$ are regular enough. The detailed proof can be found on Appendix \ref{a:regularity}.
  
\textit{Uniqueness.}  
	To show the uniqueness the solution, it is enough to solve (\ref{pb:halfspace_linear_homogeneous}) when $\psi_{0}^{*}=\psi_{1}^{*}=0$ and verify $\Psi_{w}\equiv 0$.  Applying the Fourier transform results in the following system  
	\[
	\Lambda\left(\begin{array}{c}
	A_1^{+}(\xi)\\
	A_2^{+}(\xi)
	\end{array}\right)=\left(\begin{array}{c}
	0\\
	0
	\end{array}\right),
	\]
	
	where $\Lambda$ is the invertible matrix previously defined. We conclude that $A_1 = A_2 = 0$, and thus $\widehat{\underline{\Psi}}(X, \xi) \equiv 0$. Since $\underline{\Psi}_w$ is an absolutely integrable function whose Fourier transform is identically equal to zero, then  $\underline{\Psi}_w= 0$.
\end{proof}

\subsubsection{Non homogeneous problem}
We begin the proof of the existence and uniqueness of the solution $\Psi^{F}_w$ of (\ref{pb:halfspace_linear_nonhomogeneous}) with the analysis of the equation
\begin{equation*}
\partial_X\Psi^{F}_w-\Delta_w^{2}\Psi^{F}_w=F,\quad\textrm{in}\quad\mathbb{R}_+^2.
\end{equation*} 
Our approach consists of constructing a particular solution of this equation, satisfying for some large enough $m$ an estimate where the norm of  $F$ controls the norm of the solution in $L^{\infty}({H^{m-2}_{\mathrm{uloc}}})$.
 We look for a integral solution of the form
\begin{equation}\label{solution_v}
\Psi^F_w(X, \cdot)=\int_{0}^{+\infty}G(X-X', D)F(X', \cdot)dX',
\end{equation}
where $G$ is the Green function verifying the equation
\begin{eqnarray}\label{eq_G}
 \partial_{X}G-\Delta_{w}^2 G&=&\delta_{0}.
\end{eqnarray}
Here, $\delta_{0}(\cdot)=\delta(\cdot-X')$ denotes the Dirac delta function.  In other words, $G$ is the fundamental solution over $\mathbb{R}^2$ of the Fourier multiplier $L(X,\xi)$ for any $\xi\in\mathbb{R}$ and satisfies
\begin{equation*}
L(X,\xi)G(X,\xi)=\delta_{0}(X).
\end{equation*}
Away from $X = 0$, $G( X, \xi)$ satisfies the homogeneous equation (\ref{pb:halfspace_linear_homogeneous}), see Section \ref{ss:linear_case}. 

For $X\neq 0$, $G(X,\xi)$ is a linear combination of $e^{-\lambda^{\pm}_iX}$, where $\lambda^{\pm}_i(\xi)$, $i=1,2$ are continuous functions of $\xi$ and roots of the polynomial \eqref{Dalibard2014_eq24}. We define $G$  as follows
\begin{equation}\label{G_def}
G=\left\{
\begin{array}{rcl}
&\sum_{i=1}^{2}B_i^{-}e^{-\lambda_i^{-}X},&\textrm{in}\quad X< 0\\
&\sum_{i=1}^{2}B_i^{+}e^{-\lambda_i^{+}X},&\textrm{in}\quad X> 0,
\end{array}\right.\quad B_i^{\pm}:\mathbb{R}\rightarrow\mathbb{C},\medspace k=1,2.
\end{equation}
Note that if $G$ is considered discontinuous at $X = 0$, with the discontinuity modeled by a step function, then, $\partial_X G \propto \delta_{0}(X)$ and consequently, $\partial_{X}^{k}G \propto \delta^{(k)}_{0}(X)$, $k=2,3,4$. However, \eqref{eq_G} does not involve generalized
functions beyond $\delta_{0}(X)$, and contains no derivatives of $\delta$-functions. Thus,
we conclude that $G(X,X')$ must be continuous throughout the domain and in particular at $X = 0$.
$G(X,\cdot)$ is a $\mathcal{C}^{2}$ function and $\partial^{3}G/\partial X^{3}$  has a finite jump discontinuity of magnitude $-1/(1+\alpha^{2})^{2}$ at $X=X'$. More precisely, $G$ satisfies 
$$[\partial_X^{k}G]|_{X=0} =0,\medspace k=0,1,2,\quad\textrm{and}\quad [\partial_X^{3}G]|_{X=0} =-\dfrac{1}{(1+\alpha^{2})^{2}}.$$
Substituting \eqref{G_def} in the above interface conditions provides a linear system on the coefficients $B_i^{\pm}$. 
\begin{equation}\label{system_Bi}
\left\{
\begin{array}{rcl}
B_1^{+}+B_2^{+}-\left(B_1^{-}+B_2^{-}\right)&=&0\\
\lambda_1^{+}B_1^{+}+\lambda_2^{+}B_2^{+}-\left(\lambda_1^{-}B_1^{-}+\lambda_2^{-}B_2^{-}\right)&=&0\\
(\lambda_1^{+})^{2}B_1^{+}+(\lambda_2^{+})^{2}B_2^{+}-\left((\lambda_1^{-})^{2}B_1^{-}+(\lambda_2^{-})^{2}B_2^{-}\right)&=&0\\
(\lambda_1^{+})^{3}B_1^{+}+(\lambda_2^{+})^{3}B_2^{+}-\left((\lambda_1^{-})^{3}B_1^{-}+(\lambda_2^{-})^{3}B_2^{-}\right)&=&\dfrac{1}{(1+\alpha^{2})^{2}}.
\end{array}
\right.
\end{equation}

The determinant of the Vandermonde matrix associated to \eqref{system_Bi} is of the form
\begin{equation*}
\bar{D}=(\lambda_1^{+} - \lambda_2^{+}) (\lambda_1^{+} - \lambda_1^{-}) (\lambda_2^{+} - \lambda_1^{-}) (\lambda_1^{+} - \lambda_2^{-}) (\lambda_2^{+} - \lambda_2^{-}) (\lambda_1^{-} - \lambda_2^{-}).
\end{equation*}
Since all the $\lambda_i^{\pm}$ are distinct (see Lemma \ref{p:roots}), $\bar{D}$ is non-zero and system \eqref{system_Bi} has a unique set of solutions
\begin{eqnarray}\label{Bi_form}
B_1^{+}&=-&\frac{1}{\left(\alpha ^2+1\right)^2 \left(\lambda_1^{+}-\lambda_2^{+}\right) \left(\lambda_1^{+}-\lambda_1^-\right) \left(\lambda_1^{+}-\lambda_2^{-}\right)},\nonumber\\
B_2^{+}&=&\frac{1}{\left(\alpha ^2+1\right)^2 \left(\lambda_1^{+}-\lambda _2^{+}\right) \left(\lambda_2^{+}-\lambda_1^{-}\right)\left(\lambda_2^{+}-\lambda_2^{-}\right)},\label{coefficients_green}\\
B_1^{-}&=&\frac{1}{\left(\alpha ^2+1\right)^2 \left(\lambda_1^{+}-\lambda_1^{-}\right) \left(\lambda _2^{+}-\lambda_1^{-}\right) \left(\lambda_1^{-}-\lambda_2^{-}\right)},\nonumber\\
B_2^{-}&=&-\frac{1}{\left(\alpha ^2+1\right)^2 \left(\lambda_1^{+}-\lambda_2^{-}\right) \left(\lambda _2^{+}-\lambda _2^{-}\right) \left(\lambda _2^{-}-\lambda_1^{-}\right)}.\nonumber
\end{eqnarray}

This establishes the function $G$. Its asymptotic behavior as $|\xi| \rightarrow 0$ and $|\xi | \rightarrow \infty$ is summarized in the following lemma. 

\begin{lemma}\label{lemma2}
	We have
	$$G=\left\{
	\begin{array}{r}\sum_{j=1}^{2}B_j^{-}e^{-\lambda_j^{-}X},\quad X<0 \\
\;\sum_{j=1}^{2}B_j^{+}e^{-\lambda_j^{+}X},\quad X>0
	\end{array}
	\right.,$$
	where the coefficients $B_j^{\pm}$ are given by (\ref{coefficients_green}) and $\lambda^+_j$ are the ones in Lemma \ref{Dalibard2014_lemma24}. Here, as $|\xi|\rightarrow 0$
	\begin{eqnarray*}
	\lambda_{1}^-&=-|\xi|^4+O\left(|\xi|^5\right),\quad \lambda_{2}^-&=-\dfrac{1}{\left(\alpha ^2+1\right)^{2/3}}+O\left(|\xi|\right).
	\end{eqnarray*}
	If $|\xi|\rightarrow \infty$,
	\begin{eqnarray*}
	\lambda_{j}^{-}(\xi)&=&p_1|\xi|+(-1)^{j}p_0|\xi|^{-\frac{1}{2}}+O( |\xi|^{-2}),
	\end{eqnarray*}
	where $p_i$, $i=0,1$, are defined as in Lemma \ref{Dalibard2014_lemma24} and $\zeta'=\sqrt{\frac{-1-i\alpha}{\alpha^{2} +1}}$.
	
	Asymptotic behavior:
	\begin{itemize}
		\item As $|\xi| \rightarrow 0$, we have that $B_j^\pm\rightarrow \bar{B}_j\in\mathbb{\mathbb{R}}$ independent of $\alpha$.
		
		\item For $|\xi|\gg 1$, we have $B_j^\pm=O(|\xi|^{-3/2})$. 
	\end{itemize}
\end{lemma}

We refer to Appendix \ref{a:Green_asymptotic} for a proof.

We  proceed to rigorously prove that the field
\begin{equation}\label{DGV_314}
\Psi^F_G(X,\cdot)=\int_{0}^{+\infty}G(X-X',D)F(X',\cdot)dX'=\int_{0}^{+\infty}\mathcal{F}^{-1}_{\xi\rightarrow Y'}\left(G(X-X',\cdot)\mathcal{F}_{Y'\rightarrow\xi} F(X',\cdot)\right)dX'
\end{equation}
is well-defined and satisfies $(\ref{pb:halfspace_linear_nonhomogeneous})$. This is made more precise in the following lemma:
\begin{lemma}\label{lemma_3}
	Let $F$ be smooth and compactly supported. The formula (\ref{solution_v}) defines
	a solution $\Psi^F_w$ of (\ref{pb:halfspace_linear_nonhomogeneous}) in $ L^{\infty}_{loc}(\mathbb{R}_+,H^{m+2}(\mathbb{R}))$ for any $m\geq  0$.
\end{lemma}

\begin{proof} 
 From the hypotheses, we can assert that $\widehat{F} = \widehat{F}(X',\xi)$ is in the Schwartz class with respect to $\xi\neq 0$, smooth and compactly supported in $X'$. We have additionally that $G(X-X', \xi)$ is smooth in $\xi\neq 0$, and continuous in $X, X'$. As a result, the function $J_{X,X'}: \xi\longrightarrow G(X-X', \xi)\widehat{F}(X',\xi)$ belongs to $L^{2}((1+|\xi|^2)^{m/2+1}d\xi)\times L^{2}((1+|\xi|^2)^{m/2+1}d\xi)$ for any $X,X'\geq 0$ and is smooth in $\xi$ and  continuous in $X,X'$. 
At high frequencies, the functional satisfies
	\[
|J_{X,X'}(\xi)|\leq C|\xi|^{N}|\widehat{F}|,
\]
where $N=-3/2$. This value comes from computing the $B_i$ using Lemma \ref{lemma2}. We can conclude that it belongs to $L^{2}$ since $\widehat{F}$ and its $X'$-derivatives are rapidly decreasing in $\xi$ by definition of Schwartz class.	
	
Furthermore,  when $|\xi|\ll1 $, using once again the bounds derived in Lemma \ref{lemma2},
	\[	G(X-X',\xi)\widehat{F}(X',\xi)= O(1).\]
	
	Thus, $\Psi^F_G$ defines a continuous function of $X$ with values in $H^{m+2}(\mathbb{R}^{2}_+)$ for all $m\geq 0$. Moreover, the smoothness of $F$ implies $\Psi^F_G$ is smooth in $X$ with values in the same space.  It will be a solution of the problem due to classical results in the construction of Green functions. 
	
	Finally, as its Fourier 
	transform is a linear combination of $e^{-\lambda_i(\xi)X}B_i(\xi)$, it also satisfies the linear problem treated in section \ref{ss:linear_case}.\end{proof}

\subsubsection{Bounds in Kato spaces}\label{ss:inhomogeneous}
In this section we establish that $\Psi_w$ is controlled by the norms of $\psi_0,\psi_1$ and $F$ in $L^{\infty}\left(H^{m-2}_{\mathrm{uloc}}\right)$ for a large enough $m$. The proof follows \cite{Dalibard2017}.  

Now we need to derive a representation formula for $\Psi_w$ when $\psi_0\in H^{3/2}_{\mathrm{uloc}}(\mathbb{R})$ and $\psi_1\in H^{1/2}_{\mathrm{uloc}}(\mathbb{R})$ by using its Fourier transform. The critical point  is to understand the action of the operators on $L^2_{\mathrm{uloc}}$ functions.

Due to the form of the solution, the end goal will be to establish  that for any $X > 0$, the kernel type $K^{\pm}_{i}=\mathcal{F}^{-1}_{\xi\rightarrow Y}\left( B_i^{\pm}(\xi)e^{-\lambda_i^{+}(\xi)X}\right),$ $i=1,2,$ defines an element of $L^{1}(\mathbb{R})$. The advantage of proving the latter is that $\Psi^{\pm}_i = K^{\pm}_i ( \cdot , z) * \underline{\psi}_i$ will then be (at least) an
$L^{1}_{\mathrm{uloc}}$ function.

\begin{lemma}\label{Dalibard2017_lemma7}
Let $\underline{\psi}$ be $L^1_{\mathrm{uloc}}(\mathbb{R})$. We define $\Psi^{\pm}_i (X,Y)$ by
	\begin{equation}\label{Dalibard2017_319}
	\Psi^{\pm}_{i}(X,\cdot)\defeq\chi(D)P(D)e^{-\lambda^{\pm}_i(D)Z}\underline{\psi},\quad \textrm{for}\quad\medspace i=1,2,
	\end{equation}
	where $\chi=\chi(\xi)\in C_c^{\infty}(\mathbb{R})$ and $P=P(\xi)\in C^{\infty}(\mathbb{R})$ is a is a homogeneous polynomial of degree $k\geq 0$ in the same vicinity. Then, there exists $C$ and $\delta > 0$ independent of $\underline{\psi}$ such that
	\[
	\begin{split}
	\forall Z\geq 0,&\quad\|e^{\delta Z}\Psi_{1}^+\|_{L^{\infty}(\mathbb{R}^2_+)} +\|e^{\delta Z}\Psi_{2}^+\|_{L^{\infty}(\mathbb{R}^2_+)}\leq  C\|\underline{\psi}\|_{L^1_{\mathrm{uloc}}},\\
\forall Z\leq 0,&\quad
	\|(1+|Z|)^{k/4}\Psi_{1}^-\|_{L^{\infty}(\mathbb{R}^2_+)} +\|e^{\delta |Z|}\Psi_{2}^-\|_{L^{\infty}(\mathbb{R}^2_+)}\leq  C\|\underline{\psi}\|_{L^1_{\mathrm{uloc}}}.
	\end{split}
	\]
\end{lemma}
This lemma follows the same idea as \cite[Lemma 7]{Dalibard2017}.  For the reader's convenience, we repeat the main ideas of the proof, thus making our exposition self-contained.

\begin{proof}
	We introduce a partition of unity 
	$(\varphi_q)_{q\in\mathbb{Z}}$ with $\varphi_q\in C^\infty_0(\mathbb{R})$, where $\mathrm{Supp}\thinspace\varphi_q \subset B(q,2)$ for $q \in\mathbb{Z}$ and $\sup_q \|\varphi_q\|_{W^{k,\infty}} < +\infty$ for all $k$. We also introduce
	functions $\tilde{\varphi}_q \in C_0^{\infty}(\mathbb{R})$ such that $\tilde{\varphi}_q \equiv1$ on $\mathrm{Supp}\thinspace\varphi_q$, and, say $\mathrm{Supp}\thinspace\tilde{\varphi}_q \subset B(q,3)$. Then, for $i=1,2$
	\begin{eqnarray}\label{Dalibard2017_A1}
	\Psi^{\pm}_i(Z,Y)&=&  \sum_{q\in\mathbb{Z}}\chi(D)P(D)(\varphi_q\underline{\psi})e^{-\lambda^{\pm}_i(D)Z}\\
	&=&\sum_{q\in\mathbb{Z}}\int_{\mathbb{R}}K^{\pm}_i (Z,Y-Y')\underline{\psi}(Y')\varphi_q(Y')dY'\\
	&=&
	\sum_{q\in\mathbb{Z}}\int_{\mathbb{R}}K_{i,q}^{\pm}(Z,Y,Y')\varphi_{q}(Y')\underline{\psi}(Y')dY',\nonumber
	\end{eqnarray}
	where
	\[
	K^{\pm}_{i}(Z,Y)=\int_{\mathbb{R}}e^{iY\cdot\xi}\chi(\xi)P(\xi)e^{-\lambda^{\pm}_i(\xi)Z}d\xi,\quad K_{i,q}^{\pm}(Z,Y,Y')=K^{\pm}_{i}(Z,Y-Y')\tilde{\varphi}_q(Y').
	\]
	We show  that the following  estimate holds:
	\begin{lemma}\label{lemma7}
		There exists $\delta > 0$ and for all $n\in\mathbb{N}$ a constant $C_n \geq 0$ such that for all $Y \in\mathbb{R}$, $Z > 0$ and $i=1,2$
		\begin{equation}\label{Dalibard2017_A2}
		|K^+_{i}(Z,Y)|\leq  C_n\frac{e^{-\delta Z}}{(1+|Y|)^{n}},\quad \textrm{for}\quad\medspace i=1,2.
		\end{equation}
		 When  $Z < 0$, we have for all $Y \in\mathbb{R}$, for all $n\in\mathbb{N}$
		\begin{equation}\label{Dalibard2017_A3}
		|K_{1}^-(Z,Y)|\leq C_n\dfrac{|Z|^{(n-1-k)/4}}{|Z|^{n/4}+|Y|^n},\quad|K^-_{2}(Z,Y)|\leq  C_n\frac{e^{-\delta |Z|}}{(1+|Y|^{n})}.
		\end{equation}
	\end{lemma}
We finish the proof of the current lemma and then show the result in Lemma \ref{lemma7}.

Combining (\ref{Dalibard2017_A1}) and \eqref{Dalibard2017_A2} when $i=1,2$ and $Z>0$ yields
\begin{eqnarray}
|\Psi^{+}_{i} (Z,Y)|
&\leq  &Ce^{-\delta Z}\left(\sum_{q\in\mathbb{Z},\;|q-Y|\geq 4}\int|\varphi_q(Y')\underline{\psi}(Y')|dY'\dfrac{1}{|Y-q|^2-3}\right.\\
&&\hspace{1in}\left.+\sum_{q\in\mathbb{Z},\;|q-Y|< 4}\int|\varphi_q(Y')\underline{\psi}(Y')|dY'\right)\nonumber\\
&\leq&  Ce^{-\delta Z}\|\underline{\psi}\|_{L^1_{\mathrm{uloc}}}\nonumber.
\end{eqnarray}
The latter also applies to $\Psi^-_2$. Furthermore, for $|Z|>1$, we have
 \begin{eqnarray}
 |\Psi^-_{1} (Z,Y)|
 &\leq  &C\left(\sum_{q\in\mathbb{Z},\;|q-Y|\geq 4}\int|\varphi_q(Y')\underline{\psi}(Y')|dY'\dfrac{|Z|^{\tfrac{1-k}{4}}}{|Z|^{\tfrac{1}{2}}+|Y-q|^2-3}\right.\\
 &&\qquad\left.+\sum_{q\in\mathbb{Z},\;|q-Y|< 4}\int|\varphi_q(Y')\underline{\psi}(Y')|dY'\dfrac{1}{|Z|^{(1+k)/4}} \right)\nonumber\\
 &\leq  &C\left(|Z|^{1/4}+1\right) |Z|^{-(1+k)/4}\|\underline{\psi}\|_{L^1_{\mathrm{uloc}}}\leq C\left(|Z|+1\right)^{-k/4}\|\underline{\psi}\|_{L^1_{\mathrm{uloc}}}.\nonumber
 \end{eqnarray}
When $|Z|\leq 1$, a similar reasoning yields $|\Psi^-_{1} (Z,Y)|\leq C\|\underline{\psi}\|_{L^1_{\mathrm{uloc}}}$.

Estimate in Lemma \ref{Dalibard2017_lemma7} follows. 
\end{proof}

\begin{proof}[Proof of Lemma \ref{lemma7}]
	With the notation introduce in Lemma \ref{Dalibard2017_lemma7}, we follow the ideas in \cite{Dalibard2017} to obtain the estimates.
	Let us consider $Z>0$. Since $\lambda^{+}_i$, $i=1,2$ are continuous and have non-vanishing real part on the support of $\chi$, there exists a constant $\delta > 0$ such that $\Re(\lambda^{+}_i) \geq \delta$ for all $\xi \in \mathrm{Supp}\thinspace\chi$ and for $i=1,2$. When $|Y|\leq 1$, we obtain
	\begin{equation*}
	|K^+_i(Z,Y)| \leq e^{-\delta Z}\|\chi P\|_{L^1} .
	\end{equation*}
However, this estimate is not enough for greater values of $Y$. Let us define $\chi_i(Z,\xi)=\chi(\xi)\exp(-\lambda_{i}^{+}Z)$ which is  an $L^{\infty}(\mathbb{R}_+, \mathcal{S}(\mathbb{R}))$ function. It follows that for all $n_1, n_2\in\mathbb{N}$ 
	\begin{equation*}
	|\partial_{\xi}^{n_2}\chi_i(Z,\xi)|\leq C_{n_1,n_2}\dfrac{e^{-\delta Z}}{(1+|\xi|^2)^{n_1}}.
	\end{equation*}
	
Integrating by parts with respect to the frequency variable yields
	\begin{equation}\label{ipp_kernel}
	\begin{split}
	Y^n K^+_i(Z,Y)&=
	\int_{\mathbb{R}}e^{iY\cdot\xi}D_{\xi}^n\left[\chi_i(\xi)P(\xi)e^{-\lambda_i^+(\xi)Z}\right]d\xi\\
	&\leq e^{\delta Z}	\int_{\mathbb{R}}e^{iY\cdot\xi}\sum_{m=0}^n\binom{n}{m-2}\partial_{\xi}^{n-m}\chi_i \partial_{\xi}^{m-2}(P (\xi)e^{-\lambda_i^+(\xi)Z})d\xi.
	\end{split}
	\end{equation}
	Note that for any $k\in\{0,...,n\}$, $\partial_{\xi}^{k}P (\xi)$ remains an homogeneous polynomial. Thus, expression \eqref{ipp_kernel} is bounded by a linear combination of integrals of the form
	\begin{equation*}
	\int_{\mathbb{R}}e^{iY\cdot\xi}e^{-\delta_n Z}\mathbbm{1}_{\{|\xi|\leq C\}}d\xi,
	\end{equation*}
	and \eqref{Dalibard2017_A2} follows. The estimate of $K^-_2$ stems from the same ideas.

We proceed to compute a useful estimate on $K^-_1$. When $|Z|\leq 1$, we have $$|K^-_1(Z,Y)|\leq\|\exp(-\delta|\xi|^4)\chi(\xi) P(\xi)\|_{L^1}<+\infty,$$ for all $Y\in\mathbb{R}$ and $k>-1$. Let us now consider $|Z|\geq 1$. By introducing the change of variables $\xi'=|Z|^{1/4}\xi$ and $Y'=|Z|^{-1/4}Y$, $K$ can be rewritten as
\begin{equation*}
K^-_1(Z,Y') =\dfrac{1}{|Z|^{(1+k)/4}}\int_{\mathbb{R}}e^{iY'\cdot\xi'}\chi\left(\frac{\xi'}{|Z|^{1/4}}\right)P(\xi' )e^{-\lambda^-_1\left(\frac{\xi'}{|Z|^{1/4}}\right)|Z|}d\xi'.
\end{equation*}
Since $\lambda_1^-/|\xi|^4\sim 1$ and does not vanish on the support of $\chi$, there exists a positive constant $\delta$ such that $-\lambda_1^-(\xi)\leq-\delta|\xi|^4$ on $\mathrm{Supp}\chi$. Therefore, for $|Y'|\leq 1$, it is easy to see that $$|K^-_1(Z,Y)|\leq |Z|^{-(1+k)/4}\|\exp(-\delta|\xi'|^4)P(\xi')\|_{L^1}.$$
Now, for $|Y'|\geq 1$, we perform integration by parts and obtain for any $n\in\mathbb{N}$,
\begin{align*}
Y'^nK^-_1(Z,Y)=\dfrac{1}{|Z|^{(1+k)/4}}\int_\mathbb{R}e^{iY'\xi'}D^n_{\xi'}&\left[\chi\left(\frac{\xi'}{|Z|^{1/4}}\right)P(\xi')\exp\left(-\lambda_1^-\left(\frac{\xi'}{|Z|^{1/4}})\right)|Z|\right)\right]d\xi'.
\end{align*}
The main issues arise when the derivative acts on the exponential. Note that
$\lambda_1^- (\xi )=|\xi |^4\Lambda_1^{-}(\xi)$, where $\Lambda_1\in C^\infty(\mathbb{R})$ and $\Lambda_1(0)=1$ therefore, for
all $\xi'\in\mathbb{R}$, $Z < 0$,
\begin{equation*}
\exp\left(-\lambda_1\left(\frac{\xi'}{|Z|^{1/4}}\right)|Z|\right)=\exp\left(-|\xi'|^4\Lambda_1^{-}\left(\frac{\xi'}{|Z|^{1/4}}\right)\right).
\end{equation*}
We infer that for $\xi\in\mathrm{Supp}\chi$,
\begin{equation*}
\left|P(\xi')\partial^n_{\xi'}\left(\chi\left(\frac{\xi'}{|Z|^{1/4}}\right)e^{-\lambda_1^{-}\left(\frac{\xi'}{|Z|^{1/4}}\right)|Z|}\right)\right|\leq P_{3n}(\xi')e^{-\delta|\xi'|^4},
\end{equation*}
where  $P_{3n}$ denotes a polynomial on $\xi'$ of degree $3n$. Hence, we have for all $n\in\mathbb{N}$
\begin{equation*}
|Y'^n K^-_1(Z,Y)|\leq C_n |Z|^{-(1+k)/4},
\end{equation*}
which provides in turn the following result for all $n\in\mathbb{N}$
\begin{equation}\label{estimate_n}
|K^-_1(Z,Y)|\leq C \dfrac{|Z|^{-(1+k)/4}}{1+|Y'|^n}= C\dfrac{|Z|^{(n-1-k)/4}}{|Z|^{n/4}+|Y|^n}.
\end{equation}
Taking $n=2$ in the previous inequality guarantees the convergence of integral controlling $\Psi^-_1$, completing the second estimate in the lemma.
\end{proof}

Exponential decay is obtained at high frequencies by using the following result:
\begin{lemma}\label{Dalibard2017_lemma9}
	Let $\chi\in C_c^{\infty}(\mathbb{R})$, with $\chi\equiv 1$ in a ball $B_r:=B(0,r)$, for $r>0$, and $P=P(\xi)\in C_b(B_r^c)$. For $\underline{\psi}\in L^2_{\mathrm{uloc}}(\mathbb{R})$ , we define $\Psi^{i}=\Psi^{i}(X,Y)$ by
	\begin{equation*}
	\Psi^\pm_{i}(X,Y)=(1-\chi(D))e^{-\lambda^\pm_i(D)Z}P(D)\underline{\psi}.
	\end{equation*}
	Then, for $Z\in\mathbb{R}$ and $\delta>0$ small enough,
	\begin{equation}\label{lemma_9_est}
	\|e^{\delta Z}\Psi^{1}\|_{L^{\infty}(\mathbb{R}\times\mathbb{R}^{+})}+\|e^{\delta Z}\Psi^{2}\|_{L^{\infty}(\mathbb{R}\times\mathbb{R}^{+})}\leq C\|\underline{\psi}\|_{H^{N}_{\mathrm{uloc}}(\mathbb{R})}.
	\end{equation} 
\end{lemma} 
The proof of the previous lemma follows almost exactly the one of  Lemma 9 in \cite{Dalibard2017}, and consequently, it is not repeated here. The authors in \cite{Dalibard2017} showed that for $n$ large enough, and any $|Z| \neq 0$,
\begin{equation*}
K_n(Z,Y) := \mathcal{F}^{-1}\left(1 + |\xi|^2)^{-n}(1 - \chi(\xi))P(\xi)e^{\mp\lambda_{i}^{\pm}(\xi)Z}\right)\in L^1(\mathbb{R}).
\end{equation*}
Consequently, $\Psi^i = K_n \star ((1 - \partial_Y^2)^n\underline{\psi})$ is at least an element of $L^2_\mathrm{uloc}$ when $\underline{\psi}\in H^N_{\mathrm{uloc}}(\mathbb{R})$, for $N \geq 2n$. The choice of $N$ is linked to the degree of polynomial $P(D)$, and thus, to the asymptotic behavior of the eigenvalues. 

\begin{proposition}\label{prop:halfspace_linear_nonhomogeneous}
	Let $l\geq 0$ and  $F$ a compactly supported function of $ H^{N}_\mathrm{uloc}(\mathbb{R}^2_+)$, for $N\geq\left\lceil\tfrac{2l-1}{4}\right\rceil$. Then, the solution $\Psi_w^F$ of \eqref{pb:halfspace_linear_nonhomogeneous} satisfies for $0<\delta<\bar{\delta}$ and ,  
\begin{equation}\label{control_F}
\|e^{\delta X}D^l_{X,Y}\Psi^{F}_w\|_{L^{\infty}}\leq C\|e^{\bar{\delta} X}F\|_{L^{\infty}(H^{N}_{\mathrm{uloc}})},
\end{equation}
where $D_{X,Y}$ is the differential operator with respect to the variables $X$ and $Y$.
	\end{proposition}
\begin{proof}
	We distinguish between high and low frequencies. We introduce some $\chi=\chi(\xi)\in C^{\infty}_c(\mathbb{R})$ equal to $1$ near $\xi=0$.
	
	Let $\Psi_w^{\flat}$ denote the integral expression when $\xi$ is in a vicinity of zero
	\begin{equation}\label{eq:ALGV3-20}
	\Psi_w^{\flat}=\int_{0}^{+\infty}I(X,X',\cdot)dX',\quad I(X,X',\cdot)=\chi(D)G(X-X',D)F(X',\cdot).
	\end{equation}
Here,
\begin{equation*}
\Psi_w^{\flat}=\int_{0}^X\chi(D)G(X-X',D)F(X',\cdot)dX'+ \int_{X}^{+\infty}\chi(D)G(X-X',D)F(X',\cdot)dX'.
\end{equation*}
 Let us first show \eqref{control_F}. Using Lemmas \ref{lemma2} and \ref{Dalibard2017_lemma7}, there exists a $\delta>0$
\begin{equation*}
\|G(X-X',D)F(X',\cdot)\|_{L^\infty(\mathbb{R})}\sim\left\{\begin{array}{lcl}
e^{-\delta (X-X')}\|F(X',\cdot)\|_{L^1_{\mathrm{uloc}}(\mathbb{R})},&\textrm{if}&X-X'>0\\
\|F(X',\cdot)\|_{L^1_{\mathrm{uloc}}(\mathbb{R})},&\textrm{if}&X-X'<0
\end{array}.\right.
\end{equation*}
Then, assuming $F$ decays exponentially at rate $\delta<\bar{\delta}$
\begin{equation*}
\begin{split}
\left\|\int_{0}^X\chi(D)G(X-X',D)F(X',\cdot)dX'\right\|_{L^\infty(\mathbb{R})}&\leq C\underset{Z\in \mathbb{R}_+}{\sup}\|e^{\bar{\delta }Z}F(Z,\cdot)\|_{L^1_{\mathrm{uloc}}(\mathbb{R})}\int_{0}^Xe^{- \delta(X-X')}e^{-\bar{\delta} X'}dX'\\
&\leq C e^{-\delta X}\underset{Z\in \mathbb{R}_+}{\sup}\|e^{\bar{\delta }Z}F(Z,\cdot)\|_{L^1_{\mathrm{uloc}}(\mathbb{R})}.
\end{split}
\end{equation*}
Moreover,
\begin{equation*}
\begin{split}
\left\|\int_{X}^{\infty}\chi(D)G(X-X',D)F(X',\cdot)dX'\right\|_{L^\infty(\mathbb{R})}&\leq C\underset{Z\in \mathbb{R}_+}{\sup}\|e^{\bar{\delta }Z}F(Z,\cdot)\|_{L^1_{\mathrm{uloc}}(\mathbb{R})}\int_{X}^{\infty}e^{-\bar\delta X'}dX'\\
&\leq Ce^{-\delta X}\underset{Z\in \mathbb{R}_+}{\sup}\|e^{\bar{\delta }Z}F(Z,\cdot)\|_{L^1_{\mathrm{uloc}}(\mathbb{R})}.
\end{split}
\end{equation*}

Consequently,
	\begin{equation}\label{low_freq}
	\|e^{\delta X}\Psi_w^{\flat}\|_{L^{\infty}(\mathbb{R}^2_+)}\leq C\|e^{\bar{\delta} X}F\|_{L^{\infty}(L^{1}_{\mathrm{uloc}}(\mathbb{R}^2_+))}.
	\end{equation}

For high frequencies, we define 
	\begin{equation}
\Psi_w^{\sharp}=\int_{0}^{+\infty}J(\cdot,X,X')dX',\quad J(X,X',\cdot)=(1-\chi(D))G(X-X',D)F(X',\cdot).
\end{equation}

On account of  Lemma \ref{Dalibard2017_lemma9}, we have that for $m\geq 1$
\[
\|J(X,X',\cdot)\|_{L^{\infty}(\mathbb{R})}\leq C e^{-\delta|X-X'|}\|e^{\bar{\delta }X'}F(X',\cdot)\|_{H^{N}_{\mathrm{uloc}}(\mathbb{R})}, \quad l=0,1,
\]
and, consequently,
\begin{equation}\label{high_freq}
\|\Psi_{w}^{\#}\|_{L^{\infty}(\mathbb{R}^2_+)}\leq C \int_0^{+\infty} e^{-\delta|X-X'|}dX'\|F\|_{L^{\infty}(L^{2}_{\mathrm{uloc}}(\mathbb{R}))}\leq C e^{-\delta X}\|e^{\bar{\delta }X'}F(X',\cdot)\|_{L^{\infty}(L^{2}_{\mathrm{uloc}}(\mathbb{R}^2_+))},
\end{equation}
which combined with (\ref{low_freq}) provides \eqref{control_F} for $l=0$.

It remains to show the result for the derivatives of $\Psi^F_w$.  At low frequencies, the coefficients associated to $e^{-\lambda_{i}^+X}$, $i=1,2$ and $e^{\lambda_{2}^-X}$ satisfy the same properties of exponential decay. The terms containing $e^{\lambda_{1}^-X}$ converge to a constant or decay to zero with polynomial weight. In particular, $D^{4k+l}_{X,Y}\Psi_w^\flat$ decays at a rate $O(X^{1-l/4})$, for $l=1,\cdots,3$. Hence, following the same reasoning as for $l=0$ we have
\begin{equation}\label{low_freq_high}
\|e^{\delta X}D^l_{X,Y}\Psi_w^{\flat}\|_{L^{\infty}(\mathbb{R}^2_+)}\leq C\|e^{\bar{\delta} X}F\|_{L^{\infty}(L^{1}_{\mathrm{uloc}}(\mathbb{R}^2_+))}.
\end{equation}
At high frequencies, applying the differential operator $D$ to $\Psi_w$  adds at most a $|\xi|$ factor. Consequently, the $l$-th derivative at high frequencies behaves like $|\xi|^{l-3/2}\exp(-|\xi||Z|)$. We have that $P(D)=|\xi|^{l-3/2-2n}\in C_c(\mathbb{R}\setminus B(0,r))$, $r>0$ if $l-3/2-2n<-1$. Lemma \ref{Dalibard2017_lemma9} gives for $N\geq 2\left\lceil \dfrac{2l-1}{4}\right\rceil$
\begin{equation}\label{high_freq_dev}
\|D^l_{X,Y}\Psi_{w}^{\#}\|_{L^{\infty}(\mathbb{R}^2_+)}\leq C \int_0^{+\infty} e^{-\delta|X-X'|}dX'\|F\|_{L^{\infty}(H^{N}_{\mathrm{uloc}}(\mathbb{R}))}\leq C e^{-\delta X}\|e^{\bar{\delta }X'}F(X',\cdot)\|_{L^{\infty}(H^{N}_{\mathrm{uloc}}(\mathbb{R}))}. 
\end{equation}
Finally, \eqref{control_F} results from gathering (\ref{low_freq_high}) and (\ref{high_freq_dev}).
\end{proof}

\subsubsection{Proof of Theorem \ref{Theorem2_DGV2017}}
In previous sections, we have constructed the solutions $\underline{\Psi}_w$ and $\Psi_w^{F}$ for the subproblems (\ref{pb:halfspace_linear_homogeneous}) and (\ref{pb:halfspace_linear_nonhomogeneous}), respectively. This paragraph deals with the connection to the solution of (\ref{p:linear_nonhomgeneous0}). 

The remarks following Theorem \ref{Theorem2_DGV2017} justify the existence of such a solution for smooth and compactly supported data, and it belongs to $H^{m+2}_{\mathrm{uloc}}$, for $m\gg 1$. We will now focus on retrieving estimate (\ref{DGV2017_31}).

Let us consider $\underline{\Psi}_w=\Psi_w-\Psi_{w}^{F}$. From Section \ref{s:linear_west}, we know the solution of the problem (\ref{pb:halfspace_linear_homogeneous}) will be well-defined for $\psi_{0}^{*}=\psi_{0}-\Psi_{w}^{F}\big|_{X=0}$ and $\psi_{1}^{*}=\psi_{1}-\partial_{X}\Psi_{w}^{F}\big|_{X=0}$ regular enough. Formal solutions of the homogeneous linear with zero source term and inhomogeneous Dirichlet data are given by the equation \eqref{linear_fourier}. Using Lemma \ref{Dalibard2014_lemma24} and Lemma \ref{lemma2}, we study the behavior of $\underline{\Psi}_w$ at low and high frequencies following the ideas of the previous section. 
\begin{lemma}
	Let $m_0\gg 1$. Then, there exists $\delta>0$ and $C>0$ such that the solution $\underline{\Psi}_w$ of \eqref{pb:halfspace_linear_homogeneous} satisfies the estimate
	\begin{equation*}
		\left\|e^{\delta X}\underline{\Psi}_w\right\|_{L^{\infty}(\mathbb{R}^2)}\leq C \left(\|\psi_{0}\|_{H^{m_0+3/2}_{\mathrm{uloc}}(\mathbb{R})}+\|\psi_{1}\|_{H^{m_0+1/2}_{\mathrm{uloc}}(\mathbb{R})}+\|e^{\bar{\delta} X} F\|_{H^{m_0-2}_{\mathrm{uloc}}(\mathbb{R}^2_+)}\right).
	\end{equation*} 
\end{lemma} 
\begin{proof}
	Here, we make use once again of the function $\chi \in C_c^{\infty}(\mathbb{R})$ equal to one in a vicinity of $\xi=0$ and zero elsewhere. At low frequencies, the asymptotic behavior in Lemmas \ref{Dalibard2014_lemma24} and \ref{lemma2} paired with Lemma \ref{Dalibard2017_lemma7} yield 
	\begin{equation}\label{low_linear}
	\|\chi(D)\underline{\Psi}_w\|_{L^\infty}\leq Ce^{-\delta X}\left(\|\psi_{0}\|_{L^{2}_{\mathrm{uloc}}(\mathbb{R})}+\|\psi_{1}\|_{L^{2}_{\mathrm{uloc}}(\mathbb{R})}+\|e^{\bar{\delta} X} F\|_{L^{2}_{\mathrm{uloc}}(\mathbb{R}^2_+)}\right).
	\end{equation}
	Computation of estimates of $\underline{\Psi}_w$ at high frequencies relies on Lemma \ref{lemma7}. From the asymptotic behavior listed in Lemma \ref{Dalibard2014_lemma24} and \ref{lemma2}, the coefficient multiplying $\psi^i$ behaves as $|\xi|^{3/2-i}e^{-|\xi||Z|}$, $i=0,1$. If $m_0> k>1$ and $\psi_i\in H^{3/2-i}_{\mathrm{uloc}}(\mathbb{R})$,  considering  $P(\xi)=|\xi|^{3/2-k-i}$ and $\underline{\psi}=(1-\Delta)^{k/2}\psi_i$ in Lemma \ref{Dalibard2017_lemma7} gives
	\begin{equation}\label{high_linear}
	\left\|(1-\chi(D))\underline{\Psi}_w\right\|_{L^{\infty}(\mathbb{R})}\leq C e^{-\delta X}\left(\|\psi_{0}\|_{H^{m_0+3/2}_{\mathrm{uloc}}(\mathbb{R})}+\|\psi_{1}\|_{H^{m_0+1/2}_{\mathrm{uloc}}(\mathbb{R})}+\|e^{\bar{\delta} X} F\|_{H^{m_0-2}_{\mathrm{uloc}}(\mathbb{R}^2_+)}\right).
	\end{equation}
	Combining (\ref{low_linear}) and (\ref{high_linear}), and in view of the estimate (\ref{control_F}) satisfied by $\Psi^{F}_w$, we have
	\begin{equation*}
		\left\|e^{\delta X}\Psi_w\right\|_{L^{\infty}}\leq C \left(\|\psi_{0}\|_{H^{m_0+3/2}_{\mathrm{uloc}}}+\|\psi_{1}\|_{H^{m_0+1/2}_{\mathrm{uloc}}}+\|e^{\bar{\delta} X} F\|_{H^{m_0-2}_{\mathrm{uloc}}}\right).
	\end{equation*}\end{proof}

	We are left with the task of determining the higher regularity bound (\ref{DGV2017_31}). Taking the derivatives directly on  \eqref{linear_fourier} and \eqref{DGV_314}, it is clear that by considering  larger values of $m_0$ 
	\begin{equation}
		\left\|e^{\delta X}\underline{\Psi}_w\right\|_{W^{2,\infty}}\leq C \left(\|\psi_{0}\|_{H^{m_0+3/2}_{\mathrm{uloc}}}+\|\psi_{1}\|_{H^{m_0+1/2}_{\mathrm{uloc}}}+\|e^{\bar{\delta} X} F\|_{H^{m0-2}_{\mathrm{uloc}}}\right).
		\end{equation}
	
Hence,  $\Psi_w=\underline{\Psi}_w+\Psi^{F}_w$ verifies
\begin{equation}\label{ineq_2}
	\left\|e^{\delta X}\Psi_w\right\|_{W^{2,\infty}}\leq C \left(\|\psi_{0}\|_{H^{m_0+3/2}_{\mathrm{uloc}}}+\|\psi_{1}\|_{H^{m_0+1/2}_{\mathrm{uloc}}}+\|e^{\bar{\delta} X} F\|_{H^{m0-2}_{\mathrm{uloc}}}\right).
	\end{equation}
	From this, it may be concluded that $\Psi_w\in H^2_{\mathrm{uloc}}(\omega^+)$. 
	
	Notice that for all $k\in\mathbb{Z}^2$ and all $m\geq 0$, there exists a constant $C_m>0$ such that 
	\begin{equation*}
\begin{split}
	\|\Psi_w\|_{H^{m+2}(B(k,1)\cap\omega_w^+)}&\leq C\left(\|\psi_{0}\|_{H^{m+3/2}_{\mathrm{uloc}}}+\|\psi_{1}\|_{H^{m+1/2}_{\mathrm{uloc}}}+\| F\|_{H^{m-2}(B(k,2)\cap\omega_w^+)}+\|\Psi_w\|_{H^2(B(k,2)\cap\omega_w^+)}\right).
\end{split}
	\end{equation*}
	This is a classical local elliptic regularity result, see \cite{Necasa}. Using the previous inequality and \eqref{kato_def} provides
	\begin{equation*}
	\begin{split}
	\|e^{\delta X}\Psi_w\|_{H^{m_0+2}_{\mathrm{uloc}}(\omega^+)}&=\underset{k\in\mathbb{Z}^2}{\sup}\|e^{\delta X}\Psi_w\|_{H^{m_0+2}(B(k+1)\cap \omega^+)}\\	&=C_{m_0}\underset{k\in\mathbb{Z}^2}{\sup}\left(\|\psi_{0}\|_{H^{m_0+3/2}_{\mathrm{uloc}}}+\|\psi_{1}\|_{H^{m_0+1/2}_{\mathrm{uloc}}}\right.\\
	&\hspace{0.9in}\left.+e^{\delta k}\| F\|_{H^{m_0-2}(B(k,2)\cap\omega_w^+)}+e^{\delta k}\|\Psi_w\|_{H^2(B(k,2)\cap\omega_w^+)}\right).
	\end{split}
	\end{equation*}
From $\underset{k\in\mathbb{Z}^2}{\sup}e^{\delta k}\| F\|_{H^{m_0-2}(B(k,2)\cap\omega_w^+)}\leq C\| e^{\delta k}F\|_{H^{m_0-2}_{\mathrm{uloc}}(\omega_w^+)}$ and \eqref{ineq_2}, we obtain \eqref{DGV2017_31} in Theorem \ref{Theorem2_DGV2017} for $\bar{\delta}>\delta.$
     \subsection{Differential operators at the transparent boundary}\label{transparent_w}
This paragraph is devoted to the well-posedness of the Poincaré-Steklov type operators defined at the boundary $X=M$. 

Providing explicit representations for the Poincaré-Steklov operator in terms of boundary data and the source term $F\neq 0$ is quite technical and exceeds the scope of this work. From now on, we are only interested in the case where $F=0$. Once again, without loss of generality, we assume $M=0$.

Using Proposition \ref{proposition:DP} and the variational formulation of problem \eqref{linear_homogeneous}, we have the following result:
\begin{definition}
	Let $\underline{\Psi}_w \in H^2(\mathbb{R}^2_+)$ be the unique weak solution of the Dirichlet problem \eqref{linear_homogeneous} for $(\psi_0^*,\psi_1^*)\in H^{3/2}(\mathbb{R})\times H^{1/2}(\mathbb{R})$. Then, the biharmonic matrix-valued Poincaré-Steklov operator is defined by
	\begin{equation}\label{def_PS}
	\begin{split}
	&PS_w: H^{3/2}(\mathbb{R})\times H^{1/2}(\mathbb{R})\rightarrow H^{-1/2}(\mathbb{R})\times H^{-3/2}(\mathbb{R})\\
	&PS_w\begin{pmatrix}
	\psi_0^*\\
	\psi_1^*
	\end{pmatrix}\defeq\begin{pmatrix}
	(1+\alpha^2)\Delta_w\underline{\Psi}\Big|_{X=0}\\
-\left((1+\alpha^2)\partial_X-2\alpha\partial_{Y}\right)\Delta_w\underline{\Psi}\big|_{X=0}+\displaystyle\frac{\underline{\Psi}}{2}\bigg|_{X=0}
	\end{pmatrix}=K_w\ast\begin{pmatrix}
	\psi_0^*\\
	\psi_1^*
	\end{pmatrix},
	\end{split}
	\end{equation}
	where $K_w$ is the distributional kernel.
\end{definition}
	Let us derive the expression of the operator in the Fourier space. We know that the unique solution $\underline{\Psi_w}$ of \eqref{p:linear_nonhomgeneous0} in $H^2(\mathbb{R}^2_+)$ of the linear problem \eqref{pb:halfspace_linear_homogeneous} for boundary data $\psi^*_0\in H^{3/2}(\mathbb{R})$ and $\psi^*_1\in H^{1/2}(\mathbb{R})$ can be written as
	\begin{equation*}
	\widehat{\underline{\Psi}}_w(X,\xi)=\sum_{i=1}^2A_i^+(\xi)\exp(-\lambda^+_i(\xi)X),\end{equation*}
	where $A_i^+(\xi)$ and $\lambda^+_i$, $i=1,2$ are the ones in Lemma \ref{Dalibard2014_lemma24}. Going forward and for simplicity of notation, we drop the $+$ sign from both the coefficients and the eigenvalues in this subsection.
	
	Then, taking the Fourier transform of $PS_w$ with respect to $Y$ provides the following has the Fourier representation at the ``transparent'' boundary
	\begin{equation}\label{coeff_western_rep}
	\widehat{PS}_w\begin{pmatrix}
	\widehat{\psi^*_0}\\\widehat{\psi^*_1}
	\end{pmatrix}=\begin{pmatrix}
	(1+\alpha^2)\sum_{i=1}^2A_i(\xi)\left(\lambda_i(\xi)^2+(\alpha\lambda_i(\xi)+i\xi)^2\right)\\
	\sum_{i=1}^2 A_i(\xi)\left[\left((1+\alpha^2)\lambda_i+2\alpha i\xi\right)(\lambda_i^2+(\alpha\lambda_i+i\xi)^2)+\dfrac{1}{2}\right]
	\end{pmatrix}.
	\end{equation}
	Plugging in the above equation the coefficients $A_i(\xi)$ computed in \eqref{inverse_et} yields 
	\begin{equation*}
	\widehat{PS_w\begin{pmatrix}
	\psi^*_0\\\psi^*_1
	\end{pmatrix}}=M_w\begin{pmatrix}
	\widehat{\psi^*_0}\\\widehat{\psi^*_1}
	\end{pmatrix}.
	\end{equation*}
	
	We investigate the behavior of the matrix $M_w=(m_{i,j})_{2\leq i\leq 3,0\leq j\leq 1}\in M_2(\mathbb{C})$ for $\xi$ close to zero and for $|\xi| \rightarrow \infty$. The results are gathered in the following lemma:
	\begin{lemma}\label{beh_DtN}
		\begin{itemize}
			\item Behavior at low frequencies: when $|\xi|\ll 1$
			\begin{equation*}
			M_w=
			\left(
			\begin{array}{cc}
			-\left(\alpha ^2+1\right)^{2/3}+O\left(|\xi|\right)& -\left(\alpha ^2+1\right)^{4/3}+O\left(|\xi|\right) \\
			-\frac{1}{2}+O\left(|\xi|\right)& -8\alpha i\left(\alpha ^2+1\right)|\xi|^1+O\left(|\xi| ^2\right)
			\end{array}
			\right).
			\end{equation*}
			\item Behavior at high frequencies: when $|\xi|\gg 1$
			
			\begin{equation*}
			M_w=\left(
			\begin{array}{cc}
			 \overline{m}_{2,0}|\xi| ^2+O\left(|\xi|^{-1/2}\right)& \overline{m}_{2,1} |\xi| +O\left(|\xi|^{-1/2}\right)\\
			\overline{m}_{3,0}|\xi| ^3+O\left(|\xi|^{3/2}\right)&\overline{m}_{3,1} |\xi| ^2+O\left(|\xi|^{1/2}\right)
			\end{array}
			\right),
			\end{equation*}
			where $\overline{m}_{i,j}$ is a complex quantity depending on $\alpha$, for $i=2,3, j=0,1$. Notice that the value of this constant at $\xi\rightarrow +\infty$ differs from the one at $\xi\rightarrow -\infty$ (see Lemma \ref{Dalibard2014_lemma24} and Appendix \ref{ss:appendix_B}).
		\end{itemize}
	\end{lemma}
	
	The proof of this lemma is elementary and will be given in Appendix \ref{a:matrix_t}. 
	
We have an additional result for the matrix $M_w$:
\begin{lemma}\label{DP2014_lemma2.21}
	At all frequencies, 
	\begin{equation*}
	\partial_\xi^N M_w(\xi)=O((1+|\xi|)^{3-N}),
	\end{equation*}
	for $N\in\mathbb{N}$, $0 \leq  N \leq  5$.
\end{lemma}
Then, $\nabla^N[(1-\chi(\xi ))M_w(\xi )] \in L^1(\mathbb{R})$ that for $N=5$, there exists a constant $C > 0$ such that
\begin{equation*}
\mathcal{F}^{-1}\left((1-\chi(\xi))M_w\right)\leq \dfrac{C}{|Y|^5}.
\end{equation*}
The Poincaré-Steklov operator $PS_w$ associated to $\mathbb{R}^2_+$ has been defined as a continuous operator from $H^{3/2}(\mathbb{R})\times H^{1/2}(\mathbb{R})$ to $H^{-1/2}(\mathbb{R})\times H^{-3/2}(\mathbb{R})$. Our aim here is to prove that it has a unique extension to the space $H^{3/2}_{\mathrm{uloc}}(\mathbb{R})\times H^{1/2}_{\mathrm{uloc}}(\mathbb{R})$.

Let us first show this general result:
\textcolor{black}{\begin{lemma}\label{tech_lemma_2}
        Let $s,s'\in \mathbb{R}$. If for $\Psi\in H^s(\mathbb{R})$, the differential operator $\bar{\mathcal{A}}:H^s(\mathbb{R})\rightarrow H^{s'}(\mathbb{R})$ is continuous, then, there exists a unique continuous extension $\mathcal{A}:H^s_{\mathrm{uloc}}(\mathbb{R})\rightarrow H^{s'}_{\mathrm{uloc}}(\mathbb{R})$.
    \end{lemma}
    \begin{proof}
        First, we recall the definition of $H^s_{\mathrm{uloc}}(\mathbb{R})$, that is:
    \begin{equation}\label{def_kato_rec}
    \psi\in H^s_{\mathrm{uloc}}(\mathbb{R})\quad \textit{iff}\quad \|\psi\|_{ H^s_{\mathrm{uloc}}(\mathbb{R})}=\underset{q\in \mathbb{Z}}{\sup}\|\eta_q\Psi\|_{ H^s(\mathbb{R})}<+\infty,
    \end{equation}
    where $(\eta_q)_{q\in\mathbb{Z}}$ is a partition of unity satisfying $\eta_q\in C^\infty_0(\mathbb{R})$ and $\mathrm{Supp}\thinspace\eta_q \subset B(q,1)$ for $q \in\mathbb{Z}$ and $\sup_q \|\eta_q\|_{W^{k,\infty}} < +\infty$ for all $k$. Definition \eqref{def_kato_rec} is independent of the choice of the function $\eta_q$ (see Lemma 7.1 in \cite{Alazard2016}).
    Let $\psi^*$ be function of $H^s_{\mathrm{uloc}}(\mathbb{R})$, we introduce the notation $\psi^*_q$ to denote $\eta_q\psi^*$. Then, we have
    \begin{equation*}
    \psi^*=\sum_{q\in\mathbb{Z}}\psi^*_q.
    \end{equation*}
    We are interested in verifying that $\mathcal{A}\psi^*$ belongs to $H^{s'}_{\mathrm{uloc}}(\mathbb{R})$, that is the same, as showing that $\eta_q'\mathcal{A}\psi^*\in H^{s'}(\mathbb{R})$, $\forall q'\in \mathbb{Z}$. 
    We have the following decomposition
    \begin{equation}\label{trying}
    \begin{split}
    \eta_q'\mathcal{A}\psi^*&=\sum\limits_{q\in\mathbb{Z}}\eta_q'\mathcal{A}(\eta_q\psi^*)\\
    &=\sum\limits_{|q-q'|\leq 4}\eta_q'\mathcal{A}(\eta_q\psi^*)+\sum\limits_{|q-q'|> 4}\eta_q'\mathcal{A}(\eta_q\psi^*).
    \end{split}
    \end{equation}
    The first term in r.h.s can be easily bounded as follows: if $\eta_q\psi^*\in H^s(\mathbb{R})$, we have $A(\eta_q\psi^*)\in H^{s'}(\mathbb{R})$ and furthermore, $\eta_q'A(\eta_q\psi^*)\in H^{s'}(\mathbb{R})$. Then,
    \begin{equation*}
    \left\|\sum\limits_{|q-q'|\leq 4}\eta_q'\mathcal{A}(\eta_q\psi^*)\right\|_{H^s}\leq C\|\bar{\mathcal{A}}\|_{\mathcal{L}(H^s,H^{s'})}\sum\limits_{|q-q'|\leq 4}\|\eta_q\psi^*\|_{H^s}\leq C\|\bar{\mathcal{A}}\|_{\mathcal{L}(H^s,H^{s'})}\|\psi^*\|_{H^s_{\mathrm{uloc}}(\mathbb{R})}<+\infty.
    \end{equation*}
    For the remaining term in \eqref{trying}, we consider the kernel representation of the operators. We have for $d(\textrm{Supp}\eta_{q'},\textrm{Supp}\eta_{q})\geq 1$ and all $Y'\in \textrm{Supp}\;\eta_{q'}$, 
    \begin{equation}
    \begin{split}
    \left|\bar{\mathcal{A}}(\psi^*_q)(Y')\right|&\leq \int_{|Y'-Y|\geq 1}\dfrac{1}{|Y'-Y|^m}|\psi^*_q(y)|dY\leq \dfrac{1}{|q-q'|^m}\|\psi^*_q\|_{L^2}\\
    &\leq \dfrac{1}{|q-q'|^m}\|\psi^*\|_{L^2_{\mathrm{uloc}}(\mathbb{R})}\leq \dfrac{1}{|q-q'|^m}\|\psi^*\|_{H^s_{\mathrm{uloc}}(\mathbb{R})},
    \end{split}
    \end{equation}
     for $m\geq 5$. Thus,
     \begin{equation*}
     \left|\sum\limits_{|q-q'|> 4}\eta_{q'}\mathcal{A}(\psi^*_q)\right|\leq \sum\limits_{|q-q'|> 4}\dfrac{1}{|q-q'|^m}\|\psi^*\|_{H^s_{\mathrm{uloc}}(\mathbb{R})}<+\infty.
 \end{equation*}
     Consequently, $\mathcal{A}(\psi^*)\in H^{s'}_{\mathrm{uloc}}(\mathbb{R})$ if $\psi^*\in H^s_{\mathrm{uloc}}(\mathbb{R})$, which ends the proof.
    \end{proof}
}

Now, it is possible to link the solution of the \eqref{pb:halfspace_linear_homogeneous} with $(\psi_0^{*},\psi_{1}^{*})\in H^{3/2}_{\mathrm{uloc}}(\mathbb{R})\times H^{1/2}_{\mathrm{uloc}}(\mathbb{R})$ and $PS_w(\psi_0^{*},\psi_{1}^{*})$. 

\begin{proposition}\label{ext_op_uloc}
	Let $(\psi_0^{*},\psi_{1}^{*})\in H^{3/2}_{\mathrm{uloc}}(\mathbb{R})\times H^{1/2}_{\mathrm{uloc}}(\mathbb{R})$, and let $\underline{\Psi}_w$ be the unique solution of \eqref{p:linear_nonhomgeneous0} with $F=0$ and  boundary data $\underline{\Psi}_w|_{X=0} =\psi_0^{*}$  and $\partial_X\underline{\Psi}_w|_{X=0} =\psi_1^{*}$.  Then, for all $\varphi\in C^\infty_0(\bar{\mathbb{R}}^2_+)$
	\begin{equation}\label{link_PS}
	\int_{\mathbb{R}^2_+}\partial_{X}\underline{\Psi}_w\varphi-\int_{\mathbb{R}^2_+}\Delta_w\underline{\Psi}_w\Delta_{w}\varphi=\left\langle \mathcal{A}_3[\psi_0^{*},\psi_{1}^{*}]-\dfrac{\psi_{0}^{*}}{2},\varphi\big|_{X=0}\right\rangle+\left\langle \mathcal{A}_2[\psi_0^{*},\psi_{1}^{*}],\partial_X\varphi\big|_{X=0}\right\rangle.
	\end{equation}
	Namely, for $(\psi_0^{*},\psi_{1}^{*})\in H^{3/2}(\mathbb{R})\times H^{1/2}(\mathbb{R})$, the Poincaré-Steklov operator satisfies the condition
	\begin{equation}\label{negativity}
	\left\langle \mathcal{A}_3[\psi_0^{*},\psi_{1}^{*}],\psi_0^{*}\right\rangle+\left\langle \mathcal{A}_2[\psi_0^{*},\psi_{1}^{*}],\psi_1^{*}\right\rangle\leq 0.
	\end{equation}
\end{proposition}

The proof of \eqref{link_PS} relies once again on defining a smooth function $\tilde{\chi }$, with $\tilde{\chi } = 1$ in an open set containing $\mathrm{Supp}\varphi$ and using the kernel representation formulae of the boundary differential operators. Estimate \eqref{negativity} results from considering $\underline{\Psi}_w$ as test function in \eqref{link_PS}. The detailed verification is left to the reader.

This section ends with other useful estimates on the Poincaré-Steklov operator:
\begin{proposition}\label{prop:estimates_PS}
	Let $\varphi \in C^\infty_0(\mathbb{R})$ such that $\mathrm{Supp}\varphi\subset B(Y_0, R)$, $R\geq 1$, and $(\psi_0^{*},\psi_{1}^{*})\in H^{3/2}_{\mathrm{uloc}}(\mathbb{R})\times H^{1/2}_{\mathrm{uloc}}(\mathbb{R})$. Then, there exists a  constant $C>0$ such that the following property holds.
	\begin{equation}\label{e:estimate_PS}
	\left|\left\langle \mathcal{A}_3[\psi_0^{*},\psi_{1}^{*}],\varphi\right\rangle\right|+\left|\left\langle \mathcal{A}_2[\psi_0^{*},\psi_{1}^{*}],\partial_{X}\varphi\right\rangle\right|\leq C\sqrt{R}\left(\|\varphi\|_{H^{3/2}(\mathbb{R})}+\|\partial_X\varphi\|_{H^{1/2}(\mathbb{R})}\right)\left(\|\psi_0^{*}\|_{H^{3/2}_{\mathrm{uloc}}(\mathbb{R})}+\|\psi_1^{*}\|_{H^{1/2}_{\mathrm{uloc}}(\mathbb{R})}\right)
	\end{equation}
	In particular, if $\psi_j\in H^{3/2-j}(\mathbb{R})$, $j=0,1$,
	\begin{equation}\label{e:estimate_PS_normal}
	\left|\left\langle \mathcal{A}_3[\psi_0^{*},\psi_{1}^{*}],\varphi\right\rangle\right|+\left|\left\langle \mathcal{A}_2[\psi_0^{*},\psi_{1}^{*}],\partial_{X}\varphi\right\rangle\right|\leq C\left(\|\varphi\|_{H^{3/2}(\mathbb{R})}+\|\partial_X\varphi\|_{H^{1/2}(\mathbb{R})}\right)\left(\|\psi_0^{*}\|_{H^{3/2}(\mathbb{R})}+\|\psi_1^{*}\|_{H^{1/2}(\mathbb{R})}\right)
	\end{equation}
	\end{proposition}
\begin{proof}
	This construction is adapted from \cite{Dalibard2014}. We consider a truncation function $\chi\in C^\infty_c(\mathbb{R})$ such that $\chi \equiv 1$ on $B(Y_0, R + 1)$ and $\mathrm{Supp}\chi\subset B(Y_0, R + 2)$, and such that $\|\partial^r_Y\chi\|_{\infty}\leq  C_{r}$, with $C_{r}$ independent of $R$, for all $r\in\mathbb{N}$. For the terms
	\begin{equation*}
	\int_{\mathbb{R}}K_{3,j}\ast ((1-\chi)\psi_j^*)\varphi,\quad \int_{\mathbb{R}}K_{2,j}\ast ((1-\chi)\psi_j^*)\partial_X\varphi,
	\end{equation*}
	where $|K_{i,j}(Y)|\leq C |Y|^{-5}$, $i=2,3$, $j=0,1$, for all $Y\in\mathbb{R}$, we use the following relation
	\begin{equation}\label{est_ga0}
	\begin{split}
	&C\int_{\mathbb{R}\times\mathbb{R}}\dfrac{1}{|Y'|^5}|1-\chi(Y-Y')||\psi_j^*(Y-Y')||\partial_{X}^{3-i}\varphi(Y)|\;dYdY'\\
	&\quad\leq C\int_{\mathbb{R}}|\partial_{X}^{3-i}\varphi(Y)|dY\left(\int_{|Y'|\leq 1}\dfrac{|\psi_j^*(Y-Y')|^2|}{|Y'|^5}\;dY'\right)^{\frac{1}{2}}\left(\int_{|Y'|\leq 1}\dfrac{1}{|Y'|^5}\;dY'\right)^{\frac{1}{2}}\\
	&\quad\leq C\|\psi_j^*\|_{L^2_{\mathrm{uloc}}(\mathbb{R})}\|\partial_{X}^{3-i}\varphi\|_{L^1}\leq C\sqrt{R}\|\psi_j^*\|_{L^2_{\mathrm{uloc}}(\mathbb{R})}\|\partial_{X}^{3-i}\varphi\|_{L^2}.
	\end{split}
	\end{equation}
	It remains to analyze the terms of the type
	\begin{equation*}
	\langle \mathcal{F}^{-1}(m_{i,j}\widehat{\chi \psi_j^{*}}),\partial_X^{3-i}\varphi\rangle_{H^{-(2i-3)/2},H^{(2i-3)/2}}.
	\end{equation*}
	Since $m_{i,j}(\xi)$, $i=2,3, j=0,1$ is a kernel satisfying $\mathrm{Op}(m_{i,j}): H^{3/2-j}(\mathbb{R}) \rightarrow H^{-i+3/2}(\mathbb{R})$, these terms are bounded by
	\begin{equation*}
	C\|\chi\psi^*_j\|_{H^{3/2-j}(\mathbb{R})}\|\partial_X^{3-i}\varphi\|_{H^{i-3/2}(\mathbb{R})}
	\end{equation*}

We proceed to prove the estimate
	\begin{equation}\label{estimate_ga}
	\|\chi\psi^*_j\|_{H^{3/2-j}(\mathbb{R})}\leq C\sqrt{R}\|\psi^*_j\|_{H^{3/2-j}_{\mathrm{uloc}}(\mathbb{R})}
	\end{equation}
	but first let us recall the norm definition in fractional Sobolev spaces.
	\begin{definition}
		Let  $s\in(0, 1)$ be a fractional exponent and $\omega$ be a general, possibly non-smooth, open set in $\mathbb{R}^n$ . For any $p \in[1,+\infty)$, the fractional Sobolev space $W^{s,p}(\omega)$ is defined as follows
		\begin{equation}
		W^{s,p}(\omega) :=\left\{u\in L^p(\omega):\:\dfrac{|u(Y)-u(Y')|}{|Y-Y'|^{\frac{n}{p}+s}}\in L^p(\omega\times\omega) \right\};
		\end{equation}
		i.e. an intermediary Banach space between $L^p(\omega)$ and $W^{1,p}(\omega)$, endowed with the natural norm
		\begin{equation}
		\|u\|_{W^{s,p}(\omega)} :=\left(\int_{\omega}|u|^p\;dY+[u]^p_{W^{s,p}(\omega)} \right)^{1/p}.
		\end{equation}
		Here, the term $[u]_{W^{s,p}(\omega)}$ is the so-called Gagliardo (semi)norm of $u$ defined as
		\begin{equation}
		[u]_{W^{s,p}(\omega)}= \left(\int_{\omega}\int_{\omega}\dfrac{|u(Y)-u(Y')|^p}{|Y-Y'|^{n+sp}}\right)^{1/p}.
		\end{equation}
		If $s = m + \eta$, where $m\in\mathbb{Z}$ and $\eta \in (0, 1)$. The space $W^{s,p}(\omega)$ consists of
		\begin{equation}
		W^{s,p}(\omega) :=\left\{u\in W^{m,p}(\omega):\:D^\zeta u\in W^{\eta ,p}(\omega)\:\text{for any}\:\zeta \:\text{such that}\:|\zeta|=m\right\};
		\end{equation}
		which is a Banach space with respect to the norm
		\begin{equation}\label{Gaglairdo}
		\|u\|_{W^{s,p}(\omega)} :=\left(\int_{\omega}\|u\|^p_{W^{m,p}(\omega)}\;dY+\sum_{|\zeta|=m}\|D^\zeta u\|^p_{W^{\eta,p}(\omega)} \right)^{1/p}.
		\end{equation}
		\end{definition}
	Before proving the estimates, we introduce a cut-off function $\theta$ satisfying \eqref{DP:1-4}, which will allow us to use \eqref{kato_def}. Following the same ideas in \cite[Lemma 2.26]{Dalibard2014}, we have that for a certain $u_0$
	\begin{eqnarray}
	\|\chi u_0\|^2_{L^{2}(\mathbb{R})}\leq\sum_{k\in\mathbb{Z}}\|(\tau_k\theta)\chi u_0\|^2_{L^{2}(\mathbb{R})}\leq \|\chi \|^2_{\infty }\sum_{\substack{k\in\mathbb{Z}\\ k\leq CR\;}}\|(\tau_k\theta)u_0\|^2_{L^{2}(\mathbb{R})}\leq CR\; \|\chi \|^2_{\infty }\underset{k\in\mathbb{Z}}{\sup}\|(\tau_k\theta)u_0\|^2_{L^{2}(\mathbb{R})}.
	\end{eqnarray}
	To deal with the Gagliardo norm, notice that the denominator in \eqref{Gaglairdo} for $p=2$ can be written as
	\begin{align*}
	&|\chi u_0(Y) - \chi u_0(Y')|^2\\
	&=\left(\sum_{\substack{k\in\mathbb{Z}}}k\theta(Y)\chi (Y)u_0(Y) - \tau_k \theta(Y')\chi (Y')u_0(Y')\right)^2\\
	&=\sum_{\substack{k,l\in\mathbb{Z}\\ |k-l|\leq 3}}\left(\tau_k \theta(Y)\chi (Y)u_0(Y) - \tau_k \theta(Y')\chi (Y')u_0(Y')\right)\left(\tau_l\theta(Y)\chi (Y)u_0(Y) - \tau_l\theta(Y')\chi (Y')u_0(Y')\right)\\
	&\quad+\sum_{\substack{k,l\in\mathbb{Z}\\ |k-l|> 3}}\left(\tau_k \theta(Y)\chi (Y)u_0(Y) - \tau_k \theta(Y')\chi (Y')u_0(Y')\right)\left(\tau_l\theta(Y)\chi (Y)u_0(Y) - \tau_l\theta(Y')\chi (Y')u_0(Y')\right).
	\end{align*}
	As a result of the assumptions on $\theta$, for $|k - l| > 3$, we obtain that  $\tau_k \theta(Y)\tau_l \theta(Y) = 0$ for all $Y \in \mathbb{R}$. Moreover, if $\tau_k  (Y)\tau_l (Y') \neq 0$, then, $|x - y| \geq |k - l| - 2$. Also, the first sum above contains $O(R)$ nonzero terms. Hence, using the Cauchy–Schwarz inequality gives
	\begin{align*}
	&[\chi u]^2_{W^{s,p}(\omega)}\\
	&= \int_{\mathbb{R}}\int_{\mathbb{R}}\dfrac{|\chi u_0(Y)-\chi u_0(Y')|^2}{|Y-Y'|^{3}}dY dY'\\
	&\leq CR\;\underset{k\in\mathbb{Z}}{\sup}\int_{\mathbb{R}}\int_{\mathbb{R}}\dfrac{|(\tau_k \theta\chi u(Y))-(\tau_k \theta\chi u(Y'))|^2}{|Y-Y'|^{3}}dY dY'\\
	&+\sum_{\substack{k,l\in\mathbb{Z}\\ |k-l|> 3}}\dfrac{1}{(|k-l|-2)^3}\left(\tau_k \theta(Y)\chi (Y)u_0(Y) - \tau_k \theta(Y')\chi (Y')u_0(Y')\right)\left(\tau_l\theta(Y)\chi (Y)u_0(Y) - \tau_l\theta(Y')\chi (Y')u_0(Y')\right)\\
	&=I_1+I_2.
	\end{align*}
	We have
	\begin{equation*}
	|I_1|\leq CR\;\|\chi\|^2_{W^{1,\infty}}\|u_0\|^2_{H^{1/2}_{\mathrm{uloc}}(\mathbb{R})},\quad\text{and}\quad |I_2|\leq C\|u_0\|^2_{L^2_{\mathrm{uloc}}}.
	\end{equation*}
	Then, for $u_0=\psi_1^*$ and $R>1$, it follows that
	\begin{eqnarray*}
	\|\chi\psi_1^*\|^2_{H^{1/2}(\mathbb{R})}&=&\|\chi\psi_1^*\|^2_{L^{2}(\mathbb{R})}+[\chi\psi_1^*]^2_{H^{1/2}}\leq CR\;\|\psi_1^*\|^2_{H^{1/2}_{\mathrm{uloc}}(\mathbb{R})}+C\|\psi_1^*\|^2_{L^2_{\mathrm{uloc}}}\\
	&\leq& CR\;\|\psi_1^*\|^2_{H^{1/2}_{\mathrm{uloc}}(\mathbb{R})}.
	\end{eqnarray*}
The remaining term is dealt with in a similar manner
\begin{eqnarray*}
	\|\chi\psi_0^*\|^2_{H^{3/2}(\mathbb{R})}&=&\|\chi\psi_0^*\|^2_{L^{2}(\mathbb{R})}+\|\chi D\psi_0^*\|^2_{L^{2}(\mathbb{R})}+[\chi D\psi_0^*]^2_{H^{1/2}}\\
	&\leq& CR\;\|\psi_0^*\|^2_{H^{3/2}_{\mathrm{uloc}}(\mathbb{R})}+C\|D\psi_0^*\|^2_{L^2_{\mathrm{uloc}}}\\
	&\leq& CR\;\|\psi_0^*\|^2_{H^{3/2}_{\mathrm{uloc}}(\mathbb{R})}.
\end{eqnarray*}
From \eqref{est_ga0} and \eqref{estimate_ga} we obtain \eqref{e:estimate_PS}.
The proof of inequality \eqref{e:estimate_PS_normal} is classical and follows from the Fourier representation of the differential operators.
	\end{proof}

\subsection{The problem in the rough channel} \label{s:rough_linear_west}
The section is devoted to proving the existence and uniqueness of weak solutions for the linear problem  \eqref{e:linear_app} by studying an equivalent problem defined in a channel $\omega^b_w$ presenting a transparent boundary at the interface $\{X=M\}$, $M>0$. Here, only an accurate
representation of the solution of the problem linear problem at $\{X=M\}$ is needed in order to obtain a good approximation of the solution of the original problem while solving a similar set of equations in the rough channel (step (L5) in Section \ref{l_method}). The linear problem \eqref{e:linear_app} acts on the new system through the coupling conditions described employing the Poincaré-Steklov operator in \eqref{def_PS}. As before, we are going to consider the linear problem without a source term, i.e., $F=0$ in \eqref{e:linear_app}. 

We define the following problem equivalent to \eqref{e:linear_app} in the bounded channel $\omega^b_w=\omega_w\cap\{X\leq M\}$, $M>0$,

\begin{eqnarray}\label{a:linear_equiv}
\partial_{X}\Psi_w-\Delta^{2}_w\Psi&=&0\quad\textrm{in}\;\omega^b_w\setminus\sigma_w,\nonumber\\
\left[\partial^k_X\Psi_w\right]\big|_{\sigma_w}&=&g_k,\quad k=0,\ldots,3,\\
\left[(1+\alpha^2)\Delta_w\Psi_w\right]\big|_{\sigma_w^M}&=&\mathcal{A}_2\left[\Psi_w|_{X=M},\partial_X\Psi_w|_{X=M}\right],\nonumber\\
\left[-\left((1+\alpha^2)\partial_X-2\alpha\partial_{Y}\right)\Delta_w\Psi_w+\displaystyle\frac{\Psi_w}{2}\right]\Bigg|_{\sigma_w^M}&=&\mathcal{A}_3\left[\Psi_w|_{X=M},\partial_X\Psi_w|_{X=M}\right],\nonumber,\nonumber\\
\Psi_w\big|_{X=-\gamma_w(Y)}=\partial _{\mathrm{n}}\Psi_w\big|_{X=-\gamma_w(Y)}&=&0.\nonumber
\end{eqnarray}

The equivalence between the solution of \eqref{a:linear_equiv} and the one of the original problem is given in the following lemma:
\begin{lemma}
		Let $\gamma \in W^{2,\infty}(\mathbb{R})$ and $g_k \in L^\infty(\mathbb{R})$, for $k=0,\dots,3$. 
		\begin{itemize}
			\item Let $\Psi_w$ be a solution of \eqref{e:linear_app} in $\omega_w$ such that $\Psi_w\in H^2_{\mathrm{uloc}}(\omega)$. Then, $\Psi|_{\omega^b_w}$ is a solution of \eqref{a:linear_equiv}, and for $X > M$, $\Psi$ solves the homogeneous equivalent of problem \eqref{e:linear_app} defined on the half-space $M\times\mathbb{R}$, with $\psi_0 := \Psi_w|_{X=M} \in H^{3/2}_{\mathrm{uloc}}(\mathbb{R})$ and $\psi_1 := \partial_X\Psi_w|_{X=M} \in H^{1/2}_{\mathrm{uloc}}(\mathbb{R})$.
			\item Furthermore, let $\Psi^-_w \in H^2_{\mathrm{uloc}}(\omega^b_w)$ and $\Psi^+_w \in H^2_{\mathrm{uloc}}(\mathbb{R}^2_+)$ be solutions of \eqref{a:linear_equiv} and \eqref{e:linear_app}, respectively. Taking
			
			\begin{equation*}
				\Psi_w(X,\cdot):=\left\{\begin{array}{ccc}
					\Psi_w^-(X,\cdot)&\mathrm{for}&-\gamma_w(\cdot)<X<M , \\
					\Psi_w^+(X,\cdot)& \mathrm{for}&X>M,
				\end{array}\right.
			\end{equation*}
			the function $\Psi \in H^2_{\mathrm{loc}}(\omega)$ is a solution of the problem \eqref{e:linear_app}.
		\end{itemize}
		\end{lemma}
Note that $\Psi_w^-$ solves \eqref{a:linear_equiv} in the trace sense, and for
all $\varphi\in C^\infty_c(\overline{\omega^b_w})$ satisfies 
	\begin{equation}
	\begin{split}
	    	\int_{\omega_w^b}\partial_{X}\underline{\Psi}_w\varphi-\int_{\omega_w^b}\Delta_w\underline{\Psi}_w\Delta_{w}\varphi&=-\left\langle \mathcal{A}_3\left[\Psi_w|_{X=M},\partial_X\Psi_w|_{X=M}\right]-\dfrac{\psi_{0}^{*}}{2},\varphi\big|_{X=M}\right\rangle\\
	    	&\quad-\left\langle \mathcal{A}_2\left[\Psi_w|_{X=M},\partial_X\Psi_w|_{X=M}\right],\partial_X\varphi\big|_{X=M}\right\rangle.
	\end{split}
	\end{equation}
The above result easily follows from Theorem \ref{Theorem2_DGV2017} and Proposition \ref{ext_op_uloc}. Consequently, we focus our attention on showing a well-posedness of problem \eqref{a:linear_equiv} in the remainder of the section.
\begin{proposition}\label{prop:result_rough_channel_lw}
Let $\gamma_w \in W^{2,\infty}(\mathbb{R})$ and $\omega^b_w=\omega_w\cap\{X\leq M\}$, $M>0$. Assume the Poincaré-Steklov operators $\mathcal{A}_i:H^{3/2}_\mathrm{uloc}(\mathbb{R})\times H^{1/2}_\mathrm{uloc}(\mathbb{R})\rightarrow H^{3/2-i}_\mathrm{uloc}(\mathbb{R})$, $i=2,3$ satisfy the properties in Proposition \ref{ext_op_uloc} and $g_k\in L^\infty(\mathbb{R})$, for $k=0,\ldots,3$. Then, there exists a unique solution $\Psi_w\in H^2_{\mathrm{uloc}}(\omega^b_w\setminus\sigma_w)$ satisfying
\begin{equation}\label{estimate_rough_channel_lw}
    \|\Psi_w\|_{H^2_\mathrm{uloc}(\omega^b_w)}\leq C\sum\limits_{k=0}^{3}\|g_k\|_{L^\infty(\mathbb{R})},
\end{equation}
where $C>0$ is a universal constant.
\end{proposition}

\begin{proof}
From now on, we lose the $w$ index to simplify the notation when no confusion can arise.

Before stating the main ideas of the proof, we first lift the nonhomogeneous jump conditions at $\sigma_w$ by introducing the function $\Psi^L$ as in \eqref{lift_lin}.   Then, for $\tilde{\Psi}=\Psi-\Psi^{L}$, we have
 \begin{eqnarray}\label{a:left_original}
\partial_{X}\tilde{\Psi}-\Delta^{2}_w\tilde{\Psi}&=&F^L\quad\textrm{in}\;\;\omega^b_w\setminus\sigma^M_w,\nonumber\\
\left[(1+\alpha^2)\Delta_w\tilde{\Psi}\right]\Big|_{\sigma^M_w}&=&\mathcal{A}_2\left[\Psi_w|_{X=M},\partial_X\Psi_w|_{X=M}\right],\\
\left[-\left((1+\alpha^2)\partial_X-2\alpha\partial_{Y}\right)\Delta_w\tilde{\Psi}+\displaystyle\frac{\tilde{\Psi}}{2}\right]\Bigg|_{\sigma^M_w}&=&\mathcal{A}_3\left[\Psi_w|_{X=M},\partial_X\Psi_w|_{X=M}\right],\nonumber\\
\tilde{\Psi}\big|_{X=-\gamma_w(Y)}=\partial _{\mathrm{n}}\tilde{\Psi}\big|_{X=-\gamma_w(Y)}&=&0,\nonumber
 \end{eqnarray}
 where $F^L$ is a function depending on $g_k$, $k=0,\ldots,3$. The truncation technique introduced by Lady\v{z}enskaya and Solonnikov \cite{Ladyzenskaja1980} is used to prove the existence and uniqueness of the solution of system \eqref{a:left_original} by means of a local uniform bound on $\Delta_w\tilde{\Psi}_n$, where $\tilde{\Psi}_n$ is the solution of the problem
\begin{equation}\label{method_lift_problem_n_west}
\begin{split}
	\partial_{X}\tilde{\Psi}_n-\Delta_w^{2}\tilde{\Psi}_n&=F_n^L\quad\text{in}\:\:\omega_n\\[8pt]
	(1+\alpha^2)\Delta_w\tilde{\Psi}_n\Big|_{\sigma^M_n}&=\mathcal{A}_2\left[\Psi_w|_{X=M},\partial_X\Psi_w|_{X=M}\right],\\
	-\left[(1+\alpha^2)\partial_X-2\partial_Y\right]\Delta_w\tilde{\Psi}_n+\dfrac{\tilde{\Psi}_n}{2}\bigg|_{\sigma^M_n}&=\mathcal{A}_3\left[\Psi_w|_{X=M},\partial_X\Psi_w|_{X=M}\right],\\
	\tilde{\Psi}_n\big|_{\Gamma_n}&=\frac{\partial\tilde{\Psi}}{\partial n_w}\big|_{\Gamma_n}=0,
\end{split}
\end{equation}
where $\omega_n, \sigma^M_n$ and $\Gamma_n$ are the ones in \eqref{dom_n}. The problem on $\tilde{\Psi}_n$ has the following  weak formulation: {\it Let $\varphi \in C^\infty_0(\omega^b)$ such that
\begin{equation}\label{hyp_test}
\varphi= 0\;\;\text{on}\;\;\omega^b\setminus\omega_n,\quad \varphi|_{\Gamma_n}= 0\:\:\text{and}\:\:\partial_{X}\varphi|_{\Gamma_n}=0.
\end{equation}
Then, the solution $\tilde{\Psi}_n\in H^2(\omega_n)$ of \eqref{method_lift_problem_n_west} satisfies 
\begin{equation}\label{link_PSn}
\begin{split}
    \int_{\omega^b}\partial_{X}\tilde{\Psi}_n\varphi-\int_{\omega^b}\Delta_w\tilde{\Psi}_n\Delta_{w}\varphi&=-\left\langle \mathcal{A}_3\left[\Psi_w|_{X=M},\partial_X\Psi_w|_{X=M}\right]-\dfrac{\tilde{\Psi}_n}{2}\Big|_{X=M},\varphi\big|_{X=M}\right\rangle\\
    &\quad-\left\langle \mathcal{A}_2\left[\Psi_w|_{X=M},\partial_X\Psi_w|_{X=M}\right],\partial_X\varphi\big|_{X=M}\right\rangle-\int_{\omega^b}F^L_n\varphi.
\end{split}
\end{equation}}

Taking $\tilde{\Psi}_n$ as test function gives
\begin{equation}
\begin{split}
\int_{\omega^b}|\Delta_w\tilde{\Psi}_n|^2&=\underset{\leq 0}{\underbrace{\left(\left\langle\mathcal{A}_3\left[\Psi_w|_{X=M},\partial_X\Psi_w|_{X=M}\right],\tilde{\Psi}_n\big|_{X=M}\right\rangle+\left\langle \mathcal{A}_2\left[\Psi_w|_{X=M},\partial_X\Psi_w|_{X=M}\right],\partial_X\tilde{\Psi}_n\big|_{X=M}\right\rangle\right)}}\\
&-\int_{\omega^b}F^L_n\tilde{\Psi}_n\\
&\leq C\sqrt{n}\sum_{k=0}^{3}\|g_k\|_{L^{\infty}(\mathbb{R})}\|\tilde{\Psi}_n\|_{H^2(\omega_{n})},
\end{split}
\end{equation}
where the constant $C$ only depends on $\|\gamma_w\|_{W^{2,\infty}}$. Then, applying Poincaré inequality we have
\begin{equation}\label{C0_lw}
\int_{\omega_{k}}|\Delta_w\tilde{\Psi}_n|^2\leq C_0n,
\end{equation}
where $C_0=C_0(g_0,\ldots,g_3)$. The existence of $\tilde{\Psi}_n$ in $H^2(\omega^b)$ follows from Lemma \ref{lemme_tech}. Uniqueness is obtained by following similar arguments as the ones presented in (L5), see Section \ref{l_method}.

 We work with the energy 
\begin{eqnarray}
E_k^n:=\int_{\omega_{k}}|\Delta_w\tilde{\Psi}_n|^2,
\end{eqnarray}
for which  we prove an inequality of the type 
\begin{equation}\label{DP:3-1}
E_k^n \leq C_1\left(k+1+m\underset{k\leq j\leq k+m}{\sup}\left(E_{j+1}^n-E_{j}^n\right)+\frac{1}{m^{4-2\eta}}\underset{j\geq k+m}{\sup}(E_{j+1}^n-E_j^n)\right)\quad \text{for all}\quad k \in \{m, \ldots , n\},
\end{equation}
for any $m > 1$ and $\eta\in ]0,2[$. The constant $C_1 > 0$ is uniform constant in $n$ depending only on $\|\gamma\|_{W^{2,\infty}}$ and $\|g_j\|_{W^{2-j,\infty}}$, $j=0,\dots,3$. The bound in $H^2_{\mathrm{uloc}}(\omega^b)$ is then obtained via a nontrivial induction argument.

The remaining of the overall strategy is the same as the one detailed in step (L5) in Section \ref{l_method}. Consequently, as we advance, we only discuss in detail the computations of the estimates involving the nonlocal differential operators and their incidence on the induction argument.

\begin{itemize}
	\item \textit{Induction.} To shorten the notation in the following paragraphs, we write $E_k$ and $\tilde{\Psi}$ instead of  $E_k^n$ and $\tilde{\Psi}_n$. Let us show by induction on $n-k$ that for $m$ large enough, \eqref{DP:3-1} amounts to
	\begin{equation}\label{DP:3-4}
	E_k \leq C_1\left(k+1+m^3+\frac{1}{m^{4-2\eta}}\underset{j\geq k+m}{\sup}(E_{j+1}-E_j)\right)\quad \text{for all}\quad \forall k\leq n,
	\end{equation}
	where the positive constant $C_2$ depends only on $C_0$ and $C_1$, appearing respectively in \eqref{C0_lw} and \eqref{DP:3-1}. The inequality is clearly true when $k = n$, as soon as $C_2 > C_0$. Let us now assume that 
	\begin{equation}\label{DP:3-5}
	E_{k'} \leq C_2\left({k}+1+m^3+\frac{1}{m^{4-2\eta}}\underset{j\geq k'+m}{\sup}(E_{j+1}-E_j)\right)\quad \text{for all}\quad k\in\left\{\leq k+1,\ldots,n\right\},
	\end{equation}
	holds and show it remains true for index $k$. If it were false, we have
	\begin{equation}\label{MGV_2.20}
	E_{k'} \geq C_2\left({k}+1+m^3+\frac{1}{m^{4-2\eta}}\underset{j\geq k'+m}{\sup}(E_{j+1}-E_j)\right).
	\end{equation}
	Combining inequalities \eqref{DP:3-5} and \eqref{MGV_2.20} implies for all $k+m\geq j\geq k,$
	\begin{equation}
	E_{j+1} - E_{j} \leq E_{j+1} - E_k \leq C_2(m +1).
	\end{equation}
	Then, \eqref{DP:3-1} yields
	\begin{equation}\label{MG:2.22}
	E_k \leq C_1\left(k+1+C_2m(m+1)+\frac{1}{m^{4-2\eta}}\underset{j\geq k+m}{\sup}(E_{j+1}-E_j)\right)\quad \text{for all}\quad k \in \{m, \ldots , n\}.
	\end{equation}
	Note that if $C_2 > C_1$ and $C_1C_2m(m+1) \leq C_2m^3$ in \eqref{MG:2.22} we have a contradiction.  This is verified when  $C_2 > C_1$ and $m$ is large enough. Hence, inequality \eqref{MGV_2.20} is valid for all $k \leq n$. Since equation \eqref{DP:3-1} is invariant by a horizontal translations (see (L5) in Section \ref{l_method}), we obtain
	\begin{equation*}
	E_{k+1}-E_k \leq C_1\left(2+m^3+\frac{1}{m^{4-2\eta}}\underset{j\in\mathbb{N}}{\sup}(E_{j+1}-E_j)\right)\quad \text{for all}\quad k \in \{m, \ldots , n\},
	\end{equation*}
	for all $k$, so that for $m$ large enough, we conclude that
	\begin{equation*}
	\underset{k\in\mathbb{N}}{\sup}(E_{k+1}^n-E_k^n)\leq C_1\left(\frac{2+m^3}{1-m^{-4+2\eta}}\right)=C<+\infty,
	\end{equation*}
	which is a $H^2_{\mathrm{uloc}}$ bound on $\tilde{\Psi}_n$. Hence, we can extract a subsequence of $\tilde{\Psi}_n$ that converges weakly to some $\tilde{\Psi}_n$ satisfying \eqref{a:left_original}. Existence follows from the ideas presented at the beginning of the current section.
	
\item\textit{Establishing the Saint-Venant estimate.}  This paragraph contains the proof of \eqref{DP:3-1}. The main difficulty in computing estimates independent of the size of the support of $\tilde{\Psi}_n$ resides on the nonlocal nature of the Poincaré-Steklov operators. 

Thanks to the representation formula of the Poincaré-Steklov operators, the above formulation makes sense for $\varphi \in H^2(\omega^b_w)$ satisfying \eqref{hyp_test}. To establish the estimates of $E_k$, we first introduce the cut-off function $\chi_k(Y)$ supported in $\sigma^M_{k+1}$ and identically equal to 1 on $\sigma^M_{k}$. Considering $\varphi=\chi_k\tilde{\psi}$, $k<n$, as a test function in \eqref{link_PS} yields for elements in l.h.s an expression equivalent to \eqref{bt_1}. Namely,
\begin{equation}\label{bt_1_lw}
\begin{split}E_k&=\left(\left\langle\mathcal{A}_3\left[\tilde{\Psi}|_{X=M},\partial_X\tilde{\Psi}|_{X=M}\right],\chi_k\tilde{\Psi}\big|_{X=M}\right\rangle+\left\langle \mathcal{A}_2\left[\tilde{\Psi}|_{X=M},\partial_X\tilde{\Psi}|_{X=M}\right],\chi_k\partial_X\tilde{\Psi}\big|_{X=M}\right\rangle\right)\\
&-\int_{\omega^b}F^L\chi_k\tilde{\Psi}+\text{commutator terms stemming from the bilaplacian.}
\end{split}
\end{equation}
All commutator terms are bounded by $C(E_{k+1}-E_k)$. The proof involves applying Poincaré and Young inequalities similarly to \eqref{commut_bound}. Moreover,
\begin{equation}\label{FP}
\left|\int_{\omega^b}F^L\chi_k\tilde{\Psi}\right|\leq C\left(\sum_{k=0}^{3}\|g_k\|_{L^{\infty}}\right)\sqrt{k+1}E_{k+1}^{1/2}\leq C(g_0,\ldots,g_3)\sqrt{k+1}E_{k+1}^{1/2}.
\end{equation}
It remains to handle the non-local terms, i.e., the Poincaré-Steklov operator. Drawing inspiration from \cite{Dalibard2014} and \cite{Masmoudi2000}, we introduce the auxiliary parameter $m \in\mathbb{N}_*$  appearing in \eqref{DP:3-1} and the following decomposition for $\psi_{j}$
\begin{equation}\label{decomp_psi}
\psi_j=\left(\chi_k+(\chi_{k+m}-\chi_k)+(1-\chi_{k+m})\right)\psi_j,\quad j=0,1.
\end{equation}
Then, for $i=2,3$, the transparent operators can be written as
\begin{equation}\label{up}
\begin{split}
&\langle \mathcal{A}_i\left[\tilde{\Psi}|_{X=M},\partial_X\tilde{\Psi}|_{X=M}\right],\chi_k\partial_{X}^{3-i}\tilde{\Psi}|_{X=M}\rangle\\
&\qquad=\langle \mathcal{A}_i\left[\chi_k\tilde{\Psi}|_{X=M},\chi_k\partial_X\tilde{\Psi}|_{X=M}\right],\chi_k\partial_{X}^{3-i}\tilde{\Psi}|_{X=M}\rangle\\
&\qquad+\langle \mathcal{A}_i[\chi_k\tilde{\Psi}|_{X=M},1-\chi_k\partial_X\tilde{\Psi}|_{X=M}],\chi_k\partial_{X}^{3-i}\tilde{\Psi}|_{X=M}\rangle\\
&\qquad+\mathcal{A}_i[1-\chi_k\tilde{\Psi}|_{X=M},\chi_k\partial_X\tilde{\Psi}|_{X=M}],\chi_k\partial_{X}^{3-i}\tilde{\Psi}|_{X=M}\\
&\qquad+\mathcal{A}_i[(1-\chi_k)\tilde{\Psi}|_{X=M},(1-\chi_k)\partial_X\tilde{\Psi}|_{X=M}],\chi_k\partial_{X}^{3-i}\tilde{\Psi}|_{X=M}\\
&\qquad\leq \left|\mathcal{A}_i[(\chi_{k+m}-\chi_k)\tilde{\Psi}|_{X=M},(\chi_{k+m}-\chi_k)\partial_X\tilde{\Psi}|_{X=M}],\chi_k\partial_{X}^{3-i}\tilde{\Psi}|_{X=M}\rangle\right|\\
&\qquad+\left|\langle \mathcal{A}_i[(1-\chi_{k+m})\tilde{\Psi}|_{X=M},(1-\chi_{k+m})\partial_X\tilde{\Psi}|_{X=M}],\chi_k\partial_{X}^{3-i}\tilde{\Psi}|_{X=M}\rangle\right|\\
&\qquad=I_{i,1}+I_{i,2}.
\end{split}
\end{equation}
The inequality in \eqref{up} results from considering the negativity condition satisfied by the transparent operators. For the term $I_{i,1}$ we use to Proposition \ref{prop:estimates_PS} and the estimate
\begin{equation*}
\|\chi_k\tilde{\Psi}\|_{H^2(\omega^b)}+\|\chi_k\tilde{\Psi}|_{X=M}\|_{H^{3/2}(\omega^b)}+\|\chi_k\partial_X\tilde{\Psi}|_{X=M}\|_{H^{1/2}(\omega^b)}\leq C E_{k+1}^{1/2}.
\end{equation*}
Then,
\begin{equation*}
\begin{split}
|I_{2,1}|+|I_{3,1}|\leq&\|A_2[(\chi_{k+m}-\chi_k)\psi_0,(\chi_{k+m}-\chi_k)\psi_1\|_{H^{-1/2}}\left\|\chi_k\partial_{X}\chi_k\tilde{\Psi}|_{X=M}\right\|_{H^{1/2}}\\
&+\|A_3[(\chi_{k+m}-\chi_k)\psi_0,(\chi_{k+m}-\chi_k)\psi_1\|_{H^{-3/2}}\left\|\chi_k\chi_k\tilde{\Psi}|_{X=M}\right\|_{H^{3/2}}\\
\leq& C\left(\|\partial_X\chi_k\tilde{\Psi}|_{X=M}\|_{H^{1/2}(\mathbb{R})}+\|\chi_k\tilde{\Psi}|_{X=M}\|_{H^{3/2}(\mathbb{R})}\right)\\
&\quad\times\left(\|(\chi_{k+m}-\chi_k)\partial_X\tilde{\Psi}|_{X=M}\|_{H^{1/2}(\mathbb{R})}+\|(\chi_{k+m}-\chi_k)\tilde{\Psi}|_{X=M}\|_{H^{3/2}(\mathbb{R})}\right)\\
\leq& C(E_{m+k+1}-E_k)^{1/2}E_{k+1}^{1/2}.
\end{split}
\end{equation*} 
We are left with the task of finding bounds for $I_{i,2}$. Note that for $m \geq 2$, $\mathrm{Supp} \chi_{k+1} \cap \mathrm{Supp} (1-\chi_{k+m}) = \emptyset$, so, for $i=2,3$
\begin{equation}\label{ant}
\begin{split}
\langle \mathcal{A}_{i}\left[\tilde{\Psi}|_{X=M},\partial_X\tilde{\Psi}|_{X=M}\right],\partial_X^{3-i}\chi_k\tilde{\Psi}\rangle
=&\int_{\mathbb{R}}K_{i,0}\ast((1-\chi_{k+m})\tilde{\Psi}|_{X=M})\partial_X^{3-i}\chi_k\tilde{\Psi}\\
&+\int_{\mathbb{R}}K_{i,1}\ast((1-\chi_{k+m})\partial_X\tilde{\Psi}|_{X=M})\partial_X^{3-i}\chi_k\tilde{\Psi}.
\end{split}
\end{equation}
The convolution terms in \eqref{ant} decay like $|Y|^{-5}$. We have the following estimate:
\begin{lemma}\label{DP:lemma_3.1}
	For all $k \geq m$ and $\eta\in\left.\right]0,2\left[\right.$
	\begin{equation}\label{lemma_hf}
	\left\|K_{i,j}\ast(1-\chi_{k+m})\psi_j\right\|_{L^2(\sigma_{k+1}^M)}\leq C_\eta m^{-2+\eta}\left(\underset{j\geq k+m}{\sup}(E_{j+m}-E_j)\right)^{1/2}.
	\end{equation}
\end{lemma}
\begin{proof}
	We use an idea of Gérard-Varet and Masmoudi [2010], that was later used in \cite{Dalibard2014}, to treat the large scales: we decompose the set $\sigma^M\setminus\sigma^M_{k+m}$ as
	\begin{equation*}
	\sigma^M\setminus\sigma^M_{k+m}=\underset{j\geq k+m}{\bigcup}\sigma^M_{j+1}\setminus\sigma^M_{j}.
	\end{equation*}
	On every set $\sigma^M_{j+1}\setminus\sigma^M_{j}$, we bound the $L^2$ norm of $(1-\chi_{k+m})\psi_i$, $i=0,1$,  by $E_{k+m+1}-E_{k+m}$. Thus we work with the quantity
	\begin{equation*}
	\underset{j\geq k+m} {\sup}(E_{j+1} - E_{j}),
	\end{equation*}
	which we expect to be bounded uniformly in $n$, $k$. Now, applying the Cauchy–Schwarz inequality yields,
	\begin{equation}\label{b_3}
	\begin{split}
	&\int_{\sigma_{k+1}^M}dY\int_{\sigma^M\setminus\sigma_{k+m}^M}\dfrac{1}{|Y'-Y|^5}|\partial_X^i\tilde{\Psi}\big|_{X=M}|dY'\\
	&\quad\leq C\left(\int_{\sigma_{k+1}^M}dY\int_{\sigma^M\setminus\sigma_{k+m}^M}\dfrac{1}{|Y'-Y|^{5+2\eta}}dY'\int_{\sigma^M\setminus\sigma_{k+m}^M}\dfrac{1}{|Y'-Y|^{5-2\eta}}|\partial_X^i\tilde{\Psi}\big|_{X=M}|^2dY'\right)^{1/2}\\
&\quad\leq C_{\eta}m^{-2+\eta}\left(\int_{\sigma_{k+1}^M}\int_{\sigma^M\setminus\sigma_{k+m}^M}\dfrac{1}{|Y'-Y|^{5+2\eta}}dY'dY\times\underset{j\geq k+m}{\sup}(E_{j+1}-E_j)\right)^{1/2}.
	\end{split}
	\end{equation}
	The previous result is obtained using the following computations: for $Y\in\sigma_{k+1}^M$
	\begin{equation}\label{b_2}
	\begin{split}
	\int_{\sigma^M\setminus\sigma_{k+m}^M}\dfrac{1}{|Y'-Y|^{5-2\eta}}|\partial_X^i\tilde{\Psi}\big|_{X=M}|^2dY'&=\sum\limits_{j\geq k+m}\int_{\sigma^M_{j+1}\setminus\sigma_{j}^M}\dfrac{1}{|j-Y|^{5-2\eta}}|\partial_X^i\tilde{\Psi}\big|_{X=M}|^2dY'\\
	&\leq C\sum\limits_{j\geq k+m}\left(E_{j+1}-E_{j}\right)\dfrac{1}{|j-Y|^{5-2\eta}}\\
	&\leq C\underset{j\geq k+m}{\sup}\left(E_{j+1}-E_{j}\right)\sum\limits_{j\geq k+m}\dfrac{1}{|j-Y|^{5-2\eta}}\\
	&\leq C\underset{j\geq k+m}{\sup}\left(E_{j+1}-E_{j}\right)\sum\limits_{j-k\geq m+1}\dfrac{1}{|j-k|^{5-2\eta}}\\
	&\leq C_{\eta}m^{-4+2\eta}\underset{j\geq k+m}{\sup}\left(E_{j+1}-E_{j}\right).
	\end{split}
	\end{equation}
	The series above correspond to the Hurwitz zeta function which is absolutely convergent for $\eta\in\left.\right]0,2\left[\right.$. On the other hand,
\begin{equation}\label{b_1}
\begin{split}
\int_{\sigma_{k+1}^M}\int_{\sigma^M\setminus\sigma_{k+m}^M}\dfrac{1}{|Y'-Y|^{5+2\eta}}dY'dY\leq C\int\limits_{\mathbb{R}\setminus\left[0,1\right]}\dfrac{dX}{X^{4+2\eta}}\leq C_{\eta}< +\infty.
\end{split}
\end{equation}
	Estimate \eqref{lemma_hf} is easily obtained from \eqref{b_3}, \eqref{b_2} and \eqref{b_1}.
\end{proof}
Applying several times Lemma \ref{DP:lemma_3.1} combined with \eqref{FP} gives
\begin{equation}
E_k\leq C_{\eta,g}\left(\sqrt{k+1}E_{k+1}^{1/2}+E_{k+1}^{1/2}(E_{k+m+1}-E_{k})^{1/2}+m^{-2+\eta}E_{k+1}^{1/2}\left(\underset{j\geq k+m}{\sup}(E_{j+1}-E_j)\right)^{1/2}\right),
\end{equation}
for $k\geq m\geq 1$. Since $E_k$ is a monotonically increasing function with respect to $k$, we have
\begin{equation*}
E_{k+1}\leq E_k+\left(E_{k+m+1}-E_{k}\right).
\end{equation*}
Furthermore, \begin{equation*}
E_{k+m+1}-E_{k}=\sum\limits_{k\leq j\leq k+m}\left(E_{j+1}-E_{j}\right)\leq m\underset{k\leq j\leq k+m}{\sup}\left(E_{j+1}-E_{j}\right).
\end{equation*} Taking $C=\underset{\eta\in]0,2[}{\max}C_\eta$ and using Young's inequality gives that for all $\nu>0$ there exists $C_\nu$, such that
\begin{equation}
E_k\leq \nu E_k+C_{\nu,g}\left(k+1+m\underset{k\leq j\leq k+m}{\sup}\left(E_{j+1}-E_{j}\right)+\frac{1}{m^{4-2\eta}}\underset{j\geq k+m}{\sup}(E_{j+1}-E_j)\right).
\end{equation}
Inequality \eqref{DP:3-1} follows from choosing $\nu$ sufficiently small. 
		\end{itemize}

\end{proof}
     \section{Nonlinear boundary layer formation near the western coast}\label{western_nl}
     This section is devoted to showing the well-posedness of the western boundary layer when the model presents an advection term. We study the problem 

\begin{equation}\label{e:nonlinear_west}
\begin{array}{rcl}
\partial_{X_w}\Psi_w+Q_w(\Psi_w,\Psi_w)-\Delta_w^{2}\Psi_w&=&0,\quad\textrm{in}\quad \omega_w^+\cup\omega^-\\
\left[\Psi_w\right]|_{X_w=0}&=&\phi,\\
\left[\partial_{X_w}^k\Psi_w\right]|_{X_w=0}&=&0,\quad k=1,\ldots,3,\\
\Psi_w\big|_{X=-\gamma(Y)}=0,&&\dfrac{\partial\Psi_w}{\partial n_w}\big|_{X_w=-\gamma_w(Y)}=0.
\end{array}
\end{equation}
where $\phi\in W^{2,\infty}(\mathbb{R})$ and the nonlinear term is given by $$Q_w(\Psi_w,\tilde{\Psi}_w)=\nabla_{w}^{\perp}\Psi_w\cdot\nabla_{w}(\Delta_{w}\tilde{\Psi}_w)=\nabla_{w}^\perp\left[\left(\nabla_{w}^{\perp}\Psi_w\cdot\nabla_{w}\right)\nabla_{w}^\perp\tilde{\Psi}_w\right].$$  As before, in this section we will write $X$ instead of $X_w$.

The proof of Theorem \ref{theorem:existence} for the nonlinear problem under a smallness assumption follows the general scheme presented in Section \ref{nl_linearized_method}. There are three main parts in our analysis: showing the well-posedness of the nonlinear problem in the half-space; proving the existence, uniqueness and regularity of the solution in the rough channel; and, finally, connecting both solutions at the ``transparent'' interface. 

Later, in Section \ref{s:linearized} special attention will be paid to linearized problems in the western boundary layer domain. In its general form, this kind of problem is crucial in constructing the approximate solution since it describes the behavior of higher-order western profiles and additional correctors.

     \subsection{Nonlinear problem in the half-space}\label{fixed_point}
In this section, the well-posedness of the system \eqref{pb:halfspace_nl} in the half-space is established under a smallness assumption. Namely, we study the problem
\begin{equation}\label{pb:halfspace_nl_w}
\left\{\begin{array}{rcl}
Q_w(\Psi_{w},\Psi_{w})+\partial_X\Psi_{w}-\Delta_w^{2}\Psi_{w}&=&0,\quad\textrm{in}\;X>M\\
\Psi_{w}\big|_{\sigma^M_w}=\psi_0,&&\partial_X\Psi_{w}\big|_{\sigma^M_w}=\psi_1.
\end{array}\right.
\end{equation}
We shall solve  (\ref{pb:halfspace_nl_w}) by means of a fixed point theorem using the a priori estimated provided in Theorem \ref{Theorem2_DGV2017}. Basically,
we use a
contraction mapping argument in a suitable Banach space which will be norm invariant under the
transformations that preserve the set of solutions, mainly the
translations with respect to the $X$ variable.  We introduce the functional spaces:
\begin{equation}\label{def_bs}
\mathcal{H}^m:=\left\{f\in H^m_{\mathrm{uloc}}(\mathbb{R}^2_+):\:\:\left\|e^{\delta X}f \right\|_{H^m_{\mathrm{uloc}}}<+\infty\right\}, \quad m\geq 0,
\end{equation}
with the norm $\|f\|_{\mathcal{H}^m} = C_m \|e^{\delta X}f\|_{H^m_{\mathrm{uloc}}}$. Here, $\delta>0$ is the one in Theorem \ref{Theorem2_DGV2017} and the constant $C_m$ is chosen so that if $g,f\in \mathcal{H}^{m+1}$
\begin{equation*}
\|(\nabla_w^{\perp}f\cdot\nabla_w)\nabla_w^{\perp}g\|_{\mathcal{H}^{m-1}} \leq \|f\|_{\mathcal{H}^{m+1}}\|g\|_{\mathcal{H}^{m+1}} .
\end{equation*} 

We show the following:
\begin{proposition}\label{prop:fix}
Let $m \in\mathbb{N}$, $m \gg 1$. There is $\delta_0>0$ such that for all $\psi_{0} \in H^{m+3/2}_{\mathrm{uloc}}(\mathbb{R})$ and $\psi_{1} \in H^{m+1/2}_{\mathrm{uloc}}(\mathbb{R})$, 
\begin{equation}\label{eq:41}
\|\psi_{0} \|_{H^{m+3/2}_{\mathrm{uloc}}(\mathbb{R})}+\|\psi_{1} \|_{H^{m+1/2}_{\mathrm{uloc}}(\mathbb{R})}<\delta_0,
\end{equation}
the system
\begin{equation}\label{halfspace_M=0}
\left\{\begin{array}{rclr}
Q_w(\Psi_{w},\Psi_{w})+\partial_X\Psi_{w}-\Delta_w^{2}\Psi_{w}&=&0,&\textrm{in}\quad\mathbb{R}^2_+\\
\Psi_{w}\big|_{X=0}=\psi_0,&&\partial_X\Psi_{w}\big|_{X=0}=\psi_1.&
\end{array}\right.
\end{equation}
has a unique solution in $\mathcal{H}^{m+2}$.

\end{proposition}
\begin{proof}
For any functions $\psi_{0} \in H^{m+3/2}_{\mathrm{uloc}}(\mathbb{R})$ and $\psi_{1} \in H^{m+1/2}_{\mathrm{uloc}}(\mathbb{R})$, let the operator $T_{(\psi_{0},\psi_{1})}$ be defined as follows: given a function $\Psi\in \mathcal{H}^{m+2}$, set $T_{(\psi_{0},\psi_{1})}(\Psi)=\tilde{\Psi}$, where $\tilde{\Psi}$ is the solution of (\ref{p:linear_nonhomgeneous0}) when $F=-\nabla_{w}^\perp\left[(\nabla_w^{\perp}\Psi\cdot\nabla_w)\nabla_w^{\perp}\Psi\right]$.  According to Theorem \ref{Theorem2_DGV2017}, there exists a constant $C_0$ such that for all $\tilde{\Psi}\in \mathcal{H}^{m+2}$,
\[
\|T_{(\psi_{0},\psi_{1})}(\Psi)\|_{\mathcal{H}^{m+2}}\leq C_0\left(\|\psi_{0} \|_{H^{m+3/2}_{\mathrm{uloc}}(\mathbb{R})}+\|\psi_{1} \|_{H^{m+2}_{\mathrm{uloc}}(\mathbb{R})}+\|\Psi\|^{2}_{\mathcal{H}^{m+2}}\right).
\]
The previous inequality results from taking into account that when $\bar{\delta}=2\delta$
\begin{equation*}
\|e^{\bar{\delta}X} F\|_{H^{m-2}_{\mathrm{uloc}}(\mathbb{R}^{2})}\leq C_{\bar{\delta}}\| e^{\bar{\delta} X}(\nabla_w^{\perp}\Psi\cdot\nabla_w)\nabla_w^{\perp}\Psi\|_{H^{m-1}_{\mathrm{uloc}}(\mathbb{R}^{2})}\|<\|\Psi\|^{2}_{\mathcal{H}^{m+2}}.
\end{equation*}

Let us verify that $T_{(\psi_{0},\psi_{1})}$ is a strict contraction under the smallness assumption \eqref{eq:41}. This implies the function has a fixed point in a closed ball of $\mathbb{R}^2_+$. Let $\delta_0 < 1/(4C_0^{2})$, and suppose that (\ref{eq:41}) holds. Thanks to the assumption on $\delta_0$, there exists
$R_0 > 0$ such that
\begin{equation}\label{eq:42}
C_0\left(\delta_0+R_0^{2}\right)\leq R_0.
\end{equation}
Hence, $R_0$ belongs to $[R_{-}, R_+]$, where
\[
R_{\pm}=\frac{1}{2C_0}\left(1\pm\sqrt{1-4\delta_0C_0^{2}}\right).
\]
Therefore $0 < R_- < (2C_0^{-1}$, and it is always possible to choose $2R_0C_0<1$. (\ref{eq:41}) and (\ref{eq:42}) imply
\[
\|\Psi\|_{\mathcal{H}^{m+2}}\leq R_0\Longrightarrow\|T_{(\psi_{0},\psi_{1})}(\Psi)\|_{H^{m+1}_{\mathrm{uloc}}}<R_0.
\]
Now, let $\Psi_i\in \mathcal{H}^{m+2}$ be a function satisfying   $T_{(\psi_{0},\psi_{1})}(\Psi_i)=\tilde{\Psi}_i$, where $\Psi_i$ is the solution of \eqref{pb:halfspace_linear_nonhomogeneous} when $F^i=(\nabla_w^{\perp}\Psi_i\cdot\nabla_w)\nabla_w^{\perp}\Psi_i$, $i=1,2$. If $\|\Psi_{1}\|_{H^{m+2}_{\mathrm{uloc}}}\leq R_0$,  $\|\Psi_{2}\|_{H^{m+2}_{\mathrm{uloc}}}\leq R_0$ and $\bar{\Psi} = T_{(\psi_{0},\psi_{1})}(\Psi_{1})-T_{(\psi_{0},\psi_{1})}(\Psi_{2})$, we have that $\bar{\Psi}$ is a solution of
\eqref{pb:halfspace_linear_nonhomogeneous} with $\bar{\Psi}\big|_{X=0}=\partial_{X}\bar{\Psi}\big|_{X=0}=0$ and source term $F^{1}-F^{2}=\nabla_{w}^\perp\left[(\nabla_w^{\perp}\Psi_{1}\cdot\nabla_w)\nabla_w^{\perp}\Psi_{1}\right]-\nabla_{w}^\perp\left[(\nabla_w^{\perp}\Psi_{2}\cdot\nabla_w)\nabla_w^{\perp}\Psi_{2}\right]$. Thus, using once again
Theorem \ref{Theorem2_DGV2017},
\[
\|T_{(\psi_{0},\psi_{1})}(\Psi^{1})-T_{(\psi_{0},\psi_{1})}(\Psi_{2})\|_{\mathcal{H}^{m+2}}\leq C_0\|F^{1}-F^{2}\|_{\mathcal{H}^{m-1}}\leq2C_0R_0\|\Psi_{1}-\Psi_{2}\|_{\mathcal{H}^{m+2}}.
\]

Since $2C_0R_0 < 1$, $T_{(\psi_{0},\psi_{1})}$ is a contraction over the ball of radius $R_0$ in $\mathcal{H}^{m+2}$. We can then assert that $T_{(\psi_{0},\psi_{1})}$ has a fixed point in $\mathcal{H}^{m+2}$ as a result of Banach's fixed point theorem which concludes the proof of Proposition \ref{prop:fix}.
\end{proof}

	\begin{remark}
	We can retrieve the solution for $X>M$ when $M>0$ thanks to the problem being invariant with respect to translations along the $X$-axis. Let $\Psi_0$ be the solution of \eqref{halfspace_M=0}. Then, the solution  $\Psi_M$ of \eqref{pb:halfspace_nl_w} in the half-space $X>M$, $M>0$, satisfies  $\Psi_M=\Psi_0(X-M)$.
\end{remark}

     \subsection{The problem in the rough channel}\label{s:rough}

The goal in this section is to prove, by the truncation technique employed in \cite{Gerard-Varet2010}, the existence of a solution of problem
\begin{equation}
\left\{\begin{array}{rcl}\label{eq_left1}
	Q_w(\Psi^{-}_{w},\Psi^{-}_{w})+\partial_X\Psi^{-}_{w}-\Delta_w^{2}\Psi^{-}_{w}&=&0,\quad \textrm{in}\quad\omega^{b}\setminus\sigma_w\\
\left[\Psi^{-}_w\right]\big|_{\sigma_w}=\phi,&&\left[\partial_X^{k}\Psi^{-}_w\right]\big|_{\sigma_w}=0,\;k=1,\ldots,3,\\
\Psi^{-}_{w}\big|_{X=-\gamma_w(Y)}&=&\dfrac{\partial\Psi_{w}^{-}}{\partial n_w}\big|_{X=-\gamma_w(Y)}=0,\\ 
\mathcal{A}_2[\Psi_{w}^{-}\big|_{\sigma_w^M},\partial_X\Psi_{w}^{-}\big|_{\sigma_w^M}]=\rho_2,&&\mathcal{A}_3[\Psi_{w}^{-}\big|_{\sigma_w^M},\partial_X\Psi_{w}^{-}\big|_{\sigma_w^M}]=\rho_3,
\end{array}\right.
\end{equation}
where $\omega^{b}_w=\omega_w\setminus\left(\{X>M\}\times\mathbb{R}\right)$ denotes the rough channel. We recall that \[\begin{split}
\mathcal{A}_2[\Psi\big|_{\sigma_w^M},\partial_X\Psi\big|_{\sigma_w^M}]&=(1+\alpha^2)\Delta_w\Psi,\\
 \mathcal{A}_3[\Psi\big|_{\sigma_w^M},\partial_X\Psi\big|_{\sigma_w^M}]&=-\left((1+\alpha^2)\partial_X-2\alpha\partial_{Y}\right)\Delta_w\Psi+\partial_{Y}\left(\displaystyle\frac{|\nabla_w^{\perp}\Psi|^{2}}{2}\right)\\
 &\hspace*{1in}+(\nabla_w^{\perp}\Psi\cdot\nabla_w)\left((1+\alpha^2)\partial_{X}-\alpha\partial_{Y}\right)\Psi+\displaystyle\frac{\Psi}{2}.
\end{split}\]

This part corresponds to step (NL2) in Section \ref{nl_linearized_method}. Although the idea is the same as for the linear case, an important difference resides in working indirectly with the values of the Poincaré-Steklov operators. Indeed, here we  ``join'' the solutions obtained at both sides of the artificial boundary using the implicit function theorem. To do so, higher regularity estimates of the solution near $\sigma^M_w$ are essential. 

As for the tensors at $X=M$, since we will need to construct solutions in $H^{m+2}_{\mathrm{uloc}}$ and due to the form of the differential operators, we look for $\rho_2$ and $\rho_3$ in the $H^{m-1/2}_{\mathrm{uloc}}(\mathbb{R})$ and $H^{m-3/2}_{\mathrm{uloc}}(\mathbb{R})$, respectively. We then claim that the following result holds:
\begin{proposition}\label{Lemma_14}
Let $m\gg 1$ be arbitrary. There exists $\delta> 0$ such that for all $\phi\in W^{2,\infty}(\mathbb{R})$, $\rho_2\in H^{m-1/2}_{\mathrm{uloc}}(\mathbb{R})$ and $\rho_3\in H^{m-3/2}_{\mathrm{uloc}}(\mathbb{R})$ with $\|\phi\|_{W^{2,\infty}(\mathbb{R})}<\delta$, $\|\rho_2\|_{ H^{m-1/2}_{\mathrm{uloc}}(\mathbb{R})}+\|\rho_3\|_{ H^{m-3/2}_{\mathrm{uloc}}(\mathbb{R})}<\delta$, system (\ref{eq_left1}) has a unique solution $\Psi_w^{-}\in H^{2}_{\mathrm{uloc}}(\omega_w^b\setminus\sigma_w)$.

Moreover, $\Psi_w^{-}\in H^{m+2}_{\mathrm{uloc}}((M',M)\times \mathbb{R})$, for all $M'\in ]0,M[$ and
\[
	\|\Psi_w^{-}\|_{H^{m+2}_{\mathrm{uloc}}((M',M)\times \mathbb{R})}\leq C_{M'}\left(\|\phi\|_{W^{2,\infty}}+\|\rho_2\|_{ H^{m-1/2}_{\mathrm{uloc}}}+\|\rho_3\|_{ H^{m-3/2}_{\mathrm{uloc}}}\right).
	\]
\end{proposition}

	\begin{proof} First, we will briefly discuss the existence and uniqueness of a solution in $H^2$, as well as the validity of the estimate. Then, the regularity result will be presented. Throughout the proof, we will drop the $w$ from the notation when there is no confusion.
		
\textit{\textbf{Step 1. Existence and uniqueness of the solution.}}  We look for the solution $\tilde{\Psi}=\Psi^{-}-\Psi^{L}$ of  the problem 
\begin{eqnarray}\label{pb:left_lift}
Q_w(\tilde{\Psi},\tilde{\Psi}+\Psi^{L})+Q_w(\Psi^{L},\tilde{\Psi})+\partial_{X}\tilde{\Psi}-\Delta^{2}_w\tilde{\Psi}&=&F^L\quad\textrm{in}\;\omega^b,\nonumber\\
\mathcal{A}_2[\tilde{\Psi}\big|_{\sigma_w^M},\partial_X\tilde{\Psi}\big|_{\sigma_w^M}]=\rho_2,&&
\mathcal{A}_3[\tilde{\Psi}\big|_{\sigma_w^M},\partial_X\tilde{\Psi}\big|_{\sigma_w^M}]=\rho_3,\nonumber\\
\tilde{\Psi}\big|_{X=-\gamma(Y)}=0,&&\dfrac{\partial\tilde{\Psi}}{\partial n_w}\big|_{X=-\gamma(Y)}=0.
\end{eqnarray}
Here, $\Psi^{L}$ is defined as in \eqref{lift_lin} for $g_0=\phi$ and $g_k\equiv0$ for $k=1,2,3$. In \eqref{pb:left_lift}, $F^L$ denotes $\Delta^{2}_w(\Psi^{L})-\partial_{X}(\Psi^{L})-Q_w(\Psi^{L},\Psi^{L})$. Notice that thanks to the regularity and smallness assumptions on $\phi$, we have, 
\begin{align*}
\|F^L\|_{L^2_{\mathrm{uloc}}(\omega^b)}&\leq C\left(\|\phi\|_{W^{2,\infty}}+\|\phi\|_{W^{2,\infty}}^2\right)\leq C\|\phi\|_{W^{2,\infty}}.
\end{align*}
In the nequality above, the constant $C$ depends on $\alpha$ and $\|\gamma_w\|_{W^{2,\infty}}$.  Before computing the a priori estimates, we write the weak formulation of (\ref{pb:left_lift}). 

\begin{definition}\label{p:weak_formulation}
	Let $\mathcal{V}$ be the space of functions $\varphi\in C^\infty_0(\overline{\omega^b})$ such that $\mathrm{Supp}\varphi\cap \partial\Omega=\emptyset$ and $\mathcal{D}^2_0(\omega^b)$ its completion for the norm $\|\Psi\|=\|\Delta_w \Psi\|_{L^2}$. Define for $(\Psi,\tilde{\Psi},\varphi)\in \mathcal{D}^2_0\times\mathcal{D}^2_0\times\mathcal{V}$, the trilinear form  $b(\Psi,\tilde{\Psi},\varphi)=-\int_{\Omega}(\nabla_{w}^{\perp}\Psi\cdot\nabla_{w})\nabla_{w}^{\perp}\tilde{\Psi}\cdot\nabla_w^{\perp}\varphi$. A function $\tilde{\Psi}\in H^2_{\mathrm{uloc}}(\omega^b)$ is a solution of \eqref{pb:left_lift}  if it satisfies the homogeneous conditions $\tilde{\Psi}\big|_{\Gamma_w}=\partial_{n_w}\tilde{\Psi}\big|_{\Gamma_w}=0$ at the rough boundary, and if, for all $\varphi\in \mathcal{V}$, we get
	\begin{eqnarray}\label{weak_formulation}
	&&\int_{\omega^b_w}\partial_{X}\tilde{\Psi}\varphi+b(\tilde{\Psi},\tilde{\Psi},\varphi)-\int_{\omega^b_w}\Delta_w\tilde{\Psi}\Delta_w\varphi\nonumber\\
	&&\hspace*{1in}=-\int_\mathbb{R}\left(\rho_3-\dfrac{\tilde{\Psi}}{2}\bigg|_{X=M}\right)\varphi\big|_{X=M}\;dY\\
	&&\hspace*{1.15in}+\int_\mathbb{R}\left(\dfrac{|\nabla_w^\perp\tilde{\Psi}|^2}{2}\right)\bigg|_{X=M}\partial_{Y}\tilde{\tilde{\Psi}}\bigg|_{X=M}\;dY-\int_\mathbb{R}\rho_2\big|_{X=M}\partial_{X}\varphi\big|_{X=M}\;dY.\nonumber
	\end{eqnarray}
\end{definition}

We consider the system (\ref{pb:left_lift}) in $\omega_{n}$
\begin{eqnarray}\label{pb:left_n}
Q_w(\tilde{\Psi}_{n},\tilde{\Psi}_{n}+\Psi^{L}_n)+Q_w(\Psi^{L}_n,\tilde{\Psi}_n)+\partial_{X}\tilde{\Psi}_n-\Delta^{2}_w\tilde{\Psi}_n&=&F^L_n\quad\textrm{in}\;\omega_n,\nonumber\\
\tilde{\Psi}_{n}&=&0,\quad\textrm{in}\;\omega^{b}\setminus\omega_n,\nonumber\\
\tilde{\Psi}_{n}\big|_{X=-\gamma_n(Y)}=\partial _{n}\tilde{\Psi}_{n}\big|_{X=-\gamma_n(Y)}&=&0,\\ 
\mathcal{A}_2[\tilde{\Psi}_n\big|_{\sigma^M_n}, \partial_X\tilde{\Psi}_n\big|_{\sigma^M_n}]&=&\rho_{2},\nonumber\\
\mathcal{A}_3[\tilde{\Psi}_n\big|_{\sigma^M_n}, \partial_X\tilde{\Psi}_n\big|_{\sigma^M_n}]&=&\rho_{3}.\nonumber
\end{eqnarray}

The domain $\omega_n$ and its components are the same as in  \eqref{dom_n}.

By taking $\tilde{\Psi}_n$ as a test function in (\ref{pb:left_n}), we get a first energy estimate on $\tilde{\Psi}_n$
\begin{eqnarray}\label{Prange_326}
\|\Delta_w\tilde{\Psi}_n\|^{2}_{L^{2}(\omega_n)}&=&b(\tilde{\Psi}_n,\Psi^{L}_n,\tilde{\Psi}_n)+\int_{\sigma_n^{M}}\rho_3\tilde{\Psi}_n+\int_{\sigma_n^{M}}\rho_2\partial_X\tilde{\Psi}_n-\int_{\omega_{n}}F^L\tilde{\Psi}_{n}\nonumber\\
&\leq&C_1\|\Delta_w\phi\|_{L^\infty(\omega_n)}\|\nabla_{w}^{\perp}\tilde{\Psi}_n\|_{L^{2}(\omega_n)}^2+C_2 \sqrt{n}\left(\|\rho_3\|_{H^{m-3/2}_{\mathrm{uloc}}(\mathbb{R})}\|\tilde{\Psi}_n\big|_{X=M}\|_{L^{2}([-n,n])}\right.\nonumber\\
&&\quad\left.+\|\rho_2\|_{H^{m-1/2}(\omega_n)}\|\nabla_{w}\tilde{\Psi}_n\big|_{X=M}\|_{L^{2}(([-n,n])}\right)
+\|F^L\|_{L^{2}_{\mathrm{uloc}}}\|\tilde{\Psi}_n\|_{L^{2}(\omega_n)},
\end{eqnarray}
using the Cauchy-Schwartz and Poincar\'e inequalities over $\omega_n$. Moreover, we have
\begin{equation*}
\|\Delta_w\tilde{\Psi}_n\|^{2}_{L^{2}(\omega_n)}\leq C\left(\|\phi\|_{W^{2,\infty}}+\|\rho_3\|_{H^{m-3/2}_{\mathrm{uloc}}}+\|\rho_2\|_{H^{m-1/2}_{\mathrm{uloc}}}\right)\|\tilde{\Psi}_n\|_{H^{2}(\omega_n)}.
\end{equation*}
Notice that as a consequence of the smallness assumption on $\|\phi\|_{H^2(\omega^{b})}$, the first term on the r.h.s of inequality \eqref{Prange_326} can be absorbed by the one on the l.h.s for $\delta$ small enough. Then, using Poincar\'e inequality over the whole channel yields
\begin{equation}\label{DP:2-37}
E_n:=\int_{\omega^{b}}|\Delta_w\tilde{\Psi}_n|^{2}\leq\int_{\omega_n}|\Delta_w\tilde{\Psi}_n|^{2}\leq C_0n,
\end{equation}
where constant $C_0$ depends on $\alpha$,  $\|\phi\|_{W^{2,\infty}}$, $\|\rho_3\|_{H^{m-3/2}_{\mathrm{uloc}}}$, $\|\rho_2\|_{H^{m-1/2}_{\mathrm{uloc}}}$ and $\|\gamma\|_{W^{2,\infty}}$. 
The existence of $\tilde{\Psi}_n$ in $H^2(\omega_n)$ follows. 

Following the same reasoning as in the linear case ((L5) in Section \ref{l_method} and Section \ref{s:rough_linear_west}), we establish an induction inequality on $(E_k^{n})_{k\in\mathbb{N}}$ for all $n\in\mathbb{N}$. Recall that
\[
E_k^{n}:=\int_{\omega^b}\chi_k|\Delta_w\tilde{\Psi}_n|^{2},\]
where $\chi_k\in\mathcal{C}^{\infty}_0 (\mathbb{R})$ is a cut-off function in the tangential variable such that $\mathrm{Supp}\chi_k \subset [-k-1,k+1]$ and $\chi_n\equiv 1$ on $[-k,k]$ for $k\in\mathbb{N}$. The induction relation allows one to obtain a uniform bound on the $E_k$, from which we deduce a $H^2_{\mathrm{uloc}}$ bound on $\tilde{\Psi}_n$ uniformly in $n$. From this, an exact solution follows by compactness, see (L5) in Section \ref{l_method}. 

Here, we show the inequality for all $k\in\left\{1,\ldots,n\right\}$ 
	\begin{equation}\label{inequality}
	E_k^{n}\leq C_1\left((E_{k+1}^{n}-E_k^{n})^{3/2}+(E_{k+1}^{n}-E_k^{n})+\left(\|\rho_2\|_{H^{m-1/2}_{\mathrm{uloc}}}^{2}+\|\phi\|_{W^{2,\infty}}^{2}+\|\rho_3\|_{H^{m-3/2}_{\mathrm{uloc}}}^{2}\right)(k+1)\right),
	\end{equation}
where  $C_1$ is a constant depending only on the characteristics of the domain. Then, by backwards induction on $k$, we deduce that
	\begin{equation*}
	E_k^{n}\leq Ck,\quad\forall k\in\{k_0,\ldots,n\},
	\end{equation*}
	where $k_0\in\mathbb{N}$ is a large, but fixed integer (independent of $n$)  and $E_{k_0}^{n}$ is bounded uniformly in $n$ for a constant $C$ depending on $\omega^{b}$, $\phi$ and $\rho_i$, $i=2,3$. Then, we use the fact that the derivation of energy estimates is invariant by translation on the tangential variable to prove that the uniform boundness holds not only for a maximal energy of size $k_0$, but for all $k$, similarly to (L5) in Section \ref{l_method}.
	
	To lighten notations in the subsequent proof, we shall denote $E_k$ instead of $E^{n}_k$ .
	\begin{itemize}
	\item\textit{Energy estimates.} This part is devoted to the proof of (\ref{inequality}). We carry out the energy
estimate on the system (\ref{pb:left_lift}), focusing on having constants uniform in $n$ as explained before. Since the linear part of the equation has already been analyzed on Section \ref{western_l}, we discuss in detail only the nonlinear terms. In fact, the main issue consists in handling the quadratic terms $Q_w(\tilde{\Psi},\tilde{\Psi}+\Psi^{L}),\;Q_w(\Psi^{L},\tilde{\Psi})$, which justifies the
presence of the $|\nabla_w^{\perp}\tilde{\Psi}|^{2}$ in one of the tensors at $\sigma^M_w$. Plugging $\varphi=\chi_k\tilde{\Psi}$ into the nonlinear terms of (\ref{weak_formulation}) gives
\begin{eqnarray}
b(\tilde{\Psi}+\Psi^{L},\tilde{\Psi},\chi_k\tilde{\Psi})&=&-\int_{\omega}\chi_k(\nabla_w^{\perp}(\tilde{\Psi}+\Psi^{L})\cdot\nabla_w)\nabla_w^{\perp}\tilde{\Psi}\cdot\nabla_w^{\perp}\tilde{\Psi}\nonumber\\
&&\hfill-\int_{\omega}(\nabla_w^{\perp}(\tilde{\Psi}+\Psi^{L})\cdot\nabla_w)\nabla_w^{\perp}\tilde{\Psi}\cdot(\nabla_w^{\perp}\chi_k)\tilde{\Psi}\label{first_nlt},\\
b(\tilde{\Psi},\Psi^{L},\chi_k\tilde{\Psi})&=&-\int_{\omega}\chi_k(\nabla_w^{\perp}\tilde{\Psi}\cdot\nabla_w)\nabla_w^{\perp}\Psi^{L}\cdot\nabla_w^{\perp}\tilde{\Psi}\nonumber\\
&&\hfill-\int_{\omega}\left(\nabla_w^{\perp}\tilde{\Psi}\cdot\nabla\right)\nabla_w^{\perp}\Psi^{L}\cdot(\nabla_w^{\perp}\chi_k)\tilde{\Psi}\label{second_nlt},
\end{eqnarray}

To bound each one of the terms we will frequently use the Sobolev inequality for all $\omega'\subset\omega^b$,
\begin{equation}\label{sobolev_embedding}
\forall u\in H^1(\omega'),\:u\big|_{\Gamma}=0,\quad\|u\|_{L^q(\omega')}\leq C_q\|\nabla u\|_{L^2(\omega')},\quad q\in[1,+\infty).
\end{equation}
The constant $C_q$ does not depend on $\omega '$.
Let us now illustrate the procedure for the first term in the l.h.s. of \eqref{second_nlt}. By using Cauchy-Schwarz inequality and the properties of $\chi_k$, we find that 
\begin{eqnarray*}
\left|\int_{\omega}\chi_k(\nabla_w^{\perp}\tilde{\Psi}\cdot\nabla_w)\nabla_w^{\perp}\Psi^{L}\cdot\nabla_w^{\perp}\tilde{\Psi}\right|&\leq& C\|\phi\|_{W^{2,\infty}(\omega_{k+1})}\|\nabla_{w}^{\perp}\tilde{\Psi}\|^2_{L^2(\omega_{k+1})}\\
&\leq C&\|\phi\|_{W^{2,\infty}}E_{k+1},
\end{eqnarray*}
where $C$ is a strictly positive constant depending on $\alpha$, $M$ and $\|\gamma_w\|_{W^{2,\infty}}$. On the other hand,
\begin{eqnarray*}
-\int_{\omega}\chi_k(\nabla_w^{\perp}(\tilde{\Psi}+\Psi^{L})\cdot\nabla_w)\nabla_w^{\perp}\tilde{\Psi}\cdot\nabla_w^{\perp}\tilde{\Psi}&=-&\int_\omega\chi_k\nabla_w^{\perp}(\tilde{\Psi}+\Psi^{L})\cdot\nabla_w\left(\frac{|\nabla_w^{\perp}\tilde{\Psi}|^{2}}{2}\right)\\
&=&\int_\omega\nabla_w\chi_k\cdot\nabla_w^{\perp}(\tilde{\Psi}+\Psi^{L})\left(\frac{|\nabla_w^{\perp}\tilde{\Psi}|^{2}}{2}\right)\\
&&\hfil-\int_{\mathbb{R}}\chi_k\left(\frac{|\nabla_w^{\perp}\tilde{\Psi}|^{2}}{2}\partial_{Y}\tilde{\Psi}\right)\Bigg|_{X=M}\;dY.
\end{eqnarray*}
Proceeding as before, it can be easily checked that the first term on the right-hand side is bounded by $C(E_{k+1}-E_{k})^{3/2}+C\|\phi\|_{W^{2,\infty}}(E_{k+1}-E_{k})$.  

Applying integration by parts on the second integral gives
\begin{eqnarray}\label{bord_1}
-\int_{\mathbb{R}}\chi_k\left(\frac{|\nabla_w^{\perp}\tilde{\Psi}|^{2}}{2}\partial_{Y}\tilde{\Psi}\right)\Bigg|_{X=M}\;dY
&=&\int_{\mathbb{R}}\chi_k\partial_{Y}\left(\frac{|\nabla_w^{\perp}\tilde{\Psi}|^{2}}{2}\right)\Bigg|_{X=M}\tilde{\Psi}\;dY\\
&&+\int_{\mathbb{R}}\partial_Y\chi_k\left(\frac{|\nabla_w^{\perp}\tilde{\Psi}|^{2}}{2}\right)\Bigg|_{X=M}\tilde{\Psi}\;dY.\nonumber
\end{eqnarray}
The first term can be grouped with other boundary terms stemming from the bilaplacian, while the second is bounded by $C(E_{k+1}-E_{k})^{3/2}$ as a consequence of the trace theorem.

There remains to consider the r.h.s of \eqref{first_nlt} and \eqref{second_nlt}, i.e, the terms $\int_{\omega}(\nabla_w^{\perp}(\tilde{\Psi}+\Psi^{L})\cdot\nabla_w)\nabla_w^{\perp}\tilde{\Psi}\cdot(\nabla_w^{\perp}\chi_k)\tilde{\Psi}$ and $\int_{\omega}\left(\nabla_w^{\perp}\tilde{\Psi}\cdot\nabla\right)\nabla_w^{\perp}\Psi^{L}\cdot(\nabla_w^{\perp}\chi_k)\tilde{\Psi}$ which are bounded by $C(E_{k+1}-E_{k})^{3/2}+C\|\phi\|_{W^{1,\infty}}(E_{k+1}-E_k)$ and $C\|\phi\|_{W^{2,\infty}}(E_{k+1}-E_k)$, respectively. 

The linear terms defined on $\omega^b$ satisfy \eqref{bt_1} and are bounded by $C(E_{k+1}-E_k)$ as seen in \eqref{commut_bound}. Lastly, we deduce
\begin{equation}\label{bord_x}
\int_{\omega^b_w}\partial_X\tilde{\Psi}\chi_k\tilde{\Psi}=\dfrac{1}{2}\int_{\omega^b_w}\chi_k\partial_X(|\tilde{\Psi}|^2)=\dfrac{1}{2}\int_{\mathbb{R}}\chi_k|\Psi|^2\big|_{X=M}.
\end{equation}
From collecting the boundary terms coming from the bilaplacian with \eqref{bord_1} and \eqref{bord_x}, we get
\begin{equation*}
-\int_{\mathbb{R}}\chi_k\left(\tilde{\Psi}\rho_3+\partial_X\tilde{\Psi}\rho_2\right)\Bigg|_{X=M}.
\end{equation*}
The term above is bounded for any $\delta> 0$ by
\[
C \left(\|\rho_2\|_{L^{2}_{\mathrm{uloc}}}+\|\rho_3\|_{L^{2}_{\mathrm{uloc}}}\right)E_{k+1}+C'\left(\|\rho_2\|_{L^{2}_{\mathrm{uloc}}}+\|\rho_3\|_{L^{2}_{\mathrm{uloc}}}\right)(k+1),
\]
where $C, C'$ depend only on $M$, $\alpha$ and on $\|\gamma\|_{W^{2,\infty}}$. The computation of this bound relies on the trace theorem and Young's inequality. 
We are left with
\begin{eqnarray*}
\left|\int_{\omega^b_w}\chi_k F^L\tilde{\Psi}\right|
&\leq&C\|\phi\|_{W^{2,\infty}}E_{k+1}^{1/2}\sqrt{k+1}.
\end{eqnarray*}
The last bound is not optimal but it suffices for our purposes. Applying once again Young's inequality yields
\begin{equation*}
\left|\int_{\omega^b_w}\chi_k F^L\tilde{\Psi}\right|\leq C\|\phi\|_{W^{2,\infty}}E_{k+1}+C'\|\phi\|_{W^{2,\infty}}(k+1). 
\end{equation*}
For $\|\phi\|_{W^{2,\infty}}, \|\rho_2\|_{H^{m-1/2}_{\mathrm{uloc}}}$ and $\|\rho_3\|_{H^{m-3/2}_{\mathrm{uloc}}}$, gathering all the terms provides the following inequality
\begin{equation}\label{DP_3-3_al}
E_k\leq C_1(E_{k+1}-E_k)^{3/2}+C_2\|\phi\|_{W^{2,\infty}}(E_{k+1}-E_k)+C_3\left(\|\phi\|_{W^{2,\infty}}+\|\rho_2\|_{H^{m-1/2}_{\mathrm{uloc}}}+\|\rho_3\|_{H^{m-3/2}_{\mathrm{uloc}}}\right)(k+1).
\end{equation}
\item \textit{Induction.}
We aim to deduce from (\ref{inequality}) that there exists $k_0\in\mathbb{N}\setminus\{0\}$, $C>0$ such that, for all $n\in\mathbb{N}$
\begin{equation}\label{DP:3-2}
\int_{\omega_{k_0}}|\Delta_w\tilde{\Psi}_n|^{2}\leq C.
\end{equation}
Let $C_2(\phi)$ and $C_3(\phi,\rho_2,\rho_3)$ denote the coefficients associated to second and third $k$-dependent terms in inequality \eqref{DP_3-3_al}.  From (\ref{inequality}), we prove by downward induction on $k$, that there exists a positive constant $C_4$ depending only on $C_0$, $C_1$, $C_2(\phi)$ and $C_3(\phi,\rho_2,\rho_3)$, appearing respectively in (\ref{DP:2-37}) and (\ref{DP_3-3_al}), such that, for all $k> k_0$,
\begin{equation}\label{DP_3-3_n}
E_k\leq C_4C_3(\phi,\rho_2,\rho_3)(k+1).
\end{equation}
Here, $C_4$ is independent of $n$, $k$. 

Note that (\ref{DP_3-3_n}) is holds for $k= n$ if $C_4> C_0(C_3(\phi,\rho_2,\rho_3))^{-1}$
, remembering that $\tilde{\Psi}_n=0$ on $\omega^{b}\setminus\omega_n$. We then assume that (\ref{DP_3-3_n}) holds for $n, n-1,\ldots,k+1$, where $k$ is a positive integer.

To obtain the contradiction that allows us to claim (\ref{DP_3-3_n}) holds at the rank $k$, we assume that (\ref{DP_3-3_n}) is no longer true for $k$. Then, the induction yields
\[
E_{k+1}-E_k<C_4C_3(\phi,\rho_2,\rho_3).
\]
Substituting the above inequality in (\ref{DP_3-3_al}) gives
\begin{eqnarray}\label{DP:3-6}
C_4C_3(\phi,\rho_2,\rho_3)<&E_k&\leq C_1C_4^{3/2}C_3(\phi,\rho_2,\rho_3)^{3/2}+C_2(\phi)C_4C_3(\phi,\rho_2,\rho_3)+C_3(\phi,\rho_2,\rho_3)(k+1)\nonumber
\end{eqnarray}
Even when the values of $C_0, C_1, C_2(\phi), C_3(\phi,\rho_2,\rho_3) > 0$ are fixed, $C_4$ can be conveniently chosen.  Taking $C_4\geq 2$ and plugging it in (\ref{DP:3-6}) results in a contradiction for $k>k_0$, where $k_0=\left \lfloor{C_1C_4^{3/2}C_3(\phi,\rho_2,\rho_3)^{1/2}+C_2(\phi)C_4}\right \rfloor$. Consequently, (\ref{DP_3-3_n}) is true at the rank $k>k_0$ and it also holds when $k\leq k_0$, since $E_k$ is increasingly monotonic with respect to $k$. 
\begin{remark}
	The reader can find a detailed description of the method for the Stokes problem in \cite{Gerard-Varet2010} and for the Stokes-Coriolis system, in \cite{Dalibard2014}. The backward induction in our case is less involved than the works mentioned above since we are not dealing directly with a non-local, non-linear Dirichlet-to-Neumann operator.
\end{remark}
By taking into account the translation invariance of the energy estimates and considering $k_0$ sufficiently large, we conclude that $\tilde{\Psi}_n$ is uniformly bounded in $H^2(\omega^b)$. This implies in turn  that $\tilde{\Psi}$ and consequently, $\Psi$ belong to $H^2_{\mathrm{uloc}}(\omega^b)$.
	\item \textit{Uniqueness.}  Let $\bar{\Psi}=\Psi_1-\Psi_2$, where $\Psi_i$, $i=1,2$, are solutions of \ satisfying the smallness condition $\|\Psi_i\|_{H^{m}_{\mathrm{uloc}}}<\delta$. We show that the solution $\bar{\Psi}$ of the problem
\begin{eqnarray}\label{pb:uniqueness_nl}
Q_w(\bar{\Psi},\Psi_1)+Q_w(\Psi_2,\bar{\Psi})+\partial_{X}\bar{\Psi}-\Delta_w^{2}\bar{\Psi}&=&0\quad\textrm{in}\;\omega^b,\nonumber\\
(1+\alpha^2)\Delta_w\bar{\Psi}\Big|_{\sigma^M_w}&=&0,\nonumber\\
\left(-\left((1+\alpha^2)\partial_{X}-2\alpha\partial_{Y}\right)\Delta_w\bar{\Psi}+\partial_{Y}\left(\displaystyle\frac{\nabla_w^{\perp}\bar{\Psi}\nabla_w^{\perp}(\Psi_1+\Psi_2)}{2}\right)\right.\hspace{7mm}&&\\
+(\nabla_w^{\perp}\bar{\Psi}\cdot\nabla_w)\left((1+\alpha^2)\partial_{X}-\alpha\partial_{Y}\right)
\Psi_1\hspace{7mm}&&\\
\left.+(\nabla_w^{\perp}\Psi_2\cdot\nabla_w)\left((1+\alpha^2)\partial_{X}-\alpha\partial_{Y}\right)\bar{\Psi}+\displaystyle\frac{\bar{\Psi}}{2}\right)\Bigg|_{\sigma^M_w}&=&0,\nonumber\\
\Psi\big|_{X=-\gamma(Y)}=\frac{\partial\Psi}{\partial n_w}\big|_{X=-\gamma(Y)}&=&0,\\
\end{eqnarray}
	is $\bar{\Psi}\equiv 0$. The smallness assumption on $\Psi_i$, $i=1,2$ leads the to following inequality on the truncated energies
	\begin{equation*}
	E_k\leq(\|\Psi_1\|_{H^2_{\mathrm{uloc}}}+\|\Psi_2\|_{H^2_{\mathrm{uloc}}}) C_1(E_{k+1}-E_k),
	\end{equation*}	
	where the constant $C_1$ depends only on the characteristics of the domain. Since $\underset{k}{\sup}(E_{k+1}-E_k)\leq \|\bar{\Psi}\|^2_{H^2_{\mathrm{uloc}}}$, it is possible to show that $E_k$ is uniformly bounded in $k$. Therefore, the difference between two solutions belongs to $H^{2}(\mathbb{R})$ and we can repeat the method but without $\chi_k$. The smallness assumption on  $\|\Psi_i\|_{H^2_{\mathrm{uloc}}}$, $i=1,2,$ ensures for a constant $C>0$
	\begin{equation*}
	(1-C\delta)\int_{\omega^{b}}|\Delta\Psi|^{2}\leq 0,
	\end{equation*}
	which provides the uniqueness result when $\delta <C^{-1}$.
\end{itemize}

\textit{\textbf{Step 2. Regularity.}} Higher regularity estimates for the solution are necessary for the subsequent application of the implicit function theorem. The analysis of the interior regularity and the regularity up to the boundary starts with the case of $m=1$. This case is later used to obtain higher regularity estimates through induction.
\paragraph{First order interior regularity}\label{local_regularity}\label{s2}
\begin{theorem}\label{theorem_firstlocal} For any, $\Psi\in H^2_{\mathrm{loc}}(\omega^b_w)$ which is a solution of
	\begin{equation}
	\partial_X \Psi+(\nabla^\perp_w \Psi\cdot\nabla_w)\Delta_w \Psi-\Delta^2_w \Psi=0\quad\textrm{in}\quad\omega^{b}
	\label{eq:interior-reg}
	\end{equation}
	then $\Psi\in H^s_{\mathrm{loc}}(\omega^b_w)$, $s\in \mathbb{N}$, $s\geq 2$. More precisely, for any bounded open set $\omega '$
		satisfying $\omega'\subset\subset\omega^b_w$,
	\begin{equation}\label{ch4:eq_4.4}
	\|\Psi\|_{H^s(\omega') }\leq C \|\Psi\|_{H^2_{\mathrm{\mathrm{uloc}}}(\omega^b_w)},
	\end{equation}
	where $C$ is a constant depending on $\omega '$.
\end{theorem}
\begin{proof}
	Let $\Psi$ be a solution of \eqref{eq:interior-reg} belonging to  $H^{2}_{\mathrm{uloc}}(\omega^b_w)$ and $\omega'\subset\subset \omega^b_w$. Note that
		$(\nabla^{\perp}_{w}\Psi\cdot \nabla_{w})\Delta_{w}\Psi = \nabla_{w}\cdot((\nabla^{\perp}_{w}\Psi)\Delta_{w}\Psi)$ and
		$\nabla^{\perp}_{w}\Psi\in H^{1}(\omega')$. Therefore, by Sobolev embedding 
		$$\nabla^{\perp}_{w}\Psi\in L^{p}(\omega')\;\text{ for any }p\in [1,+\infty).$$ 
		As $\Delta_{w}\Psi \in L^{2}(\omega')$, we have that $(\Delta_{w}\Psi)\nabla^{\perp}_{w}\Psi \in L^{q}(\omega')$ for any $q\in [1,2)$ and, on that account
		$$\nabla_{w}\cdot((\nabla^{\perp}_{w}\Psi)\Delta_{w}\Psi)\in W^{-1,q}(\omega')\;\text{ for any }q\in [1,2).$$
		Taking into account $\partial_{X}\Psi \in H^{1}(\omega')\subset W^{-1,q}(\omega ')$ and using \eqref{eq:interior-reg} provide $\Delta_{w}(\Delta_{w}\Psi)\in W^{-1,q}(\omega')$ for any $q\in [1,2)$. Therefore, by means of classical elliptic regularity arguments, we obtain $\Delta_{w}\Psi\in W^{1,q}(\omega')$ for any $q\in [1,2)$. This implies in turn that $\Psi\in W^{3,q}(\omega')$ for any $q\in [1,2)$ and
		\begin{equation*}
		\|\Psi\|_{W^{3,q}(\omega ')}\leq C\|\Psi\|_{H^{2}(\omega'')}\leq C \|\Psi\|_{H^{2}_{\mathrm{uloc}}(\omega^b_w)},
		\end{equation*}
	where $\omega'\subset\omega''\subset\omega_w^b$	and $C$ is a constant depending only $\omega '$.
		
		Consequently, $\nabla^{\perp}\Psi\in W^{2,q}(\omega')$ for any $q\in [1,2)$ leading to $\nabla^{\perp}\Psi\in W^{1,p}(\omega')$ for any $p\in [1,+\infty)$. Similarly, $\Delta_{w}\Psi \in W^{1,q}(\omega')$ for any $q\in [1,2)$ which gives that $\Delta_{w}\Psi\nabla^{\perp}\Psi\in W^{1,q}(\omega')$ for any $q\in [1,2)$ and $\Psi\in W^{4,q}(\omega')$ for any $q\in [1,2)$.
		Repeating the procedure results in $\Psi\in W^{s,q}(\omega')$ for any $s\in \mathbb{N}^{*}$. Therefore, $\Psi\in H^{s}(\omega')$ for any $s\in \mathbb{R}$ by Sobolev embedding. In particular, there exists $C$ depending only on $\omega '$ such that 
		\begin{equation*}
		\|\Psi\|_{H^{3}(\omega ')}\leq C \|\Psi\|_{H^{2}_{\mathrm{uloc}}(\omega^b_w)}.
		\end{equation*}
	
\end{proof}

\paragraph{Regularity up to the boundary}
Since we are only interested in the regularity near the artificial boundary, we can consider without loss of generality that the behavior at the interface at $X=0$ and the rough boundary does not influence our analysis. We tackle our regularity analysis for $X>M'$, where $M'\in (0,M)$.
To prove $H^3_{\mathrm{\mathrm{uloc}}}$- regularity up to the boundary, we need to compute a priori estimates for $\partial_Y\Psi$ in $H^2_{\mathrm{\mathrm{uloc}}}$. First, we are going to localize the equation near a fixed $k \in \mathbb{Z}$. Let $\tilde{\varphi}_k\in\mathcal{C}^\infty_0(\mathbb{R})$ be equal to $1$ in a neighborhood of $k \in \mathbb{Z}$, and such that the size of $\mathrm{Supp}\tilde{\varphi}_k$ is bounded uniformly in $k$. Moreover, we set $\varphi_k=\tilde{\varphi}^2_k$.

The idea is to apply a finite difference operator with a step $h>0$ in the direction parallel to the boundary, that is to say parallel to the $Y$-axis, and then, pass to the limit when $h$ goes to zero. This shows that $\partial_{Y}(\varphi_{k} \Psi_{w}^{-} )\in H^{2}(\omega^{b}_{w})$. Then, using the equations, it implies that  $\partial_{X}(\varphi_{k} \Psi_{w}^{-})\in H^{2}(\omega^{b}_{w})$ and thus $\varphi_{k} \Psi_{w}^{-} \in H^{3}(\omega^{b}_{w})$. From the arbitrariness of $k\in\mathbb{Z}$ and $\varphi_{k}$, and from the interior regularity provided for the case when $m=1$, this in turn implies that $\Psi_{w}^{-}\in H^{3}_{\mathrm{\mathrm{uloc}}}(\omega_{w}^{b})$. Going forward, to alleviate the notation,
we omit the $k$-dependence of $\varphi$ and we denote $\Psi$ instead of $\Psi_{w}^{-}$.
We define the finite difference operator $\delta_{h}$ as follows:
	\begin{equation*}
	\delta_h u = \dfrac{\tau_h u-u}{h},\quad  \tau_h u(X,Y)=u(X,Y+h).
	\end{equation*}
	Then, for $\omega'\subset\subset \omega^b_w$, there exists a constant $C>0$ such that $|h|<\text{dist}(\partial\omega,\omega')$, and $f\in W^{k,p}(\omega)$,
	\begin{eqnarray*}
		\|\delta_h f\|_{W^{k-1,p}(\omega')} &\leq&C\|f\|_{W^{k,p}(\omega)},\\
		\underset{h\rightarrow 0}{\lim}\left\|\delta_h(f)-\partial_{Y}f\right\|_{W^{k-1,p}(\omega')}&=&0.
	\end{eqnarray*}

Considering $\psi_h=\delta_{h}(\varphi_k\Psi)$ in \eqref{eq_left1} yields
\begin{eqnarray*}
\partial_{X}\psi_h-\Delta^{2}_w\psi_h&=&f_h\quad\textrm{in}\;\omega^{b}\cap\mathrm{Supp}\varphi_k,\nonumber\\
(1+\alpha^2)\Delta_w\psi_h&=&\rho_2^h,\nonumber\\
-\left[(1+\alpha^2)\partial_{X}-2\alpha\partial_{Y}\right]\Delta_w\psi_h+ \frac{\psi_{h}}{2}\Big|_{X=M}&=&\rho_3^h,\nonumber\\
\psi_h=\partial_X \psi_h&=&0,\quad\textrm{on}\; (M', M)\times\partial\mathrm{Supp}\varphi_k.
\end{eqnarray*}
Taking into account that $\varphi_k$ is independent of $X$, we have that
\begin{align*}
f_h&=-\delta_{h}\left[4\varphi_k^{(3)}\left(\partial_{Y}-\alpha\partial_{X}\right)\Psi+6\varphi_k''\Delta_w\Psi+4\varphi_k'\left(\partial_{Y}-\alpha\partial_{X}\right)\Delta_w\Psi+\varphi_k^{(4)}\Psi\right.\\
&\hspace{.6in}\left.-\nabla_w\Psi\cdot\left(3(0,\varphi_k')\Delta_w\Psi+3\varphi_k''\nabla_w\Psi+(0,\varphi_k^{(3)})\Psi\right)\right]-\delta_{h}\left[(\nabla_w^\perp\Psi\cdot\nabla_w)\Delta_w(\varphi_k\Psi)\right],\\
\rho_2^{h}&=\delta_{h}[\varphi_k\rho_2]-(1+\alpha^2)\delta_h\left[2\varphi_k'\left(\partial_{Y}-\alpha\partial_{X}\right)\Psi+\varphi_k''\Psi\right],\\
\rho_3^{h}&=
\delta_{h}[\varphi_k\rho_3]+\delta_{h}\left[\left((1+\alpha^2)\partial_{X}-2\alpha\partial_{Y}\right)\left(2\partial_{Y}\varphi_k\left(\partial_{Y}-\alpha\partial_{X}\right)\Psi+\partial_{Y}^2\varphi_k\Psi\right)\right.\\
&\left.-\nabla_w^\perp\Psi\cdot\left(-\alpha\nabla_w\left(\Psi\partial_{Y}\varphi_k\right)+\left((1+\alpha^2)\partial_{X}-\alpha\partial_{Y}\right)\Psi(0,\partial_{Y})\varphi_k\right)-\left(\partial_{Y}\left(\tfrac{|\nabla_w^\perp\Psi|^2}{2}\right)\varphi_k\right)\right]\\
&+\delta_{h}\left[(\nabla_w^\perp\Psi\cdot\nabla_w)\left((1+\alpha^2)\partial_{X}-\alpha\partial_{Y})(\varphi_k\Psi)\right)\right]-\delta_{h}[2\alpha\Delta_{w}\Psi\partial_{Y}\varphi_k].
\end{align*}
Let us now state some technical lemmas which are necessary to the proof.
\begin{lemma}\label{lem:estimate_terms} Let $\sigma^M_w=\left\{(M, Y):\: Y\in\mathbb{R}\right\}$. Define $\sigma_k=\sigma^M_w\cap \mathrm{Supp}\varphi$ where $\varphi\in \mathcal{C}^\infty_0(\mathbb{R})$ and it is equal to $1$ in a neighborhood of $\sigma^M_w$. Consider the functions $\Psi\in H^{2}_{\mathrm{\mathrm{uloc}}}(\omega^{b})$, $\rho_2\in H^{1/2}_{\mathrm{\mathrm{uloc}}}(\sigma_{M})$ and $\rho_3\in H^{-1/2}_{\mathrm{\mathrm{uloc}}}(\sigma_{M})$. Then, for any $h\in \mathbb{R}$, we have the estimates
	\begin{equation*}
	\begin{split}
	\|\rho_{3}^{h}\|_{H^{-3/2}(\sigma_{k})}\leq& C(1+\|\varphi\|_{W^{4,\infty}})(\|\rho_{3}\|_{H^{-1/2}_{\mathrm{\mathrm{uloc}}}(\sigma^M_w)}+\|\rho_{2}\|_{H^{1/2}_{\mathrm{\mathrm{uloc}}}(\sigma^M_w)}\\
	&+ \|\Psi\|_{H^{2}_{\mathrm{\mathrm{uloc}}}}(1+\|\psi_{h}\|_{H^{2}_{\mathrm{\mathrm{uloc}}}}+\|\Psi\|_{H^{2}_{\mathrm{\mathrm{uloc}}}})),\\
	\|\rho_{2}^{h} \|_{H^{-1/2}(\sigma_{k})}\leq& C\left( \|\rho_{2}\|_{H^{1/2}_{\mathrm{\mathrm{uloc}}}(\sigma^M_w)}+\|\Psi\|_{H^{2}_{\mathrm{\mathrm{uloc}}}}\right)\|\varphi\|_{W^{3,\infty}},\\
	\|f_{h}\|_{H^{-2}(\omega^b_w)}\leq& C(\|\Psi\|_{H^{2}_{\mathrm{\mathrm{uloc}}}}(1+\|\Psi\|_{H^{2}_{\mathrm{\mathrm{uloc}}}}+\|\psi_{h}\|_{H^{2}_{\mathrm{\mathrm{uloc}}}}+\|\Psi_{L}\|_{H^{2}_{\mathrm{\mathrm{uloc}}}}))(1+\|\varphi\|_{W^{3,\infty}(\omega^b_w)}),
	\end{split}
	\end{equation*}
	where $C$ is a constant depending only on the domain $\omega^b_w$.
	
\end{lemma}
\begin{lemma}\label{lem:elliptic-reg}
	Consider the linear problem
	\begin{align}
	\partial_{X}z-\Delta^{2}_w z&=f_{1}\quad\textrm{in}\;\omega_k,\nonumber\\
	(1+\alpha^2)\Delta_w z&=f_{2},\label{linearpb}\\
	-\left[(1+\alpha^2)\partial_{X}-2\alpha\partial_{Y}\right]\Delta_w z +\frac{z}{2}\Big|_{\sigma_{k}}&=f_{3},\nonumber\\
	z=\partial_Xz&=0,\;\text{on}\;\partial\omega_{k}\setminus\sigma_{k}\nonumber,
	\end{align}
	where $\omega_k=\omega^{b}\cap\mathrm{Supp}\varphi$, $\sigma_{k}=\partial\omega_{k}\cap\partial\omega^b_w$ and $\nu$ stands for the unit outer normal vector at the boundary. If $f_{1}\in H^{-2}(\omega_{k})$, $f_{2}\in H^{-1/2}(\sigma_{k})$, $f_{3}\in H^{-3/2}(\sigma_{k})$. Then, problem \eqref{linearpb} has a unique solution $z\in H^{2}(\omega_{k})$ and satisfies the estimate
	\begin{equation}\label{ellip-reg-est}
	\|z\|_{H^{2}(\omega_{k})}\leq C\left(\|f_{1}\|_{H^{-2}(\omega_{k})}+\|f_{2}\|_{H^{-1/2}(\sigma_{k})}+\|f_{3}\|_{H^{-3/2}(\sigma_{k})}\right).
	\end{equation}
\end{lemma}
	The regularity up to the boundary and the fact that $\varphi\Psi\in H^{3}(\omega_{w}^{b})$ is a consequence of Lemmas \ref{lem:estimate_terms} and \ref{lem:elliptic-reg} for $z=\psi_h$. Combining these two Lemmas with the smallness assumption on $\|\Psi\|_{H^{2}_{\mathrm{\mathrm{uloc}}}(\omega^b_w)}$, we have that $\psi_{h}\in H^{2}(\omega_{w}^{b})$ and, for any $h\in\mathbb{R}^{*}$
	\begin{equation*}
	\|\psi_{h}\|_{H^{2}(\omega_{w}^{b})}\leq C\left(\|\Psi\|_{H^{2}_{\mathrm{\mathrm{uloc}}}(\omega^b_w)}+\|\rho_{2}\|_{H^{1/2}_{\mathrm{\mathrm{uloc}}}(\sigma_{k})}+\|\rho_{3}\|_{H^{1/2}_{\mathrm{\mathrm{uloc}}}(\sigma_{k})}\right),
	\end{equation*}
	for some constant $C$ depending on the bound on $\|\Psi\|_{H^{2}_{\mathrm{\mathrm{uloc}}}(\omega^b_w)}$ but independent of $h$. This implies that $\psi_{h}\in H^{2}(\omega^b_w)$ and therefore, $\partial_{Y}(\Psi\varphi) \in H^{2}(\omega^b_w)$. From the arbitrariness of $\varphi$ and as $\Psi\in H^{2}(\omega^b_w)$, this means that $\partial_{Y}\Psi \in H^{2}_{\mathrm{uloc}}(\omega^b_w)$. Hence, that $\partial_{Y}\partial_{X}\Psi\in H^{1/2}_{\mathrm{uloc}}(\sigma_{k})$. In particular $\partial_{X}\Psi\in H^{3/2}_{\mathrm{uloc}}(\sigma_{k})$. The first boundary condition of \eqref{eq_left1} yields
	\begin{equation*}
	\partial_{X}^{2}\Psi =(1+\alpha)^{-1} \left(-\partial_{Y}^{2}\Psi+2\alpha\partial_{Y}\partial_{X}\Psi +(1+\alpha)^{-1}\rho_{2}\right),
	\end{equation*}
	which implies that $\partial_{X}^{2}\Psi\in H^{1/2}_{\mathrm{uloc}}(\sigma_{k})$. This result combined with the second boundary condition in \eqref{eq_left1} gives $\partial_{X}\nabla_{w}\Psi\in H^{-1/2}_{\mathrm{uloc}}(\sigma_{k})$ and, consequently, $\Delta_{w}(\partial_{X}\Psi)\in H^{-1/2}_{\mathrm{uloc}}(\sigma_{k})$. Using now the main equation in \eqref{eq_left1}, we deduce that $\Delta_{w}^{2}\Psi\in H^{-1}_{\mathrm{uloc}}(\omega^b_w)$ and $\partial_{X}^{2}\Delta_{w}\Psi\in H^{-1}_{\mathrm{uloc}}(\omega^b_w)$. Hence, $\partial_{X}^{2}\Delta\Psi\in H^{-3/2}_{\mathrm{uloc}}(\sigma_{k})$ and $A_{3}[\partial_{X}\Psi]\in H^{-3/2}_{\mathrm{uloc}}(\sigma_{k})$. Let $\chi\in C^{\infty}_c(\omega^b_w)$, setting 
	\begin{equation*}
	\begin{split}
	f_{1}&= \Delta_{w}^{2}(\chi\partial_{X}\Psi) -(\chi\partial_{X}\Psi),\\
	f_{2} &= (1+\alpha^{2})\Delta_{w}(\chi\partial_{X}\Psi),\\
	f_{3}&= -[(1+\alpha^{2})\partial_{X}-2\alpha\partial_{Y}]\Delta_{w}(\chi\partial_{X}\Psi)+\frac{\chi\partial_{X}\Psi}{2},
	\end{split}
	\end{equation*}
	yields that $\chi\partial_{X}\Psi$ is solution to a linear problem of the from \eqref{linearpb} where $f_{1}\in H^{-2}(\sigma_{k})$, $f_{2}\in H^{-1/2}(\sigma_{k})$, $f_{3}\in H^{-2}(\sigma_{k})$. From Lemma \ref{lem:elliptic-reg}, $\chi\partial_{X}\Psi\in H^{2}(\omega^b_w)$ and paired with the arbitrariness  of $\chi$ $\partial_{X}\Psi\in H^{2}(\omega^b_w)$. Finally, combining this with $\partial_{Y}\Psi\in H^{2}(\omega^b_w)$ gives $\Psi\in H^{3}(\omega^b_w)$.
	
It remains to prove the results in Lemmas \ref{lem:estimate_terms} and \ref{lem:elliptic-reg}.
Lemma \ref{lem:elliptic-reg} is a standard elliptic regularity result. The main difficulty resides in the proof of the estimates in Lemma \ref{lem:estimate_terms} when handling the nonlinear terms.
\\

\begin{proof}[Proof of Lemma \ref{lem:elliptic-reg}]
	Before computing the elliptic regularity estimates, let us first briefly comment on the existence of an
	appropriate weak solution $z$ of the boundary-value problem \eqref{linearpb}.
	Note that the weak formulation associated with \eqref{linearpb} is
	\begin{equation}\label{variational_formulation_psi1}
	\begin{split}
	\int_{\omega_{k}} \partial_Xz\theta-\int_{\omega_{k}} \Delta_wz\Delta_w\theta&=\int_{\omega_{k}} f_{1}\theta-\int_{\sigma_{k}}f_{3}\theta\big|_{X=M}\\
	&-\int_{\sigma_{k}}f_{2}\partial_X\theta\big|_{X=M}+\int_{\sigma_{k}\frac{1}{2}z\big|_{X=M}\theta\big|_{X=M}}, \;\;\forall\;\theta\in H^{2}(\omega_{k}),
	\end{split}
	\end{equation}
	which can be written as follows:
	\begin{equation}\label{var_variational_formulation_psi1}
	A(z,\theta)+a(z,\theta)=(f_{1},\theta)-\int_{\sigma_{k}}f_{3}\theta\big|_{X=M}-\int_{\sigma_{k}}f_{2}\partial_X\theta\big|_{X=M}+\int_{\sigma_{k}}
	\frac{1}{2}z\big|_{X=M}\theta\big|_{X=M},
	\end{equation}
	where $A(z,\theta)=\int_{\omega_{k}} \partial_Xz\theta-\int_{\omega_{k}} \Delta_wz\Delta_w\theta$ and $a(z,\theta)$ contains the integral boundary terms. Here, $(\cdot, \cdot)$ refers to the usual product in $L^2$. The differential operator $A$ is a bilinear and continuous form in $H^2(\omega_k)\times H^2(\omega_k)$, as well as, $H^2_0(\omega_k)$-elliptic. The existence and uniqueness of a solution $z\in H^2(\omega_{k})$ is guaranteed by applying \cite[Theorem 3.1, Section 1.3.2]{Necasa}.
	Taking $\theta=z$ in \eqref{variational_formulation_psi1} yields the expression
	\begin{equation}\label{var_for_psi1}
	\int_{\omega_{k}}|\Delta_wz|^2=-\int_{\omega_{k}}f_{1}z+\int_{\sigma_{k}}f_{3}z+\int_{\sigma_{k}}f_2\partial_X z,
	\end{equation}
	which leads naturally to the estimate by using Cauchy-Schwarz inequality, the definition of the norm in dual spaces and the trace theorem
	\begin{equation*}
	\int_{\omega_{k}}|\Delta_wz|^2\leq C\left(\|f\|_{H^{-2}(\omega_{k})}+\|f_2\|_{H^{-1/2}(\sigma_{k})}+\|f_3\|_{H^{-3/2}(\sigma_{k})}\right)\|z\|_{H^2(\omega_{k})}.
	\end{equation*}
	Finally, we obtain the desired result
	\begin{equation*}
	\|z\|_{H^2(\omega_{k})}\leq C\left(\|f\|_{H^{-2}(\omega_{k})}+\|f_{2}\|_{H^{-1/2}(\sigma_{k})}+\|f_{3}\|_{H^{-3/2}(\sigma_{k})}\right).
	\end{equation*}
\end{proof}

\begin{proof}[Proof of Lemma \ref{lem:estimate_terms}]
	This proof is divided in three parts corresponding to each one of the estimates of $\rho_{3}^{h}$, $\rho_{2}^{h}$ and$f^{h}$.

	{\it Estimate of $\rho_{3}^{h}$.}
	We have
	\begin{equation}
	\begin{split}
	\rho_3^{h}&=
	\delta_{h}[\varphi_k\rho_3]+\delta_{h}\left[\left((1+\alpha^2)\partial_{X}-2\alpha\partial_{Y}\right)\left(2\partial_{Y}\varphi_k\left(\partial_{Y}-\alpha\partial_{X}\right)\Psi+\partial_{Y}^2\varphi_k\Psi\right)\right.\\
	&\left.-\nabla_w^\perp\Psi\cdot\left(-\alpha\nabla_w\left(\Psi\partial_{Y}\varphi_k\right)+\left((1+\alpha^2)\partial_{X}-\alpha\partial_{Y}\right)\Psi(0,\partial_{Y})\varphi_k\right)-\left(\partial_{Y}\left(\tfrac{|\nabla_w^\perp\Psi|^2}{2}\right)\varphi_k\right)\right]\nonumber\\
	&+\delta_{h}\left[(\nabla_w^\perp\Psi\cdot\nabla_w)\left((1+\alpha^2)\partial_{X}-\alpha\partial_{Y})(\varphi_k\Psi)\right)\right]-\delta_{h}[2\alpha\Delta_{w}\Psi\partial_{Y}\varphi_k].
	\end{split}
	\label{estim-rho3}
	\end{equation}
	The first term on the r.h.s satisfies
	\begin{equation}\label{estim-r3-0-1}
	\begin{split}
	\lVert \delta_{h}[\varphi_k\rho_3]\rVert_{H^{-3/2}(\sigma_{k})}&\leq C\lVert\rho_3\rVert_{H^{-1/2}_{\mathrm{\mathrm{uloc}}}(\sigma_{M})},
	\end{split}
	\end{equation}
	
	Let us give a closer look at the second term. Enclosed in brackets are terms involving at most one order derivative of $\Psi$ multiplied by a derivative of $\varphi_k$, which has compact support on $\sigma_{k}$. Let us now recall that $\Psi$ belongs to $H^{3/2}_{\mathrm{\mathrm{uloc}}}$ and $\partial_{X}\Psi$ belongs to $H^{1/2}_{\mathrm{\mathrm{uloc}}}$ at $X=M$. 
	
	Moreover, using the conditions at $\sigma^M_w$, we get
	\begin{equation}\label{bc-x-m}
	\lVert \partial_{X}^{2}\Psi\rVert_{H^{-1/2}_{\mathrm{\mathrm{uloc}}}(\sigma^M_w)}\leq C(\lVert\Psi\rVert_{H^{2}_{\mathrm{\mathrm{uloc}}}(\omega^{b})}+\lVert\rho_{2}\rVert_{H^{-1/2}_{\mathrm{\mathrm{uloc}}}(\sigma^M_w)}).
	\end{equation}
	Hence,
	\begin{equation}\label{estim-r3-1-0}
	\begin{split}
	&\left\|\delta_{h}\left[\left((1+\alpha^2)\partial_{X}-2\alpha\partial_{Y}\right)\left(2\partial_{Y}\varphi_k\left(\partial_{Y}-\alpha\partial_{X}\right)\Psi+\partial_{Y}^2\varphi_k\Psi\right)\right]\right\|_{H^{-3/2}(\sigma_{k})}\\
	&\leq C \left(\left\|\rho_{2}\right\|_{H^{-1/2}_{\mathrm{\mathrm{uloc}}}(\sigma^M_w)}+\left\|\Psi\right\|_{H^{3/2}_{\mathrm{\mathrm{uloc}}}(\sigma^M_w)}+\left\|\partial_{X}\Psi\right\|_{H^{1/2}_{\mathrm{\mathrm{uloc}}}(\sigma^M_w)}\right)\|\varphi_k\|_{W^{4,\infty}}.
	\end{split}
	\end{equation}
	Now, observe that for $f\in H^{1/2}_{\mathrm{\mathrm{uloc}}}(\sigma_{k},\mathbb{R}^{2})$, $g\in H^{1/2}_{\mathrm{\mathrm{uloc}}}(\sigma_{k},\mathbb{R}^{2})$, $\theta\in C^{\infty}_{c}(\sigma_{k},\mathbb{R}_{+})$ and as a result of the Sobolev embedding $H^{1/2}(\mathbb{R})\subset L^{p}(\mathbb{R})$, for any $p\in[1,+\infty)$, we have
	\begin{equation}\label{lemma-cool}
	\|fg\theta\|_{H^{-1/2}(\sigma_{k})}\leq\|fg\theta\|_{L^{p}(\sigma_{k})}\leq \|f\|_{H^{1/2}_{\mathrm{\mathrm{uloc}}}(\sigma^M_w)}\|g\|_{H^{1/2}_{\mathrm{\mathrm{uloc}}}(\sigma^M_w)}\|\theta\|_{L^{\infty}}.
	\end{equation}
	Consequently, using Trace theorem
	\begin{equation}\label{estim-r3-1}
	\begin{split}
	&\left\|\delta_{h}\left[-\nabla_w^\perp\Psi\cdot\left(-\alpha\nabla_w\left(\Psi\partial_{Y}\varphi_k\right)+\left((1+\alpha^2)\partial_{X}-\alpha\partial_{Y}\right)\Psi(0,\partial_{Y})\varphi_k\right)\right]\right\|_{H^{-3/2}(\sigma_{k})}\\
	&\leq C\left(\left\|\Psi\right\|_{H^{3/2}_{\mathrm{\mathrm{uloc}}}(\sigma^M_w)}+\left\|\partial_{X}\Psi\right\|_{H^{1/2}_{\mathrm{\mathrm{uloc}}}(\sigma^M_w)}\right)\|\varphi_k\|_{W^{2,\infty}(\omega^b_w)}\\
	&\leq C\|\Psi\|^{2}_{H^{2}_{\mathrm{\mathrm{uloc}}}(\omega^b_w)}\|\varphi_k\|_{W^{2,\infty}(\omega^b_w)}.
	\end{split}
	\end{equation}
	Finally, note that
	\begin{equation*}
	\begin{split}
	(\partial_{Y}^{2}\Psi)\varphi_k &= \partial_{Y}^{2}(\Psi\varphi_k)-2(\partial_{Y}\Psi\partial_{Y}\varphi_k)-(\Psi\partial_{Y}^{2}\varphi_k),\\
	(\partial_{X}^{2}\Psi)\varphi_k &=\partial_{X}^{2}(\Psi\varphi_k),\\ (\partial_{YX}^{2}\Psi)\varphi_k &= \partial_{YX}^{2}(\Psi\varphi_k)-(\partial_{X}\Psi\partial_{Y}\varphi_k).
	\end{split}
	\end{equation*}
	Therefore,
	\begin{equation}\label{estim-carre}
	\begin{split}
	&\delta_{h}\left[-\left(\partial_{Y}\left(\tfrac{|\nabla_w^\perp\Psi|^2}{2}\right)\varphi_k\right)\right] \\
	&=\delta_{h}\left[-\left(\partial_{Y}\left(\tfrac{|\nabla_w^\perp\Psi|^2}{2}\varphi_k\right)\right)\right]-\delta_{h}\left[(\partial_{Y}\varphi_k)\tfrac{|\nabla_w^\perp\Psi|^2}{2}\right].
	\end{split}
	\end{equation}
	The second term is a quadratic in $\Psi$, linear in $\varphi_k$ and involves at most one derivative of $\Psi$ and one derivative of $\varphi_k$. Therefore,
		\begin{equation}\label{estim-carre-0}
		\begin{split}
		\|\delta_{h}\left[(\partial_{Y}\varphi_k)\tfrac{|\nabla_w^\perp\Psi|^2}{2}\right]\|_{H^{-3/2}(\sigma_{k})}&\leq C
		\|(\partial_{Y}\varphi_k)\tfrac{|\nabla_w^\perp\Psi|^2}{2}\|_{H^{-1/2}(\sigma_{k})}\\
		&\leq C \|\varphi_k\|_{W^{1,\infty}(\sigma_{k})}\left(\|\partial_{X}\Psi\|_{H^{1/2}_{\mathrm{\mathrm{uloc}}}(\sigma_{k})}+\|\Psi\|_{H^{3/2}_{\mathrm{\mathrm{uloc}}}(\sigma_{k})}\right)^{2}
		.
		\end{split}
		\end{equation}
		The first term of \eqref{estim-carre} can be treated as follows
		\begin{equation}\label{estim-carre-1}
		\begin{split}
		\delta_{h}(|\nabla_w^\perp\Psi|^2\varphi_k)&=\delta_{h}\left(\left[(1+\alpha^{2})(\partial_{X}\Psi)^{2}-2\alpha\partial_{X}\Psi\partial_{Y}\Psi+(\partial_{Y}\Psi)^{2}\right]\varphi_k\right)\\
		&=2(1+\alpha^{2})(\partial_{X}\psi_{h})(\partial_{X}\Psi)-2\alpha\partial_{X}\psi_{h}\partial_{Y}\Psi-2\alpha\partial_{X}\Psi\partial_{Y}\psi_{h}+2(\partial_{Y}\Psi)(\partial_{Y}\psi_{h})\\
		&+F_{1}(\Psi,\delta_{h}\varphi_k,\varphi_k)+F_{2}(\Psi,\delta_{h}\Psi,\varphi_k),
		\end{split}
		\end{equation}
		where $F_{1}(\Psi,\delta_{h}\varphi_k)$ is a sum of terms quadratic in $\Psi$, linear in $\delta_{h}\varphi_k$ and involving at most one derivative of $\Psi$ and $\delta_{h}\varphi_k$. $F_{2}(\Psi,\delta_{h}\Psi,\varphi_k)$, on the other hand, is a sum of terms linear in $\Psi$, linear in $\delta_{h}\Psi$ and linear in $\varphi_k$ involving at most one derivative in $\Psi$, no derivative in $\delta_{h}\Psi$ and one derivative in $\varphi_k$.
	As a result of \eqref{lemma-cool}, we have
		\begin{equation}\label{estim-r3-2}
		\begin{split}
		\|F_{1}(\Psi,\delta_{h}\varphi_k)\|_{H^{-1/2}(\sigma_{k})}&\leq C \left(\|\Psi\|_{H^{3/2}_{\mathrm{\mathrm{uloc}}}(\sigma^M_w)}+\|\partial_{X}\Psi\|_{H^{1/2}_{\mathrm{\mathrm{uloc}}}(\sigma^M_w)}\right)^{2}\|\varphi_k\|_{W^{1,\infty}},
		\end{split}
		\end{equation}
		and 
		\begin{equation}\label{estim-carre-2}
		\begin{split}
		\|F_{2}(\Psi,\delta_{h}\Psi,\varphi_k),\|_{H^{-1/2}(\sigma_{k})}&\leq C (\|\Psi\|_{H^{3/2}_{\mathrm{\mathrm{uloc}}}(\sigma^M_w)}+\|\partial_{X}\Psi\|_{H^{1/2}_{\mathrm{\mathrm{uloc}}}(\sigma^M_w)})\|\delta_{h}\Psi\|_{H^{1/2}_{\mathrm{\mathrm{uloc}}}(\sigma^M_w)}\|\varphi_k\|_{W^{1,\infty}}\\
		&\leq C (\|\Psi\|_{H^{3/2}_{\mathrm{\mathrm{uloc}}}(\sigma^M_w)}+\|\partial_{X}\Psi\|_{H^{1/2}_{\mathrm{\mathrm{uloc}}}(\sigma^M_w)})^{2}\|\varphi_k\|_{W^{1,\infty}}.
		\end{split}
		\end{equation}
		The reminder of \eqref{estim-carre-1} is now easy to handle using \eqref{lemma-cool}
		\begin{equation}\label{estim-carre-3}
		\begin{split}
		&\|    2(1+\alpha^{2})(\partial_{X}\psi_{h})(\partial_{X}\Psi)-2\alpha\partial_{X}\psi_{h}\partial_{Y}\Psi-2\alpha\partial_{X}\Psi\partial_{Y}\psi_{h}+2(\partial_{Y}\Psi)(\partial_{Y}\psi_{h})\|_{H^{-1/2}_{\mathrm{\mathrm{uloc}}}(\sigma_{k})}\\
		&\leq C
		\left(\|\partial_{X}\Psi\|_{H^{1/2}_{\mathrm{\mathrm{uloc}}}(\sigma_{k})}+\|\Psi\|_{H^{3/2}_{\mathrm{\mathrm{uloc}}}(\sigma_{k})}\right)\left(\|\partial_{X}\psi_{h}\|_{H^{1/2}_{\mathrm{\mathrm{uloc}}}(\sigma_{k})}+\|\psi_{h}\|_{H^{3/2}_{\mathrm{\mathrm{uloc}}}(\sigma_{k})}\right).
		\end{split}
		\end{equation}
		Using \eqref{estim-carre-0}, \eqref{estim-carre-1}, \eqref{estim-r3-2}, \eqref{estim-carre-2}, and \eqref{estim-carre-3} one has
		\begin{equation*}
		\begin{split}
		\hspace{-3mm}&\left\|\delta_{h}\left[-\left(\partial_{Y}\left(\tfrac{|\nabla_w^\perp\Psi|^2}{2}\right)\varphi_k\right)\right]\right\|_{H^{-3/2}(\sigma_{k})}\\
		&\quad\leq
		\left\|\delta_{h}\left[\partial_{Y}\left(\tfrac{|\nabla_w^\perp\Psi|^2}{2}\varphi_k\right)\right]\right\|_{H^{-3/2}(\sigma_{k})}
		+C \|\varphi_k\|_{W^{1,\infty}(\sigma_{k})}\left(\|\partial_{X}\Psi\|_{H^{1/2}_{\mathrm{\mathrm{uloc}}}(\sigma_{k})}+\|\Psi\|_{H^{3/2}_{\mathrm{\mathrm{uloc}}}(\sigma_{k})}\right)^{2}\\
		&\quad\leq C\left\|\delta_{h}\left[\tfrac{|\nabla_w^\perp\Psi|^2}{2}\varphi_k\right]\right\|_{H^{-1/2}(\sigma_{k})}
		+C \|\varphi_k\|_{W^{1,\infty}(\sigma_{k})}\left(\|\partial_{X}\Psi\|_{H^{1/2}_{\mathrm{\mathrm{uloc}}}(\sigma_{k})}+\|\Psi\|_{H^{3/2}_{\mathrm{\mathrm{uloc}}}(\sigma_{k})}\right)^{2}
		\\&\quad
		\leq 
		C \left(\|\varphi_k\|_{W^{1,\infty}(\sigma_{k})}\left(\|\partial_{X}\Psi\|_{H^{1/2}_{\mathrm{\mathrm{uloc}}}(\sigma_{k})}+\|\Psi\|_{H^{3/2}_{\mathrm{\mathrm{uloc}}}}\right)^{2}\right.
		\\&\quad
		+\left.\left(\|\partial_{X}\Psi\|_{H^{1/2}_{\mathrm{\mathrm{uloc}}}(\sigma_{k})}+\|\Psi\|_{H^{3/2}_{\mathrm{\mathrm{uloc}}}(\sigma_{k})}\right)\left(\|\partial_{X}\psi_{h}\|_{H^{1/2}_{\mathrm{\mathrm{uloc}}}(\sigma_{k})}+\|\psi_{h}\|_{H^{3/2}_{\mathrm{\mathrm{uloc}}}(\sigma_{k})}\right)\right).
		\end{split}
		\end{equation*}
	Combining the above result with the Trace theorem gives
	
	\begin{equation*}
	\begin{split}
	\left\|\delta_{h}\left[-\left(\partial_{Y}\left(\tfrac{|\nabla_w^\perp\Psi|^2}{2}\right)\varphi_k\right)\right]\right\|_{H^{-3/2}(\sigma_{k})} \leq C \left(\lVert \psi_{h}\rVert_{H^{2}(\omega^b_w)}\|\Psi\|_{H^{2}_{\mathrm{\mathrm{uloc}}}(\omega^b_w)}+\|\Psi\|_{H^{2}_{\mathrm{\mathrm{uloc}}}(\omega^b_w)}^{2}\right)\|\varphi_k\|_{W^{1,\infty}}.
	\end{split}
	\end{equation*}
	
	It remains to tackle the two last terms of \eqref{estim-rho3}, starting with
	\begin{equation*}
	\begin{split}
	\delta_{h}\left[(\nabla_w^\perp\Psi\cdot\nabla_w)\left((1+\alpha^2)\partial_{X}-\alpha\partial_{Y}\right)(\varphi_k\Psi)\right].
	\end{split}
	\end{equation*}
	Since too many derivatives are involved in the above expression, this term cannot be controlled roughly by controlling $\delta_{h}[\nabla_w^\perp\Psi]$ by $|\nabla (\nabla_w^\perp\Psi)|$. To address this issue, we make $\delta_{h}(\varphi_k\Psi)=\psi_{h}$ appear, similarly as we did for \eqref{estim-carre}.
	We claim
	
		\begin{equation}\label{estim-hard-r3-0}
		\begin{split}
		&    \delta_{h}\left(([\nabla_w^\perp\Psi]\cdot\nabla_w)\left((1+\alpha^2)\partial_{X}-\alpha\partial_{Y}\right)(\varphi_k\Psi)\right)\\
		&= 
		(\tau_{h}[\nabla_w^\perp\Psi]\cdot\nabla_w)\left((1+\alpha^2)\partial_{X}-\alpha\partial_{Y}\right)(\delta_{h}(\varphi_k\Psi))
		+    (\delta_{h}[\nabla_w^\perp\Psi]\cdot\nabla_w)\left((1+\alpha^2)\partial_{X}-\alpha\partial_{Y}\right)(\varphi_k\Psi)\\
		&= (\tau_{h}[\nabla_w^\perp\Psi]\cdot\nabla_w)\left((1+\alpha^2)\partial_{X}-\alpha\partial_{Y}\right)(\psi_{h})
		+(\delta_{h}[\nabla_w^\perp\Psi]\cdot\nabla_w)\left((1+\alpha^2)\partial_{X}-\alpha\partial_{Y}\right)(\varphi_k\Psi)    
		.
		\end{split}
		\end{equation}
		For the first term
		\begin{equation*}
		\begin{split}
		&\|(\tau_{h}[\nabla_w^\perp\Psi]\cdot\nabla_w)\left((1+\alpha^2)\partial_{X}-\alpha\partial_{Y}\right)(\psi_{h})\|_{H^{-3/2}(\sigma_{k})}
		\\
		&\leq C \|(\tau_{h}[\nabla_w^\perp\Psi]\cdot\nabla_w)\left((1+\alpha^2)\partial_{X}-\alpha\partial_{Y}\right)(\psi_{h})\|_{L^{1}(\sigma_{k})}\\
		&\leq\left(\|\psi_{h}\|_{H^{3/2}(\sigma_{k})}+\|\partial_{X}\psi_{h}\|_{H^{1/2}(\sigma_{k})}\right)\left( \|\Psi\|_{H^{3/2}_{\mathrm{\mathrm{uloc}}}(\sigma_{k})}+\|\partial_{X}\Psi\|_{H^{1/2}_{\mathrm{\mathrm{uloc}}}(\sigma_{k})}\right).
		\end{split}
		\end{equation*}
		This inequality results from the Sobolev embedding $L^{1}(\sigma_{k})\subset H^{-3/2}(\sigma_{k})$. It can also be seen coming back to the definition of $H^{-3/2}(\sigma_{k})$ as done below at \eqref{estim-sobolev0}-\eqref{estim-sobolev1}. 
		For the second term of \eqref{estim-hard-r3-0}, we have
	
	\begin{equation}\label{estim-hard-r3}
	\begin{split}
	  (\delta_{h}[\nabla_w^\perp\Psi]&\cdot\nabla_w)\left((1+\alpha^2)\partial_{X}-\alpha\partial_{Y}\right)(\varphi_k\Psi)\\
	=&\left[(\delta_{h}[\nabla_w^\perp\Psi]\cdot\nabla_w)\left((1+\alpha^2)\partial_{X}-\alpha\partial_{Y}\right)\Psi\right]\varphi_k+(\delta_{h}[\nabla_w^\perp\Psi]\cdot\nabla_w)\left(\left[\left((1+\alpha^2)\partial_{X}-\alpha\partial_{Y}\right)\varphi_k\right]\Psi\right)\\
	&+\left[(\delta_{h}[\nabla_w^\perp\Psi]\cdot\nabla_w)\varphi_k\right]\left[\left((1+\alpha^2)\partial_{X}-\alpha\partial_{Y}\right)\Psi\right].
	\end{split}
	\end{equation}
	Note that the second and third terms above are easier to deal with as they have fewer derivatives of $\Psi$. Indeed, it consists of the product of two terms: the first one involves $\delta_{h}[\nabla_w^\perp\Psi]$ while the second one contains at most a first order derivative of $\Psi$ and is multiplied by a term with compact support. Accordingly, the latter belongs to $H^{1/2}(\sigma_{k})$. Using that for $f\in H^{1/2}(\sigma_{k})$ and $g\in H^{3/2}_{0}(\sigma_{k})$, we have $g\in W^{1/2,\infty}$ and $fg\in H^{1/2}(\sigma_{k})$. Thus, 
	\begin{equation}\label{estim-sobolev0}
	\|fg\|_{H^{1/2}(\sigma_{k})}\leq C \|f\|_{H^{1/2}(\sigma_{k})}\|g\|_{H^{3/2}(\sigma_{k})}.\quad
	\end{equation}
	For $f_1\in H^{1/2}_{\mathrm{\mathrm{uloc}}}(\sigma^M_w)$, $f_2\in H^{1/2}_{\mathrm{\mathrm{uloc}}}(\sigma^M_w)$, $\varphi_k\in C_{c}^{\infty}(\sigma_{k})$, and $g\in H^{3/2}_{0}(\sigma_{k})$, using \eqref{lemma-cool} gives
		\begin{equation}
		\begin{split}\label{estim-sobolev}  
		\left|\int_{\mathbb{R}}(\delta_{h}f_{1})f_{2}\varphi_k g\right|&=  \left|\int_{\mathbb{R}}(\delta_{h}f_{1}\bar\varphi_k)f_{2}\varphi_k g\right|\\
		&\leq \|\delta_{h}(f_{1}\bar\varphi_k)\|_{H^{-1/2}(\mathbb{R})}\|f_{2}\varphi_k g\|_{H^{1/2}(\mathbb{R})}\\
		&\leq C \|f_{1}\bar\varphi_k\|_{H^{1/2}(\mathbb{R})}\|f_{2}\|_{H^{1/2}_{\mathrm{\mathrm{uloc}}}(\sigma_{k})}\|\varphi_k\|_{L^{\infty}(\mathbb{R})} \|g\|_{H^{3/2}(\sigma_{k})}\\
		&\leq C \|f_{1}\|_{H^{1/2}_{\mathrm{\mathrm{uloc}}}(\sigma_{k})}\|f_{2}\|_{H^{1/2}_{\mathrm{\mathrm{uloc}}}(\sigma_{k})}\|\varphi_k\|_{W^{1,\infty}(\mathbb{R})} \|g\|_{H^{3/2}(\sigma_{k})},
		\end{split}
		\end{equation}
where $\bar\varphi_k$ belongs to $C^{\infty}_{0}(\mathbb{R})$ and satifies $\text{Supp}(\varphi_k)\subset\text{Supp}(\bar\varphi_k)$. This implies that
		\begin{equation}\label{estim-sobolev1}
		\|(\delta_{h}f_{1})f_{2}\varphi_k\|_{H^{-3/2}(\sigma_{k})}\leq C \|f_{1}\|_{H^{1/2}_{\mathrm{\mathrm{uloc}}}(\sigma_{k})}\|f_{2}\|_{H^{1/2}_{\mathrm{\mathrm{uloc}}}(\sigma_{k})}\|\varphi_k\|_{W^{1,\infty}(\mathbb{R})}.
		\end{equation}
	Hence,
	\begin{equation}\label{estim-hard-r3-1}
	\begin{split}
	\|&(\delta_{h}[\nabla_w^\perp\Psi]\cdot\nabla_w)\left(\left[\left((1+\alpha^2)\partial_{X}-\alpha\partial_{Y}\right)\varphi_k\right]\Psi\right)\\
	&
	+\left[(\delta_{h}[\nabla_w^\perp\Psi]\cdot\nabla_w)\varphi_k\right]\left[\left((1+\alpha^2)\partial_{X}-\alpha\partial_{Y}\right)\Psi\right]\|_{H^{-3/2}(\sigma_{k})}\\
	&\leq  C  (\|\Psi\|_{H^{1/2}_{\mathrm{\mathrm{uloc}}}(\sigma^M_w)}+\|\partial_{X}\Psi\|_{H^{-1/2}_{\mathrm{\mathrm{uloc}}}(\sigma^M_w)})^{2}\|\varphi_k\|_{W^{2,\infty}}.
	\end{split}
	\end{equation}
	Now, the first term in the r.h.s of \eqref{estim-hard-r3} satisfies
	\begin{equation}\label{estim-r3-3}
	\begin{split}
	&\left[(\delta_{h}[\nabla_w^\perp\Psi]\cdot\nabla_w)\left((1+\alpha^2)\partial_{X}-\alpha\partial_{Y}\right)\Psi\right]\varphi_k\\
	&\quad=\left[\varphi_k([\nabla_w^\perp\delta_{h}\Psi]\cdot\nabla_w)\left((1+\alpha^2)\partial_{X}-\alpha\partial_{Y}\right)\Psi\right]\\
	&\quad=\left[([\nabla_w^\perp\delta_{h}(\varphi_k\Psi)]\cdot\nabla_w)\left((1+\alpha^2)\partial_{X}-\alpha\partial_{Y}\right)\Psi\right]-\left[([(\nabla_w^\perp\varphi_k)\delta_{h}(\Psi)]\cdot\nabla_w)\left((1+\alpha^2)\partial_{X}-\alpha\partial_{Y}\right)\Psi\right]\\
	&\quad-\left[([\nabla_w^\perp[\delta_{h}(\varphi_k)\tau_{h}\Psi]]\cdot\nabla_w)\left((1+\alpha^2)\partial_{X}-\alpha\partial_{Y}\right)\Psi\right],
	\end{split}
	\end{equation}
	where $\tau_{h}$ is the translation of amplitude $h$ on the Y axis. Once again, the second and third terms above are easier to deal with as they include fewer derivatives and can be bounded similarly to \eqref{estim-hard-r3-1}. Now, we are left with analyzing the first term. We have
	\begin{equation*}
	\left[([\nabla_w^\perp\delta_{h}(\varphi_k\Psi)]\cdot\nabla_w)\left((1+\alpha^2)\partial_{X}-\alpha\partial_{Y}\right)\Psi\right]=\left[([\nabla_w^\perp\psi_{h}\cdot\nabla_w)\left((1+\alpha^2)\partial_{X}-\alpha\partial_{Y}\right)\Psi\right].
	\end{equation*}
	Similarly to the case of \eqref{estim-hard-r3-1}, using that $\psi_{h}\in H^{3/2}(\sigma_{k})$ and $\partial_{X}\psi_{h}\in H^{1/2}(\sigma_{k})$ yields
	\begin{equation}\label{estim-r3-4}
	\begin{split}
	&\|\left[([\nabla_w^\perp\delta_{h}(\varphi_k\Psi)]\cdot\nabla_w)\left((1+\alpha^2)\partial_{X}-\alpha\partial_{Y}\right)\Psi\right]\|_{H^{-3/2}(\sigma_{k})}\\
	&=\|\left[([\nabla_w^\perp\psi_{h}\cdot\nabla_w)\left((1+\alpha^2)\partial_{X}-\alpha\partial_{Y}\right)\Psi\right]\|_{H^{-3/2}(\sigma_{k})}\\
	&\leq C(\|\psi_{h}\|_{H^{3/2}(\sigma_{k})}+\|\partial_{X}\psi_{h}\|_{H^{1/2}(\sigma_{k})})(\|\Psi\|_{H^{3/2}_{\mathrm{\mathrm{uloc}}}(\sigma^M_w)}+\|\partial_{X}\Psi\|_{H^{1/2}_{\mathrm{\mathrm{uloc}}}(\sigma^M_w)}+\|\rho^{2}_{h}\|_{H^{-1/2}_{\mathrm{\mathrm{uloc}}}(\sigma^M_w)}).
	\end{split}
	\end{equation}
	
		Finally, taking $f\in H^{-1/2}_{\mathrm{\mathrm{uloc}}}(\sigma_{k})$, $l\in C^{\infty}_c(\sigma_{k})$, and $g\in H^{3/2}(\sigma_{k})$ we have
		\begin{equation*}
		\begin{split}
		\left|\int_{\mathbb{R}}fgl \right|&\leq \|f\|_{H^{-1/2}_{\mathrm{\mathrm{uloc}}}(\sigma_{k})} \|gl\|_{H^{1/2}(\sigma_{k})}\\
		&\leq\|f\|_{H^{-1/2}_{\mathrm{\mathrm{uloc}}}(\sigma_{k})}
		\|g\|_{H^{3/2}(\sigma_{k})}\|l\|_{W^{2,\infty}(\sigma_{k})}.
		\end{split}
		\end{equation*}
		Then, we infer
		\begin{equation*}
		\|fl\|_{H^{-3/2}(\sigma_{k})}\leq \|f\|_{H^{-1/2}_{\mathrm{\mathrm{uloc}}}(\sigma_{k})}\|l\|_{W^{2,\infty}(\sigma_{k})}.
		\end{equation*}
		From \eqref{bc-x-m} follows that
		\begin{equation}\label{estim-r3-7}
		\begin{split}
		\|\Delta_{w}\Psi\partial_{Y}\varphi_k\|_{H^{-3/2}(\sigma_{k})}&\leq C \|\Delta_{w}\Psi\|_{H^{-1/2}_{\mathrm{\mathrm{uloc}}}(\sigma_{k})}\|\varphi_k\|_{W^{2,\infty}(\sigma_{k})}\\
		&\leq C\left(\|\Psi\|_{H^{2}_{\mathrm{\mathrm{uloc}}}(\omega^b_w)}+\|\rho_{2}\|_{H^{-1/2}_{\mathrm{\mathrm{uloc}}}(\sigma_{k})}\right)
		\|\varphi_k\|_{W^{2,\infty}(\sigma_{k})}.
		\end{split}
		\end{equation}
	
	Therefore, from \eqref{estim-r3-0-1}, \eqref{estim-r3-1}, \eqref{estim-r3-2}, \eqref{estim-r3-3}, \eqref{estim-hard-r3-1}, \eqref{estim-r3-4}, \eqref{estim-r3-7} and applying the Trace theorem
	\begin{equation*}
	\|\rho_{3}^{h}\|_{H^{-3/2}(\sigma_{k})}\leq C(1+\|\varphi_k\|_{W^{4,\infty}})(\|\rho_{2}\|_{H^{-1/2}_{\mathrm{\mathrm{uloc}}}(\sigma^M_w)}+ \|\Psi\|_{H^{2}_{\mathrm{\mathrm{uloc}}}}(1+\|\psi_{h}\|_{H^{2}_{\mathrm{\mathrm{uloc}}}}+\|\Psi\|_{H^{2}_{\mathrm{\mathrm{uloc}}}})),
	\end{equation*}
	which is the desired estimate.
	\paragraph{Estimate of $f_{h}$.} Let us recall the expression of $f_{h}$,
	\begin{equation}\label{estim-f}
	\begin{split}
	f_h&=-\delta_{h}\left[4\varphi_k^{(3)}\left(\partial_{Y}-\alpha\partial_{X}\right)\Psi+6\varphi_k''\Delta_w\Psi+4\varphi_k'\left(\partial_{Y}-\alpha\partial_{X}\right)\Delta_w\Psi+\varphi_k^{(4)}\Psi\right.\\
	&\hspace{.6in}\left.-\partial_X\Psi\left(3\varphi_k'\Delta_w\Psi+\varphi_k^{(3)}\Psi\right)-3\varphi_k''\Delta_w\Psi\right]-\delta_{h}\left[(\nabla_w^\perp\Psi\cdot\nabla_w)\Delta_w(\varphi_k\Psi)\right].
	\end{split}
	\end{equation}
	Let us analyze the first (large) term of \eqref{estim-f}. Note that all the terms enclosed in brackets that are not quadratic in $\Psi$ involve at most a third order derivative of $\Psi$ and are all proportional to a derivative of $\varphi_k$ (hence compactly supported). This implies that 
	\begin{equation}\label{estim-f-2}
	\begin{split}
	&\left\|        \delta_{h}\left[4\varphi_k^{(3)}\left(\partial_{Y}-\alpha\partial_{X}\right)\Psi+6\varphi_k''\Delta_w\Psi+4\varphi_k'\left(\partial_{Y}-\alpha\partial_{X}\right)\Delta_w\Psi+\varphi_k^{(4)}\Psi\right]\right\|_{H^{-2}(\omega^b_w)}\\
	&\leq C\left\|4\varphi_k^{(3)}\left(\partial_{Y}-\alpha\partial_{X}\right)\Psi+6\varphi_k''\Delta_w\Psi+4\varphi_k'\left(\partial_{Y}-\alpha\partial_{X}\right)\Delta_w\Psi+\varphi_k^{(4)}\Psi\right\|_{H^{-1}(\omega^b_w)}\\
	&\leq C\left\|\Psi\right\|_{H^{2}_{\mathrm{\mathrm{uloc}}}}\|\varphi_k\|_{W^{4,\infty}(\omega^b_w)}.
	\end{split}
	\end{equation}
	The quadratic terms between the brackets in the second term of \eqref{estim-f} can be treated in an analogous manner
	\begin{equation}\label{estim-f-3}
	\begin{split}
	&\left\|    \delta_{h}\left[-\partial_X\Psi\left(3\varphi_k'\Delta_w\Psi+\varphi_k^{(3)}\Psi\right)-3\varphi_k''\Delta_w\Psi\right]\right\|_{H^{-2}(\omega^b_w)}\\
	&\leq 
	C\left\|  
	-\partial_X\Psi\left(3\varphi_k'\Delta_w\Psi+\varphi_k^{(3)}\Psi\right)-3\varphi_k''\Delta_w\Psi\right\|_{H^{-1}(\omega^b_w)}\\
	&\leq 	C\left\|\Psi\right\|^2_{H^{2}_{\mathrm{\mathrm{uloc}}}}\|\varphi_k\|_{W^{3,\infty}(\omega^b_w)}.
	\end{split}
	\end{equation}
	On the last term on the r.h.s of \eqref{estim-f}, we use the property
		\begin{equation*}
		\delta_{h}\left[(\left(\nabla_w^\perp\Psi\right)\cdot\nabla_w)\Delta_w(\varphi_k\Psi)\right] = \delta_{h}\left(\nabla_{w}\cdot\left(\Delta_w(\varphi_k\Psi)\nabla_w^\perp\Psi\right)\right),
		\end{equation*}
		by virtue of $\nabla_{w}\cdot\nabla_{w}^{\perp}(\delta_{h}\Psi)=0$. Therefore,
		\begin{equation*}
		\|\delta_{h}\left[(\left(\nabla_w^\perp\Psi\right)\cdot\nabla_w)\Delta_w(\varphi_k\Psi)\right] \|_{H^{-2}(\omega^b_w)}\leq C       \|\delta_{h}\left(\Delta_w(\varphi_k\Psi)\nabla_w^\perp\Psi\right) \|_{H^{-1}(\omega^b_w)},
		\end{equation*}
		and it suffices to estimate $\delta_{h}\left(\Delta_w(\varphi_k\Psi)\left(\nabla_w^\perp\Psi\right)\right) $ is the $H^{-1}(\omega^b_w)$ norm. We get
		\begin{equation}\label{estim-hard-f-0}
		\begin{split}
		\delta_{h}\left(\Delta_w(\varphi_k\Psi)\left(\nabla_w^\perp\Psi\right)\right)&=
		\Delta_w(\varphi_k\Psi)\delta_{h}\left(\nabla_w^\perp\Psi\right)+
		\delta_{h}\Delta_w(\varphi_k\Psi)\tau_{h}\left(\nabla_w^\perp\Psi\right)\\
		&=\Delta_w(\varphi_k\Psi)\delta_{h}\left(\nabla_w^\perp\Psi\right)+
		\Delta_w(\psi_{h})\tau_{h}\left(\nabla_w^\perp\Psi\right).
		\end{split}
		\end{equation}
		The second term satisfies
		\begin{equation}\label{estim-f-plus}
		\begin{split}
		\|
		\Delta_w(\psi_{h})\tau_{h}\left(\nabla_w^\perp\Psi\right)
		\|_{H^{-1}(\omega^b_w)}
		\leq C \|\psi_{h}\|_{H^{2}(\omega^b_w)}\|\Psi\|_{H^{2}_{\mathrm{\mathrm{uloc}}}(\omega)}.
		\end{split}
		\end{equation}
		The above estimate is obtained from applying the following result: {\it let $f\in H^{1}_{\mathrm{\mathrm{uloc}}}(\omega^b_w)$, $g\in L^{2}(\omega^b_w)$ with compact support, and $l\in H^{1}(\omega^b_w)$,
			\begin{equation}\label{Sobolev-1}
			\begin{split}
			\left|\int_{\mathbb{R}} f g l\right|&\leq \|g\|_{L^{2}(\omega^b_w)}\|f l\|_{L^{2}_{\mathrm{\mathrm{uloc}}}(\omega^b_w)}\leq C\|g\|_{L^{2}(\omega^b_w)}\|f\|_{L^{4}_{\mathrm{\mathrm{uloc}}}(\omega^b_w)}\|l\|_{L^{4}(\omega^b_w)},\\
			&\leq C\|g\|_{L^{2}(\omega^b_w)}\|f\|_{H^{1}_{\mathrm{\mathrm{uloc}}}(\omega^b_w)}\|l\|_{H^{1}(\omega^b_w)},
			\end{split}
			\end{equation}
			hence,
			\begin{equation}\label{Sobolev-2}
			\|fg\|_{H^{-1}(\omega^b_w)}\leq C\|g\|_{L^{2}(\omega^b_w)}\|f\|_{H^{1}_{\mathrm{\mathrm{uloc}}}(\omega^b_w)}.
			\end{equation}}
			
		The first term in \eqref{estim-hard-f-0} satisfies
	\begin{equation}\label{estim-hard-f}
	\begin{split}
	\Delta_w(\varphi_k\Psi)\delta_{h}\left(\nabla_w^\perp\Psi\right) =\varphi_k\Delta_w(\Psi)\delta_{h}\left(\nabla_w^\perp\Psi\right)
	+F(\delta_{h}\Psi,\Psi,\varphi_k),
	\end{split}
	\end{equation}
	where $F(\delta_{h}\Psi,\Psi,\varphi_k)$ is a linear combination of functions depending linearly of $\Psi$,  $\partial_{h}\Psi$ and $\varphi_k$, compactly supported and involving at most a first order derivative of $\delta_{h}\Psi$, a first order derivative of $\Psi$ and a second order derivative in $\varphi_k$. Therefore, proceeding similarly as in \eqref{estim-sobolev} but in $H^{1}(\omega^b_w)$ norm, we have
	\begin{equation}\label{estim-f-4}
	\begin{split}
	\|F(\delta_{h}\Psi,\Psi,\varphi_k)\|_{H^{-1}(\omega^b_w)}&\leq C\|\delta_{h}\Psi\|_{H^{1}_{\mathrm{\mathrm{uloc}}}}\|\Psi\|_{H^{2}_{\mathrm{\mathrm{uloc}}}}\|\varphi_k\|_{W^{2,\infty}}\\
	&\leq C\|\Psi\|_{H^{2}_{\mathrm{\mathrm{uloc}}}}^{2}\|\varphi_k\|_{W^{2,\infty}}.
	\end{split}
	\end{equation}
	The first term of \eqref{estim-hard-f} can be estimated noting that $\delta_{h}(fg) = (\delta_{h}f)g+\delta_{h}g\tau_{h}f$ where $\tau_{h}$ denotes the translation operator of size $h$ along the $Y$-axis. This gives
	\begin{equation*}
	\begin{split}
	&\varphi_k\Delta_w(\Psi)\delta_{h}\left(\nabla_w^\perp\Psi\right)=
	\Delta_w(\Psi)\delta_{h}\left(\nabla_w^\perp(\Psi\varphi_k)\right)\\
	&-
	\Delta_w(\Psi)\delta_{h}\left(\Psi\nabla_w^\perp\varphi_k\right)
	-\Delta_w(\Psi)\left(\delta_{h}\varphi_k\tau_{h}(\nabla_w^\perp\Psi)\right)\\
	&=\Delta_w(\Psi)\left(\nabla_w^\perp\psi_{h}\right)-
	\Delta_w(\Psi)\delta_{h}\left(\Psi\nabla_w^\perp\varphi_k\right)
	-\Delta_w(\Psi)\left(\delta_{h}\varphi_k\tau_{h}(\nabla_w^\perp\Psi)\right).
	\end{split}
	\end{equation*}
	The last two terms are now easier to bound. Using again \eqref{Sobolev-1}-\eqref{Sobolev-2} for a function $f\in H^{1}(\omega^b_w)$ with compact support and $g\in L^{2}_{\mathrm{\mathrm{uloc}}}(\omega^b_w)$ leads to
		\begin{equation}\label{estim-f-5-b}
		\begin{split}
		\|\varphi_k\Delta_w(\Psi)\delta_{h}\left(\nabla_w^\perp\Psi\right)\|_{H^{-1}_{\mathrm{\mathrm{uloc}}}(\omega^b_w)}&\leq C\|\Delta_{w}\Psi\|_{L^{2}_{\mathrm{\mathrm{uloc}}}(\omega^b_w)}\|\nabla_{w}^{\perp}\psi_{h}\|_{H^{1}(\omega^b_w)}
		+\|\varphi_k\|_{W^{2,\infty}(\omega^b_w)}\|\Psi\|^{2}_{H^{2}(\omega^b_w)}\\
		&\leq C \|\Psi\|_{H^{2}_{\mathrm{\mathrm{uloc}}}(\omega^b_w)}\left(\|\psi_{h}\|_{H^{2}(\omega^b_w)}+\|\varphi_k\|_{W^{2,\infty}(\omega^b_w)} \|\Psi\|_{H^{2}_{\mathrm{\mathrm{uloc}}}(\omega^b_w)}\right).
		\end{split}
		\end{equation}
	
	Therefore, from equations \eqref{estim-f-2}, \eqref{estim-f-3}, \eqref{estim-f-plus}, \eqref{estim-f-4},
	and \eqref{estim-f-5-b},
	we have 
	\begin{equation*}
	\|f_{h}\|_{H^{-2}(\omega^b_w)}\leq C(\|\Psi\|_{H^{2}_{\mathrm{\mathrm{uloc}}}}(1+\|\Psi\|_{H^{2}_{\mathrm{\mathrm{uloc}}}}+\|\psi_{h}\|_{H^{2}_{\mathrm{\mathrm{uloc}}}}))(1+\|\varphi_k\|_{W^{4,\infty}(\omega^b_w)}).
	\end{equation*}
	\paragraph{Estimate of $\rho_{2}^{h}$.}
	We have
	\begin{equation*}
	\begin{split}
	\rho_{2}^{h} =& \delta_{h}[\varphi_k\rho_2]+(1+\alpha^2)\delta_h\left[2\partial_{Y}\varphi_k\left(\partial_{Y}-\alpha\partial_{X}\right)\Psi+\partial_{Y}^2\varphi_k\Psi\right]\\
	=& \delta_{h}[\varphi_k\rho_2]+(1+\alpha^2)\left[2\partial_{Y}\varphi_k\left(\partial_{Y}-\alpha\partial_{X}\right)\psi_{h}+\partial_{Y}^2\varphi_k\psi_{h}\right]\\
	&+
	(1+\alpha^2)\left[2\partial_{Y}\delta_{h}\varphi_k\left(\partial_{Y}-\alpha\partial_{X}\right)\Psi+\partial_{Y}^2\delta_{h}\varphi_k\Psi\right].
	\end{split}
	\end{equation*}
	Since $\|\delta_{h}f\|_{H^{p}(\sigma_{k})}\leq C\|f\|_{H^{p+1}(\sigma_{k})}$ for $h$ sufficiently small, where $C$ is independent of $h$, we have
	\begin{equation*}
	\begin{split}
	\|\rho_{2}^{h} \|_{H^{-1/2}(\sigma_{k})}\leq& C\left( \|\rho_{2}\|_{H^{1/2}_{\mathrm{\mathrm{uloc}}}(\sigma^M_w)}\|\varphi_k\|_{W^{1,\infty}}+\|\varphi_k\|_{W^{1,\infty}}\|\Psi\|_{H^{3/2}(\sigma_{k})}\right.\\
	&\left.+\|\varphi_k\|_{W^{3,\infty}}\|\Psi\|_{H^{-1/2}(\sigma_{k})}+\|\varphi_k\|_{W^{2,\infty}}\|\Psi\|_{H^{1/2}(\sigma_{k})}\right)\\
	\leq& C\left(\|\rho_{2}\|_{H^{1/2}_{\mathrm{\mathrm{uloc}}}(\sigma^M_w)}\|\varphi_k\|_{W^{1,\infty}}+\|\varphi_k\|_{W^{3,\infty}}\|\Psi\|_{H^{3/2}(\sigma_{k})}\right).
	\end{split}
	\end{equation*}
	Therefore, using the Trace theorem
	\begin{equation*}
	\|\rho_{2}^{h} \|_{H^{-1/2}(\sigma_{k})}\leq C\left( \|\rho_{2}\|_{H^{1/2}_{\mathrm{\mathrm{uloc}}}(\sigma^M_w)}+\|\Psi\|_{H^{2}_{\mathrm{\mathrm{uloc}}}}\right)\|\varphi_k\|_{W^{3,\infty}},
	\end{equation*}
	which is the desired estimate.
\end{proof}

\paragraph{Higher interior regularity.}
We intend to iterate the argument used in the case $m=1$, thereby deducing
our solution in various higher Sobolev spaces. As before, we start with the interior regularity analysis.

\label{sec:interior-higher}
\
\begin{proposition}\label{t:interior-higher}
	Let $m$ be a nonnegative integer and  $\Psi\in H^2_{\mathrm{\mathrm{uloc}}}(\omega^b_w)$ be a solution of the PDE
	\begin{equation*}
	\partial_X \Psi+(\nabla_w^\perp \Psi\cdot\nabla_w)\Delta_w \Psi-\Delta^2_w \Psi=0\quad in\: \omega^b_w.
	\end{equation*}
	Then, $\Psi\in H^{m+2}_{\mathrm{\mathrm{uloc}}}(\omega^b_w)$ and for each $\omega '\subset\subset\omega$
	\begin{equation*}
	\| \Psi\|_{H^{2+m}(\omega')}\leq c(\omega') \| \Psi\|_{H^{2}_{\mathrm{\mathrm{uloc}}}(\omega^b_w) }
	\end{equation*}
\end{proposition}

This proposition is a direct consequence of Theorem \ref{theorem_firstlocal}.

\paragraph{Regularity up to the boundary for $m>1$.} We now complete the proof of Theorem \ref{Lemma_14} for $m\geq 2$ using an induction argument. Taking $m=1$ as the base case, we assume that the theorem holds up to $m\in\mathbb{N}^{*}$ and, then, prove that it is true as well for $m+1$. Again we localize the solution near a fixed $k\in\mathbb{Z}$ using $\tilde \varphi_k$ and apply a finite difference operator $\delta_{h}$ to show that 
\begin{equation}\label{prin-estimate-higher}
\|\delta_{h}(\varphi_k \Psi)\|_{H^{m+2}(\omega^{b}_{w})}\leq C(\varphi_k) (\|\rho_{2}\|_{H^{m+1/2}_{\mathrm{\mathrm{uloc}}}}+\|\rho_{2}\|_{H^{m-1/2}_{\mathrm{\mathrm{uloc}}}}).
\end{equation}
Therefore, $\tilde \varphi_k\Psi$ belongs to $H^{m+3}(\omega^{b}_{w})$ if $\rho_{2}\in H^{m+1/2}_{\mathrm{\mathrm{uloc}}}(\omega^{b}_{w})$ and $\rho_{3}\in H^{m-1/2}_{\mathrm{\mathrm{uloc}}}(\omega^{b}_{w})$ which are exactly the hypotheses of Theorem \ref{Lemma_14} when adding one degree of regularity to $m$. From the interior regularity result given earlier, we have that $\Psi\in H^{m+3}_{\mathrm{\mathrm{uloc}}}(\omega^{b}_{w})$. Hence, Theorem \ref{Lemma_14} indeed holds for $m+1$, which concludes the induction.

The proof of estimate \eqref{prin-estimate-higher} follows the same ideas as the ones presented for the case $m = 1$. First, we have the following result
\begin{lemma}\label{lem:elliptic-reg-higher}
	Consider the linear problem
	\begin{align}
	\partial_{X}z-\Delta^{2}_w z&=f_{1}\quad\textrm{in}\;\omega_k,\nonumber\\
	(1+\alpha^2)\Delta_w z&=f_{2},\label{linearpb-higher}\\
	\left[(1+\alpha^2)\partial_{X}-2\alpha\partial_{Y}\right]\Delta_w z -\frac{z}{2}\Big|_{\sigma_{k}}&=f_{3},\nonumber\\
	z=\partial_Xz&=0,\;\text{on}\;\partial\omega_{k}\setminus\sigma_{k}\nonumber,
	\end{align}
	where $\omega_k=\omega^{b}\cap\mathrm{Supp}\varphi_k$, $\sigma_{k}=\partial\omega_{k}\cap\partial\omega^b_w$ and $\mathrm{n}$ stands for the unit outer normal vector at the boundary. If $f_{1}\in H^{m-2}(\omega_{k})$, $f_{2}\in H^{m-1/2}(\sigma_{k})$, $f_{3}\in H^{m-3/2}(\sigma_{k})$. This problem has a unique solution $z\in H^{m+2}(\omega_{k})$ and it satisfies the estimate
	\begin{equation}\label{ellip-reg-est-higher}
	\|z\|_{H^{m+2}(\omega_{k})}\leq C\left(\|f_{1}\|_{H^{m-2}(\omega_{k})}+\|f_{2}\|_{H^{m-1/2}(\sigma_{k})}+\|f_{3}\|_{H^{m-3/2}(\sigma_{k})}\right).
	\end{equation}
\end{lemma}
This lemma is a direct consequence of Lemma \ref{lem:elliptic-reg}. As before, we need the following result
\begin{lemma}
	\label{lem:estimate_terms-higher} Let $\sigma_k=\sigma^M_w\cap \mathrm{Supp}\varphi_k$ where $\varphi_k\in \mathcal{C}^\infty_0(\mathbb{R})$ and is equal to $1$ in a neighborhood of $\sigma^M_w$. Consider the functions $\Psi\in H^{m+2}_{\mathrm{\mathrm{uloc}}}(\omega)$, $\rho_2\in H^{m-1/2}_{\mathrm{\mathrm{uloc}}}(\mathbb{R})$ and $\rho_3\in H^{m-3/2}_{\mathrm{\mathrm{uloc}}}(\mathbb{R})$. Then, for any $h\in \mathbb{R}$ we have the estimates 
	\begin{equation*}
	\begin{split}
	\|\rho_{3}^{h}\|_{H^{m-3/2}(\sigma_{k})}&\leq C(1+\|\varphi_k\|_{W^{m+4,\infty}})(\|\rho_{3}\|_{H^{m-1/2}_{\mathrm{\mathrm{uloc}}}(\sigma^M_w)}+ \|\Psi\|_{H^{m+2}_{\mathrm{\mathrm{uloc}}}}(1+\|\psi_{h}\|_{H^{m+2}_{\mathrm{\mathrm{uloc}}}}+\|\Psi\|_{H^{m+2}_{\mathrm{\mathrm{uloc}}}})),\\
	\|\rho_{2}^{h} \|_{H^{m-1/2}(\sigma_{k})}&\leq C\left( \|\rho_{2}\|_{H^{m+1/2}_{\mathrm{\mathrm{uloc}}}(\sigma^M_w)}+\|\Psi\|_{H^{l}_{\mathrm{\mathrm{uloc}}}}\right)\|\varphi_k\|_{W^{l+3,\infty}},\\
	\|f_{h}\|_{H^{m-2}(\omega^b_w)}&\leq C(\|\Psi\|_{H^{m+2}_{\mathrm{\mathrm{uloc}}}}(1+\|\Psi\|_{H^{m+2}_{\mathrm{\mathrm{uloc}}}}+\|\psi_{h}\|_{H^{m+2}_{\mathrm{\mathrm{uloc}}}}+\|\Psi_{L}\|_{H^{m+2}_{\mathrm{\mathrm{uloc}}}}))(1+\|\varphi_k\|_{W^{l+1,\infty}(\omega^b_w)}),
	\end{split}
	\end{equation*}
	where $C$ is a constant depending only on the domain $\omega^b_w$.
	
\end{lemma}
Assume that Lemma \ref{lem:elliptic-reg-higher} and \ref{lem:estimate_terms-higher} holds. Then, for $l\in \mathbb{N}^{*}$ and $l\geq 2$, \eqref{prin-estimate-higher} follows directly by induction on $\{2,...,l\}$. We now show a way to adapt the proof of \ref{lem:estimate_terms-higher} to prove Lemma \ref{lem:estimate_terms-higher}.\\

\begin{proof}[Proof of Lemma \ref{lem:estimate_terms-higher}]
	Let assume that $\Psi \in H^{l}$ for $l\geq 3$ (the case $l=2$ was shown in Lemma \ref{lem:elliptic-reg}).
	The estimate of the linear terms are exactly the same as in the proof of \ref{lem:estimate_terms}. Therefore we are left with showing the regularity of the quadratic terms in $f_{h}$ and $\rho_{3}^{h}$. \paragraph{Estimate of $\rho_{3}^{h}$.}  We have
	\begin{equation*}
	\begin{split}
	\rho_3^{h}&=
	\delta_{h}[\varphi_k\rho_3]+\delta_{h}\left[\left((1+\alpha^2)\partial_{X}-2\alpha\partial_{Y}\right)\left(2\partial_{Y}\varphi_k\left(\partial_{Y}-\alpha\partial_{X}\right)\Psi+\partial_{Y}^2\varphi_k\Psi\right)\right.\\
	&\left.-\nabla_w^\perp\Psi\cdot\left(-\alpha\nabla_w\left(\Psi\partial_{Y}\varphi_k\right)+\left((1+\alpha^2)\partial_{X}-\alpha\partial_{Y}\right)\Psi(0,\partial_{Y})\varphi_k\right)-\left(\partial_{Y}\left(\tfrac{|\nabla_w^\perp\Psi|^2}{2}\right)\varphi_k\right)\right]\\
	&+(\delta_{h}[\nabla_w^\perp\Psi]\cdot\nabla_w)\left((1+\alpha^2)\partial_{X}-\alpha\partial_{Y}\right)(\varphi_k\Psi).
	\end{split}
	\end{equation*}
	All the terms can be estimated as before. In fact most of them are easier to compute as $H^{s}(\sigma_{k})$ is an algebra for $s>1/2$. For $f\in H^{l-3/2}(\sigma_{k})$ with compact support and $g\in H^{l-3/2}_{\mathrm{uloc}}(\sigma_{k})$, we have
		\begin{equation*}
		\|fg\|_{H^{l-3/2}(\sigma_{k})}\leq C\|f\|_{H^{l-3/2}(\sigma_{k})}\|g\|_{H^{l-3/2}(\sigma_{k})}.
		\end{equation*}
		which replaces \eqref{lemma-cool}.
		Let us illustrate the case of 
		$\delta_{h}\left[\partial_{Y}\left(\frac{|\nabla^{\perp}\Psi|^{2}}{2}\right)\varphi_k\right]$. As before, this term can be decomposed as in \eqref{estim-carre} and the difficulty consists of estimating
		\begin{equation*}
		\begin{split}
		\delta_{h}\left[\partial_{Y}\left(\frac{|\nabla^{\perp}\Psi|^{2}}{2}\varphi_k\right)\right]
		\end{split}
		\end{equation*}
		From the algebra property, it follows that
		\begin{equation*}
		\begin{split}
		\left\|\delta_{h}\left[\partial_{Y}\left(\frac{|\nabla^{\perp}\Psi|^{2}}{2}\varphi_k\right)\right]\right\|_{H^{l-7/2}(\sigma_{k})}&\leq C \||\nabla^{\perp}\Psi|^{2}\varphi_k\|_{H^{l-3/2}(\sigma_{k})}\\
		&\leq C\|\nabla^{\perp}\Psi\|^{2}_{H^{l-3/2}_{\mathrm{uloc}}(\sigma_{k})}\|\varphi_k\|_{H^{l-3/2}(\sigma_{k})}.\\
		&\leq C\left(\|\partial_{X}\Psi\|_{H^{l-3/2}_{\mathrm{uloc}}(\sigma_{k})}+\|\Psi\|_{H^{l-1/2}_{\mathrm{uloc}}(\sigma_{k})}\right)\|\varphi_k\|_{H^{l-3/2}(\sigma_{k})}.
		\end{split}
		\end{equation*}
		Note that such a direct estimate does not work for $l=2$, which explains why this term was treated previously in a lengthier way, also involving the norm of $\psi_{h}$.
	
	\paragraph{Estimate of $\rho_{3}^{h}$.} Concerning the estimate of $f_{h}$ the same can be done by assuming that $l\geq 3$ and noticing that $H^{2}(\omega^b_w)\in L^{\infty}(\omega^b_w)$. Therefore, for $f \in H^{l-1}(\omega^b_w)$, $g\in H^{l-2}(\omega^b_w)$ and $\theta\in C_{c}^{\infty}(\overline{\omega^b_w})$, we obtain
	\begin{equation*}
	\|f g \theta \|_{H^{l-3}(\omega^b_w)}\leq \|f\|_{H^{l-1}(\omega^b_w)}\|g\|_{H^{l-2}(\omega^b_w)}\|\theta \|_{W^{l,\infty}(\omega^b_w)}.
	\end{equation*}
	Noticing then that $H^{s}_{\mathrm{\mathrm{uloc}}}(\omega^b_w)$ is an algebra as soon as $s>1$, we have
	\begin{equation*}
	\begin{split}
	\|\left[(-\partial_{Y}+\alpha\partial_X)\psi_{h}\partial_{X} +\partial_{X}\psi_{h}(\partial_{Y}-\alpha\partial_X))\right]\Delta_{w}\Psi\|_{H^{-2}(\omega^b_w)}\leq C\|\psi_{h}\|_{H^{2}(\omega^b_w)}\|\Psi\|_{H^{2}_{\mathrm{\mathrm{uloc}}}}^{2}.
	\end{split}
	\end{equation*}
	all other terms of $f_{h}$ can be dealt with exactly as in the proof of Lemma \ref{lem:estimate_terms} and this ends the proof of Lemma \ref{lem:estimate_terms-higher}.
\end{proof}

This concludes the proof of Theorem \ref{Lemma_14}.
	\end{proof}
     \subsection{Connecting the solutions at the artificial boundary}\label{s:join}
In previous sections, we showed the existence and uniqueness of the solution on both the half-space and the bumped domain. It remains to connect both results at the artificial boundary $X=M$. This local analysis is based on the implicit function theorem and will allow us to establish the solution of the problem on the whole boundary layer domain.

On account of Theorem \ref{Lemma_14}, we know the solution $\Psi^{-}$ of \eqref{eq_left1} is well-defined and satisfies the estimate
\begin{equation*}
\left\|\Psi^{-}\big|_{X=M}\right\|_{H^{m+3/2}_{\mathrm{uloc}}}+\left\|\partial_X\Psi^{-}\big|_{X=M}\right\|_{H^{m+1/2}_{\mathrm{uloc}}}\leq C\left(\|\phi\|_{W^{2,\infty}}+\|\rho_2\|_{ H^{m-1/2}_{\mathrm{uloc}}}+\|\rho_3\|_{ H^{m-3/2}_{\mathrm{uloc}}}\right),
\end{equation*}
for some positive constant $C$ depending on $\alpha$, $M$ and $\|\gamma_w\|_{W^{2,\infty}}$. The existence and uniqueness of the solution $\Psi^{+}$ of (\ref{pb:halfspace_nl_w})  with $\Psi^{+}\big|_{X=M}=\Psi^{-}\big|_{X=M}$ and $\partial_{X}\Psi^{+}\big|_{X=M}=\partial_{X}\Psi^{-}\big|_{X=M}$ is guaranteed by Proposition \ref{prop:fix}  when $C\left(\|\phi\|_{W^{2,\infty}}+\|\rho_2\|_{ H^{m-1/2}_{\mathrm{uloc}}}+\|\rho_3\|_{ H^{m-3/2}_{\mathrm{uloc}}}\right)$  smaller than a certain quantity $\delta_0>0$.

Furthermore, $\mathcal{A}_2[\Psi^{+}\big|_{X=M},\partial_X\Psi^{+}\big|_{X=M}]$ and $\mathcal{A}_3[\Psi^{+}\big|_{X=M},\partial_X\Psi^{+}\big|_{X=M}]$ belong to $H^{m-1/2}_{\mathrm{uloc}}$ and $H^{m-3/2}_{\mathrm{uloc}}$, respectively. Thus, the mapping $$\mathcal{F}:W^{2,\infty}(\omega_w)\times H^{m-1/2}_{\mathrm{uloc}}(\mathbb{R})\times H^{m-3/2}_{\mathrm{uloc}}(\mathbb{R})\rightarrow H^{m-1/2}_{\mathrm{uloc}}(\mathbb{R})\times H^{m-3/2}_{\mathrm{uloc}}(\mathbb{R})$$ given by
\begin{equation*}
\begin{array}{rcl}
\mathcal{F}(\phi,\rho_2,\rho_3)&=&\left(\mathcal{A}_2[\Psi^{+}\big|_{X=M},\partial_X\Psi^{+}\big|_{X=M}]-\rho_2,\mathcal{A}_3[\Psi^{+}\big|_{X=M},\partial_X\Psi^{+}\big|_{X=M}]-\rho_3\right)
\end{array}
\end{equation*}
is well-defined. Note that, when $\phi, \rho_2$ and $\rho_3$ are simultaneously equal to zero, $\mathcal{F}(0,0,0)=0$ as a direct consequence of Theorem \ref{Lemma_14}.  The main idea consists in applying the implicit function theorem $\mathcal{F}$ to find a solution of $\mathcal{F} (\phi, \rho_2, \rho_3 ) = 0$ for $\phi$ in a vicinity of zero. Thus, we need to first check the following hypotheses:
\begin{itemize}
	\item $\mathcal{F}$ is continuously Fréchet differentiable;
	\item $(v_1,v_2)\mapsto d\mathcal{F}(0,0,0)(0,v_1,v_2)$ is a Banach space isomorphism on $H^{m-1/2}_{\mathrm{uloc}}(\mathbb{R})\times H^{m-3/2}_{\mathrm{uloc}}(\mathbb{R})$, where $d$ is the differential with respect to $\rho_2$ and $\rho_3$.
\end{itemize}

\textit{$F$ is a $\mathcal{C}^{1}$ mapping in a neighborhood of zero:} Let $(\phi^{0},\rho_2^{0}, \rho_3^{0})$ and $(\phi,\rho_2, \rho_3)$ be in a vicinity of zero in the sense of the functional norm of $H^{m+3/2}_{\mathrm{uloc}}(\mathbb{R})\times H^{m-1/2}_{\mathrm{uloc}}(\mathbb{R})\times H^{m-3/2}_{\mathrm{uloc}}(\mathbb{R})$. Let $\Psi_0^{\pm}$ and $\Psi^{\pm}$ be  solutions of (\ref{pb:halfspace_nl_w}), (\ref{eq_left1}) associated to $(\phi^{0},\rho_2^{0}, \rho_3^{0})$ and $(\phi+\phi^{0},\rho_2+\rho_2^{0}, \rho_3+\rho_3^{0})$ respectively. We introduce the functions $\psi^{\pm}=\Psi^{\pm}-\Psi^{\pm}_0$. Then, easy computations show that $\psi^{-}$ satisfies
\begin{eqnarray}\label{problem_c1}
\partial_X\psi^{-}-\Delta_w^{2}\psi^{-}+Q_w(\psi^{-},\psi^{-}+\Psi^{-}_0)+Q_w(\Psi^{-}_0,\psi^{-})&=&0,\quad\textrm{in}\quad\omega_w^b\setminus\sigma_w\nonumber\\
\left[\psi^{-}\right]\Big|_{\sigma_w}=\phi,\quad
\left[\partial_X^{k}\psi^{-}\right]\Big|_{\sigma_w}&=&0,\quad k=1,2,3\nonumber\\
(1+\alpha^2)\Delta_w\psi^{-}\Big|_{\sigma_w^M}&=&\rho_2,\nonumber\\
\left[-\left((1+\alpha^2)\partial_{X}-2\alpha\partial_{Y}\right)\Delta_w\psi^-+\partial_{Y}\left(\displaystyle\frac{\nabla_w^{\perp}\psi^{-}\nabla_w^{\perp}(\psi^{-}+\Psi_0^-)}{2}\right)\right.\hspace{7mm}&&\\
+(\nabla_w^{\perp}\psi^{-}\cdot\nabla_w)\left((1+\alpha^2)\partial_{X}-\alpha\partial_{Y}\right)
(\psi^{-}+\Psi_0^-)\hspace{7mm}&&\nonumber\\
\left.+(\nabla_w^{\perp}\Psi_0^-\cdot\nabla_w)\left((1+\alpha^2)\partial_{X}-\alpha\partial_{Y}\right)\psi^{-}+\displaystyle\frac{\psi^-}{2}\right]\Bigg|_{\sigma_w^M}&=&\rho_3,\nonumber\\
\psi^{-}\big|_{X=-\gamma(Y)}=\partial _{n}\psi^{-}\big|_{X=-\gamma(Y)}&=&0.\nonumber
\end{eqnarray}
Similar arguments to the ones of Theorem \ref{Lemma_14} apply to (\ref{problem_c1}). This leads to the estimate
\begin{equation*}
\left\|\psi^{-}\right\|_{H^{2}_{\mathrm{uloc}}(\omega_w)}+\left\|\psi^{-}\big|_{X=M}\right\|_{H^{m+3/2}_{\mathrm{uloc}}}+\left\|\partial_X\psi^{-}\big|_{X=M}\right\|_{H^{m+1/2}_{\mathrm{uloc}}}\leq C\left(\|\phi\|_{W^{2,\infty}}+\|\rho_2\|_{ H^{m-1/2}_{\mathrm{uloc}}}+\|\rho_3\|_{ H^{m-3/2}_{\mathrm{uloc}}}\right),
\end{equation*}
for $\|\phi^{0}\|_{H^2_{\mathrm{uloc}}}+\|\rho_2^{0}\|_{ H^{m-1/2}_{\mathrm{uloc}}}+\|\rho_3^{0}\|_{ H^{m-3/2}_{\mathrm{uloc}}}$ and $\|\phi\|_{W^{2,\infty}}+\|\rho_2\|_{ H^{m-1/2}_{\mathrm{uloc}}}+\|\rho_3\|_{ H^{m-3/2}_{\mathrm{uloc}}}$ small enough. Hence, the solution $\psi^{-}$ belongs to $H^{2}_{\mathrm{uloc}}(\omega_w^b)$ and to $H^{m+2}_{\mathrm{uloc}}((M,M')\times\mathbb{R})$, for $M'>\sup\,(-\gamma_w)$. Moreover, we can assume
\[
\psi^{-}=\psi^{-}_{\ell}+O(\|\phi\|^2_{W^{2,\infty}}+\|\rho_2\|^2_{ H^{m-1/2}_{\mathrm{uloc}}}+\|\rho_3\|^2_{ H^{m-3/2}_{\mathrm{uloc}}}),
\]
where $\psi^{-}_{\ell}$ is the solution of 
\begin{eqnarray}\label{pb:psiL_-}
\partial_X\psi^{-}_{\ell}-\Delta_w^{2}\psi^{-}_{\ell}+Q_w(\psi^{-}_{\ell},\Psi^{-}_0)+Q_w(\Psi^{-}_0,\psi^{-}_{\ell})&=&0,\quad\textrm{in}\quad\omega_w^b\nonumber\\
\left[\psi^{-}_{\ell}\right]\Big|_{X=0}=\phi,\quad
\left[\partial_X^{k}\psi^{-}_{\ell}\right]\Big|_{X=0}&=&0,\quad k=1,2,3\nonumber\\
(1+\alpha^2)\Delta_w\psi^{-}_{\ell}\Big|_{X=M}&=&\rho_2,\nonumber\\
\left[-\left((1+\alpha^2)\partial_{X}-2\alpha\partial_{Y}\right)\Delta_w\psi^-_{\ell}\right.\hspace{7mm}&&\\
+(\nabla_w^{\perp}\psi^{-}_{\ell}\cdot\nabla_w)\left((1+\alpha^2)\partial_{X}-\alpha\partial_{Y}\right)\Psi_0^-\hspace{7mm}&&\nonumber\\
\left.+(\nabla_w^{\perp}\Psi_0^-\cdot\nabla_w)\left((1+\alpha^2)\partial_{X}-\alpha\partial_{Y}\right)\psi^{-}_{\ell}+\displaystyle\frac{\psi^-}{2}\right]\Bigg|_{X=M}&=&\rho_3,\nonumber\\
\psi^{-}_{\ell}\big|_{X=-\gamma(Y)}=\partial _{n}\psi^{-}_{\ell}\big|_{X=-\gamma(Y)}&=&0,\nonumber
\end{eqnarray}

Note that (\ref{pb:psiL_-}) is similar to system (\ref{problem_c1}) but lacks the quadratic terms. Solution of problem (\ref{pb:psiL_-}) can be sought using once again Lady\v{z}henskaya and Solonnikov's truncated energy method, which provides a dependence on the $\phi$ and $\rho_i$, $i=2,3,$ at the main order. 
An analogous conclusion can be drawn from applying Theorem \ref{Theorem2_DGV2017} to the problem defined on the half space. In this case, we have $\psi^{+}=\psi^{+}_{\ell}+O(\|\phi\|_{W^{2,\infty}}^2+\|\rho_2\|^2_{ H^{m-1/2}_{\mathrm{uloc}}}+\|\rho_3\|^2_{ H^{m-3/2}_{\mathrm{uloc}}})$. Here, $\psi^{+}_{\ell}$ satisfies
\begin{eqnarray}\label{pb:psiL_+}
Q_w(\psi^{+}_{\ell},\Psi^{+}_0)+Q_w(\Psi^{+}_0,\psi^{+}_{\ell})+\partial_X\psi^{+}_{\ell}-\Delta_w^{2}\psi^{+}_{\ell}&=&0,\quad\textrm{in}\;X>M\nonumber\\
\psi^{+}_{\ell}\big|_{X=M}&=&\psi^{-}_{\ell}\big|_{X=M},\\
\partial_X\psi^{+}_{\ell}\big|_{X=M}&=&\partial_X\psi^{-}_{\ell}\big|_{X=M}.\nonumber
\end{eqnarray}
From Theorem \ref{Theorem2_DGV2017}, we infer that
\begin{equation*}
\|\psi^{+}_{\ell}\|_{H^{m+2}_{\mathrm{uloc}}}\leq C\left(\left\|\psi^{+}_{\ell}\big|_{X=M}\right\|_{H^{m+3/2}_{\mathrm{uloc}}}+\left\|\partial_{X}\psi^{+}_{\ell}\big|_{X=M}\right\|_{H^{m+1/2}_{\mathrm{uloc}}}\right)\leq C\left(\|\phi\|_{W^{2,\infty}}+\|\rho_2\|_{ H^{m-1/2}_{\mathrm{uloc}}}+\|\rho_3\|_{ H^{m-3/2}_{\mathrm{uloc}}}\right).
\end{equation*}
More details on the resolution of problems similar to \eqref{pb:psiL_-} and \eqref{pb:psiL_+} can be found in Section \ref{s:linearized}. Furthermore,
\begin{equation}\label{frechet_nl}
\begin{split}
\mathcal{F}_1(\phi+\phi^{0},\rho_2+\rho_2^{0}, \rho_3+\rho_3^{0})-\mathcal{F}_1(\phi^{0},\rho_2^{0}, \rho_3^{0})&=-\rho_2+(1+\alpha^2)\Delta_w\psi^{+}_{\ell}\\
&+O\left(\|\phi\|^2_{W^{2,\infty}}+\|\rho_2\|^2_{ H^{m-1/2}_{\mathrm{uloc}}}+\|\rho_3\|^2_{ H^{m-3/2}_{\mathrm{uloc}}}\right),\\
\mathcal{F}_2(\phi+\phi^{0},\rho_2+\rho_2^{0}, \rho_3+\rho_3^{0})-\mathcal{F}_2(\phi^{0},\rho_2^{0}, \rho_3^{0})&=-\rho_3+\left((1+\alpha^2)\partial_{X}-2\alpha\partial_{Y}\right)\Delta_w\psi^{+}_{\ell}-\dfrac{\psi^{+}_{\ell}}{2}\\
&+(\nabla_w^{\perp}\psi^{+}_{\ell}\cdot\nabla_w)\partial_{X}\left((1+\alpha^2)\partial_{X}-\alpha\partial_{Y}\right)\Psi^{+}_0\\
&+(\nabla_w^{\perp}\Psi^{+}_0\cdot\nabla_w)\left((1+\alpha^2)\partial_{X}-\alpha\partial_{Y}\right)\psi^{+}_{\ell}\\&+O\left(\|\phi\|^2_{W^{2,\infty}}+\|\rho_2\|^2_{ H^{m-1/2}_{\mathrm{uloc}}}+\|\rho_3\|^2_{ H^{m-3/2}_{\mathrm{uloc}}}\right).\end{split}
\end{equation}
We see that the Fréchet differential  of $\mathcal{F}$ at  $(\phi^{0},\rho_2^{0}, \rho_3^{0})$ is defined by $L=(L_1, L_2)$ where
\begin{equation*}
\begin{split}
L_1&=-\rho_2+(1+\alpha^2)\Delta_w\psi^{+}_{\ell},\\ L_2&=-\rho_3-\left((1+\alpha^2)\partial_{X}-2\alpha\partial_{Y}\right)\Delta_w\psi^{+}_{\ell}+\dfrac{\psi^{+}_{\ell}}{2}+(\nabla_w^{\perp}\psi^{+}_{\ell}\cdot\nabla_w)\partial_{X}\left((1+\alpha^2)\partial_{X}-\alpha\partial_{Y}\right)\Psi^{+}_0\\
&+(\nabla_w^{\perp}\Psi^{+}_0\cdot\nabla_w)\left((1+\alpha^2)\partial_{X}-\alpha\partial_{Y}\right)\psi^{+}_{\ell}.
\end{split}
\end{equation*}

It is important to emphasize that (\ref{pb:psiL_+}) is a linear problem with respect to the perturbation function $\psi^{+}_{\ell}$. Therefore, it is possible to obtain the exact form of its solution by applying the same reasoning as the one in Section \ref{s:linear_west}. The use of Fourier analysis to (\ref{pb:psiL_+})  provides a solution  $\psi^{+}_{\ell}$ showing continuous dependence on $\Psi^{+}_0$, as well as, on the boundary condition. Note that $\psi^{+}_{\ell}\big|_{X=M}$ and $\partial_{X}\psi^{+}_{\ell}\big|_{X=M}$ depend in turn on $\Psi_{0}^{-}$. Hence, $\psi^{\pm}$ depend continuously on $\Psi_0^{\pm}$, and conversely on $\phi^{0}, \rho^{0}_2, \rho^{0}_3$. Therefore, $\mathcal{F}$ is a $\mathcal{C}^{1}$ function in a neighborhood of zero.

\textit{$d\mathcal{F}(0,0,0)$ is invertible:} Since $d\mathcal{F}(0,0,0) = L^0(\cdot, \cdot)$, we consider the systems satisfied by $\psi^{\pm}_{\ell}$ with $\Psi_{0}^{\pm}=0$ and $\phi=0$. 
\begin{equation}\label{pb:psiL_w_psi_zero}
\left\{\begin{array}{rcl}
\partial_X\psi^{-}_{\ell}-\Delta_w^{2}\psi^{-}_{\ell}&=&0,\quad\textrm{in}\quad X\leq M\nonumber\\
(1+\alpha^2)\Delta_w\psi^{-}_{\ell}\Big|_{X=M}&=&\rho_2,\\
\begin{array}{l}\left[-\left((1+\alpha^2)\partial_{X}-2\alpha\partial_{Y}\right)\partial_{X}\Delta_w\psi^{-}_{\ell}+\displaystyle\frac{\psi^{-}_{\ell}}{2}\right]\end{array}\Bigg|_{X=M}&=&\rho_3,\nonumber\\
\psi^{-}_{\ell}\big|_{X=-\gamma(Y)}=\partial _{n}\psi^{-}_{\ell}\big|_{X=-\gamma(Y)}&=&0,\nonumber
\end{array}\right.,
\end{equation}
and
\begin{equation}\label{pb:psiL_w_psi_zero_+}
\left\{\begin{array}{rcl} \partial_X\psi^{+}_{\ell}-\Delta_w^{2}\psi^{+}_{\ell}&=&0,\quad\textrm{in}\;X>M\nonumber\\
\psi^{+}_{\ell}\big|_{X=M}&=&\psi^{-}_{\ell}\big|_{X=M},\\
\partial_X\psi^{+}_{\ell}\big|_{X=M}&=&\partial_X\psi^{-}_{\ell}\big|_{X=M}.\end{array}\right.
\end{equation}
If $L^0(\rho_2,\rho_3)=(0,0)$, then $\psi_{\ell}:=\mathbbm{1}_{X\leq M}\psi^{-}_{\ell}+\mathbbm{1}_{X>M}\psi^{+}_{\ell}$ is a solution of the linear system in the whole western boundary layer domain 
\begin{equation}\label{pb:psiL_linear}
\begin{array}{rcl}
\partial_X\psi_{\ell}-\Delta_w^{2}\psi_{\ell}&=&0,\quad\textrm{in}\quad \omega_w\\
\psi_{\ell}\big|_{X=-\gamma(Y)}=\partial _{n}\psi_{\ell}\big|_{X=-\gamma(Y)}&=&0.
\end{array}
\end{equation}

The existence and uniqueness of the solution $\psi_{\ell}\equiv 0$ of the linear elliptic problem \eqref{pb:psiL_linear} is guaranteed by Theorem \ref{theorem:existence_linear_all},  and therefore $\rho_2=0$ and $\rho_3=0$. Consequently,  $\mathrm{ker}\;d\mathcal{F}(0,0,0)=\left\{(0,0)\right\}$, and $\mathrm{ker}\;d\mathcal{F}(0,0,0)$ is one-to-one. 
Solving the equation
\begin{equation*}
L^0(\rho_2,\rho_3)=(\bar{\rho}_{2},\bar{\rho}_{3}),
\end{equation*}
for a given $(\bar{\rho}_{2},\bar{\rho}_{3})\in H^{m-1/2}(\mathbb{R})\times H^{m-3/2}(\mathbb{R})$ is equivalent to finding the solution of the problem
\begin{equation}\label{ps_problem}
\begin{split}
	\partial_X\psi^{-}_{\ell}-\Delta_w^{2}\psi^{-}_{\ell}&=0,\quad\textrm{in}\quad X\leq M\\
	\left((1+\alpha^2)\Delta_w\psi^{-}_{\ell},-\left((1+\alpha^2)\partial_{X}-2\alpha\partial_{Y}\right)\Delta_w\psi^{-}_{\ell}+\displaystyle\frac{\psi^{-}_{\ell}}{2}\right)\Bigg|_{X=M}&=-\left(\bar{\rho}_2,\bar{\rho}_3\right)+PS_w(\psi^{-}_{\ell}\big|_{X=M}, \partial_X\psi^{-}_{\ell}\big|_{X=M}),\\
	\psi^{-}_{\ell}\big|_{X=-\gamma(Y)}=\partial _{n}\psi^{-}_{\ell}\big|_{X=-\gamma(Y)}&=0,
\end{split}
\end{equation}
where $PS_w$ denotes the Poincar\'e-Steklov operator $$PS_w(\psi_0,\psi_1) :  H^{m+3/2}_{\mathrm{uloc}}(\mathbb{R})\times\in H^{m+1/2}_{\mathrm{uloc}} (\mathbb{R})\rightarrow H^{m-1/2}_{\mathrm{uloc}}(\mathbb{R})\times\in H^{m-3/2}_{\mathrm{uloc}}(\mathbb{R}).$$  The existence of a unique solution $\psi^{-}_{\ell} \in H^{2}_{\mathrm{uloc}}(\omega^{b}_w)$ of problem \eqref{ps_problem} follows from the ideas of the proof of Proposition \ref{prop:result_rough_channel_lw} in Section \ref{western_l}.

The only point remaining concerns proving that  $\psi_{\ell}^{-}$ is a  $H^{m+2}_{\mathrm{uloc}}$ function in $[M',M]\times\mathbb{R}$ for all $\sup(-\gamma_w)<M'<M$. Therefore, we notice that
in the domain $(M',M)\times\mathbb{R}$, the derivatives up to order $k$ of $\psi_{\ell}^{-}$  satisfy a linear system similar to the one above. It follows that $\partial_{Y}^l\psi_{\ell}^{-}\in H^{2}_{\mathrm{uloc}}([M',M]\times\mathbb{R}$ for all $l\leq k$. In particular, $\partial_{Y}^l\psi_{\ell}^{-}\big|_{X=M} \in H^{3/2}_{\mathrm{uloc}}([M',M]\times\mathbb{R})$ and $\partial_X\partial_{Y}^l\psi_{\ell}^{-}\big|_{X=M} \in H^{1/2}_{\mathrm{uloc}}([M',M]\times\mathbb{R})$, hence, $\psi_{\ell}^{-} \in H^{m+3/2}_{\mathrm{uloc}}(\mathbb{R})$ and $\partial_n\psi_{\ell}^{-} \in H^{m+1/2}_{\mathrm{uloc}}(\mathbb{R})$. Consequently, $$(\rho_2,\rho_3)=\begin{pmatrix}(1+\alpha^2)\Delta_w\psi^{-}_{\ell},-\left((1+\alpha^2)\partial_X-2\alpha\partial_{Y}\right)\Delta_w\psi^{-}_{\ell}+\frac{\psi^{-}_{\ell}}{2}
\end{pmatrix}.$$ We can finally assert that $d\mathcal{F}(0,0,0)$ is an isomorphism of $H^{m-1/2}(\mathbb{R})\times H^{m-3/2}(\mathbb{R})$.
Using the implicit function theorem, we deduce that for all $\phi\in W^{2,\infty}$ in a neighborhood of zero, there exists $(\rho_2,\rho_3)\in H^{m-1/2}_{\mathrm{uloc}}(\mathbb{R})\times H^{m-3/2}_{\mathrm{uloc}}(\mathbb{R})$ such that $\mathcal{F}(\phi,\rho_{2},\rho_{3})=0$. Let $\Psi_w:=\mathbbm{1}_{X\leq M}\Psi_w^{-}+\mathbbm{1}_{X> M}\Psi_w^{-}$, where $\Psi_w^{-},\Psi_w^{+}$ are the solutions of \eqref{eq_left1} and \eqref{pb:halfspace_nl_w} associated to $(\phi,\rho_2, \rho_3)$. By definition, the jump of $\Psi_w$ across the transparent boundary $\{X=M\}$ is zero, and since $\mathcal{F}(\phi,\rho_2, \rho_3)=(0,0)$,
\begin{equation*}
\mathcal{A}_i\left[\Psi^{-}_w\big|_{X=M},\partial_X\Psi^{-}_w\big|_{X=M}\right]=\rho_i=\mathcal{A}_i\left[\Psi^{+}_w\big|_{X=M},\partial_X\Psi^{+}_w\big|_{X=M}\right],\quad i=2,3.
\end{equation*}
Since $\partial^k_X\Psi_w^{+}\big|_{X=M}=\partial^k_X\Psi_w^{-}\big|_{X=M}$, $k=0,1$ we deduce that
\begin{equation*}
\begin{split}
&\begin{pmatrix}(1+\alpha^2)\Delta_w\Psi^{-},-\left((1+\alpha^2)\partial_X-2\alpha\partial_{Y}\right)\Delta_w\Psi^{-}+\dfrac{\Psi^{-}}{2}
\end{pmatrix}\bigg|_{X=M}\\
&\hspace{4cm}=\\
&\begin{pmatrix}(1+\alpha^2)\Delta_w\Psi^{+},-\left((1+\alpha^2)\partial_X-2\alpha\partial_{Y}\right)\Delta_w\Psi^{+}+\dfrac{\Psi^{+}}{2}
\end{pmatrix}\bigg|_{X=M}.
\end{split}
\end{equation*}
Accordingly, these operators are continuous across $\{X=M\}\times \mathbb{R}$, and therefore $\Psi_w$ is a solution of the western boundary layer system in the whole domain $\omega_w$. This concludes the proof of Theorem \ref{theorem:existence}.
     \section{Linearized problem}\label{s:linearized}
This section focuses on the well-posedness analysis of the linearized problems driving the higher-order western profiles of the approximate solution and the correctors needed to deal with the influence of the east boundary layer on the western side of $\Omega^\varepsilon$. 

We are interested in the system
\begin{equation}\label{linearised_west}
\begin{split}
    \partial_X\Psi_w+Q_w(\Psi_w,\Psi^0_w)+Q_w(\Psi^0_w,\Psi_w)-\Delta^2_w\Psi_w&=F,\quad\textrm{in}\:\:\omega_w^{+}\cup\omega_w^{-}\\
    [\partial^k_X\Psi_w]\big|_{X=0}&=g^k,\quad k=0,\ldots,3\\
    \Psi_w\big|_{X=-\gamma_w(Y)}=0,\quad\dfrac{\partial\Psi_w}{\partial n_w}\big|_{X=-\gamma_w(Y)}&=0,
\end{split}
\end{equation}
where $g_k\in L^{\infty}(\mathbb{R})$. Assume $\Psi^0_w\in H^{2}_{\mathrm{uloc}}(\omega_w^{+}\cup\omega_w^{-})$ and $F\in H^{-2}_{\mathrm{uloc}}(\omega_w)$ are exponentially decaying functions.  The existence and uniqueness of a solution of \eqref{linearised_west} depend additionally on a smallness hypothesis made on the first profile of the western boundary layer $\Psi^0_w$. 

The main result concerning  problem \eqref{linearised_west} is summarized in the following theorem:

\begin{theorem}\label{t:linearized_hs}
Let $\omega_w$ defined as in \eqref{west_dom} and $\gamma_w$ be a bounded Lipschitz function. Suppose $g_k\in \ch{L^{\infty}}(\mathbb{R})$, $k=0,\ldots, 3$, and $\delta>0$  and $\Psi_0\in H^2_{\mathrm{uloc}}(\omega_w^+\cup\omega_{w}^-)$  a function  satisfying
\begin{equation*}
\|e^{\delta X}\Psi^0_w\|_{H^2_{\mathrm{uloc}}(\omega_w^+\cup\omega_{w}^-)}< \delta_0.
\end{equation*}
Additionally, suppose that $e^{\delta}F\in H^{m-2}_{\mathrm{uloc}}(\omega_w)$. Then, for $\delta_0$ small enough,
\eqref{linearised_west} has a unique solution $\Psi$ in $H^2_\mathrm{uloc}(\omega_w^{+}\cup\omega_w^{-})$. Moreover, the following estimate holds
	\begin{equation}\label{regularity_est}
	\|e^{\delta X}\Psi\|_{H^{2}_{\mathrm{uloc}}(\omega_w^{+}\cup\omega_w^{-})} \leq C\left( \sum_{k=0}^3\|g_k\|_{L^{\infty}(\mathbb{R})}+\|e^{\delta X}F\|_{H^{m-2}_{\mathrm{uloc}}(\omega_w)}\right),
	\end{equation}
	for a universal constant $C$. Here, $\delta$ is the one in Theorem \ref{Theorem2_DGV2017}.
\end{theorem}
Note that the smallness condition on $\Psi^0_w$ is ensured by choosing $\phi$ small enough in \eqref{e:nonlinear_west}. Theorem \ref{t:linearized_hs} results from following the same ideas of the nonlinear case. As a result, we are going to discuss the relevant steps of the construction listed in Section \ref{nl_linearized_method} but only emphasize the differences concerning the previous section. The details are left to the reader.
\subsection{The problem on the half space}
We study the following problem on the half-space
\begin{equation}\label{linearised_west_hs}
\begin{split}
  \partial_X\Psi_w+Q_w(\Psi,\Psi^0_w)+Q_w(\Psi^0_w,\Psi_w)-\Delta^2_w\Psi_w&=F,\quad\textrm{in}\:\:\mathbb{R}^2_+\\
    \Psi_w\big|_{X=0}=\psi_0,\quad\partial_n\Psi_w\big|_{X=0}&=\psi_1,
\end{split}
\end{equation}
where $\psi_0\in H^{m-3/2}_{\mathrm{uloc}}(\mathbb{R})$ and $\psi_1\in H^{m-1/2}_{\mathrm{uloc}}(\mathbb{R})$. 

The result in this section equivalent to Proposition \ref{prop:fix} is as follows:
\begin{proposition}\label{prop:fix_linearized}
Let $m \in\mathbb{N}$, $m \gg 1$ and $\mathcal{H}^m$ the functional space defined in \eqref{def_bs}. There is small constant $\delta_0>0$ such that for all $\psi_{j} \in H^{m-j+3/2}_{\mathrm{uloc}}(\mathbb{R})$, $j=0,1$ and $\Psi^{0}_w\in H^{m+2}_{\mathrm{uloc}}$ such that for $\delta>0$, $\|e^{\delta X}\Psi^{0}_w\|_{H^{m+2}_{\mathrm{uloc}}}\leq \delta_0$ and 
\begin{equation}\label{eq:41_linearized}
\|\psi_{0} \|_{H^{m+3/2}_{\mathrm{uloc}}(\mathbb{R})}+\|\psi_{1} \|_{H^{m+1/2}_{\mathrm{uloc}}(\mathbb{R})}+\|e^{\delta X}F\|_{H^{m-2}_\mathrm{uloc}}<\delta_0.
\end{equation}
Then, the problem

\begin{equation}
\left\{\begin{array}{rclr}\label{nonlin_hs_lin}
\partial_X\Psi_w+Q_w(\Psi^0_w,\Psi_w)+Q_w(\Psi,\Psi^0_w)-\Delta_w^{2}\Psi_{w}&=&F,&\textrm{in}\quad\mathbb{R}^2_+\\
\Psi_{w}\big|_{X=0}=\psi_0,&&\partial_X\Psi_{w}\big|_{X=0}=\psi_1.&
\end{array}\right.
\end{equation}
has a unique solution in $\mathcal{H}^{m+2}$.
\end{proposition}
The strategy of proof is the same as to the one of Proposition \ref{fixed_point}. Indeed, the mapping $T_{(\psi_{0},\psi_{1})}: H^{m+2}_{\mathrm{uloc}}\rightarrow H^{m+2}_{\mathrm{uloc}}$
such that $T_{(\psi_{0},\psi_{1})}(\Psi_w)=\tilde{\Psi}$ is the solution of \eqref{p:linear_nonhomgeneous0} when $\widetilde{F}=F -Q_w(\Psi^0_w, \Psi_w)-Q_w(\Psi_w, \Psi^0_w)$ satisfies the estimate for $\bar{\delta}=2\delta$
\begin{equation*}
    \begin{split}
        \|T_{(\psi_{0},\psi_{1})}(\Psi_w)\|_{\mathcal{H}^{m+2}}&<C_1\left(\|\psi_{0} \|_{H^{m+3/2}_{\mathrm{uloc}}(\mathbb{R})}+\|\psi_{1} \|_{H^{m+2}_{\mathrm{uloc}}(\mathbb{R})}+\|e^{\delta X}F\|_{H^{m-2}_{\mathrm{uloc}}}\right.\\
&\hspace{1in}\left.+\|\Psi^0_w\|_{\mathcal{H}^{m+2}}\|\tilde{\Psi}\|_{\mathcal{H}^{m+2}}\right),
    \end{split}
\end{equation*}
for $C_1>0$. 
When considering the functions $\Psi^i$, $i=1,2$, such that $T_{(\psi_{0},\psi_{1})}(\Psi^i_w)=\tilde{\Psi}^i_w$ are solutions of \eqref{linearised_west} for $\widetilde{F}^i=F-Q_w(\Psi^i_w,\Psi^0_w)-Q_w(\Psi^0_w,\Psi^i_w)$, we can show that $T_{(\psi_{0},\psi_{1})}(\Psi^{1})-T_{(\psi_{0},\psi_{1})}(\Psi_{2})$ is a strict contraction in a ball of radius $R_0>0$. Indeed, it is always possible to choose $R_0\in ]\delta(\eta),\frac{1}{2C_1}[$ when $0<\delta_0\leq\delta(\eta)=\frac{\eta -1}{2 C_1 \eta -C_1}$, $\eta>1$. The existence result is a consequence of applying the fixed point theorem in $\mathcal{H}^{m+2}$. Since the ideas of proof are the same of Proposition \ref{prop:fix}, they are not repeated here.

\subsection{The linearized problem on the rough domain}

This section is devoted to the well-posedness of the problem in the rough channel $\omega^b=\omega_w\setminus\left\{X<M\right\}$
\begin{equation}\label{lin_rough}
\begin{split}
\partial_{X}\Psi^-_w+Q_w(\Psi^-_w,\Psi^0_w)+Q_w(\Psi^0_w,\Psi^-_w)-\Delta_{w}^2\Psi^-_w&=F,\quad\mathrm{in}\quad\omega_w,\quad\text{in}\quad\omega^b\setminus\sigma_w\\
\left[\partial_{X}^k\Psi^-_w\right]\big|_{\sigma_w}&=g_k,\quad k=0,\ldots,3,\\
\mathcal{A}_2[\Psi^-_w|_{\sigma_w^M}, \partial_{X}\Psi^-_w|_{\sigma_w^M}]&=\rho_2,\\
\tilde{\mathcal{A}}_3[\Psi^-_w|_{\sigma_w^M}, \partial_{X}\Psi^-_w|_{\sigma_w^M}]&=\rho_3,\\
\Psi^-_w|_{X=-\gamma_w(Y)}=0,&\quad\dfrac{\partial\Psi^-_w}{\partial n_w}\Big|_{X=-\gamma_w(Y)}=0,
\end{split}
\end{equation}
where $g_k\in W^{m+2-k,\infty}(\mathbb{R})$ and $\rho_2$ and $\rho_3$ belong to $H^{m-1/2}_{\mathrm{uloc}}(\mathbb{R})$ and $H^{m-3/2}_{\mathrm{uloc}}(\mathbb{R})$, respectively. The source term $F$ is a function of $H^{m-2}_{\mathrm{uloc}}(\omega^b_w)$. The transparent operators are defined by
\begin{equation}
\begin{split}
\mathcal{A}_2[\Psi, \partial_{X}\Psi]&=(1+\alpha^2)\Delta_{w}\Psi,\\
\tilde{\mathcal{A}}_3[\Psi, \partial_{X}\Psi]&=
-((1+\alpha^2)\partial_{X}-2\alpha\partial_{Y})\Delta_{w}\Psi\\
&\quad+(\nabla_{w}^\perp\Psi\cdot\nabla_w)((1+\alpha^2)\partial_X-\alpha\partial_Y)\Psi^0_w+(\nabla_{w}^\perp\Psi^0_w\cdot\nabla_w)((1+\alpha^2)\partial_X-\alpha\partial_Y)\Psi+\dfrac{\Psi}{2}.
\end{split}
\end{equation}
We claim the following result:
\begin{proposition}\label{Lemma_14_lin}
	Let $m\gg 1$ be arbitrary and $g_k\in W^{m+2-k,\infty}(\mathbb{R})$, $k=0,\ldots,3$. There exists $\delta> 0$ such that for $\delta>0$, $F\in H^{m-2}_{\mathrm{uloc}}(\omega^b_w)$, $\rho_2\in H^{m-1/2}_{\mathrm{uloc}}(\mathbb{R})$ and $\rho_3\in H^{m-3/2}_{\mathrm{uloc}}(\mathbb{R})$ satisfy
	\begin{equation}\label{smallness_lin}
		\|F\|_{H^{m-2}_{\mathrm{uloc}}(\mathbb{R})}+\|\Psi^0_w\|_{H^{m+2}_{\mathrm{uloc}}(\mathbb{R})}+\|\rho_2\|_{ H^{m-1/2}_{\mathrm{uloc}}(\mathbb{R})}+\|\rho_3\|_{ H^{m-3/2}_{\mathrm{uloc}}(\mathbb{R})}<\delta_0.
	\end{equation}
	Then, system (\ref{lin_rough}) has a unique solution $\Psi_w\in H^{2}_{\mathrm{uloc}}(\omega_w^b)$.
	
	Moreover, $\Psi_w\in H^{m+2}_{\mathrm{uloc}}((M',M)\times \mathbb{R})$, for all $M'\in ]\sup (-\gamma_w),M[$ and
	\begin{equation}\label{trans_reg_estimate}
	\|\Psi_w\|_{H^{m+2}_{\mathrm{uloc}}((M',M)\times \mathbb{R})}\leq C_{M'}\left(\sum\limits_{k=0}^3\|g_k\|_{W^{m+2-k,\infty}}+\|\rho_2\|_{ H^{m-1/2}_{\mathrm{uloc}}(\mathbb{R})}+\|F\|_{H^{m-2}_{\mathrm{uloc}}((M',M)\times\mathbb{R}}+\|\rho_3\|_{ H^{m-3/2}_{\mathrm{uloc}}(\mathbb{R})}\right)
	\end{equation}
\end{proposition}

\begin{proof}
Note that Proposition \ref{Lemma_14_lin} is the equivalent for the linearized case of Proposition \ref{Lemma_14}. Consequently, to look up the solution of problem \eqref{lin_rough}, similar arguments to the nonlinear case apply, see Section \ref{s:rough}. First, we discuss the existence and uniqueness of the solution in $H^2_{\mathrm{uloc}}(\omega^b)$ and later, its regularity near the artificial boundary.

We work with the truncated energies 
\begin{equation}
E^n_k=\int_{\omega_{n}}|\chi_k\Delta_{w}\tilde{\Psi}|^2,
\end{equation}
where the truncation function $\chi_k\in C^\infty_0(\mathbb{R})$ is such that $\chi_k \equiv 1 \in [-k, k]$, $\mathrm{Supp}\chi_k \subset [-k - 1, k + 1]$, and the derivatives $\chi_k^{(j)}$, $j=1,\ldots,4$ are bounded uniformly in $k$. Moreover, $\tilde{\Psi}=\Psi-\Psi^L$ is the solution of the system
\begin{eqnarray}\label{pb:left_lift_lin}
Q_w(\tilde{\Psi},\Psi^{0}_w)+Q_w(\Psi^0_w,\tilde{\Psi})+\partial_{X}\tilde{\Psi}-\Delta^{2}_w\tilde{\Psi}&=&F^L_0\quad\textrm{in}\;\omega^b,\nonumber\\
\mathcal{A}_2[\tilde{\Psi}|_{\sigma^M_w}, \partial_{X}\tilde{\Psi}|_{\sigma^M_w}]=\rho_2,&&
\mathcal{A}_2[\tilde{\Psi}|_{\sigma^M_w}, \partial_{X}\tilde{\Psi}|_{\sigma^M_w}]=\rho_3,\nonumber\\
\tilde{\Psi}\big|_{X=-\gamma(Y)}=0,&&\dfrac{\partial\Psi}{\partial n_w}\tilde{\Psi}\big|_{X=-\gamma(Y)}=0.
\end{eqnarray} 
Function $\Psi^L$ depends on the $g_k$ and is defined as in \eqref{lift_lin} and for $F_0^L=F-Q_w(\Psi^L,\Psi^0_w)-Q_w(\Psi^0_w,\Psi^L)$, we have the estimate
\begin{equation}
\|F_0^L\|_{H^{m-2}_{\mathrm{uloc}}(\omega^b)}\leq C\left(\|F\|_{H^{m-2}_{\mathrm{uloc}}(\omega^b)}+\sum_{k=0}^3\|g^k\|_{W^{2,\infty}(\mathbb{R})}\|\Psi^0_w\|_{H^{m+2}_{\mathrm{uloc}}(\mathbb{R})}\right).
\end{equation}

We compute the following inequality  for the sequence $(E^n_k)_{l\leq n, n\in\mathbb{R}}$ 
	\begin{equation}\label{inequality_linearized}
\ch{E_k^{n}\leq C_1\left((E_{k+1}^{n}-E_k^{n})+\left(\|\rho_2\|_{H^{m-1/2}_{\mathrm{uloc}}}^{2}+\|\rho_3\|_{H^{m-3/2}_{\mathrm{uloc}}}^{2}\right)(k+1)\right).}
\end{equation}
This a priori estimate allows one to obtain a uniform bound on $E^k_n$ which is used in turn to deduce a $H^{2}_{\mathrm{\mathrm{uloc}}}$ bound on $\tilde{\Psi}$ using backward induction on $n-k$ and later a compactness argument. This corresponds to step (NL2) in Section \ref{nl_linearized_method}. For more details we refer the reader to Section \ref{s:rough}.

From now on, we drop the $n$'s to lighten the notation and only address the particularities of the features unique to the linearized case.

First, we write the weak formulation of (\ref{pb:left_lift_lin}). 

\begin{definition}\label{p:weak_formulation_lin}
	Let $\mathcal{V}$ and $\mathcal{D}^2_0$ be spaces of functions in Definition \ref{p:weak_formulation} and $b(\Psi,\tilde{\Psi},\varphi)=-\int_{\Omega}(\nabla_{w}^{\perp}\Psi\cdot\nabla_{w})\nabla_{w}^{\perp}\tilde{\Psi}\cdot\nabla_w^{\perp}\varphi$ for $(\Psi,\tilde{\Psi},\varphi)\in \mathcal{D}^2_0\times\mathcal{D}^2_0\times\mathcal{V}$. The solution $\tilde{\Psi}\in H^2_{\mathrm{uloc}}(\omega^b)$ of \eqref{pb:left_lift_lin} satisfies for all $\varphi\in \mathcal{V}$
	\begin{eqnarray}\label{weak_formulation_lin}
		&&\int_{\omega^b_w}\partial_{X}\Psi\varphi+b(\tilde{\Psi},\tilde{\Psi}^0,\varphi)+b(\tilde{\Psi}^0,\tilde{\Psi},\varphi)-\int_{\omega^b_w}\Delta_w\tilde{\Psi}\Delta_w\varphi\nonumber\\
	&&\hspace*{1in}=\int_{\omega^b_w}F^L_0\varphi -\int_\mathbb{R}\left(\rho_3-\frac{\tilde{\Psi}}{2}\Big|_{X=M^-}\right)\varphi\big|_{X=M^-}\;dY-\int_\mathbb{R}\rho_2\partial_{X}\varphi\big|_{X=M^-}\;dY.\nonumber
	\end{eqnarray}
\end{definition}

\begin{itemize}
	\item\textit{Energy estimates.} The proof of (\ref{inequality_linearized}) is as follows. Plugging $\varphi=\chi_k\tilde{\Psi}$ into the trilinear terms in \eqref{weak_formulation_lin}, and proceeding similarly as before it is possible to show that
	\begin{eqnarray}\label{linearized_trilinear}
	\left|b(\tilde{\Psi},\Psi^0_w,\chi_k\tilde{\Psi})+b(\Psi^0_w,\tilde{\Psi},\chi_k\tilde{\Psi})\right|
	&\leq& C\|\Psi^0_w\|_{H^2_{\mathrm{uloc}}}E_{k+1}+C_1\|\Psi^0_w\|_{H^2_{\mathrm{uloc}}}(E_{k+1}-E_k),
	\end{eqnarray}
	where $C$ and $C_1$ depend on the domain. To obtain estimate \ref{linearized_trilinear}, we  first write the terms as in \eqref{first_nlt} and \eqref{second_nlt}. Then, combining the Sobolev inequality \eqref{sobolev_embedding} with the definition of norm in Kato spaces leads to the desired result. The terms associated to the biharmonic operator are bounded by $C(E_{k+1}-E_k)$ as seen in \eqref{commut_bound}. We are left with
	\begin{eqnarray}\label{lin_2}
		\left|\int_{\omega^b_w}\chi_k F^L_0\tilde{\Psi}\right|&\leq C\left(\|F\|_{H^{m-2}_{\mathrm{uloc}}(\omega^b)}+\sum_{k=0}^3\|g^k\|_{W^{2,\infty}(\mathbb{R})}\|\Psi^0_w\|_{H^{m-2}_{\mathrm{uloc}}(\mathbb{R})}\right)E_{k+1}^{1/2}.
	\end{eqnarray}
	At the artificial boundary, we get
	\begin{equation}\label{lin_1}
	\begin{split}	&-\int_{\mathbb{R}}\chi_k\left(\rho_3\tilde{\Psi}\big|_{X=M^-}+\rho_2\partial_X\tilde{\Psi}\big|_{X=M^-}\right)\\
	&\qquad\leq
	  C\left(\|\rho_2\|_{L^{2}_{\mathrm{uloc}}}+\|\rho_3\|_{L^{2}_{\mathrm{uloc}}}\right)E_{k+1}+C\left(\|\rho_2\|_{L^{2}_{\mathrm{uloc}}}+\|\rho_3\|_{L^{2}_{\mathrm{uloc}}}\right)(k+1),
	\end{split}
	\end{equation}
	where $C$ depends only on $M$, $\alpha$ and on $\|\gamma\|_{W^{2,\infty}}$. The computation of this bound relies on the trace theorem and Young's inequality. 
	
	From \eqref{linearized_trilinear}, \eqref{lin_2}, \eqref{lin_1} and the bound on the linear term, we obtain
	\begin{eqnarray*}
		E_k&\leq&C'_1(E_{k+1}-E_k)+C_2'\left(\|F\|_{H^{m-2}_{\mathrm{uloc}}(\omega^b)}+\sum_{k=0}^3\|g^k\|_{W^{2,\infty}(\mathbb{R})}\|\Psi^0_w\|_{H^{m-2}_{\mathrm{uloc}}(\mathbb{R})}\right)^2E_{k+1}\\
		&&+C_4'\left(\|\rho_2\|_{H^{m-1/2}_{\mathrm{uloc}}}+\|\rho_3\|_{H^{m-3/2}_{\mathrm{uloc}}}\right) E_{k+1}+C_4'\left(\|\rho_2\|_{H^{m-1/2}_{\mathrm{uloc}}}+\|\rho_3\|_{H^{m-3/2}_{\mathrm{uloc}}}\right)(k+1).
	\end{eqnarray*}
	Choosing $\delta_0$ small enough, the previous expressions becomes for $C_1>0$
	\begin{equation}\label{DP_3-3_linearized}
	E_k\leq C_1\left(E_{k+1}-E_k+C_2(\rho_2,\rho_3)(k+1)\right),
	\end{equation}
	where $C_2(\rho_2,\rho_3)$ refers to the term $\|\rho_2\|_{H^{m-1/2}_{\mathrm{uloc}}}+\|\rho_3\|_{H^{m-3/2}_{\mathrm{uloc}}}$.
	\item \textit{Induction.} Performing backwards induction on \eqref{DP_3-3_linearized} is easier than the nonlinear case. Indeed, considering that there exists a constant $C_3$ independent of $n$ and $k$, such that 
	\begin{equation*}
E_k>C_3C_2(\rho_2,\rho_3)(k+1),
	\end{equation*}
	implies that $E_{k+1}-E_k<C_3C_2(\rho_2,\rho_3)$. Furthermore, substitution on (\ref{DP_3-3_linearized}) gives
	\begin{eqnarray}\label{DP:3-6_linearized}
	C_3C_2(\rho_2,\rho_3)(k+1)<&E_k&\leq C_1C_2(\rho_2,\rho_3)C_3+C_1C_2(\rho_2,\rho_3) (k+1).s
	\end{eqnarray}
	Taking $C_3\geq 2C_1$ and plugging it in (\ref{DP:3-6_linearized}) provides a contradiction for $k>k_0$, where $k_0=\left \lfloor{C_3}\right \rfloor$. Therefore, (\ref{DP_3-3_linearized}) is true at the rank $k>k_0$ and also for $k\leq k_0$, since $E_k$ is an increasing functional with respect to the value of $k$. The derivation of the energy estimates is invariant by horizontal translation, and all constants depend only on norms of $\rho_i$, $i=2,3$ and $\gamma_w$. Following the same reasoning in Step (L5) of Section \ref{l_method}, it is possible to show that $\tilde{\Psi}$ is uniformly bounded in $H^2_{\mathrm{uloc}}(\omega^b_w)$.
	\item \textit{Uniqueness.}  Let $\Psi_i$, $i=1,2$, are solutions of \eqref{lin_rough}. To establish uniqueness, we need to show that the solution $\bar{\Psi}=\Psi_1-\Psi_2$ of the system
	\begin{eqnarray}\label{pb:uniqueness_linearized}
	Q_w(\bar{\Psi},\Psi^0_w)+Q_w(\Psi^0_w,\bar{\Psi})+\partial_{X}\bar{\Psi}-\Delta_w^{2}\bar{\Psi}&=&0\quad\textrm{in}\;\omega^b,\nonumber\\
	\bar{\Psi}\big|_{X=-\gamma(Y)}=\partial _{n}\bar{\Psi}\big|_{X=-\gamma(Y)}&=&0,
	\end{eqnarray}
	 is identically zero under a smallness assumption on $\Psi^0_w$. 
	 
	 Repeating the same reasoning of ``existence part'' provides the induction relation on the truncated energies $\int_{\omega_n}\chi_k|\Delta_w\bar{\Psi}|^2$
	\begin{equation*}
		E_k\leq C_{\delta^0}\left(E_{k+1}+(E_{k+1}-E_k)\right).
	\end{equation*}
	Here, the constant depends not only on the domain, but
	also on $\|\Psi_w^0\|_{H^{m}_{\mathrm{uloc}}}$. The term $E_{k+1}-E_k$ is uniformly bounded in $k$, hence, $E_k\leq C$ uniformly in $k$. As a consequence, repeating the method but this time without the truncation function leads to $\bar{\Psi}\equiv 0$.
\end{itemize}

	The last step in the proof of Proposition \ref{Lemma_14_lin} corresponds to the higher regularity estimates near $X=M$. Interior and boundary regularity results analogous to that of Lemma \ref{Lemma_14} can be easily obtained by following the same ideas. The major difficulties we encountered when performing the regularity analysis near the boundary for the first western profile came from the nonlinear terms, which are not present in the linearized case. Moreover, information on the behavior and $H^{m+2}$-regularity of $\Psi^0_w$ in the vicinity of the transparent boundary is available. These factors greatly simplify the computations, and, therefore, the detailed verification is left to the reader.

\end{proof}
\subsection{Joining \texorpdfstring{$\omega^-_w$}{omegaplus} and \texorpdfstring{$\omega^+_w$}{omegaminus}}
In this section we are concerned with finding a solution on the whole domain $\omega_w$. Let $\Psi^-_w$ be the unique solution of \eqref{lin_rough} satisfying the regularity estimate \eqref{trans_reg_estimate}. As a consequence of standard trace properties, we have that $\Psi^{-}\big|_{X=M}\in H^{m+3/2}_\mathrm{uloc}(\mathbb{R})$ and $\partial_X\Psi^{-}\big|_{X=M}\in H^{m+1/2}_\mathrm{uloc}(\mathbb{R})$. Taking $\psi_0=\Psi^{-}\big|_{X=M}$ and $\psi_1=\partial\Psi^{-}\big|_{X=M}$ in \eqref{nonlin_hs_lin} guarantees the existence a unique solution $\Psi^+$ for the problem in the half-space decaying exponentially far from the boundary having imposed smallness conditions on $\Psi^0_w$ and $\rho_i$, $i=2,3$. Additionally, we have that $\mathcal{A}_2[\Psi^+|_{X=M^+}, \partial_{X}\Psi^+|_{X=M^+}]\in H^{m-1/2}$, $\tilde{\mathcal{A}}_3[\Psi^+|_{X=M^+}, \partial_{X}\Psi^+|_{X=M^+}]\in H^{m-3/2}$. We define the mapping $\mathcal{F}: H^{m+2}_{\mathrm{uloc}}(\omega_w)\times(W^{2,\infty}(\mathbb{R}))^4\times H^{m-1/2}_{\mathrm{uloc}}(\mathbb{R})\times H^{m-3/2}_{\mathrm{uloc}}(\mathbb{R})\rightarrow H^{m-1/2}_{\mathrm{uloc}}(\mathbb{R})\times H^{m-3/2}_{\mathrm{uloc}}(\mathbb{R})$ by

\begin{equation*}
\mathcal{F}(\Psi^0_w,g_0,\ldots,g_3,\rho_2,\rho_3)=\left(\mathcal{A}_2[\Psi^+|_{X=M^+}, \partial_{X}\Psi^+|_{X=M^+}]-\rho_2,\tilde{\mathcal{A}}_3[\Psi^+|_{X=M^+}, \partial_{X}\Psi^+|_{X=M^+}]-\rho_3\right).
\end{equation*}
Once again, if $\Psi$ is a $C^3$ function at $X=0$ and $X=M$, and $F=0$, we have that $\mathcal{F}=0$. 
Showing that $\mathcal{F}$ is a $C^1$ mapping starts by considering two points in the vicinity of zero using the functional norm defined in the domain of $\mathcal{F}$. Suppose these points have an analogous form as the ones  in Section \ref{s:join}. Let $\Psi_0^\pm$ and $\Psi^\pm$ be the solutions associated to the points $(\Psi^0_0,g_0^0,\ldots,g^0_3,\rho_2^0,\rho_3^0)$ and $(\Psi^0_w+\Psi^0_0,g_0+g_0^0,\ldots,g_3+g_3^0,\rho_2+\rho_2^0,\rho_3+\rho_3^0)$, respectively. Function $\psi^-= \Psi^--\Psi_0^-$ satisfies
\begin{eqnarray}\label{problem_c1_lin}
\partial_X\psi^{-}-\Delta_w^{2}\psi^{-}+Q_w(\Psi^0_w+\Psi^0_0,\psi^{-})+Q_w(\psi^-,\Psi^0_w+\Psi^0_0)&&\\
+Q_w(\Psi_w^0,\Psi_0^-)+Q_w(\Psi_0^-,\Psi^0_w)&=&0,\quad\textrm{in}\quad\omega_w^b\setminus\{\sigma_w\}\nonumber\\
\left[\partial_X^{k}\psi^{-}\right]\Big|_{\sigma_w}&=&g_k,\quad k=0,\ldots,3\nonumber\\
(1+\alpha^2)\Delta_w\psi^{-}\Big|_{\sigma^M_w}&=&\rho_2,\nonumber\\
\left(-\left((1+\alpha^2)\partial_{X}-2\alpha\partial_{Y}\right)\Delta_w\psi^-+\displaystyle\frac{\psi^-}{2}\right.\hspace{7mm}&&\\
+(\nabla_w^{\perp}\Psi^0_w+\Psi^0_0\cdot\nabla_w)\left((1+\alpha^2)\partial_{X}-\alpha\partial_{Y}\right)\psi^{-}\hspace{7mm}&&\\
+(\nabla_w^{\perp}\Psi^0_w\cdot\nabla_w)\left((1+\alpha^2)\partial_{X}-\alpha\partial_{Y}\right)\Psi_0^-\hspace{7mm}&&\\
+(\nabla_w^{\perp}\psi^-\cdot\nabla_w)\left((1+\alpha^2)\partial_{X}-\alpha\partial_{Y}\right)(\Psi^0_w+\Psi^0_0)\hspace{7mm}&&\nonumber\\
\left.+(\nabla_w^{\perp}\Psi^-_0\cdot\nabla_w)\left((1+\alpha^2)\partial_{X}-\alpha\partial_{Y}\right)\Psi^0_w\right)\bigg|_{X=M}&=&\rho_3,\nonumber\\
\psi^{-}\big|_{X=-\gamma(Y)}=\partial _{n}\psi^{-}\big|_{X=-\gamma(Y)}&=&0,\nonumber
\end{eqnarray}
and for small enough norms $\|e^{\delta X}\Psi^0_w\|_{H^{2}_{\mathrm{uloc}}},\|\Psi^0_0\|_{H^{2}_{\mathrm{uloc}}}$, and $\|\Psi_0^-\|_{H^{m+2}_{\mathrm{uloc}}}$, we deduce 
\begin{equation*}
\begin{split}
    \|\psi^-\|_{H^{2}_{\mathrm{uloc}}(\omega^b)}+\|\psi^-|_{X=M}\|_{H^{m+1/2}_{\mathrm{uloc}}}&+\left\|\partial_{X}\psi^-\big|_{X=M}\right\|_{H^{m+1/2}_{\mathrm{uloc}}}\\
    \leq&C\left(\sum\limits_{k=0}^3\|g_k\|_{W^{m+2-k,\infty}}+\|\rho_2\|_{ H^{m-1/2}_{\mathrm{uloc}}}\|\rho_3\|_{ H^{m-3/2}_{\mathrm{uloc}}}\right).
\end{split}
\end{equation*}

Proceeding similarly on the half-space, suppose $\psi_i=\partial_{X}^k\psi^{+}|_{\sigma_w^M}$, $i=0,1$ and a smallness condition on $\Psi_w^0$. Then, the unique solution $\psi^+$ of
\begin{eqnarray}\label{problem_c1_lin+}
\partial_X\psi^{+}-\Delta_w^{2}\psi^{+}&&\\
+Q_w(\Psi^0_w+\Psi^0_0,\psi^{+})+Q_w(\Psi_w^0,\Psi_0^{+})+Q_w(\psi^{+},\Psi^0_w+\Psi^0_0)+Q_w(\Psi_0^{+},\Psi^0_w)&=&0,\quad\textrm{in}\quad\{X>M\}\nonumber\\
\psi^{+}\big|_{\sigma_w^M}=\psi_0,\quad\partial_X\psi^{+}\big|_{\sigma_w^M}&=&\psi_1,\nonumber
\end{eqnarray}
fulfills the estimate
\begin{equation*}
\begin{split}
\|e^{\delta X}\psi^+\|_{H^{m+2}_{\mathrm{uloc}}}&\leq C\left(\sum\limits_{k=0}^3\|g_k\|_{W^{m+2-k,\infty}}+\|\rho_2\|_{ H^{m-1/2}_{\mathrm{uloc}}(\mathbb{R})}+\|\rho_3\|_{ H^{m-3/2}_{\mathrm{uloc}}(\mathbb{R})}\right),
\end{split}
\end{equation*}  
provided that $\|e^{\delta X}\Psi_0^-\|_{H^{m+2}_{\mathrm{uloc}}}$ is small.
It is easily seen the Fréchet differential depends continuously on $\psi^+$, and consequently, on $\rho_2,\rho_3$ and $g_k$, $k=0,\ldots,3$.

\textit{$d\mathcal{F}(0,\ldots,0)$ is invertible:} We consider the systems satisfied by $\psi^{\pm}$ with $\Psi_w^0=0$ where the solutions are considered to be $C^3$. Here, $\psi^{\pm}$ are the unique solutions of \eqref{p:linear_nonhomgeneous0} and \eqref{a:linear_equiv}. If $\mathcal{F}(0,\ldots,0,\rho_2,\rho_3)=0$,  $\psi:=\mathbbm{1}_{X\leq M}\psi^{-}+\mathbbm{1}_{X>M}\psi^{+}$ is a solution of the linear system in $\omega_w$ boundary layer domain \eqref{e:linear_app}. From Section \ref{western_l}, we know \eqref{e:linear_app} without jumps at $X=O$ has a unique solution $\psi\equiv 0$ and therefore $\rho_2=0$ and $\rho_3=0$. Consequently,  $\mathrm{ker}\;d\mathcal{F}(0,\ldots,0)=\left\{(0,0)\right\}$.

By Lemma \ref{Lemma_14_lin}, $\psi^{-}$ is a $H^{m+2}_{\mathrm{uloc}}$ function in $[M',M]\times\mathbb{R}$ for all $0<M'<M$. Thus, $$(\rho_2,\rho_3)=\begin{pmatrix}(1+\alpha^2)\Delta_w\psi^{-},-\left((1+\alpha^2)\partial_X-2\alpha\partial_{Y}\right)\Delta_w\psi^{-}+\frac{\psi^{-}}{2}
\end{pmatrix}.$$ We see that $d\mathcal{F}(0,\ldots,0)$ is an isomorphism of $H^{m-1/2}_{\mathrm{uloc}}(\mathbb{R})\times H^{m-3/2}_{\mathrm{uloc}}(\mathbb{R})$. The implicit function theorem guarantees the existence of $(\rho_2,\rho_3)\in H^{m-1/2}_{\mathrm{uloc}}(\mathbb{R})\times H^{m-3/2}_{\mathrm{uloc}}(\mathbb{R})$ such that $\mathcal{F}(\Psi^0_w,g_0,\ldots,g_3,\rho_{2},\rho_{3})=0$ for all $g_k\in W^{m+2-k,\infty}(\mathbb{R})$ and $\Psi^0_w\in H^{m+2}_{\mathrm{uloc}}(\omega^b_w)$ near zero.  Finally, we can assert that $\Psi_w:=\mathbbm{1}_{X\leq M}\Psi_w^{-}+\mathbbm{1}_{X> M}\Psi_w^{-}$, where $\Psi_w^{-},\Psi_w^{+}$ are the solutions of \eqref{p:linear_nonhomgeneous0} and \eqref{a:linear_equiv} associated to $(\Psi^0_w,g_0,\ldots,g_3,\rho_{2},\rho_{3})$ is a solution of the western boundary layer system in the whole domain $\omega_w$. Namely, it satisfies
\begin{equation*}
\begin{split}
    \mathcal{A}_2\left[\Psi^{-}_w\big|_{X=M},\partial_X\Psi_w^{-}\big|_{X=M}\right]&=\mathcal{A}_2\left[\Psi^{+}_w\big|_{X=M},\partial_X\Psi_w^{+}\big|_{X=M}\right],\\
    \tilde{\mathcal{A}}_3\left[\Psi^{-}_w\big|_{X=M},\partial_X\Psi_w^{-}\big|_{X=M}\right]&=\tilde{\mathcal{A}}_3\left[\Psi^{+}_w\big|_{X=M},\partial_X\Psi_w^{+}\big|_{X=M}\right],
\end{split}
\end{equation*}
which implies in turn the continuity of the differential operators across $\{X=M\}$, and therefore $\Psi_w$ at the transparent boundary. This completes the proof of Theorem \ref{t:linearized_hs}.
\section{Eastern boundary layer}\label{eastern}
     This section is devoted to the proof of Theorem \ref{prop:eastern} concerning the well-posedness of problem 
\begin{equation}\label{eastern_gral}
\begin{split}
-\partial_{X_e}\Psi_e-\Delta_e^{2}\Psi_e&=0,\quad\textrm{in}\quad\omega_e^+\cup\omega_e^-\\
\left[\partial_{X_e}^{k}\Psi_e\right]\Big|_{X_e=0}&=\tilde{g}_k,\;k=0,\ldots,3,\\ \Psi_e\big|_{X_e=-\gamma_e(Y)}&=\frac{\partial\Psi_e}{\partial n_e}\big|_{X_e=-\gamma_e(Y)}=0,
\end{split}
\end{equation}
where $\tilde{g}_k\in L^\infty$, for $k=0,\dots,3$. Problem \eqref{eastern_gral} describes the behavior of the eastern boundary layer profiles. Note that the system driving $\Psi_e^0$ is obtained by choosing $\tilde{g}_k\equiv0$, $\forall k$. 

Similar to the previous case of the western boundary layer, we rely on wall laws to show the existence and uniqueness of solutions to system \eqref{eastern_gral} following the steps listed in Section \ref{l_method}. In Section \ref{linear_east}, the analysis of the problem in the half-space exhibits an additional difficulty: the presence of degeneracy at low frequencies. The singular behavior is going to impact the convergence of  $\Psi_e$ when $X_e$ goes infinity. As a result, a probabilistic setting and ergodicity properties are prescribed to prove the convergence of the solution. 

In this context, we will distinguish three different behaviors of $\Psi_e$ far from the boundary: $\Psi_{\mathrm{exp}}$ which decays exponentially to zero, $\Psi_{\mathrm{erg}}$ whose convergence to a specific constant is driven by the ergodic theorem and $\Psi_{\mathrm{alg}}$, a function converging to zero at a polynomial rate when $\varepsilon\rightarrow 0$. The analysis of each one of these functions is going to be conducted separately in subsection \ref{s:diff_behavior}, and final results are summarized in Theorem \ref{prop:eastern}.  The probabilistic scenario is only necessary for the analysis of $\Psi_{\mathrm{erg}}$ since the convergence of the remaining components is obtained using deterministic methods. We pay special attention to the link between the value of the limit of $\Psi_{\mathrm{erg}}$ far for the boundary and the choice of $\tilde{g}_0$.  Later, on Section \ref{construction}, we will see that this translates in a strong connection between the problems driving the $n$-th eastern profile and $\Psi_{int}^n$. Finally, we will briefly discuss the equivalence of solutions of the problems in the half-space and the rough channel at the ``transparent'' boundary, and, hence provide a solution for problem \eqref{rough_part_east} in the whole domain $\omega_e$.

To simplify the notation, we will write $X$ instead of $X_e$ throughout this section.

      \subsection{The problem in the half-space}\label{linear_east}
In this section, we will consider without loss of generality that  $M'=0$, to facilitate the computations. We can easily recuperate the solution when $M'\neq 0$ since the differential operators involved in the problem are translation-invariant. We are confronted with solving the following linear problem
    \begin{align*}\label{bl_east1_halfspace}
    -\partial_X\Psi_{e}-\Delta^2_e\Psi_e&=0,\quad\textrm{in}\quad\mathbb{R}^2_+,\\
    \Psi_e\big|_{X=0}=\psi_0&,\quad
    \partial_X\Psi_e\Big|_{X=0}=\psi_1.\numberthis
\end{align*}
Taking the Fourier transform with respect to $Y$ on the main equation results in the following fourth order characteristic polynomial with complex coefficients
\begin{equation}\label{characteristic_polynomial_east}
    P_e(\mu,\xi)=-\mu+(\mu^2+(-\alpha\mu+i\xi)^2)^2.
\end{equation}
Note that by considering $\mu=-\lambda$ in \eqref{characteristic_polynomial_east}, we obtain the characteristic polynomial corresponding to the linear boundary layer problem defined at the western domain.
Some properties of the polynomial are stated in the following lemma:
\begin{lemma}\label{lemma:asymp_behavior}
Let $ P_e(\mu,\xi)$ be the characteristic polynomial of \eqref{bl_east1_halfspace}, then:
\begin{enumerate}
    \item $P_e(\mu,\xi)$ has four distinct roots $\mu^{\pm}_i$, $i=1,2$, where $\Re(\mu_i^+)>0$, $i=1,2$ and $\Re(\mu_i^-)<0$. Moreover, $\mu^{\pm}=-\lambda^{\mp}$, where $\lambda^{\pm}$ are the roots of \eqref{Dalibard2014_eq24}.
    \item (Low frequencies) As $|\xi|\rightarrow 0$, we have
    \begin{align*}
        \mu_1^+&=\xi ^4+O\left(\xi ^5\right),\quad\mu_2^+=\dfrac{1}{\left(1+\alpha ^2\right)^{2/3}}+O\left(\xi \right), \\
        \mu_i^-&=-\dfrac{1}{2\left(1+\alpha ^2\right)^{2/3}}+(-1)^i\dfrac{\sqrt{3}}{2\left(1+\alpha ^2\right)^{2/3}}+O\left(\xi\right),\quad i=1,2.
    \end{align*}
    \item (High frequencies) When $\xi\rightarrow \infty$, we have
    \begin{align*}
         \mu_{i}^+&=\dfrac{1+i\alpha\mathrm{sgn}\;\xi}{1+\alpha^{2}}|\xi|+\dfrac{(-1)^i}{2}\sqrt{\dfrac{1+i\alpha\mathrm{sgn}\;\xi}{1+\alpha^{2}}}|\xi|^{-1/2}+O(|\xi|^{-3/2}),\quad i=1,2,\\
         \mu_{i}^-&=\dfrac{-1+i\alpha\mathrm{sgn}\;\xi}{1+\alpha^{2}}|\xi|+\dfrac{(-1)^i}{2}\sqrt{\dfrac{-1+i\alpha\mathrm{sgn}\;\xi}{1+\alpha^{2}}}|\xi|^{-1/2}+O(|\xi|^{-3/2}).
     \end{align*}
\end{enumerate}
\end{lemma}
The proof of the previous results follows the same ideas of the ones of Lemmas \ref{p:roots} and \ref{Dalibard2014_lemma24}.  For more details, we refer the reader to Appendix \ref{s:roots}.

When $\xi\neq 0$, polynomial \eqref{characteristic_polynomial_east} has two roots with strictly positive real part noted by $\mu_1$ and $\mu_2$, whose asymptotic behavior is summarized in Lemma \ref{lemma:asymp_behavior}. The solutions of
\begin{eqnarray*}
-\partial_X\widehat{\Psi_e}+(\partial_X^{2}+(\alpha\partial_X+i\xi)^{2})^{2}\widehat{\Psi_e}&=&0,\quad\text{in}\quad\mathbb{R}^2_+\\
\widehat{\Psi_e}\big|_{X=0}=\widehat{\psi_0},&&\partial_X\widehat{\Psi_e}\big|_{X=0}=\widehat{\psi_1},\nonumber
\end{eqnarray*} 
are linear combinations of $\exp(-\mu_k(\xi)X)$ with coefficients also depending
on $\xi$, where $(\mu_k)_{1\leq k\leq 2}$ are the roots of \eqref{characteristic_polynomial_east} satisfying $\Re{\mu}>0$. More precisely, they are of the form $\sum_{i=1}^2 A_k(\xi)\exp(-\mu_k(\xi)X)$, where $A_k:\mathbb{R}\rightarrow\mathbb{C}$, $k=1,2$. Substituting the expression of $\widehat{\Psi_e}$ on the boundary conditions provides the following linear system of equations
\begin{eqnarray*}
A_1+A_2&=&\widehat{\psi_0},\\
\mu_1A_1+\mu_2A_2&=&-\widehat{\psi_1}.
\end{eqnarray*}
It is then clear
that the matrix associated to this system is invertible when $\xi\neq 0$ since all the roots of the characteristic polynomial are simple. Thus, we obtain the coefficients 
\begin{equation}\label{coefficients}
\begin{split}
A_1&:=A_1^0\widehat{\psi_0}+A_1^1\widehat{\psi_1}=\dfrac{\mu_2}{\mu_2-\mu_1}\widehat{\psi_0}+\dfrac{1}{\mu_2-\mu_1}\widehat{\psi_1},\\ A_2&:=A_2^0\widehat{\psi_0}+A_2^1\widehat{\psi_1}=-\dfrac{\mu_1}{\mu_2-\mu_1}\widehat{\psi_0}-\dfrac{1}{\mu_2-\mu_1}\widehat{\psi_1}.
\end{split}
\end{equation}
\begin{lemma}[Asymptotic behavior of the coefficients]\label{lemma_asymp_coeff}

\begin{itemize}
 \item (Low frequencies). When $|\xi|\rightarrow 0$, we have\\
 \begin{align*}
 A_1(\xi)&=\widehat{\psi_0}+(1+\alpha ^2)^{2/3}\widehat{\psi_1}+O(|\widehat{\psi_0}||\xi|^4+|\widehat{\psi_1}||\xi|),\\
 A_2(\xi)&=-(1+\alpha ^2)^{2/3}\widehat{\psi_1}+O(|\widehat{\psi_0}||\xi|^4+|\widehat{\psi_1}||\xi|).
 \end{align*}
    \item (High frequencies) As $|\xi|$ goes to infinity, we have\\
         \begin{align*}
             &A_1(\xi)=-\sqrt{1-i\alpha\mathrm{sgn}\;\xi}|\xi|^{1/2}\widehat{\psi_1}+\left(\dfrac{1}{\sqrt{1-i\alpha\mathrm{sgn}\;\xi}}|\xi|^{3/2}+\dfrac{1}{2}\right)\widehat{\psi_0}\\
             &\hspace{1cm}+O(|\widehat{\psi_0}||\xi|^{-3/2}+|\widehat{\psi_1}||\xi|^{-5/2}),\\
             &A_2(\xi)=\sqrt{1-i\alpha\mathrm{sgn}\;\xi}|\xi|^{1/2}\widehat{\psi_1}-\left(\dfrac{1}{\sqrt{1-i\alpha\mathrm{sgn}\;\xi}}|\xi|^{3/2}-\dfrac{1}{2}\right)\widehat{\psi_0}\\
             &\hspace{1cm}+O(|\widehat{\psi_0}||\xi|^{-3/2}+|\widehat{\psi_1}||\xi|^{-5/2}).
         \end{align*}
         \end{itemize}
\end{lemma}
     \subsubsection{Behavior far from the boundary}\label{s:diff_behavior}
In the previous section, we proved that $\widehat{\Psi}_e(X,\xi)$ was a linear combination of $A_i(\xi)e^{-\mu^+_i(\xi)X}$, where $(\mu_i)_{1\leq i\leq 4}$ are the roots of the characteristic polynomial  \eqref{characteristic_polynomial_east} such that $\Re(\mu_i^+)>0$ and the form of coefficients is given in \eqref{coefficients}. The next step would be to derive a representation formula for $\Psi_e$, based on the Fourier transform results, in such a way that the formula still makes sense when $\psi_0\in H^{3/2}_{\mathrm{uloc}}(\mathbb{R})$ and $\psi_1\in H^{1/2}_{\mathrm{uloc}}(\mathbb{R})$. The behavior of $\Psi_e$ as $X \rightarrow +\infty$ is crucial to understanding the approximation of the solution. Contrary to the periodic setting (see \cite{Bresch2005}), the exponential decay does not hold in this case. Although the terms associated to $A_2(\xi)e^{-\mu_2^+(\xi)X}$ decay exponentially to zero (see paragraph \ref{lemma7_equiv_e}), when looking closely at expressions on \eqref{coefficients} and the asymptotic behavior in \eqref{lemma:asymp_behavior}, it is clear that $A_1(\xi)e^{-\mu_1(\xi)X}$ does not converge to zero at low frequencies when $X\rightarrow+\infty$ since $\mu_1(\xi)=O(|\xi|^4)$.

This section is devoted to proving the results in Theorem \ref{prop:eastern}. We have divided the proof into a sequence of lemmas, one for each component of the solution $\Psi_e$: $\Psi_{\mathrm{exp}}$, $\Psi_{\mathrm{alg}}$ and $\Psi_{\mathrm{erg}}$. $\Psi_{\mathrm{exp}}$ is a function decaying exponentially at infinity, while $\Psi_{\mathrm{alg}}$ refers to the part of $\Psi_e$ converging at polynomial rate. Finally, $\Psi_{\mathrm{erg}}$ is a function whose convergence at $X\rightarrow +\infty$ is guaranteed by using ergodic properties in a probabilistic setting.

\paragraph{Behavior at high frequencies.} Let $\Psi_{e}^{\sharp}$ denote the eastern boundary layer at high frequencies. Following the ideas in \cite[Lemma 9]{Dalibard2017} coupled with the behavior of $\mu^{+}_i(\xi )$ at infinity yields an equivalent result to the one in Lemma \ref{Dalibard2017_lemma9}:

\begin{lemma}\label{lemma:high_east}
	Let $\psi_0,\psi_1\in L_{\mathrm{uloc}}(\mathbb{R})$. Then, the behavior of the $\Psi_{e}$ at high frequencies denoted by $\Psi_{e}^{\sharp}$  satisfies the estimate
	\begin{equation*}
	\left\|e^{\delta X}\Psi_{e}^{\sharp}\right\|_{L^\infty}\leq C\left(\|\psi_0\|_{L^1_{\mathrm{uloc}}(\mathbb{R})}+\|\psi_1\|_{L^1_{\mathrm{uloc}}(\mathbb{R})}\right),
	\end{equation*}
	for some constants $C, \delta >0$ 
\end{lemma}

\paragraph{Low frequencies.}
This paragraph is devoted to the analysis of the behavior of the eastern boundary layer at low frequencies. Each component is enclosed in the corresponding lemma, starting with the deterministic ones (exponential and algebraic) followed by the probabilistic limit.
	
		Our analysis starts with the component of $\Psi_e$  decaying exponentially at infinity. Explicitly, we deal with the term whose Fourier transform satisfies
	\begin{equation}\label{eastern_exp}
	\widehat{\Psi_{\mathrm{exp}}}(X,\xi)=\sum_k\chi(\xi)A_2^k(\xi)e^{-\mu_2^+(\xi)X}\widehat{\psi}_k,
	\end{equation}
	where $A_2^k(\xi)$ is the coefficient of $A_2(\xi)$ associated to $\widehat{\psi}_k$ in \eqref{coefficients} and $\chi$ is defined as in Lemma \ref{lemma:high_east}. 
	
	Following closely the arguments in Lemma \ref{Dalibard2017_lemma7} and  \cite[Lemma 7]{Dalibard2017}, we obtain the result:
	
	\begin{lemma}\label{lemma7_equiv_e}
		Let $\Psi_{\mathrm{exp}}$ defined as in \eqref{eastern_exp} and $\psi_0,\psi_1\in L_{\mathrm{uloc}}(\mathbb{R})$. Then,  there exist constants $\delta, C>0$ independent of $\psi_i$, $i=0,1$ such that
		\begin{equation}
		\|e^{\delta X}\Psi_{\mathrm{exp}}\|_{L^\infty(\omega^+_e)}\leq C(\|\psi_0\|_{L^1_{\mathrm{uloc}}(\mathbb{R})}+\|\psi_1\|_{L^1_{\mathrm{uloc}}(\mathbb{R})}).
		\end{equation}
	\end{lemma}
	
	There remains to study the terms of $A_1$. From \eqref{coefficients}, it is clear that $A_1^k\rightarrow \bar{A}_1^k$, for $k=0,1$ as $\xi\rightarrow 0$ where $\bar{A}_1^k\in\mathbb{R}$. Therefore, it is possible to rewrite each one of the terms as   $A_1^k(\xi)=\bar{A}_1^k+\tilde{A}^k_1(\xi)$ with
	\begin{equation*}\label{decomposition}
		\bar{A}_1^k(\xi)=\left\{\begin{array}{ccl}
			1,&\textrm{for}&k=0\\
			(1+\alpha^2)^{2/3},&\textrm{for}&k=1
		\end{array}\right.,\quad \tilde{A}_1^k(\xi)=\left\{\begin{array}{rcl}
			\dfrac{\mu_1}{\mu_2-\mu_1},&\textrm{for}&k=0\\
			\dfrac{1}{\mu_2-\mu_1}-(1+\alpha^2)^{2/3},&\textrm{for}&k=1
		\end{array}\right..
	\end{equation*}

	The terms $\tilde{A}_1^k(\xi)f$, $k=0,1$ decay to zero far from the boundary at polynomial rate as a result of Lemma \ref{Dalibard2017_lemma7}. Notice that, at low frequencies, $\tilde{A}_1^1$ behaves as a constant, while $\tilde{A}_1^2$ does it similarly to a
	homogeneous polynomial of degree $1$, which is a straightforward consequence of Lemma \ref{lemma:asymp_behavior}. 
	
	We introduce the notation $\Psi_{\mathrm{alg}}$ for the part of $\Psi_e$ decaying to zero algebraically when $X\rightarrow+\infty$, i.e., the one associated to the coefficients $\tilde{A}_k$, $k=0,1$. Its behavior is summarized in the following lemma:
	
	\begin{lemma}\label{lemma7_equiv}
		Let $\psi_k\in L^1_{\mathrm{uloc}}(\mathbb{R})$, $k=0,1$. Then,  there exists a constant $C$ independent of $\psi_i$, $i=0,1$ such that
		\begin{equation}
		\|X^{1/4}\Psi_{\mathrm{alg}}\|_{L^\infty(\mathbb{R}^2)}\leq C(\|\psi_0\|_{L^1_{\mathrm{uloc}}(\mathbb{R})}+\|\psi_1\|_{L^1_{\mathrm{uloc}}(\mathbb{R})}).
		\end{equation}
	\end{lemma}
	
	It is easy to check that the convergence of the terms $\bar{A}_1^kf$  is not guaranteed for any choice of function $f$, $k=0,1$. Therefore, we require additional hypotheses, in this case, of ergodicity.
	 
	We recall for the reader's convenience the probability setting: for $\varepsilon > 0$, let $(P, \Pi, \mu)$ be a probability space where $P$ is the set of $K$-Lipschitz functions, with $K>0$; $\Pi$ the borelian $\sigma$-algebra of $P$, seen as a subset of $C_b(\mathbb{R}^2;\mathbb{R})$ and $\mu$ a probability measure. Let $(\tau_{Y} )_{Y\in \mathbb{R}}$, the measure-preserving transformation group acting
	on $P$. We recall that there exists a function $F \in  L^\infty(P)$ such that
	\begin{equation*}
	\gamma_e(m, Y) =F (\tau_{Y}m),\quad Y \in  \mathbb{R},\quad m \in  P.
	\end{equation*}
	We define the stochastic derivative of $F$ by
	\begin{equation*}
	\partial_m F (m) :=\gamma_e'
	(m,0)\quad \forall m \in  P,
	\end{equation*}
	so that $\gamma_e'(m, Y) =\partial_m F (\tau_{Y}m)$ for $(m, Y) \in  P\times \mathbb{R}$. Then, the eastern boundary layer domain can be described as follows for all $m \in  P$,
	\begin{equation*}
	\begin{split}
	\omega_e(m) &=\left\{(X, Y)\in\mathbb{R}^2:\:\: X>-\gamma_e(m, Y)\right\},
	\end{split}
	\end{equation*}
	where $\omega_{e}(m)=\omega_{e}^+(m)\cup\sigma_{e}\cup\omega_{e}^-(m)$ and $\omega_{e}^{\pm}(m)=\omega_{e}(m)\cap\{\pm X>0\}$. We assume $\gamma_{e}$ is a homogeneous and measure-preserving random process.

	\begin{lemma}\label{lem:erg_east}
		Let $\chi\in C^\infty_0(\mathbb{R})$ and $\mu_{1}^+$ defined as in Lemma \ref{lemma:asymp_behavior}. Suppose that $f$ is a stationary random function belonging to $L^\infty$
		Then,
		\begin{equation*}
		\chi(D)e^{-\mu_{1}^+(D)X}f\rightarrow \mathbb{E}(f),\:\:\textrm{as}\:\: X \rightarrow+\infty,
		\end{equation*}
		locally uniformly in $Y$, almost surely and in $L^p (P)$ for all finite $p$.
	\end{lemma}
	The hypothesis on $f$ can be relaxed since we need $f$ to be at least to $L^1_{\mathrm{uloc}}$.
	\begin{proof}
	Birkhoff Ergodic Theorem (see, e.g., Theorem 4.1.2 of \cite{katok1997introduction}) guarantees the existence of the following limit
		\begin{equation}\label{BGV2008_4-13}
		\begin{split}
		\underset{R\rightarrow\infty}{\lim}\dfrac{1}{R}\int_0^Rf(m, -Y')dY'&=\mathbb{E}(f),\quad\text{almost surely}.
		\end{split}
		\end{equation}
	 Furthermore, following the ideas in \cite[Lemma 4.6]{Basson2008} we have 
	 \begin{equation*}
	 	\begin{split}
	 		\underset{R\rightarrow\infty}{\lim}\dfrac{1}{R}\int_0^Rf(m, Y-Y')dY'&=\mathbb{E}(f),
	 	\end{split}
	 \end{equation*}
 uniformly locally in $Y$, and \eqref{BGV2008_4-13} also satisfies almost surely
 \begin{equation*}
 \underset{R\rightarrow\infty}{\lim}\dfrac{1}{R}\int_0^Rf(m, Y')dY'=\mathbb{E}(f).
 \end{equation*}
	
From Lemma \ref{Dalibard2017_lemma7}, we know that for $|X|\geq 1$,  $Y'\rightarrow K(X, Y-Y')$ is  $L^1(\mathbb{R})$. Let $m(X)=\int_{\mathbb{R}}K(X,Y)dY$, we have \begin{equation*}
			\int_\mathbb{R}K(X,Y)dY= \widehat{K}(X,0)=\chi(0)e^{-\mu_1(0)X}=1,\end{equation*} and, consequently, $m_0(X)=1$, when $X>0$. 
Note that $m(X)=1$, for all $X>0$. Then, proceeding exactly as in \cite[Lemma 4.6]{Basson2008}  
		\begin{equation}\label{op_est_1}
		\begin{split}
		\chi(D)e^{-\mu_1^+(D)X}=\mathbb{E}(f)&+\int_{\mathbb{R}}K(X,Y'-Y)\left(f(Y-Y')-\mathbb{E}(f)\right)dY'.
		\end{split}
		\end{equation}

		Integrating by parts leads to
		\begin{equation*}
			\begin{split}
				\chi(D)e^{-\mu_1^+(D)X}=\mathbb{E}(f)-\int_{\mathbb{R}}\partial_{Y'}K(X,Y')\left(\int_0^{Y'}\left(f(Y-Y_1)-\overline{f}\right)dY_1\right)dY'.
			\end{split}
		\end{equation*}

		Therefore, 
		\begin{equation*}
			\begin{split}
				\chi(D)e^{-\mu_1^+(D)X}f&=\mathbb{E}(f)-\int_{\mathbb{R}}Y'\partial_{Y'}K(X,Y')\left(\dfrac{1}{Y'}\int_0^{Y'}\left(f_0(Y-Y_1)-\overline{f}\right)dY_1\right)dY\\
				&=\mathbb{E}(f)+J.
			\end{split}
		\end{equation*}
		Now, we focus our attention on controlling the term $J$. From Lemma \ref{Dalibard2017_lemma7}, we have the key estimates 
		\begin{align*}\label{15}
			&|Y'\partial_Y'K(X,Y')|\leq C\dfrac{X^{1/4}Y}{X^{3/4}+Y^3},\quad\forall|X|\geq 1,\quad \forall Y'\in\mathbb{R},\\
			&\int_\mathbb{R}Y'\partial_Y'K(X,Y')=-1.\numberthis
		\end{align*}
		
		We conduct the analysis by distinguishing the behavior of $J$ in $|Y'|\leq R$ from the one on $|Y'|> R$ for $R\gg 1 $, denoted by $J_1$ and $J_2$ respectively.  For all $\delta>0$ and all $L>0$ there exists $R>0$ such that for all $|Y|\leq L$ and   $|Y'|\geq R$
		\begin{equation*}
	\left|\dfrac{1}{Y'}\int_0^{Y'}f(Y-Y_1)dY_1-\mathbb{E}(f)\right|\leq \delta,
		\end{equation*}
	and $|J_1|\leq C\delta$. For $|Y'|\leq R$, we have that 
		\begin{equation*}
			\begin{split}
				|J_2|&=\left|\int_{|Y'|\leq R}Y'\partial_{Y'}K(X,Y')\dfrac{1}{Y'}\int_0^{Y'}\left(f(Y-Y_1)-\mathbb{E}(f)\right)dY_1dY\right|\\
				&\leq C\|f\|_{L^1_{\mathrm{uloc}}}\left|\int_{|Y'|\leq R}\dfrac{X^{1/4}Y}{X^{3/4}+Y^3}dY'\right|\\
				&\leq C_R\|f\|_{L^1_{\mathrm{uloc}}}|X|^{-1/4}\underset{X\rightarrow\infty}{\longrightarrow} 0.
			\end{split}
		\end{equation*}
	Thus, \begin{equation*}
			\underset{X\rightarrow\infty}{\lim}J(X,Y)=0,
		\end{equation*}
		which allows us to conclude that
		\begin{equation*}
		\underset{X\rightarrow\infty}{\lim}\chi(D)e^{-\mu_1^+(D)X}f=\mathbb{E}(f).
		\end{equation*}
			\end{proof}
		Combining the results in Lemmas ~\ref{lemma_asymp_coeff} and ~\ref{lem:erg_east}, we can provide a definition for the limit of $\Psi_e$.
 \begin{definition}\label{ref_7.1}
 Let $\psi_0,\psi_1$ be stationary random functions belonging to $L^\infty(\omega_e)$. Then, the limit of $\Psi_{erg}$ as $X\rightarrow+\infty$ is the measurable function $\bar{\phi}:P\rightarrow\mathbb{R}$ of the form
 \begin{equation}\label{ergodic_limit}
    \bar{\phi}=\mathbb{E}[\psi_0]+(1+\alpha^2)^{2/3}\mathbb{E}[\psi_1].
 \end{equation}
 \end{definition}
     \paragraph{Almost sure convergence of $\Psi_{\mathrm{erg}}$.} We show almost sure estimates in the stationary ergodic setting as in \cite{dalibard2011effective}. 

\begin{proposition}\label{p:conv_erg_as}
	Let $\Psi_{\mathrm{erg}}$ be the part of the solution $\Psi_e$ of the system \eqref{eastern_bl_p1} whose convergence is guaranteed by ergodicity hypotheses. Then the following estimates hold:
	\begin{align*}
	&\left\|\Psi_{\mathrm{erg}}(\cdot/\varepsilon)-\bar{\phi}\right\|_{L^2(\Omega^\varepsilon)}=o(1)\quad\textrm{almost surely as }\varepsilon \rightarrow 0.
	\end{align*}
	\end{proposition}
\begin{proof}
	This proof follows the ideas in \cite{dalibard2011effective} inspired from the works by Souganidis (see \cite{Sougadinis}). Let $\delta>0$ be arbitrary. Then, according to Egorov's Theorem, there exist a measurable set $M_\delta \subset P$ and a number $X_\delta > 0$ such that
	\begin{align*}
	\left|\Psi_{\mathrm{erg}}(m,X,0)-\bar{\phi}\right|&\leq\delta,\quad \forall m\in M_{\delta}, \forall X>X_{\delta},\\
	\mu(M^c_\delta)&\leq \delta.
	\end{align*}
	Without loss of generality, we assume that $X_\delta\leq \varepsilon^{-1}(\chi_e(y)-\chi_w(y))$.
From Birkhoff's ergodic theorem, we have that for almost every $m$ there exists $k_\delta > 0$ (depending on $m$) such that if $k > k_\delta$
\begin{equation*}
A_\delta=A_\delta(m):=\left\{Y\in \mathbb{R},\:\: \tau_{Y}m \in  M_\delta^c\:\:\textrm{satisfies:}\:\:|A_\delta \cap (-k,k)|\leq  4k\delta\right\}.
\end{equation*}
Indeed, when $k$ goes to infinity,
\begin{equation*}
\dfrac{1}{2k}\int^k_{-k}\mathbbm{1}_{\tau_Y m\in M_\delta^c}\rightarrow \mu(M_\delta^c)\leq \delta.
\end{equation*}
 Therefore, for almost all $m\in P$, there exists some $k_\delta$ such that for all $k>k_\delta$
\begin{equation*}
|A_\delta \cap (-k,k)|=\left|\int^k_{-k}\mathbbm{1}_{\tau_Y m\in M_\delta^c}\right|\leq 4k\delta.
\end{equation*}
For all $\varepsilon>0$, we have
\begin{align*}
\left\|\Psi_{\mathrm{erg}}(\cdot/\varepsilon)-\bar{\phi}\right\|^2_{L^2(\Omega^\varepsilon)}&=\int_{y_{min}}^{y_{max}}\int_{\Gamma_e}^{\Sigma_w}|\Psi^{m,\mathrm{erg}}_e\left(\frac{x}{\varepsilon},\frac{y}{\varepsilon}\right)-\bar{\phi}|^2dxdy\\
&=\varepsilon^2\int_{y_{min}/\varepsilon}^{y_{max}/\varepsilon}\int_{\gamma_e(Y)}^{(\chi_e(\varepsilon Y)-\chi_w(\varepsilon Y))/\varepsilon}|\Psi_{\mathrm{erg}}(m,X,Y)-\bar{\phi}|^2dXdY\\
&=\varepsilon^2\int_{y_{min}/\varepsilon}^{y_{max}/\varepsilon}\int_{\gamma_e(Y)}^{(\chi_e(\varepsilon Y)-\chi_w(\varepsilon Y))/\varepsilon}|\Psi_{\mathrm{erg}}(\tau_Ym,X,0)-\bar{\phi}|^2dXdY\\
&=\varepsilon^2\int_{y_{min}/\varepsilon}^{y_{max}/\varepsilon}\int_{\gamma_e(Y)}^{X_\delta}|\Psi_{\mathrm{erg}}(\tau_Ym,X,0)-\bar{\phi}|^2dXdY\\
&+\varepsilon^2\int_{y_{min}/\varepsilon}^{y_{max}/\varepsilon}\int_{X_\delta}^{(\chi_e(\varepsilon Y)-\chi_w(\varepsilon Y))/\varepsilon}\mathbbm{1}_{\tau_{Y}m\in M_\delta}|\Psi_{\mathrm{erg}}(\tau_Ym,X,0)-\bar{\phi}|^2dXdY\\
&+\varepsilon^2\int_{y_{min}/\varepsilon}^{y_{max}/\varepsilon}\int_{X_\delta}^{(\chi_e(\varepsilon Y)-\chi_w(\varepsilon Y))/\varepsilon}\mathbbm{1}_{\tau_{Y}m\in M_\delta^c}|\Psi_{\mathrm{erg}}(\tau_Ym,X,0)-\bar{\phi}|^2dXdY\\
&=\sum_{j=1}^{3}I_j.
\end{align*}
Then,
\begin{equation}\label{I_1_est}
\begin{split}
I_1&\leq \varepsilon\left\|\Psi_{\mathrm{erg}}\left(\frac{x}{\varepsilon},\frac{y}{\varepsilon}\right)-\bar{\phi}\right\|^2_{L^2(\{x<x_\delta\}\cap\Omega^{\varepsilon})}\leq C_\delta\varepsilon\left(\|\Psi_{\mathrm{erg}}\|_{\infty}+|\bar{\phi}|\right),
\end{split}
\end{equation}
where the constant $C_\delta$ depends on the random parameter $m$ and $x_\delta=\varepsilon X_\delta$.
Taking into account that if $\tau_Ym\in M_\delta$, $\Psi_{\mathrm{erg}}(m,X,Y)=\Psi_{\mathrm{erg}}(\tau_Ym,X,0)$ and $X>X_\delta$,  
\begin{equation}\label{I_2_est}
I_2\leq \varepsilon^2\left(\dfrac{y_{max}-y_{min}}{\varepsilon}\right)\left(\dfrac{\chi_{w}(y)-\chi_{e}(y)-x_\delta}{\varepsilon}\right)\delta^2\leq C\delta^2.
\end{equation}

As for the third integral, we know that $\Psi_e \in  L^\infty((-\infty, a)\times\mathbb{R})$. Assuming $a=\bar{\chi}_{e}:=\max_{y\in(y_{min},y_{max})}\chi_{e}(y)$ and $\varepsilon < 1/k_\delta$, we have
\begin{equation*}
I_3\leq \varepsilon\|\chi_e-\chi_w\|_{\infty}\|\Psi_{\mathrm{erg}}-\bar{\phi}\|^2_{L^\infty((-\infty,\bar{\chi}_{e})\times\mathbb{R})}\lambda\left(\varepsilon^{-1}(y_{min},y_{max})\cap A_\delta\right),
\end{equation*}
where $\lambda$ denotes the Lebesgue measure. Then, 
\begin{equation}\label{I_3_est}
I_3\leq C\delta\left(|\bar{\phi}|^2+\|\Psi_{\mathrm{erg}}\|^2_{L^\infty((-\infty,\bar{\chi}_{e})\times\mathbb{R})}\right).
\end{equation}
The first estimate of the lemma is obtained by combining \eqref{I_1_est}, \eqref{I_2_est} and \eqref{I_3_est}.
\end{proof}

\paragraph{Connection between the choice of $\tilde{g}_0$ and the limit of $\Psi_{\mathrm{erg}}$.} In the formal construction of the approximate solution, it was important to reflect that boundary layers are not supposed to  have any impact far from the boundary \eqref{bc_far}. In Section \ref{s:diff_behavior}, we showed that at low frequencies and far for the boundary the eastern boundary layer can be decomposed as  $$\Psi_e=\Psi_{\mathrm{exp}}+\Psi_{\mathrm{alg}}+\Psi_{\mathrm{erg}},$$
where $\Psi_{\mathrm{exp}}$ and  $\Psi_{\mathrm{alg}}$ converge to zero when $X\rightarrow+\infty$ at different rates (exponential and with a polynomial weight, respectively), while $\Psi_{\mathrm{erg}}$ converges almost surely to a quantity $\bar{\phi}$ once ergodicity assumptions have been added. It is obvious that the far field condition does not hold for $\bar{\phi}\neq 0$. In this paragraph, we explore how choosing $\tilde{g}_0$ wisely can make
\begin{equation*}
\Psi_e\underset{X\rightarrow+\infty}{\longrightarrow} 0\quad almost\medspace surely,
\end{equation*}
which is the last result in Theorem \ref{prop:eastern}.

We are interested in the specific case when all $\tilde{g}_k$ are given but $\tilde{g}_0$ in \eqref{eastern_gral}. Moreover, $\tilde{g}_0$  is supposed to be constant with respect to the boundary layer variables. Note that if $\Psi_e$ is a solution of \eqref{eastern_gral}, by linearity of the problem, $\tilde{\Psi}_e=\Psi_e+\tilde{g}_0$ also satisfies the boundary layer problem. Passing to the limit as $X\rightarrow +\infty$ gives
\begin{equation*}
\lim\limits_{X\rightarrow+\infty}\tilde{\Psi}_e=\bar{\phi}+\tilde{g}_0.
\end{equation*}
Then, for $\tilde{g}_0=-\bar{\phi}$, the limit of $\Psi_{\mathrm{erg}}$ as $X_e$ goes to infinity equals zero and, in turn,  $\tilde{\Psi}_e$ satisfies both \eqref{eastern_gral} and condition at $+\infty$.

Let us now illustrate the procedure for the profile $\Psi^1_e$. This function satisfies the system of equations
\begin{equation}\label{eastern_gral_1}
\begin{split}
-\partial_{X_e}\Psi_e^1-\Delta_e^{2}\Psi_e^1&=0,\quad\textrm{in}\quad\omega_e^+\cup\omega_e^-\\
\left[\partial_{X_e}^{k}\Psi_e^1\right]\Big|_{\sigma_e}&=-\left[\Psi^1_{int}\right]\Big|_{x=\chi_e(y)},\\
\left[\partial_{X_e}\Psi_e^1\right]\Big|_{\sigma_e}&=\left[\partial_{x}\Psi^0_{int}\right]\Big|_{x=\chi_e(y)},\\
\left[\partial_{X_e}^{k}\Psi_e^1\right]\Big|_{\sigma_e}&=0,\;k=2,3,\\
\Psi_e^1\big|_{X_e=-\gamma_e(Y)}&=\frac{\partial\Psi_e^1}{\partial n_e}\big|_{X_e=-\gamma_e(Y)}=0.
\end{split}
\end{equation}
The jump of its derivative at $\sigma_e$ depends of $\Psi^0_{int}$, and it is, therefore, known from the previous step. Here, $\Psi^1(t,y)$ is the solution of the equation
\begin{equation*}
    \begin{split}
        \partial_x\Psi^1&=0\quad\text{in}\quad\Omega\\
    \end{split}
\end{equation*}
Consequently, $\Psi^1(t,y)=C^1(t,y)$. 

Let us consider a solution $\Psi_e^1$ of \eqref{eastern_gral_1} and define the function
\begin{equation}\label{other_sol}
    \tilde{\Psi}^1_e=\left\{\begin{array}{lcc}
        \Psi^1_e-\bar{\phi}, &\text{in}& \omega_e^+ \\
        \Psi^1_e, &\text{in}& \omega_e^-
    \end{array}\right.,
\end{equation}
where $\bar{\phi}$ is the value of the limit $\Psi^1_e$ when $X\rightarrow +\infty$. Note that \eqref{other_sol} satisfies all the conditions in \eqref{eastern_gral_1} but $[\tilde{\Psi}^1_e]|_{\sigma_e}= -\bar{\phi}$. Moreover, $\tilde{\Psi}^1_e\rightarrow 0$ almost surely far from the eastern boundary. Since the jump at the interface $\sigma_e$ is linked to the interior profile, we have
\begin{equation*}
    \bar{\phi}=\left[\Psi^1_{int}\right]\Big|_{x=\chi_e(y)}=C^1
\end{equation*}
Therefore, considering $C^1(t,y)=\bar{\phi}(t,y)$ provides a solution for the boundary layer \eqref{eastern_gral_1} satisfying the far field condition. 
This interdependence between the eastern boundary layer and the interior profile determines the construction of the approximate solution. First, we solve the problem at the East with a general function $\tilde{g}_0$ to obtain the value of the limit far from the boundary. Then, this information is considered when computing the solution of the corresponding interior profile. Finally, the western boundary layer problem is addressed.

     \subsection{Transparent operators} \label{transparent_op}
This section deals with the well-posedness of the differential operators associated with the eastern boundary layer problem at the transparent boundary. We stick in our analysis to the same ideas developed in Section \ref{transparent_w} and, as a consequence, we summarize the main steps and go into detail only when the differences with the western boundary layer are notable.

Once again, without loss of generality, suppose that the artificial boundary lays at $X=0$ and is defined as follows
\begin{definition}\label{PS_e}
	Let $\Psi_e \in H^2(\mathbb{R}^2_+)$ be the unique weak solution of the Dirichlet problem \eqref{eastern_gral} for $(\psi_0,\psi_1)\in H^{3/2}(\mathbb{R})\times H^{1/2}(\mathbb{R})$. Then, the biharmonic matrix-valued Poincaré-Steklov operator is given by
	\begin{equation}\label{def_PS_e}
	\begin{split}
	&PS_e: H^{3/2}(\mathbb{R})\times H^{1/2}(\mathbb{R})\rightarrow H^{-1/2}(\mathbb{R})\times H^{-3/2}(\mathbb{R})\\
	&PS_e\begin{pmatrix}
	\psi_0\\
	\psi_1
	\end{pmatrix}\defeq\begin{pmatrix}
	\ch{-}(1+\alpha^2)\Delta_e\Psi_e\Big|_{X=0}\\
	\left[\left((1+\alpha^2)\partial_X+2\alpha\partial_{Y}\right)\Delta_e\Psi_e+\displaystyle\frac{\Psi_e}{2}\right]\bigg|_{X=0}
	\end{pmatrix}=K_{e}\ast\begin{pmatrix}
	\psi_0\\
	\psi_1
	\end{pmatrix},
	\end{split}
	\end{equation}
	where $K_{e}$ is the distributional kernel.
\end{definition}

Substituting the Fourier representation of the solution $\Psi_e$ on $PS_e$ similarly to \eqref{coeff_western_rep} results in an explicit formula for the Poincaré-Steklov operator. Indeed, we have  $PS_e=\left(\mathcal{B}_2[\psi_0,\psi_1], \mathcal{B}_3[\psi_0,\psi_1]\right)$ for
\begin{equation}\label{transp_op}
    \begin{split}
        \mathcal{B}_k&:\; H^{3/2}\times H^{1/2}\rightarrow H^{-3/2+k}\\
\mathcal{B}_k[\psi_0,\psi_1]&:=\mathcal{F}^{-1}(n_{k,0}\widehat{\psi_0}+n_{k,1}\widehat{\psi_1}),\quad k=2,3,
    \end{split}
\end{equation}
where $n_{i,j}$ denotes the components of the matrix $M_{e}=\mathcal{F}^{-1}(K_{e})$.

The asymptotic behavior of $M_{e}$ at low and high frequencies is summarized in the following lemma:
\begin{lemma}\label{beh_DtN_e}
	\begin{itemize}
		\item When $|\xi|\ll 1$
		\begin{equation*}
		M_{e}=
		\left(
		\begin{array}{cc}
		\left(\alpha ^2+1\right) |\xi|^2+O\left(|\xi|^3\right) & \left(\alpha ^2+1\right)^{4/3}+O\left(|\xi|\right) \\
		\frac{1}{2}+O\left(|\xi|\right) & -\left(\alpha ^2+1\right)^{2/3}+O\left(|\xi|\right). \\
		\end{array}
		\right)
		\end{equation*}
		\item When $|\xi|\rightarrow +\infty$, there exist complex constants $\overline{n}_{i,j}$, $i=2,3$, $j=0,1$ depending on the parameter $\alpha$ such that
		\begin{equation*}
		M_{e}=\left(
		\begin{array}{cc}
		\overline{n}_{2,0}|\xi| ^2+O\left(|\xi|^{-1/2}\right)& \overline{n}_{2,1}|\xi| +O\left(|\xi|^{-1/2}\right)\\
		\overline{n}_{3,0}|\xi| ^3+O\left(1\right)&\overline{n}_{3,1} |\xi| ^2+O\left(|\xi|^{-1/2}\right).
		\end{array}
		\right)
		\end{equation*}
	\end{itemize}
\end{lemma}

For a proof of Lemma \ref{beh_DtN_e}, we refer the reader to Appendix \ref{a:matrix_t}. Applying the same ideas of Section \ref{transparent_w}, it is easily seen that at high and low frequencies our operator is well-defined and continuous in usual Sobolev spaces. Moreover, Lemma \ref{DP2014_lemma2.21} also holds for $M_{e}$.

Of course, we are interested in extending the definition of $PS_e$ to the case of functions that are not square-integrable in $\mathbb{R}^2$,  but rather locally uniformly integrable.  Generalizing the results above can be easily achieved by following the same reasoning of Section \ref{transparent_w}. The differences in the definition of the operators at the western and eastern boundary layers do not impact the estimates. Consequently, for the convenience of the reader, we list the relevant results without proof. 

The unique extension of $\mathcal{B}_i$, $i=2,3$ to Kato spaces is guaranteed by Lemma \ref{tech_lemma_2}. Moreover, we have the integral representation:
\begin{proposition}\label{ext_op_uloc_e}
	Let $(\psi_0,\psi_1)\in H^{3/2}_{\mathrm{uloc}}(\mathbb{R})\times H^{1/2}_{\mathrm{uloc}}(\mathbb{R})$, and let $\Psi_e$ be the unique solution of \eqref{p:linear_nonhomgeneous0} with $F=0$ and  boundary data $\Psi_e|_{X=0} =\psi_0$  and $\partial_X\Psi_e|_{X=0} =\psi_1$.  Then, for all $\varphi\in C^\infty_0(\bar{\mathbb{R}}^2_+)$
	\begin{equation}\label{link_PS_e}
	-\int_{\mathbb{R}^2_+}\partial_{X}\Psi_e\varphi-\int_{\mathbb{R}^2_+}\Delta_e\Psi_e\Delta_{w}\varphi=\left\langle \mathcal{B}_3[\psi_0,\psi_1]-\dfrac{\psi_0}{2},\varphi\big|_{X=0}\right\rangle+\left\langle \mathcal{B}_2[\psi_0,\psi_1],\partial_X\varphi\big|_{X=0}\right\rangle.
	\end{equation}
	In particular, for $(\psi_0,\psi_1)\in H^{3/2}(\mathbb{R})\times H^{1/2}(\mathbb{R})$ the Poincaré-Steklov operator satisfies
	\begin{equation}\label{positivity}
	\left\langle \mathcal{B}_3[\psi_0,\psi_1],\psi_0\right\rangle+\left\langle \mathcal{B}_2[\psi_0,\psi_1],\psi_1\right\rangle\leq 0.
	\end{equation}
\end{proposition}

It is possible to relate the solution of the \eqref{pb:halfspace_linear_homogeneous} to $(\psi_0,\psi_1)\in H^{3/2}_{\mathrm{uloc}}(\mathbb{R})\times H^{1/2}_{\mathrm{uloc}}(\mathbb{R})$ and $PS_w(\psi_0,\psi_1)$ by introducing a smooth function $\tilde{\chi }$, with $\tilde{\chi } = 1$ in an open set containing $\mathrm{Supp}\varphi$ and the kernel representation formula of the boundary operators. Estimate \eqref{positivity} follows from taking $\Psi_e\in H^2(\mathbb{R}^2_+)$ as test function in \eqref{link_PS_e} and using a density argument. 

Similarly to the linear problem driving the behavior of the western boundary layer, we have that
\begin{proposition}\label{prop:estimates_PS_e}
	Let $\varphi \in C^\infty_0(\mathbb{R})$ such that $\mathrm{Supp}\varphi\subset B(Y_0, R)$, $R\geq 1$, and $(\psi_0,\psi_1)\in H^{3/2}_{\mathrm{uloc}}(\mathbb{R})\times H^{1/2}_{\mathrm{uloc}}(\mathbb{R})$. Then, there exists a  constant $C>0$ such that the following property holds.
	\begin{equation}\label{e:estimate_PS_e}
	\left|\left\langle \mathcal{B}_3[\psi_0,\psi_1],\varphi\right\rangle\right|+\left|\left\langle \mathcal{B}_2[\psi_0,\psi_1],\partial_{X}\varphi\right\rangle\right|\leq C\sqrt{R}\left(\|\varphi\|_{H^{3/2}(\mathbb{R})}+\|\partial_X\varphi\|_{H^{1/2}(\mathbb{R})}\right)\left(\|\psi_0\|_{H^{3/2}_{\mathrm{uloc}}(\mathbb{R})}+\|\psi_1\|_{H^{1/2}_{\mathrm{uloc}}(\mathbb{R})}\right).
	\end{equation}
	Moreover, if $\psi_j\in H^{3/2-j}(\mathbb{R})$, $j=0,1$,
	\begin{equation}\label{e:estimate_PS_normal_e}
	\left|\left\langle \mathcal{B}_3[\psi_0,\psi_1],\varphi\right\rangle\right|+\left|\left\langle \mathcal{B}_2[\psi_0,\psi_1],\partial_{X}\varphi\right\rangle\right|\leq C\left(\|\varphi\|_{H^{3/2}(\mathbb{R})}+\|\partial_X\varphi\|_{H^{1/2}(\mathbb{R})}\right)\left(\|\psi_0\|_{H^{3/2}(\mathbb{R})}+\|\psi_1\|_{H^{1/2}(\mathbb{R})}\right).
	\end{equation}
\end{proposition}
The proof of the proposition relies mainly on bounds on $M_{e}$ which are used to compute estimates in fractional Sobolev spaces. Since the proof is very similar to the one in Lemma \ref{DP2014_lemma2.21}, we refer the reader to Section \ref{transparent_w} for details. 

\subsection{Equivalent problem and estimates on the rough channel}\label{rough_part_east}
In this section, we are concerned with proving the existence of weak solutions to the linear system driving the behavior of $\Psi_e$ in $\omega_e$. There difficulties we confronted before remain: the irregularities of $\gamma_e$ prevents us from using the Fourier transform in the tangential direction and the domain $\omega_e^b$ is  unbounded and, it is therefore impossible to rely on Poincar\'e type inequalities.  Here, we follow the ideas presented in Step (L5) of Section \ref{l_method}.  We start by defining a problem equivalent to \eqref{eastern_gral} yet posed in the bounded channel $\omega^b_e=\{-\gamma_e(Y)\leq X\leq M \}\times \mathbb{R}$, where a transparent boundary condition has been imposed at the interface $X=M$, $M\geq 0$. We have the
system 
\begin{eqnarray}\label{first_eastern_bumped_problem}
	-\partial_{X}\Psi^{-}_e-\Delta_{e}^{2}\Psi^{-}_e&=&F^{L}_e,\quad\text{in}\quad\omega^b_e\setminus\sigma_e\nonumber\\[8pt]
	\left[\Psi^{-}_e\right]\big|_{\sigma_e}&=&-\bar{\phi},\\
	\left[\partial_X^{k}\Psi^{-}_e\right]\big|_{\sigma_e}&=&\tilde{g}_k,\;\;k=1,\dots,3 ,\\
	(1+\alpha^2)\Delta_e\Psi^{-}_e\big|_{\sigma_e^M}&=&\mathcal{B}_2\left[\Psi^-_e\big|_{X=M},\partial_X\Psi^-_e\big|_{X=M}\right],\nonumber\\
	-(1+\alpha^2)\partial_X\Delta_{e}\Psi^-_e+2\alpha\partial_Y\Delta_{e}\Psi^-_e+\dfrac{\Psi^{-}_e}{2}\bigg|_{\sigma_e^M}&=&\mathcal{B}_3\left[\Psi^-_e\big|_{X=M},\partial_X\Psi^-_e\big|_{X=M}\right],\nonumber\\
	\Psi^{-}_e\big|_{X=-\gamma_e(Y)}&=&\frac{\partial\Psi^{-}_e}{\partial{n_e}}\big|_{X=-\gamma_e(Y)}=0,\nonumber
\end{eqnarray}
where $\tilde{g}_k\in L^{\infty}(\mathbb{R})$, $k=1,\ldots,3$ and $\bar{\phi}$ is constant function with respect to the boundary layer variables chosen as in Definition \ref{ref_7.1}. This guarantees the solution of  \eqref{first_eastern_bumped_problem} satisfies the far field condition (see Section \ref{s:diff_behavior}).  Here, $\mathcal{B}_k$ denotes the components of the Poincaré-Steklov operator.
The following lemma states the equivalence between the solutions of the problems \eqref{first_eastern_bumped_problem} and \eqref{eastern_gral}.
\begin{lemma}\label{east_equiv}
	Let $\gamma_{e}\in W^{2,\infty}(\mathbb{R})$ be an ergodic stationary random process, $K$-Lipschitz almost
    surely, for some $K > 0$ in a probability space $(P,\Pi,\mu)$. Assume $\bar{\phi}$ is a constant function with respect to the macroscopic variables and $g_k\in L^\infty(\mathbb{R})$.
		\begin{itemize}
			\item If $\Psi_e$ is a solution of \eqref{e:linear_app} in $\omega_e$ such that $\Psi_e\in H^2_{\mathrm{uloc}}(\omega_e)$, then, $\Psi|_{\omega^b_e}$ is a solution of \eqref{first_eastern_bumped_problem}, and for $X > M$, $\Psi$ solves problem \eqref{e:linear_app}, with $\psi_0 := \Psi_e|_{X=M} \in H^{3/2}_{\mathrm{uloc}}(\mathbb{R})$ and $\psi_1 := \partial_X\Psi_e|_{X=M} \in H^{1/2}_{\mathrm{uloc}}(\mathbb{R})$.
			\item Conversely, if $\Psi^-_e \in H^2_{\mathrm{uloc}}(\omega^b_e)$ and $\Psi^+_e \in H^2_{\mathrm{uloc}}(\mathbb{R}^2_+)$  are solutions of \eqref{first_eastern_bumped_problem} and \eqref{bl_east1_halfspace}, respectively; then, the function
			
			\begin{equation*}
				\Psi_e(X,\cdot):=\left\{\begin{array}{ccc}
					\Psi_e^-(X,\cdot)&\mathrm{for}&-\gamma_e(\cdot)<X<M , \\
					\Psi_e^+(X,\cdot)& \mathrm{for}&X>M,
				\end{array}\right.
			\end{equation*}
			belongs to $H^2_{\mathrm{loc}}(\omega)$ and is a solution of the problem \eqref{e:linear_app}.
		\end{itemize}
		\end{lemma}
The proof of Lemma \ref{east_equiv} follows from combining the results in Section \ref{linear_east} and Proposition \ref{prop:estimates_PS_e} and is, therefore, left to the reader.

\begin{proposition}\label{prop:result_rough_channel_le}
Let $\gamma_e \in W^{2,\infty}(\mathbb{R})$ and $\omega^b_e=\omega_e\cap\{X\leq M\}$, $M>0$. Let $\mathcal{B}_i:H^{3/2}_\mathrm{uloc}(\mathbb{R})\times H^{1/2}_\mathrm{uloc}(\mathbb{R})\rightarrow H^{3/2-i}_\mathrm{uloc}(\mathbb{R})$, $i=2,3$ be Poincaré-Steklov operators verifying Proposition \ref{prop:estimates_PS_e}. Moreover, $\bar{\phi}$ is a constant function with respect to the boundary layer variables and $\tilde{g}_k\in L^\infty(\mathbb{R})$, for $k=1,\ldots,3$. Then, there exists a unique solution $\Psi_e\in H^2_{\mathrm{uloc}}(\omega^b_e\setminus\sigma_e)$ satisfying for some constant $C>0$ the estimate
\begin{equation}\label{estimate_rough_channel_le}
    \|\Psi_e\|_{H^2_\mathrm{uloc}(\omega^b_e)}\leq C\left(\|\bar{\phi}\|_{\infty}+\sum\limits_{k=1}^{3}\|\tilde{g}_k\|_{L^\infty(\mathbb{R})}\right).
\end{equation}
\end{proposition}
It is worth noting that the above result does not need the ergodicity hypothesis.
\begin{proof}
To facilitate the computations, we lift the jump conditions at $X=0$ in order to work with a $C^3$ function at the interface between the interior and rough domains. Namely, we analyze the existence and uniqueness of a solution $\tilde{\Psi}_e=\Psi^{-}_e-\Psi^L_e$ of the system
\begin{eqnarray}\label{a:lift_east}
-\partial_{X}\tilde{\Psi}_e-\Delta_{e}^{2}\tilde{\Psi}_e&=&F^L_e\quad\text{in}\quad\omega^b_e\nonumber\\[8pt]
(1+\alpha^2)\Delta_e\tilde{\Psi}_e\big|_{\sigma_e^M}&=&\mathcal{B}_2\left[\tilde{\Psi}_e\big|_{X=M},\partial_X\tilde{\Psi}_e\big|_{X=M}\right]\nonumber\\
-(1+\alpha^2)\partial_X\Delta_{e}\tilde{\Psi}_e+2\alpha\partial_Y\Delta_{e}\tilde{\Psi}_e+\dfrac{\tilde{\Psi}_e}{2}\bigg|_{\sigma_e^M}&=&\mathcal{B}_3\left[\tilde{\Psi}_e\big|_{X=M},\partial_X\tilde{\Psi}_e\big|_{X=M}\right]\nonumber\\
\tilde{\Psi}_e\big|_{X=-\gamma_e(Y)}&=&\frac{\partial\tilde{\Psi}_e}{\partial{n_e}}\big|_{X=-\gamma_e(Y)}=0,\nonumber
\end{eqnarray}
 where $\Psi^L_e$ is defined as in \eqref{lift_lin} and $F^L_e=\Delta_{e}^{2}\Psi^L_e+\partial_{X}\Psi^L_e$.

Adapting Definition \ref{p:weak_formulation_meth} provides the weak formulation: 
{\it A function $\Psi\in H^2_{\mathrm{uloc}}(\omega^b_e)$ is a solution of \eqref{a:lift_east}  if it satisfies the homogeneous conditions $\tilde{\Psi}_e\big|_{\Gamma_e}=\partial_{\mathrm{n}}\tilde{\Psi}_e\big|_{\Gamma_e}=0$ at the rough boundary, and if, for all $\varphi\in \mathcal{\tilde{V}}$, we have
	\begin{eqnarray}\label{weak_formulation_e}
	\int_{\omega^b_e}\partial_{X}\tilde{\Psi}_e\varphi+\int_{\omega^b}\Delta_e\tilde{\Psi}_e\Delta_e\varphi&=&-\int_{\omega^b}F^L_e\varphi\\
	&&-\left\langle\mathcal{B}_3\left[\tilde{\Psi}_e\big|_{X=M},\partial_X\tilde{\Psi}_e\big|_{X=M}\right]-\frac{\tilde{\Psi}_e}{2}\Big|_{X=M},\varphi\big|_{X=M}\right\rangle_{H^{-3/2}_{\mathrm{uloc}},H^{3/2}_{\mathrm{uloc}}}\nonumber\\
	&&-\left\langle\mathcal{B}_2\left[\tilde{\Psi}_e\big|_{X=M},\partial_X\tilde{\Psi}_e\big|_{X=M}\right],\partial_{X}\varphi\big|_{X=M}\right\rangle_{H^{-1/2}_{\mathrm{uloc}},H^{1/2}_{\mathrm{uloc}}},
	\end{eqnarray}
where $\mathcal{\tilde{V}}$ is the space of functions $\varphi\in C^\infty_0(\overline{\omega^b_e})$ such that $\mathrm{Supp}\varphi\cap \partial\Omega=\emptyset$ and $\mathcal{D}^2_0(\omega^b)$ its completion for the norm $\|\Psi\|=\|\Delta_e \Psi\|_{L^2}$
}
Note that \eqref{weak_formulation_e} is quite similar to \eqref{link_PS_e}. 

To prove the existence and uniqueness of a $H^2_{\mathrm{uloc}}$ solution of problem \eqref{a:lift_east}, we use the method by Lady\v{z}enskaja and Solonnikov on the truncated energies
\begin{equation*}
E_k^n:=\int_{\omega_k}|\Delta_e^2\tilde{\Psi}_e^n|^2,
\end{equation*}
where $\Psi_{e,n}^{-}$ is equal to $\tilde{\Psi}_e$ on $\omega_n$ and zero elsewhere. Here, $\omega_n$ is defined as in \eqref{dom_n}. Then, one applies
the same reasoning on the translated channel to get a uniform local bound. The latter allows us to show that the maximal energy is uniformly bounded, and we can extract a convergent subsequence and obtain the desired result using a compactness argument.

The weak formulation \eqref{weak_formulation_e} and the estimates of the Poincaré-Steklov operator for the eastern boundary layer are very similar to the ones in Section \ref{transparent_w}. Then, it is not surprising to obtain the inequality
\begin{equation}\label{inequality1}
E_k^n \leq C_1\left(k+1+m\underset{k\leq j\leq k+m}{\sup}\left(E_{j+1}^n-E_{j}^n\right)+\frac{1}{m^{4-2\eta}}\underset{j\geq k+m}{\sup}(E_{j+1}^n-E_j^n)\right)\quad \text{for all}\quad k \in \{m, \ldots , n\},
\end{equation}
where $C_1$ is a constant depending on the characteristics of the domain and the jump functions when following the reasoning in Section \ref{s:rough_linear_west}. It is clear that being \eqref{inequality1} a key element in the reminder of the analysis, the results obtained in  Section \ref{s:rough_linear_west} also apply to \eqref{first_eastern_bumped_problem}. The proof is left to the reader.
Finally, consider $\Psi^-_e\in H^2_{\mathrm{uloc}}(\omega^b_e)$ to be the unique solution of \eqref{eastern_gral}. Then, take $\psi_0 := \Psi^-_e\big|_{X=M}\in H^{3/2}_{\mathrm{uloc}}(\mathbb{R})$ and $\psi_1 := \partial_X\Psi^-_e\big|_{X=M}\in H^{1/2}_{\mathrm{uloc}}(\mathbb{R})$, then, there exists a unique solution $\Psi^+_e\in H^2_{\mathrm{uloc}}(\{X>M\}\times\mathbb{R})$. We have that
\begin{equation}
\mathcal{B}_k\left[\Psi^+_e\big|_{X=M},\partial_X\Psi^+_e\big|_{X=M}\right]=\mathcal{B}_k\left[\Psi^-_e\big|_{X=M},\partial_X\Psi^-_e\big|_{X=M}\right],
\end{equation}
and $\Psi^+_e\big|_{X=M}=\Psi^-_e\big|_{X=M}$,  $\partial_X\Psi^+_e\big|_{X=M}=\partial_X\Psi^-_e\big|_{X=M}$. Thus, $\Psi^+=\Psi^+_e\mathbbm{1}_{\{X>M\}\times\mathbb{R}}+\Psi^-_e\mathbbm{1}_{(M\leq X\leq -\gamma_e(Y)}$ is a $H^{2}_{\mathrm{uloc}}$ solution of original problem defined in $\omega_{e}$ \eqref{eastern_gral}.
\end{proof}
     
 \section{Convergence result}\label{s:convergence}
       We are now ready to prove the convergence result stated in Theorem \ref{theorem:convergence}. The general scheme of
the proof is classical: we build an approximate solution
of the fluid system and then show that the approximation is close to an exact
solution through energy estimates. 

\subsection{Construction of the approximate solution}\label{construction}
In this section, we justify the well-posedness of each one of the functions within the approximate solution of the 2d quasigeostrophic problem \eqref{Bresch-GV_1}.

Let us recall that  an approximate solution of problem \eqref{Bresch-GV_3} is defined as follows

\begin{definition}
	A function $\Psi_{app} \in H^2(\Omega^\varepsilon)$ is an approximate solution to \eqref{Bresch-GV_3} if it satisfies the approximate equation
	\begin{equation}
	\left\{\begin{array}{rcll}
	\left(\partial_t+\nabla^{\perp}\Psi_{app}\cdot\nabla\right)\left(\Delta\Psi_{app}+\varepsilon^{-3} y\right)+\Delta\Psi_{app}-\Delta^{2}\Psi_{app}&=&\varepsilon^{-3}\,\mathrm{curl}\,\tau+r^\varepsilon_{e},&\textrm{in }\Omega^{\varepsilon}\\
	\Psi_{app}|_{\partial\Omega^{\varepsilon}}&=&\frac{\partial\Psi_{app}}{\partial n}|_{\partial\Omega^{\varepsilon}}=0,\\
	\Psi_{app}|_{t=0}&=&\Psi_{ini}.
	\end{array}
	\right.
	\end{equation}
for some $r^\varepsilon\in L^{\infty}_t([0,T],H^{-2}(\Omega^\varepsilon))$ such that $r^\varepsilon=r^\varepsilon_0+r^\varepsilon_1+r^\varepsilon_2$, where $r^\varepsilon_k=o(1)$ in $L_t^\infty([0,T],H^{-k})$, $k=0,1,2$. 
\end{definition}

\begin{itemize}
\item At main order in the interior of the domain, we get the Sverdrup relation
\begin{equation*}
\partial_x\Psi^0_{int}=\mathrm{curl}\;\tau.
\end{equation*}
In Section \ref{s:asymptotic}, it was discussed that $\Psi_{int}^0$ does not vanish on the whole boundary. Therefore, we introduce boundary layer corrections resulting from the balance between $\varepsilon^{-3}\partial_x \Psi_{bl}^0$ and 
$Q_{bl}(\Psi_{bl}^0,\Psi_{bl}^0)-\Delta^2\Psi_{bl}^0$. Since it is possible to prescribe only one boundary condition, either on the eastern
coast or on the western coast, the boundary condition for the Sverdrup
equation is chosen such that $\Psi_{bl}|_{\Sigma_e}=0$ in \eqref{form_int0}. 
\item The assumption on $\Psi^0_{int}$ determines that $\Psi_{e}^0\equiv 0$. The system driving the behavior of $\Psi_w^0$ is given by \eqref{western_bl_0}. A smallness condition on $\mathrm{curl}\;\tau$ guarantees existence and uniqueness of a solution $\Psi_w^0$ of \eqref{western_bl_0} with exponential decay when $\varepsilon\rightarrow 0$ (Theorem \ref{theorem:existence}), see Section \ref{fixed_point}.

\item Let us now present the general form of the remaining profiles in the approximation. 

The $n$-th interior profile $\Psi_{int}^{n}=\Psi_{int}^{n}(t,x,y)$ satisfies
\begin{equation}\label{int_gen}
\partial_x\Psi_{int}^{n} =F_{n},
\end{equation}
where $F_{n}$ depends on the $\Psi_{int}^{m}$, $m \leq n-1$. Note that $F_1=F_2=0$ as a result of the perturbation being of order $\varepsilon^{-3}$ and the interior part not having singularities. The terms $\Delta\Psi^{n-3}_{\mathrm{int}}$, $\Delta\Psi^{n-3}_{\mathrm{int}}$ and $\left(\nabla^{\perp}\Psi_{int}^i\cdot\nabla\right)\Delta\Psi_{int}^j$ form the source term $F^n$ in \eqref{int_gen}, when $i+j+3=n$, as well as, . Note that $\Psi_{int}^{n}$ does not meet the boundary conditions, and therefore boundary layer correctors must be defined. Following the direction of propagation of the equation in the main order, we choose $\Psi_{int}^{n}(t,x,y)=C_n(t,y)-\int_x^{\chi_e(y)}F_n(t,x',y)\; dx'$. 

 Function $C_n(t,y)$ is determined by the eastern profile $\Psi^{n}_e$ satisfying the system
\begin{equation}\label{general_eastern_problem_1}
\left\{\begin{array}{rcllc}
-\Delta_{e}^{2}\Psi^{n}_e-\partial_{X_e}\Psi^{n}_e&=&G_n,\quad\mathrm{in}\:\:\omega_e^{-}\cup\omega_w^{+},&\\[8pt]
\left[\Psi^{n}_e\right]\big|_{\sigma_e}&=&-C_n,&\\[10pt]
\left[\partial_{X_e}^k\Psi^{n}_e\right]\big|_{\sigma_e}&=&(-1)^{k+1}\left[\partial_x\Psi^{n-k}_{int}\right]\big|_{x=\chi_e(y)},&k=0,\ldots,3,\\[10pt]
\Psi^{n}_e\big|_{X_e=-\gamma_e(Y)}=0,&&\dfrac{\partial\Psi^{n}_e}{\partial n_e}\big|_{X_e=-\gamma_e(Y)}=0,&
\end{array}
\right.\\[10pt]
\end{equation}
where $G_n$ is conditioned by the behavior of $\Psi_{e}^{m}$, for all $m \leq n-1$. Singularities at low frequencies for \eqref{general_eastern_problem_1} are a consequence of the nature of the main equation as stated in Remark \ref{remark_4.1} and later discussed in Section \ref{s:diff_behavior}. This impacts the asymptotic behavior of the solution far from the boundary, i.e., in the West at macroscopic level.  $C_n(t,y)$ is chosen such that the limit stemming from the Ergodic Theorem converges to zero when $X_e\rightarrow +\infty$. Theorem \ref{prop:eastern} guarantees the well-posedness of system \ref{general_eastern_problem_1} in $\omega_e$ since $G_n$ is usually small, it can be considered as a perturbation parameter of \ref{eastern_gral}.

Similar equations are obtained for the western profiles but with additional interaction terms.   Namely,
\begin{equation}\label{general_western_problem_1}
\left\{\begin{array}{rcllc}
\partial_{X_w}\Psi^{n}_{w}+Q_w(\Psi^n_w,\Psi^0_w)+Q_w(\Psi^0_w,\Psi^n_w)-\Delta_{w}^{2}\Psi^{n}_w&=&H_{n},\quad\mathrm{in}\:\:\omega_{w}^{-}\cup\omega_w^{+},&&\\[8pt]
\left[\partial^k_{X_w}\Psi^{n}_{w}\right]\big|_{\sigma_{w}}&=&-\left[\partial_x^k\Psi^{n-k}_{int}\right]\big|_{x=\chi_w(y)}-\left[\partial^k_{X_e}\Psi^{n}_{e}\right]\big|_{\sigma_{w}},&k=0,\ldots,3,&\\[10pt]
\Psi^{n}_{w}\big|_{X_w=-\gamma_{w}(Y)}=0,&&\dfrac{\partial\Psi^{n}_w}{\partial n_w}\big|_{X_w=-\gamma_w(Y)}=0,&&
\end{array}
\right.\\[10pt]
\end{equation}
where $H_{n}$ depends on $\Psi_{w}^{m}$, $\Psi_{e}^{m}$ and $\Psi_{int}^{m}$, for all $m \leq n-1$. Problem  \eqref{general_western_problem_1} has two features that clearly distinguish it from \eqref{general_eastern_problem_1}: two linearized terms containing $\Psi^0_w$ and the influence of the eastern boundary layer function on the jump at the interface separating the interior domain from the western rough domain. As a consequence, the well-posedness of the solution of \eqref{general_western_problem_1} depends not only on the behavior of interior profiles, but on the first profile of the western boundary layer as well as the solutions of the  eastern boundary layer system. The complete analysis of \eqref{general_western_problem_1} is presented on Section \ref{s:linearized}.
\end{itemize}

To summarize, for the eastern profile $\Psi^n_e$ to be well-posed, it fixes the value of the constant in $\Psi^n_{int}$. Its role is to correct at $\Sigma_e$, $\partial_{X}^k\Psi_{int}^{n-k}$, for $k=1,2,3$. On the other hand, $\Psi^n_w$ plays same role for the interior profiles and the jump of $\Psi^n_e$ and its derivatives at $\Sigma_w$. 

\subsection{Computing the reminder}
Once we have constructed the approximate solution, a natural question arises: how ``far'' is $\Psi_{app}^{\varepsilon}$ from $\Psi^{\varepsilon}$, i.e., what is the error (in a suitable norm) when replacing $\Psi^{\varepsilon}$ by $\Psi_{app}^{\varepsilon}$ ? How far should we take the expansion?; i.e., what is the minimum value of $n$ to satisfy a suitable energy estimate?

\subsubsection{Building the correctors}
\paragraph{The interior term.}
The $\varepsilon^{-3}$ factor in the main equation of \eqref{Bresch-GV_3} dictates the asymptotic development must be taken at least till order $n=3$ to deal with the remaining stemming from the substitution of  $\Psi^0_{int}$ on the original equation. We are referring in particular to $$\partial_t\Delta \Psi^0_{int},\medspace \left(\nabla^{\perp}\Psi_{int}^0\cdot\nabla\right)\Delta\Psi_{int}^0,\quad\text{and}\quad\Delta^2\Psi_{int}^0.$$ 
All subsequent error terms containing only interior profiles are $O(\varepsilon)$.

\paragraph{Western boundary layer profiles}
We start by stressing that far from the western boundary,  all terms containing at least one $\Psi_w^i$ decay exponentially as a result of Theorem \ref{theorem:existence}. Let us, therefore, focus our analysis on the western region. 
The error terms resulting from evaluating in $\partial_t\Delta\Psi_{app}+\Delta\Psi_{app}$ can be considered as part of $r^\varepsilon$. To illustrate this, take $\Delta\Psi_w^0$. We have
\begin{equation}
\|\Delta\Psi_w^0\|_{H^{-2}(\Omega^\varepsilon)}\leq C\|\Psi_w^0\|_{L^{2}(\Omega^\varepsilon)}\leq C\sqrt{\varepsilon}.
\end{equation}
Let us now analyze the corresponding elements in the advection term. Namely,
\begin{subequations}
\begin{align}
	 &\left(\nabla^{\perp}\Psi_{w}^i\cdot\nabla\right)\Delta\Psi_{int}^j\label{96subeq1}\\
	 &\left(\nabla^{\perp}\Psi_{int}^i\cdot\nabla\right)\Delta\Psi_{w}^j\label{96subeq2}\\
	 &\left(\nabla^{\perp}\Psi_{w}^i\cdot\nabla\right)\Delta\Psi_{w}^j\label{96subeq3}\\
	 &\left(\nabla^{\perp}\Psi_{w}^i\cdot\nabla\right)\Delta\Psi_{e}^j\label{96subeq4}\\
	 &\left(\nabla^{\perp}\Psi_{e}^i\cdot\nabla\right)\Delta\Psi_{w}^j.\label{96subeq5}
\end{align}
\end{subequations}
Note that there is a part of \eqref{96subeq1}--\eqref{96subeq3} that is used to compute $\Psi^n_w$ as a component of $G^n$, $i+ j\leq n$. We need to elucidate if the remaining components can be included in $r^\varepsilon$ or if an additional corrector is needed. For \eqref{96subeq1}, we have far from the western boundary that
\begin{equation*}
\begin{split}
\norm{\left(\nabla^{\perp}\Psi_{w}^0\cdot\nabla\right)\Delta\Psi_{int}^0}{H^{-1}(\Omega^\varepsilon)}&\leq C\left(\norm{\nabla^{\perp}\left(\Psi_{w}^0\cdot\nabla\Delta\Psi_{int}^0\right)}{H^{-1}(\Omega^\varepsilon)}+\norm{\Psi_{w}^0D^2\Psi_{int}^0}{L^{2}(\Omega^\varepsilon)}\right)\\
&\leq C\sqrt{\varepsilon}.
\end{split}
\end{equation*}
Hence, there is no need for an additional corrector for \eqref{96subeq1}. The same cannot be said for \eqref{96subeq2} since
\begin{equation}
\begin{split}
\norm{\left(\nabla^{\perp}\Psi_{int}^0\cdot\nabla\right)\Delta\Psi_{w}^0}{H^{-1}(\Omega^\varepsilon)}&\leq C\varepsilon^{-1/2}.
\end{split}
\end{equation}
The troublesome term in the previous equation can be considered as part of the source term in the problem driving the behavior of $\Psi^1_w$.

Although the action of the ergodic part is corrected by the choice of the corresponding interior profile far from the boundary, the advection terms including the parts of $\Psi^n_e$ must be taken into account in the West. For example, let us consider the term of the type \eqref{96subeq4}
\begin{align*}
\norm{\varepsilon\left(\nabla^\perp\Psi_w^0\cdot\nabla\right)\Delta\Psi^1_e}{H^{-1}(\Omega^\varepsilon)}&\leq C\varepsilon\norm{\left(\nabla^\perp\Psi_w^0\cdot\nabla\right)\nabla\Psi^1_e}{L^{2}(\Omega^\varepsilon)}\\
&\leq C\varepsilon\left(\norm{\left(\nabla^\perp\Psi_w^0\cdot\nabla\right)\nabla\Psi^1_e}{L^{2}(\Omega^\varepsilon\cap\left\{x-\chi_w(y)\leq \sqrt{\varepsilon}\right\})}\right.\\
&\hspace{0,5in}\left.+\norm{\left(\nabla^\perp\Psi_w^0\cdot\nabla\right)\nabla\Psi^1_e}{L^{2}(\Omega^\varepsilon\cap\left\{x-\chi_w(y)> \sqrt{\varepsilon}\right\})}\right).
\end{align*}

As a consequence of the decreasing behavior of $\Psi_w^0$ far from the boundary, the second term in the last inequality is exponentially small. For the first element, we have
\begin{equation}
\begin{split}\left\|(\nabla^\perp\Psi_w^0\nabla)\nabla\Psi_{\mathrm{alg}}^1\right\|_{L^{2}(\Omega^\varepsilon\cap\left\{x-\chi_w(y)> \sqrt{\varepsilon}\right\})}&\leq\|\nabla^\perp\Psi_w^0\|_{L^{2}(\Omega^\varepsilon\cap\left\{x-\chi_w(y)> \sqrt{\varepsilon}\right\})}\|D^2\Psi_{\mathrm{alg}}^1\|_{L^{2}(\Omega^\varepsilon\cap\left\{x-\chi_w(y)> \sqrt{\varepsilon}\right\})}\\
&\leq C\varepsilon^{-7/4}.
\end{split}
\end{equation}
Moreover, from Lemma \ref{lem:erg_east}, we know that $\Psi_{\mathrm{erg}}^1=o(1)$ and it easy to verify that $X_e^{k/4}\nabla_e^{k}\Psi_{\mathrm{erg}}^1=o(1)$. Proceeding similarly as before yields
\begin{align*}
	\norm{\varepsilon\left(\nabla^\perp\Psi_w^0\cdot\nabla\right)\Delta\Psi^1_{\mathrm{erg}}}{H^{-2}(\Omega^\varepsilon)}=o(\varepsilon).
\end{align*}
Consequently, the first term needs to be corrected; this is possible when including it as part of the source term in the system dictating $\Psi^2_w$. 

For the term \eqref{96subeq4} when $i=1$ and $j=0$, we decompose the domain as before when analyzing the interaction terms 
\begin{equation}
\begin{split}
\norm{\varepsilon\left(\nabla^\perp\Psi^1_{\mathrm{alg}}\cdot\nabla\right)\nabla\Psi_w^0}{L^{2}(\Omega^\varepsilon\cap\left\{x-\chi_w(y)\leq \sqrt{\varepsilon}\right\})}=O(1),
\end{split}
\end{equation}
and
\begin{equation}
\begin{split}
\norm{\varepsilon\left(\nabla^\perp\Psi^1_{\mathrm{erg}}\cdot\nabla\right)\nabla\Psi_w^0}{L^{2}(\Omega^\varepsilon\cap\left\{x-\chi_w(y)\leq \sqrt{\varepsilon}\right\})}=o(\varepsilon^{-1/4}).
\end{split}
\end{equation}
 Therefore, to deal with these advection terms, we add the functions $\varepsilon^{5/4}\Psi_{ew}^{\mathrm{erg}}$ and $\varepsilon^{3/4}\Psi_{ew}^{\mathrm{alg}}$ to the approximate solution: one for the term decaying algebraically, and the other, for $\Psi^1_{\mathrm{erg}}$. These functions fulfill the equations
\begin{equation}\label{corrector_problem_1}
\left\{\begin{array}{rcllc}
\partial_{X_w}\Psi_{ew}+Q_w(\Psi^0_w,\Psi_{ew})-\Delta_w^2\Psi_{ew}&=&G^{\varepsilon,t,y}(X_w,Y),&\text{in}&\omega_w,\\
\Psi_{ew}\big|_{X=-\gamma_w(Y)}&=&\dfrac{\partial\Psi_{ew}}{\partial {n_w}}\Big|_{X=-\gamma_w(Y)}=0,&&
\end{array}
\right.\\[10pt]
\end{equation}
where, \begin{equation*}
G^{\varepsilon,t,y}(X_w,Y)=\left\{\begin{array}{rcl}
\left(\nabla^\perp_e\Psi^1_{\mathrm{erg}}\left(\frac{\chi_e(y)-\chi_w(y)}{\varepsilon}-X_w,Y\right)\cdot\nabla_w\right)\nabla_w\Psi_w^0(X_w,Y),&\text{for}&\Psi_{ew}=\Psi_{ew}^{\mathrm{erg}},\\
\left(\nabla^\perp_e\Psi^1_{\mathrm{alg}}\left(\frac{\chi_e(y)-\chi_w(y)}{\varepsilon}-X_w,Y\right)\cdot\nabla_w\right)\nabla_w\Psi_w^0(X_w,Y),&\text{for}&\Psi_{ew}=\Psi_{ew}^{\mathrm{alg}}.
\end{array}\right.
\end{equation*}
Note that $\Psi_{ew}$ is $C^3$ at the interface $\sigma_w$. From Section \ref{s:linearized}, \eqref{corrector_problem_1}  has a unique solution decaying exponentially when $\varepsilon$ goes to zero.
The remaining terms are small enough and can be considered a part of $r^\varepsilon$. Indeed, repeating the same reasoning as before, we have
\begin{align*}
\|\varepsilon\left(\nabla^{\perp}\Psi_{int}^1\cdot\nabla\right)\Delta\Psi_{w}^0\|_{H^{-1}(\Omega^\varepsilon)}&=O(\varepsilon^{1/2}),\\
\|\varepsilon^2\left(\nabla^{\perp}\Psi_{e}^1\cdot\nabla\right)\Delta\Psi_{w}^1\|_{H^{-2}(\Omega^\varepsilon)}
&\leq \varepsilon^2\|\nabla^{\perp}\Psi_{e}^1\cdot\nabla\Psi_{w}^1\|_{L^{2}(\Omega^\varepsilon)}\\
&=\leq \|\nabla^{\perp}_{X_e,Y}\Psi_{_e}^1\cdot\nabla_{X_w,Y}\Psi_{w}^1\|_{L^{2}(\Omega^\varepsilon)}=o(\varepsilon^{1/2}),\\
\|\varepsilon^2\left(\nabla^{\perp}\Psi_{e}^1\cdot\nabla\right)\Delta\Psi_{w}^1\|_{H^{-2}(\Omega^\varepsilon)}&\leq \|\left(\nabla_{X_e,Y}^{\perp}\Psi_{e}^1\cdot\nabla_{X_w,Y}\right)\Psi_{w}^1\|_{L^{2}(\Omega^\varepsilon)}=o(\varepsilon^{1/2}).
\end{align*}
\paragraph{The eastern boundary layer profiles}
The same analysis must be applied to $\Psi_e^n$, $i \geq 1$. The linear terms involving $\Psi^1_e$ are $o(\varepsilon)$. Indeed, as a consequence of Proposition \ref{p:conv_erg_as} we have
\begin{equation}
\|\varepsilon\partial_t\Delta\Psi^1_e\|_{H^{-2}(\Omega^\varepsilon)}\leq \varepsilon\|\partial_t\Psi^1_e\|_{L^{2}(\Omega^\varepsilon}=o(\varepsilon).
\end{equation}
The estimate for the Laplacian can be computed in the same manner. From the previous paragraph, we know that advection terms containing $\Psi_w^0$ are exponentially small on the eastern domain. It remains to check the terms
\begin{subequations}
\begin{align}
	&\left(\nabla^{\perp}\Psi_{int}^i\cdot\nabla\right)\Delta\Psi_{e}^j, \label{912subeq1}\\
	&\left(\nabla^{\perp}\Psi_{e}^i\cdot\nabla\right)\Delta\Psi_{w}^j, \label{912subeq2}\\
	&\left(\nabla^{\perp}\Psi_{e}^i\cdot\nabla\right)\Delta\Psi_{e}^j. \label{912subeq3}
\end{align}
\end{subequations}

\eqref{912subeq1} satisfies
\begin{equation*}
	\begin{split}
		\norm{\varepsilon\left(\nabla^{\perp}\Psi_{e}^1\cdot\nabla\right)\Delta\Psi_{int}^0}{H^{-2}(\Omega^\varepsilon)}&\leq\varepsilon\norm{\nabla^{\perp}\Psi_{e}^1\Delta\Psi_{int}^0}{H^{-1}(\Omega^\varepsilon)}\\
		&\leq \norm{\Psi^1_e\Delta\Psi^0_{int}}{L^2(\Omega^\varepsilon)}+\norm{\Psi^1_e\Delta\Psi^0_{int}}{L^2(\Omega^\varepsilon)}\\
		&\leq C\norm{\Psi^1_e}{L^2(\Omega^\varepsilon)}=o(1).
	\end{split}
\end{equation*}
Furthermore, applying the same idea to \eqref{912subeq2} yields $\norm{\varepsilon\left(\nabla^{\perp}\Psi_{int}^0\cdot\nabla\right)\Delta\Psi_{e}^1}{H^{-2}(\Omega^\varepsilon)}=o(1)$. Lastly, \eqref{912subeq3} follows
\begin{equation*}
	\begin{split}
		\norm{\varepsilon^2\left(\nabla^{\perp}\Psi_{e}^1\cdot\nabla\right)\Delta\Psi_{e}^1}{H^{-2}(\Omega^\varepsilon)}&\leq C\varepsilon^2\norm{\nabla^{\perp}\Psi_{e}^1\otimes\nabla\Psi_{e}^1}{L^{2}(\Omega^\varepsilon)}\\
		&\leq C\varepsilon^2\|\nabla\Psi^1_e\|_{L^4}^2\leq C\sqrt{\varepsilon}.
	\end{split}
\end{equation*}
We conclude there is no need for correctors in the eastern domain.

\paragraph{Traces at $\Sigma_e$ and $\Sigma_w$}
The traces of $\Psi^n_{\mathrm{exp}}$ at $\Sigma_w$  and of $\Psi^n_{w}$ at $\Sigma_e$ are exponentially small. Indeed, they satisfy
\begin{equation*}
	\norm{\partial^k_X\Psi^n_{\mathrm{exp}}}{L^\infty_t(H^{3/2-k}(\Sigma_e\cup\Sigma_w)}+\norm{\partial^k_X\Psi^n_{w}}{L^\infty_t(H^{3/2-k}(\Sigma_e\cup\Sigma_w)}=O\left(\exp\left(-\frac{\delta}{\varepsilon}\right)\right),\quad \lambda>0.
\end{equation*}
Moreover, the traces of $\Psi^n_{\mathrm{erg}}$ and $\Psi^n_{\mathrm{alg}}$ are included in the jump conditions in \eqref{general_western_problem_1} and therefore, they are lifted by the western profiles.

Note that thanks to hypotheses we made near $y_{min}$ and $y_{max}$, the traces are zero for $y \in [y_{max} - \lambda, y_{max} ] \cup [y_{min}, y_{min} + \lambda]$. We add a corrector to $\psi^\varepsilon$ to lift the jump conditions of the form
\begin{equation}
\phi_1^\varepsilon(t,x,y)=\sum_k\frac{x^k}{k!}\left(\tilde{\sigma}_k^e\theta(x-\chi_w(y))+\tilde{\sigma}_k^w\theta(x-\chi_w(y))\right),
\end{equation}
where $\tilde{\sigma}_k^w$ and $\tilde{\sigma}_e^w$ denote the values of the traces of the western and eastern profiles, respectively. Moreover,  $\theta$ is a function belonging to $C_c^{\infty}([-\delta, \delta])$ for $\delta > 0$ small enough.

For the traces of $\partial_x^k\Psi_{int}^n$, $k=0,\ldots,3$ at $\Sigma_w\cup\Sigma_e$ to be well-defined, we need correctors $\Psi_w^{i}$, $\Psi_e^{i}$, $n\leq i\leq n+3$ satisfying \eqref{general_western_problem_1} and \eqref{general_eastern_problem_1}, respectively. 
    \subsection{Energy estimates}
 
 Let us consider the difference $\psi^\varepsilon= \Psi^\varepsilon-\Psi^{\varepsilon}_{app}$, with $\Psi^{\varepsilon}$ satisfying \eqref{Bresch-GV_3} and $\Psi^{\varepsilon}_{app}$ defined as in Section \ref{construction}. Moreover, let  $r^\varepsilon$ be the error resulting from the difference between the original solution. We have
 \begin{equation}\label{estimate_uniq}
 \begin{array}{rcl}
 \partial_t\Delta\psi^\varepsilon+\nabla^{\perp}\Psi^\varepsilon\cdot\nabla(\Delta\psi^\varepsilon)+\nabla^{\perp}\psi^\varepsilon\cdot\nabla(\Delta\Psi_{app}^\varepsilon)+\varepsilon^{-3}\partial_x\psi^\varepsilon+\Delta\psi^\varepsilon-\Delta^2\psi^\varepsilon&=&r^\varepsilon\:\text{in}\:\:\Omega^\varepsilon\\
 \psi^\varepsilon|_{\partial\Omega}&=&\dfrac{\partial\psi^\varepsilon}{\partial n}\big|_{\partial\Omega}=0\\
 \psi^\varepsilon|_{t=0}&=&(\Psi_{ini}-\Psi_{app}^\varepsilon)\big|_{t=0},\end{array}
 \end{equation}
 where $\|(\Psi_{ini}-\Psi_{app}^\varepsilon)\big|_{t=0}\|_{L^t_{\infty}(H^1(\Omega^\varepsilon))}=O(\varepsilon)$, which results from hypotheses made on $\Psi_{ini}$. Moreover, $\norm{r^{\varepsilon}}{L^{\infty}_t(H^{-2}(\Omega^{\varepsilon}))}=o(1)$. The details on the computation of the remainder have been discussed in detail in Section \ref{construction} . 
It is clear that $\psi^{\varepsilon}$ belongs to $H^{2}(\Omega^{\varepsilon})$. Multiplying the main equation on \eqref{estimate_uniq} and integrating over $\Omega^\varepsilon$ provides the following
\begin{align*}
\int_{\Omega^\varepsilon}\partial_t\Delta\psi^\varepsilon\cdot \psi^\varepsilon&=-\int_{\Omega^\varepsilon}\partial_t\nabla\psi^\varepsilon\cdot\nabla \psi^\varepsilon=-\dfrac{1}{2}\dfrac{d}{dt}\int_{\Omega^\varepsilon}|\nabla\psi^\varepsilon|^2\\
\int_{\Omega^\varepsilon}\nabla^{\perp}\Psi^\varepsilon\cdot\nabla(\Delta\psi^\varepsilon)&\psi^\varepsilon=-\int_{\Omega^\varepsilon}\left((\nabla^{\perp}\Psi^\varepsilon\cdot\nabla)\nabla^{\perp}\psi^\varepsilon\right)\cdot\nabla^{\perp}\psi^\varepsilon=-\dfrac{1}{2}\int_{\Omega^\varepsilon}\nabla^\perp \Psi^\varepsilon\cdot\nabla|\nabla^\perp\psi^\varepsilon|^2=0\\
\int_{\Omega^\varepsilon}(\nabla^\perp\psi^\varepsilon\cdot\nabla) \Delta\Psi^\varepsilon_{app}\psi^\varepsilon&=-\int_{\Omega^\varepsilon}(\nabla^\perp\psi^\varepsilon\cdot\nabla)\nabla^\perp\Psi^\varepsilon_{app}\cdot \nabla^\perp\psi^\varepsilon=\int_{\Omega^\varepsilon}(\nabla^\perp\psi^\varepsilon\cdot\nabla\Psi^\varepsilon_{app})D^2 \psi^\varepsilon\\
\int_{\Omega^\varepsilon}\partial_x\psi^\varepsilon\psi^\varepsilon&=0\\
\int_{\Omega^\varepsilon}\Delta \psi^\varepsilon \psi^\varepsilon&=\int_{\partial\Omega^\varepsilon}\partial_n \psi^\varepsilon \psi^\varepsilon-\int_{\Omega^\varepsilon}|\nabla \psi^\varepsilon|^2=-\int_{\Omega^\varepsilon}|\nabla \psi^\varepsilon|^2\\
-\int_{\Omega^\varepsilon}\Delta^2 \psi^\varepsilon \psi^\varepsilon&=-\int_{\Omega^\varepsilon}|\Delta\psi^\varepsilon|^2.
\end{align*}

We claim that
\begin{align}\label{bef_gron}
\frac{1}{2}\partial_t\norm{\nabla\psi^{\varepsilon}(t,\cdot)}{L^{2}}^{2}+\norm{\nabla\psi^{\varepsilon}(t,\cdot)}{L^{2}}^{2}+\norm{D^2 \psi^{\varepsilon}(t,\cdot)}{L^{2}}^{2}&\leq\left|\int_{\Omega^{\varepsilon}}r^{\varepsilon} \psi^{\varepsilon}(t,\cdot)\right|+\left|\int_{\Omega^{\varepsilon}}(\nabla^\perp\psi^\varepsilon\cdot\nabla\Psi^\varepsilon_{app})D^2 \psi^\varepsilon\right|.
\end{align}
We proceed to analyze each one of the terms on the r.h.s. For the first term we have
\begin{eqnarray}\label{small_1}
\left|\int_{\Omega^{\varepsilon}}r^{\varepsilon} \psi^{\varepsilon}(t,\cdot)\right|&\leq& \norm{D^2 \psi^{\varepsilon}(t,\cdot)}{L^{2}(\Omega^{\varepsilon})}\|r^\varepsilon\|_{H^{-2}}\\
&\leq&\frac{1}{2}\left(\norm{D^2 \psi^{\varepsilon}(t,\cdot)}{L^{2}(\Omega^{\varepsilon})}^2+\|r^\varepsilon\|^2_{H^{-2}}\right)\nonumber
\end{eqnarray}
We are left with bounding the term  $\int_{\Omega^{\varepsilon}}(\nabla^\perp\psi^\varepsilon\cdot\nabla\Psi^\varepsilon_{app})D^2 \psi^\varepsilon$. The difficulty here comes from $\Psi^{\varepsilon}_{app}$ at the boundary layer since
\begin{align*}
\nabla\Psi^{\varepsilon}_{app} &=\nabla\Psi_{int}+\nabla\Psi^\varepsilon_{e}+\nabla\Psi^\varepsilon_{w}.
\end{align*} 

In particular, from $\nabla^\varepsilon\Psi_{w}=O(\varepsilon^{-1})$ in $L^\infty$ since $|\nabla\Psi_{e}^\varepsilon|=\varepsilon\nabla\Psi^1_e+O(\varepsilon)$, where $\nabla\Psi_{e}^1$, similarly to $\varepsilon\nabla\Psi_{int}$, is small and bounded in $L^\infty$. Hence, we focus our attention on 
\begin{equation*}
\left|\int_{\Omega^{\varepsilon}}(\nabla^\perp\psi^\varepsilon\cdot\nabla\Psi^0_w)D^2 \psi^\varepsilon\right|.
\end{equation*}
From Hardy's inequality, we have the following
\begin{equation*}
\left\|\dfrac{\nabla^\perp\psi^\varepsilon(t,\cdot)}{d(\mathbf{x},\Gamma_w^\varepsilon)}\right\|_{L^2}\leq C\|D^2\psi^\varepsilon(t,\cdot)\|_{L^2}.
\end{equation*}

Here,  $d(\mathbf{x},\Gamma_w^\varepsilon)$ denotes the distance from $\mathbf{x}=(x,y)\in\mathbb{R}^2$ to the western rough boundary. Thus,
\begin{equation}\label{small_term}
\left|\int_{\Omega^{\varepsilon}}(\nabla^\perp\psi^\varepsilon\cdot\nabla\Psi^\varepsilon_w)D^2 \psi^\varepsilon\right|\leq C\|d(\mathbf{x},\Gamma_w^\varepsilon)\nabla \Psi^{0}_{w}\|_{L^\infty}\|D^2\psi^\varepsilon(t,\cdot)\|_{L^2}^2.
\end{equation}
Note that $d(\mathbf{x},\Gamma_w)\nabla \Psi^0_{w}\sim \frac{z}{\varepsilon}\exp(-\frac{z}{\varepsilon})$, where $z$ denotes the distance to the boundary, and, therefore, it satisfies
\begin{equation*}
\|d(\mathbf{x},\Gamma_w^\varepsilon)\nabla^\perp \Psi^{0}_w\|_{L^\infty}\leq C_0,
\end{equation*}
where the small constant $C_0$ does not depend on $\varepsilon$. Hence, \eqref{small_term} can be absorbed by the diffusion term on the l.h.s. Plugging \eqref{small_1} and \eqref{small_term} in \eqref{bef_gron} and then applying the Gr\"{o}nwall's inequality complete the proof of Theorem \ref{theorem:convergence}.
    \section{Acknowledgments}
The author gratefully acknowledges the many helpful suggestions of Anne-Laure Dalibard during the preparation of this paper. The author received funding from the European Research Council (ERC) under the European Union’s Horizon 2020 research and innovation program Grant agreement No. 637653, project BLOC ``Mathematical Study of Boundary Layers in Oceanic Motion''. This work is part of the research project SingFlows funded by the French agency ANR under grant ANR-18-CE40-0027. This material is based upon work supported by the National Science Foundation under Grant No. DMS-1928930 while the author participated in a program hosted by the Mathematical Sciences Research Institute in Berkeley, California, during the Spring 2021 semester.
	\appendix
	\section*{Appendix}
	\section{Study of the roots of equation of the linear western boundary layer}\label{s:roots}
This appendix is concerned with proving Lemma \ref{p:roots}, i.e., that equation (\ref{Dalibard2014_eq24})
\begin{equation*}
P(\lambda)=-\lambda-(\lambda^2+(\alpha_w\lambda+i\varepsilon)^2)^2=0,
\end{equation*}
does not have purely imaginary nor multiple roots when $\xi\neq 0$.

The first part can be easily checked by considering $\lambda=iR$,  $R\in\mathbb{R}$ on (\ref{Dalibard2014_eq24}). This yields the following relation
\[
iR=((1+\alpha_w)^{2}R^{2}+2\alpha_w\xi R+\xi^{2})^{2}.
\]
Taking the imaginary part of the previous equality results in $R=0$. 

The second part of the analysis, although strenuous because (\ref{Dalibard2014_eq24}) is a fourth-degree polynomial with complex coefficients, leads us to discard the cases of purely imaginary roots and multiple eigenvalues. We proceed as follows:

\begin{itemize}
	\item First, we prove that for $\alpha_w$ and $\xi$ equal to zero, all the roots of the characteristic equation are simple.
	\item Then, for $\xi\neq 0$,  we show that there are no purely imaginary simple nor double roots. 
	\item The remaining cases are analyzed through equations resulting from the relation between repeated roots of a polynomial and its derivatives. 
\end{itemize}

The main ingredient in the analysis of the multiplicity of the solutions will be the classical lemma:

\begin{lemma}
	Let $\mathbb{K}$ be a commutative ring and $P(x) \in \mathbb{K}[x]$, a nonconstant polynomial of the form $P(x)=\sum_k a_kx^{k}$. Let $a\in\mathbb{K}$ be a multiple root of $P$. Denote by $P'$, the derivative of the polynomial $P$. Then, $P'(a)=0$.
\end{lemma}

In this case, $P'(x)$ is a third-degree polynomial, and its algebraic solution can be explicitly derived by using, for example,  Cardano's method. We will prove that the form of the original equation excludes the possibility of $P'(x)$ having real roots. Moreover, a system of equations describing the behavior of the real and imaginary parts of the complex roots will be solved to find the relation $\alpha_w$ and $\xi$ satisfy in the presence of a double root. 

\begin{itemize}
	\item [$\bullet$] We start by proving the result for the cases when  $\alpha_w=0$ or $\xi=0$.
	\begin{itemize}
		\item Let us first consider $\xi=0$. Then, equation (\ref{Dalibard2014_eq24}) becomes
	\begin{equation}
	\lambda+(1+\alpha_w^{2})^{2}\lambda^{4}=0.
	\end{equation}
	The roots of the above equation are 
	\begin{equation*}
		\lambda_k=-\dfrac{1}{(1+\alpha_w^{2})^{2/3}}e^{i\frac{2\pi k}{3}},\medspace\text{for}\medspace k=0,1,2, \qquad\lambda_3=0.
	\end{equation*}
It is evident that these roots are simple for all $\alpha_w\in\mathbb{R}$, and that $\Re(\lambda_0)<0$, $\Re(\lambda_1)=\Re(\lambda_2)<0$.
	\item When $\alpha_w=0$, a double root must satisfy the following system of equations
	\begin{equation}\label{p_34}
	\left\{\begin{array}{rcl}
	\lambda+(\lambda^{2}-\xi^{2})^{2}&=&0\\
	1+4\lambda(\lambda^{2}-\xi^{2})&=&0.
	\end{array}
	\right.
	\end{equation}
	Combining both equations in (\ref{p_34}) gives the cubic polynomial
	\begin{equation}\label{p_35}
	\lambda^{3}+\dfrac{1}{16}=0,
	\end{equation}
	The values satisfying the above equation are of the form
	\[
	\lambda_k=-\frac{1}{2 \sqrt[3]{2}}e^{\frac{2\pi} {3}i k},\medspace\text{for}\medspace k=0,1,2.
	\]
	Since $\xi$ is a real-valued quantity, we have that $\Im(\lambda)=0$ as a result of combining (\ref{p_34}.b) and (\ref{p_35}). The latter does not hold for $\lambda_1$ and $\lambda_2$. Moreover, a quick substitution of $\lambda_0$ in (\ref{p_34}.b) results in a contradiction since $\lambda_0\xi^{2}\leq 0$ for $\xi\in\mathbb{R}$. 
	
	Hence, equation (\ref{Dalibard2014_eq24}) does not have repeated roots when $\alpha_w=0$.
		\end{itemize}
	\item [$\bullet$]  We now analyze the case when $\alpha_w,\xi\in\mathbb{R}^{*}$.
Let us start by rewriting equation (\ref{Dalibard2014_eq24}) as follows
\begin{equation}
\dfrac{\lambda}{(1+\alpha_w^{2})^{2}}+\left(\left(\lambda+\dfrac{\alpha_w i\xi}{1+\alpha_w^{2}}\right)^{2}-\dfrac{\xi^{2}}{(1+\alpha_w^{2})^{2}}\right)^{2}=0.
\end{equation}
We then introduce the variable $\mu=\lambda+\frac{\alpha_w i\xi}{1+\alpha_w^{2}}$. Due to the affine relation between $\mu$ and $\lambda$, we can assert that a repeated root of the problem will satisfy the following system
\begin{equation}\label{sist:30}
\left\{
\begin{array}{rcl}
\dfrac{\mu}{(1+\alpha_w^{2})^{2}}-\dfrac{\alpha_w i\xi}{(1+\alpha_w^{2})^{3}}+\left(\mu^{2}-\dfrac{\xi^{2}}{(1+\alpha_w^{2})^{2}}\right)^{2}=0,\\
\dfrac{1}{4(1+\alpha_w^{2})^{2}}+\mu\left(\mu^{2}-\dfrac{\xi^{2}}{(1+\alpha_w^{2})^{2}}\right)=0.
\end{array}
\right.
\end{equation}
\begin{lemma}\label{lemma_roots}
	Let $\alpha_w,\xi\in\mathbb{R}^{*}$ and $\mu\in \mathbb{C}\setminus\mathbb{R}$ be a solution of (\ref{sist:30}). Then, $|\xi|<\dfrac{\sqrt{3}}{2}(1+\alpha_w^{2})^{1/3}$. Furthermore, setting $a=-\dfrac{\xi^{2}}{(1+\alpha_w^{2})^{2}}$, $ b=\dfrac{1}{4(1+\alpha_w^{2})^{2}}$ and
	 \begin{eqnarray*}
		A&=&\sqrt[3]{-\dfrac{b}{2}+\sqrt{\dfrac{b^{2}}{4}+\dfrac{a^{3}}{27}}},\\
		B&=&\sqrt[3]{-\dfrac{b}{2}-\sqrt{\dfrac{b^{2}}{4}+\dfrac{a^{3}}{27}}},\\
	\end{eqnarray*}
the solution $\mu$ belongs to $\left\{\mu_-,\mu_+\right\}$, for
	\begin{eqnarray*}
		\mu_+&=&-\dfrac{1}{2}(A+B)+\dfrac{\sqrt{3}i}{2}(A-B)\\
		\mu_-&=&-\dfrac{1}{2}(A+B)-\dfrac{\sqrt{3}i}{2}(A-B).
	\end{eqnarray*}
\end{lemma}
\begin{proof}
Note that \eqref{sist:30} is written in canonical form. The sign of the term $\frac{b^{2}}{4}+\frac{a^{3}}{27}$ determines the number of real roots of the cubic polynomial, going from three in the cases when it is non-positive to one, when $\frac{b^{2}}{4}+\frac{a^{3}}{27}$ is greater than zero. Let us prove that if a repeated root $\mu$ exists, it must forcibly have a non zero imaginary part. 

Assuming $\mu\in\mathbb{R}$ and then taking the imaginary part of (\ref{sist:30}a) lead to
\begin{equation}
-\dfrac{\alpha_w\xi}{(1+\alpha_w^{2})^{3}}=0,
\end{equation}
which contradicts the assumption that $\alpha_w$ and $\xi$ are not equal zero. We have thus proved that $\mu$ must be a complex quantity with nonzero imaginary part for the characteristic equation to have a repeated root. Therefore, $\frac{b^{2}}{4}+\frac{a^{3}}{27}$ has to be positive which in turn, implies that $A$ and $B$ must be real quantities and the Fourier variable should satisfy $|\xi|<\frac{\sqrt{3}}{2}(1+\alpha_w^{2})^{1/3}$.
\end{proof}

\begin{lemma}
Let $\mu\in\mathbb{C}\setminus\mathbb{R}$ be a solution of (\ref{sist:30}). Then, $\mu$ is also a root of the quadratic polynomial
\begin{equation}\label{sist:31}
\dfrac{3}{4}\mu^{2}-\dfrac{\alpha_w i\xi}{1+\alpha_w^{2}}\mu+\dfrac{\xi^{2}}{4(1+\alpha_w^{2})^{2}}=0.
\end{equation}
\end{lemma}
This equation is easily obtained by first, multiplying  (\ref{sist:30}a) by $\mu$ and  (\ref{sist:30}b) by $\left(\mu^{2}-\dfrac{\xi^{2}}{(1+\alpha_w^{2})^{2}}\right)$ and then, subtracting the resulting equations .

Substituting $\mu_{\pm}=-\dfrac{1}{2}(A+B)\pm\dfrac{\sqrt{3}i}{2}(A-B)$ in (\ref{sist:31}) yields the following system for the real and imaginary parts of the solution

\begin{equation}\label{sist:32}
\left\{\begin{array}{rcl}
\dfrac{3}{16} (A+B)^2-\dfrac{9}{16} (A-B)^2\pm\dfrac{\sqrt{3}}{2} \dfrac{\alpha_w \xi}{1+\alpha_w^{2}} (A-B)+\dfrac{\xi^{2}}{4(1+\alpha_w^{2})^{2}}&=&0\\
\mp\dfrac{3\sqrt{3} }{8} (A+B)(A-B)+\dfrac{\alpha_w\xi}{2(1+\alpha_w^{2})}(A+B) &=&0.
\end{array}\right.
\end{equation}

The root $\lambda$ cannot be purely imaginary, hence, $A+B\neq 0$ and the condition $A-B\neq 0$ is derived from the fact that $\mu$ must be complex. Dividing equation (\ref{sist:32}b) by $A+B$ gives
\begin{equation}
A-B= \pm\dfrac{4}{3\sqrt{3}}\dfrac{\alpha_w}{1+\alpha_w^{2}}\xi.
\end{equation}
 
We check at once that $A+B$ is also proportional to $\xi$. From the relation $(A+B)^{2}=(A-B)^{2}+4AB$ and the fact that $AB=-\frac{a}{3}$, we obtain
\begin{equation}\label{sist:34}
(A+B)^2=\dfrac{4}{3(1+\alpha_w^{2})^2}\left(\frac{4\alpha_w^{2}}{9}+1\right)\xi^2.
\end{equation} 

The term $A+B$ must additionally satisfy equation (\ref{sist:32}a). Taking into account that $\xi\neq 0$, the combination of (\ref{sist:32}a) and (\ref{sist:34}) provides the following condition for $\alpha_w$

\begin{equation}\label{alpha_cond}
\dfrac{8 \alpha_w ^2+9}{18 \left(\alpha_w ^2+1\right)^2}=0.
\end{equation}
The roots of the above equality are complex; hence, there is no real-valued $\alpha_w$ for which \eqref{Dalibard2014_eq24} has multiple roots. The same conclusion can be drawn from solving system (\ref{sist:32}) directly since complex or null $A+B$ terms contradict our previous assumptions. 
\end{itemize}
Finally, we conclude \eqref{Dalibard2014_eq24} has four simple roots $(\lambda^{\pm}_i)_{i=1,2}$ $\forall \xi, \alpha_w\in \mathbb{R}$ in $\mathbb{C}\setminus\mathbb{I}$, satisfying $\Re(\lambda_{i}^+)>0$ and $\Re(\lambda_{i}^-)>0$ if $\xi\neq 0$.
\section{Expansion of the eigenvalues at high frequencies}\label{ss:appendix_B}
This section is devoted to high-frequency expansions of the main functions we work with, namely, $\lambda_k$ and $A_k$.

 In high frequencies, that is, for $|\xi|\gg 1$, by considering $\lambda=\xi\rho$, where $\rho\in\mathbb{C}$, we obtain
 \begin{equation}
 \xi^{-3}\rho+\left((1+\alpha_w^{2})\rho^{2}+2i\alpha_w\rho-1\right)^{2}=0.
 \end{equation}
The above polynomial provides the following approximation of the solutions
\begin{equation}
\rho^{+}=\frac{1-i\alpha_w}{\alpha_w^{2} +1}+O( |\xi|^{-3/2}),\quad\rho^{-}=-\frac{1+i\alpha_w}{\alpha_w^{2} +1}+O( |\xi|^{-3/2}).
\end{equation}
The definition of $\lambda$ exhibits a clear relation between the sign of its real part and the sign of $\rho\xi$. In particular, $\lambda$ has positive real part if and only if $\mathrm{sgn}(\Re(\rho))=\mathrm{sgn}(\xi)$.

Since we have already proved that all roots of (\ref{Dalibard2014_eq24}) are simple, we will provide a second term on the expansion of the solutions to make this assertion evident to the reader. Let $\bar{\rho}$ be a root of the polynomial $\left((1+\alpha_w^{2})\rho^{2}+2i\alpha_w\rho-1\right)^{2}$  and $\rho=\bar{\rho}+\xi^{-\eta}\tilde{\rho}+O(\xi^{-3})$, where $\eta< 3$, we have
\begin{equation}
 \xi^{-3}\bar{\rho}+\left(2(1+\alpha_w^{2})\tilde{\rho}\bar{\rho}\xi^{-\eta}+2i\alpha_w\tilde{\rho}\xi^{-\eta}+O(\xi^{-2\eta})\right)^{2}+O(\xi^{-3-\eta})=0.
\end{equation}
This is the same as
\[
\xi^{-3}\bar{\rho}+4\tilde{\rho}^{2}\xi^{-2\eta}\left((1+\alpha_w^{2})\bar{\rho}+i\alpha_w+O(\xi^{-\eta})\right)^{2}+O(\xi^{-3-\eta})=0,
\]
from which we conclude that $\eta=3/2$ and
\begin{equation}
\tilde{\rho}^{\pm}=\pm\frac{i\sqrt{\bar{\rho }}}{2 \sqrt{(\alpha_w ^2 +1)\bar{\rho }+i \alpha_w }}.
\end{equation}
Consequently, for $j=1,2$,
\begin{eqnarray}
\rho_j^{+}&=&\frac{1-i\alpha_w}{\alpha_w^{2} +1}+(-1)^{j}\frac{i\xi^{-\frac{3}{2}}}{2}\left(\frac{1-i\alpha_w}{\alpha_w^{2} +1}\right)+O( \xi^{-3}),\nonumber\\
\rho_{j}^{-}&=&-\frac{1+i\alpha_w}{\alpha_w^{2} +1}+(-1)^{j}\frac{i\xi^{-\frac{3}{2}}}{2}\sqrt{\frac{1+i\alpha_w}{\alpha_w^{2} +1}}+O( \xi^{-3}).\nonumber
\end{eqnarray}

We now turn towards the expressions of  the $A_k$'s which satisfy the linear system
\[
\begin{pmatrix}
1&1\\
\lambda_1&\lambda_2
\end{pmatrix}\begin{pmatrix}
A_1\\A_2
\end{pmatrix}=\begin{pmatrix}\hat{\psi}^{*}_0\\
\hat{\psi}^{*}_1
\end{pmatrix}.
\]
Thus,
\begin{equation}
\begin{split}
A_1=\frac{\lambda _2\hat{\psi}^{*}_0}{\lambda _2-\lambda _1}-\frac{\hat{\psi}^{*}_1}{\lambda _2-\lambda _1},\quad&A_2=-\frac{\lambda _1\hat{\psi}^{*}_0}{\lambda _2-\lambda _1}+\frac{\hat{\psi}^{*}_1}{\lambda _2-\lambda _1}.
\end{split}
\end{equation}

\underline{\textit{High frequency expansions}}.
At infinity, the sign of real part of the $\lambda$ will depend on the sign of $\xi$, hence, for $\xi \rightarrow +\infty$
\begin{eqnarray}
\lambda_{j}^{+}(\xi)&=&\frac{1-i\alpha_w}{\alpha_w^{2} +1}\xi+(-1)^{j}\frac{i\xi^{-\frac{1}{2}}}{2}\sqrt{\frac{1-i\alpha_w}{\alpha_w^{2} +1}}+O( \xi^{-2}),\nonumber\\
A_{j}^{+}(\xi)&=&(-1)^{j} i\xi^{\frac{3}{2}}\sqrt{\frac{1-i\alpha_w}{\alpha_w^{2} +1}}\psi_{0}^{*}+(-1)^{j-1} i\xi^{\frac{1}{2}}\sqrt{1+i\alpha_w}\psi_{1}^{*}+O( |\psi_{0}^{*}|+|\xi|^{-5/2}|\psi_{1}^{*}|),\nonumber
\end{eqnarray}
and, when $\xi\rightarrow -\infty$
\begin{eqnarray}
\lambda_{j}^{+}(\xi)&=&\frac{1+i\alpha_w}{\alpha_w^{2} +1}|\xi|+(-1)^{j}\frac{|\xi|^{-\frac{1}{2}}}{2}\sqrt{\frac{1+i\alpha_w}{\alpha_w^{2} +1}}+O( \xi^{-2}),\nonumber\\
A_{j}^{+}(\xi)&=&(-1)^{j-1}|\xi|^{\frac{3}{2}}\sqrt{\frac{1+i\alpha_w}{\alpha_w^{2} +1}}\psi_{0}^{*}+(-1)^{j}|\xi|^{\frac{1}{2}}\sqrt{1-i\alpha_w}\psi_{1}^{*}+O( |\psi_{0}^{*}|+|\xi|^{-5/2}|\psi_{1}^{*}|).\nonumber
\end{eqnarray}

\section{Computations of the regularity estimates}\label{a:regularity}
	We now focus our attention on the term  $\int_{\mathbb{R}}|\xi|^{2k}|\widehat{\Psi}_{w}|^{2}d\xi$, $k=0,1,2$. First, we decompose the integral into two pieces, one on $\{|\xi| > \xi_0\}$ and $\{|\xi| \leq \xi_0\}$.

\begin{itemize}
	\item[$\triangleright$] On the set $|\xi|\leq \xi_0$, 
	\begin{eqnarray}
	\int_{|\xi|\leq \xi_0}\int_0^{\infty}|\xi|^{2k}|\widehat{\Psi}_{w}|^{2}d\xi dX&\leq&C\sum_{j=1}^{2}\int_{|\xi|\leq \xi_0}|\xi|^{2k}\dfrac{|A_j(\xi)|^{2}}{2\Re(\lambda_j)}.
	\end{eqnarray}
	
	Using \eqref{inverse_et} and Lemma \ref{Dalibard2014_lemma24}
	\begin{equation*}
	\int_{|\xi|\leq \xi_0}\int_0^{\infty}|\xi|^{2k}|\widehat{\Psi}_{w}|^{2}d\xi dX\leq C\int_{|\xi|\leq \xi_0}\left(|\widehat{\psi}^*_0|^2+|\widehat{\psi}^*_1|^2\right)<+\infty.
	\end{equation*}
	
	\item[$\triangleright$] We now analyze the case when $|\xi|>\xi_0$. We are only illustrating the case when $\xi>0$ since the negative case can be obtained in the same manner.  For $k=0,1,2$, we have
	\begin{equation}\label{asymp_inf}
	|\xi|^{2k}\int_{0}^{\infty}|\widehat{\underline{\Psi}}_w|^{2}dX=|\xi|^{2k}\left(\sum_{1\leq l,m\leq 2}A_l^{+}\bar{A}_m^{+}\dfrac{1}{\lambda_l^{+}+\bar{\lambda}_m^{+}}\right).
	\end{equation}
	From Lemma \ref{Dalibard2014_lemma24}, we have that $\lambda_l^{+}+\bar{\lambda}_m^{+}=a|\xi|+b_{lm}|\xi|^{-1/2}+O(|\xi|^{-2})$, $a=\Re(\zeta^{2})$ and $b_{lm}\in\mathbb{R}\cup\mathbb{R}$, $l,m\in\{1,2\}$. Hence,
	\[
	\frac{1}{\lambda_k+\bar{\lambda}_l}=\frac{1}{a |\xi| }\left(1-\frac{b_{lm} |\xi|^{-3/2}}{a}\right)+O\left(\xi^{-3}\right),\quad k,l=1,2.
	\]
	As a consequence,
	\begin{equation*}
	\begin{array}{r}|\xi|^{2k}\displaystyle\int_{0}^{\infty}|\hat{\underline{\Psi}}_w|^{2}dX=\dfrac{|\xi|^{2k-1}}{a}\left(|A_1^++A_2^+|^{2}-\dfrac{\xi^{-3/2}}{a}\sum_{l,m}b_{lm}A_l^+\bar{A}^+_m\right)+\;O\left((|\widehat{\psi}^{*}_0|+|\widehat{\psi}^{*}_1|)|\xi|^{2k-1}\right).
	\end{array}
	\end{equation*}
	Here, $|A_1^++A_2^+|^{2}=|\hat{\psi}_{0}^{*}|^{2}$ and $\sum_{l,m}b_{lm}A_l^+\bar{A}_m^+=O(|\xi|^{3}|\hat{\psi}_{0}^{*}|^2+|\xi|||\hat{\psi}_{1}^{*}|^2)$. 

Asymptotic expansions in Lemma \ref{Dalibard2014_lemma24} lead to
\begin{equation*}
\lambda_l^{+}+\bar{\lambda}_m^{+}=\left\{\begin{array}{lcl}2\Re{\left(\zeta^{2}\right)}\xi+(-1)^{l} \xi^{-1/2}\Im\left(\zeta\right)+O(\xi^{-2}),&\textrm{for}&l=m,\\
2\Re{\left(\zeta^{2}\right)}\xi+(-1)^{l-1} i\xi^{-1/2}\Re\left(\zeta\right)+O(\xi^{-2}),&\textrm{for}&l\not=m,
\end{array}\right.
\end{equation*}
for $\zeta=\sqrt{\frac{1-i\alpha_w}{\alpha_w^{2} +1}}$, and
\hspace*{-5mm}\begin{equation*}
A_l^{+}\bar{A}_m^{+}=\left\{
\begin{array}{lcl}
\left(\frac{1}{4}+i\Im(\zeta)\xi^{3/2}+|\zeta|^{2}\xi^{3}\right)|\hat{\psi}^{*}_0|^{2}-\frac{\xi}{|\zeta|^{2}}|\hat{\psi}^{*}_1|^{2}+\frac{(-1)^{l-1}}{|\zeta|^{2}}\left(i\Im(\zeta)\xi^{1/2}-2\Re(\zeta^{2})\xi^{2}\right)|\hat{\psi}^{*}_0\hat{\psi}^{*}_1|,\;l =m\\
\left(\frac{1}{4}+i\Im(\zeta)\xi^{3/2}+|\zeta|^{2}\xi^{3}\right)|\hat{\psi}^{*}_0|^{2}-\frac{\xi}{|\zeta|^{2}}|\hat{\psi}^{*}_1|^{2}+\frac{(-1)^{l-1}}{|\zeta|^{2}}\left(i\Re(\zeta)\xi^{1/2}-2\Im
(\zeta^{2})\xi^{2}\right)|\hat{\psi}^{*}_0\hat{\psi}^{*}_1|,\;l\not=m.
\end{array}
\right.
\end{equation*}
These expressions could mistakenly lead to a need for more robust results on regularity at the boundary. We will show that
nontrivial cancellations occur for the usual elliptic regularity result to hold in this case.
All sums of the eigenvalues are of the type $a\xi+b_{lm}\xi^{-1/2}+O(\xi^{-2})$, $a=\Re(\zeta^{2})$ and $b_{lm}\in\mathbb{R}\cup\mathbb{I}$, $l,m\in\{1,2\}$. Thus, 
		\[
		\frac{1}{\lambda_k+\bar{\lambda}_l}=\frac{1}{a \xi }\left(1-\frac{b_{lm} |\xi|^{-3/2}}{a}+\frac{b^2_{lm}\xi^{-3}}{a^2}\right)+O\left(\xi^{-9/2}\right),\quad k,l=1,2,
		\]
		where
		\[
		b_{lm}=\left\{\begin{array}{ll}
		(-1)^{l}\Im(\zeta),&l=m,\\
		(-1)^{l-1}i\Re(\zeta),&l\not=m.
		\end{array}\right.
		\]
Substituting the above expression in (\ref{asymp_inf}) leads to
\begin{equation}\label{regularity_key}
|\xi|^{2k}\int_{0}^{\infty}|\hat{\underline{\Psi}}_w|^{2}dX=\frac{|\xi|^{2k-1}}{a}\left(|A_1+A_2|^{2}-\frac{|\xi|^{-3/2}}{a}\sum_{l,m}b_{lm}A_l\bar{A}_m +\frac{\xi^{-3}}{a^2}\sum_{l,m}b^2_{lm}A_l\bar{A}_m\right)+O\left(\xi^{2k-9/2}\right).
\end{equation}
Moreover,
\begin{eqnarray}\label{blm}
|A_1+A_2|^{2}&=&|\psi_{0}^{*}|^{2},\nonumber\\
\sum_{l,m}b_{lm}A_l\bar{A}_m&=&A_2 \bar{A}_2 \Im(\zeta)+A_1 \bar{A}_1 (-\Im(\zeta))-i A_2 \bar{A}_1 \Re(\zeta)+i A_1 \bar{A}_2 \Re(\zeta)\nonumber\\
&=&2 \sqrt{\xi }\left(2 \xi ^{3/2} \Im(\zeta)-1\right)\psi_{0}^{*} \psi_{1}^{*},\\
\sum_{l,m}b_{lm}^{2}A_l\bar{A}_m&=&\Im(\zeta)^{2}\left(|A_1|^{2}+|A_2|^{2}\right)-\Re(\zeta)^{2}\left(A_1 \bar{A}_2+\bar{A}_1 A_2\right)\nonumber\\
&=&\frac{\Im(\zeta)^2}{2}  \left(\frac{4}{\left|\zeta\right| ^2}\xi  |\psi_{1}^{*}|^2+ \left(4 i \xi ^{3/2} \Im(\zeta)+4 \xi ^3 |\zeta|^{2}+1\right)|\psi_{0}^{*}|^2\right)\\
&&+\frac{2\Re(\zeta)^3 \left(-4 \xi ^{3/2}\Im(\zeta)+i\right)}{\left|\zeta\right| ^2}\sqrt{\xi } \psi_{0}^{*} \psi_{1}^{*}.\nonumber
\end{eqnarray}
Combining the equations (\ref{blm}) and (\ref{regularity_key}) we obtain,
	\[
	|\xi|^{2k}\int_{0}^{\infty}|\hat{\underline{\Psi}}_w|^{2}dX=O\left(|\xi|^{2k-1}|\psi_{0}^{*}|^{2}+|\xi|^{2k-3}|\psi_{1}^{*}|^{2}\right).
	\]
\end{itemize}

To complete the proof of \eqref{relations}, it remains to check the behavior of the derivatives of $\widehat{\Psi}_w$ with respect to $X$ up to the second order. Each derivation adds a factor $(-1)\left(\lambda^{+}_l+\bar{\lambda}^{+}_m\right)$. It is clear the expression is bounded when $|\xi|\leq \xi_0$. Moreover, simple computations show $|\lambda^{+}_l+\bar{\lambda}_m|^{k}=O(|\xi|^{k})$ for $1\leq l,m\leq 2$ and $k=0,1,2$ when $|\xi|>\xi_0$. Note that the term $\int_{\{|\xi|>\xi_0\}\times\mathbb{R}^{+}}|\partial_X^{k}\widehat{\underline{\Psi}}_w|dXd\xi$ will behave asymptotically as $\int_{\{|\xi|>\xi_0\}\times\mathbb{R}^{+}}|\xi|^{2k}|\underline{\widehat{\Psi}}_w|dXd\xi$, and therefore, its boundeness depends on the regularity of the functions $\psi^{*}_0$ and $\psi^{*}_1$.\label{s:appendix}
	\section{Asymptotic behavior of the Green function coefficients}\label{a:Green_asymptotic}

This section is devoted to the proof of Lemma \ref{lemma2} dealing with low and high-frequency expansions of the coefficients of the Green function \eqref{G_def}. The form of the coefficients
\begin{eqnarray*}
B_1^{+}&=&-\frac{1}{\left(\alpha ^2+1\right)^2 \left(\lambda_1^{+}-\lambda_2^{+}\right) \left(\lambda_1^{+}-\lambda_1^-\right) \left(\lambda_1^{+}-\lambda_2^{-}\right)},\nonumber\\
B_2^{+}&=&\frac{1}{\left(\alpha ^2+1\right)^2 \left(\lambda_1^{+}-\lambda _2^{+}\right) \left(\lambda_2^{+}-\lambda_1^{-}\right)\left(\lambda_2^{+}-\lambda_2^{-}\right)},\\
B_1^{-}&=&\frac{1}{\left(\alpha ^2+1\right)^2 \left(\lambda_1^{+}-\lambda_1^{-}\right) \left(\lambda _2^{+}-\lambda_1^{-}\right) \left(\lambda_1^{-}-\lambda_2^{-}\right)},\nonumber\\
B_2^{-}&=&-\frac{1}{\left(\alpha ^2+1\right)^2 \left(\lambda_1^{+}-\lambda_2^{-}\right) \left(\lambda _2^{+}-\lambda _2^{-}\right) \left(\lambda _2^{-}-\lambda_1^{-}\right)}.\nonumber
\end{eqnarray*}
combined with the asymptotic behavior of the eigenvalues stemming from the characteristic equation provide the desired results. Let us now illustrate this for each case.
\begin{itemize}
    \item [] \textit{Low frequencies:} When $|\xi|\ll 1$, we have
   \begin{equation*}
            \lambda_1^-=-|\xi| ^4+O\left(|\xi| ^5\right),\quad\lambda^-_2=-\frac{1}{\left(\alpha ^2+1\right)^{2/3}}+O\left(|\xi|\right),\quad
        \lambda_j^+=\frac{1+(-1)^ji \sqrt{3}}{2 \left(\alpha ^2+1\right)^{2/3}}+O\left(|\xi|\right),\quad j=1,2.
    \end{equation*}
    Consequently,
    \begin{eqnarray*}
B_j^{+}=B_2^{-}=\frac{1}{3}+O(|\xi|),\quad j=1,2,\quad B_1^{-}=1+O(|\xi|).
\end{eqnarray*}
    \item [] \textit{High frequencies:} As $|\xi|\rightarrow +\infty$, the eigenvalues behave for $j=1,2$ as
    \begin{equation*}
            \lambda_j^-=-\frac{i |\xi| }{\alpha +i}+(-1)^j\frac{|\xi|^{-1/2}}{2 \sqrt{1-i \alpha }}+O\left(|\xi|^{-3/2}\right),\quad
        \lambda_j^+=-\frac{i \xi }{\alpha -i}+(-1)^j\frac{\sqrt{\frac{1}{\xi }}}{2 \sqrt{-1-i \alpha }}+O\left(|\xi|^{-3/2}\right).
    \end{equation*}
    
\end{itemize} 

	\section{Proof of Lemma \ref{lemme_tech}}\label{a:lemm_tech}

Note that we have the following Poincaré inequality for $H^2_0(U)$:
$$\|f\|_{H^2_0(U)}\leq C\|D^2f\|^2_{L^2(U)}.$$
The previous result is obtained by chaining the Poincaré inequality for $f$ with the Poincaré inequality for $Df$. We considered the norm
$\|f\|^2_*=\|D^2f\|^2_{L^2(U)}$ on $H^2_0(U)$ which is equivalent to the standard $H^2_0(U)$ norm.

We claim that
$$\|\Delta f\|_{L^2(U)}=\|D^2f\|_{L^2(U)}=\|f\|_*$$
for any $f\in H^2_0(U)$. 

Indeed, let us consider $f\in  C^\infty_0(U)$. Then integration by parts and commutativity of partial derivatives for smooth functions implies $$\int_U f_{x_ix_i}f_{x_jx_j}dx=-\int_U f_{x_i}f_{x_jx_jx_i} dx=-\int_U f_{x_i}f_{x_jx_ix_j} dx=\int_U f_{x_ix_j}f_{x_ix_j} dx,$$
for all $1\leq i, j\leq n$. Summing over all $i$ and $ j$ yields
$$\|\Delta f\|_{L^2(U)}=\|D^2f\|_{L^2(U)},$$
for all $f\in  C^\infty_0(U)$. Since $ C^\infty_0(U)$ is dense in $H^2_0(U)$, passing to limits we find that
$$\|\Delta f\|_{L^2(U)}=\|D^2f\|_{L^2(U)}\quad \text{for all}\quad f\in H^2_0(U).$$
This gives the desired equality of norms. Hence, in $H^2_0(U)$ we have
\begin{equation*}
\|f\|_{H^2_0(U)}\leq C\|\Delta f\|_{L^2(U)}.
\end{equation*}

Let us now show that
\begin{equation}
\|\Delta f\|_{L^2(U)}\leq C\|\Delta_\alpha f\|_{L^2(U)},
\end{equation}
where $C>0$ is a constant depending on $\alpha\in\mathbb{R}$.

As a result of Plancherel theorem, we have that
\begin{equation*}
\|\Delta_\alpha f\|^2_{L^2(U)}=\|\widehat{\Delta_\alpha f}\|^2_{L^2(U)},\quad\text{and}\quad \|\Delta f\|^2_{L^2(U)}=\|\widehat{\Delta f}\|^2_{L^2(U)}.
\end{equation*}
Let $(\xi_1,\xi_2)$ be the Fourier variables associated to $X$ and $Y$, respectively. We have
\begin{equation}
\begin{split}
\|\widehat{\Delta_\alpha f}\|^2_{L^2(U)}&=\int_{U}|\xi_1^2+(\xi_2\pm\alpha\xi_1)^2|^2|\widehat{f}|^2\\&\geq\int_{U}|\xi_1^2+\xi_2^2-\alpha^2\xi_1^2|^2|\widehat{f}|^2\\
&\geq(1-\alpha^2)^2\|\widehat{\Delta f}\|^2_{L^2(U)}
\end{split}
\end{equation} 

Here, we have used that $(a+b)^2\geq a^2-b^2$, for $a,b\in\mathbf{R}$. Therefore, if $\Delta_{\alpha}f\in L^2(U)$, we have that $\Delta f\in L^2(U)$.

For non-homogeneous Dirichlet boundary data, a similar result can be obtained by applying the Poincar\'e--Wirtinger inequality. The additional term in the estimate is introduced to control the boundary terms. The proof is left to the reader.

\begin{remark}
In \cite{Bresch2005}, the authors work with a norm equivalent to $\|\cdot\|_{H^2}$.

Indeed, let $U$ be a bounded subset of $\mathbb{R}^2$. Then, the scalar product
\begin{eqnarray*}
(u,v)_{\alpha}=\int_{U}\langle D^2_\alpha u,D^2_\alpha v\rangle_F,
\end{eqnarray*}
defines a norm on $H^2_0(U)$ for
\begin{equation*}
D^2_\alpha=\begin{pmatrix}
	(1+\alpha^2)\partial_{X}^2&\pm\alpha\partial_{XY}^2\\
	\pm\alpha\partial_{XY}^2&\partial_{Y}^2\\
\end{pmatrix}
\end{equation*}
and the Frobenius inner product\footnote{The Frobenius norm on ${\displaystyle \mathrm {M} _{m,n}(K)}$ is derived from the scalar or standard Hermitian product on this space, namely
\begin{equation*}
	{\displaystyle (A,B)\in \mathrm {M} _{m,n}(K)^{2}\mapsto \langle A,B\rangle =\operatorname {tr} (A^{*}B)=\operatorname {tr} (BA^{*})},
\end{equation*}
	where ${\displaystyle A^{*}}$ denotes the conjugate transpose of  ${\displaystyle A}$.}.

The equivalence is easily obtained since a function $u$ of $H^2_0(U)$ satisfies
\begin{equation*}
\dfrac{1}{1+\alpha^2}\|u\|_{\alpha}\leq\|u\|^2_{H^2_0(U)}\leq\max(1+\alpha^2,2\alpha^2)\|u\|^2_{\alpha}.
\end{equation*}
\end{remark}
	\section{Proof of Lemmas \ref{beh_DtN} and \ref{beh_DtN_e}}\label{a:matrix_t}

\begin{proof}[Proof of Lemma~\ref{beh_DtN}]The Fourier multiplier $M_w=(m_{i,j})_{2\leq i\leq 3,0\leq j\leq 1}\in M_2(\mathbb{C})$ has the components
	\begin{align*}
		m_{2,0}&=-\left(\alpha_w ^2+1\right) \left(\left(\alpha_w ^2+1\right) \lambda _1 \lambda _2+\xi ^2\right),\\
	m_{2,1}&=-\left(\alpha_w ^2+1\right) \left(\left(\alpha_w ^2+1\right) \lambda _1+\left(\alpha_w ^2+1\right) \lambda _2+2 i \alpha_w  \xi \right),\numberthis\label{exp_MGW}\\
		m_{3,0}&= -\frac{1}{2} \left(2 \left(\alpha_w ^2+1\right) \lambda _1 \lambda _2 \left(\left(\alpha_w ^2+1\right) \lambda _1+\left(\alpha_w ^2+1\right) \lambda _2+4 i \alpha_w  \xi \right)+4 i \alpha_w  \xi ^3-1\right),\\
	m_{3,1}&=\left(5 \alpha_w ^2+1\right) \xi ^2\\
	&\quad+\left(\alpha_w ^2+1\right) \left(-\lambda _1 \left(\left(\alpha_w ^2+1\right) \lambda _2+4 i \alpha_w  \xi \right)-\lambda _2 \left(\left(\alpha_w ^2+1\right) \lambda _2+4 i \alpha_w  \xi \right)-\left(\alpha_w ^2+1\right) \lambda _1^2\right).
	\end{align*}
	
 Recall that $\lambda_i$ are the roots of the characteristic equation
 \begin{equation*}
	P(\lambda)=-\lambda-(\lambda^{2}+(\alpha\lambda+i\xi)^{2})^{2} =0,
	\end{equation*}
satisfying $\Re(\lambda_i)>0$. 

Expressions in \eqref{exp_MGW} together with the asymptotic expansions in
	Lemma \ref{Dalibard2014_lemma24} are the core ingredients of the proof. 
	\begin{itemize}
	    \item [$\bullet$] At low frequencies, the eigenvalues are complex conjugate constants depending on the parameter $\alpha_{w}$. In particular, substituting $\lambda_{1}(\xi)=\displaystyle\frac{1- i\sqrt{3}}{2 \left(\alpha_{w} ^2+1\right)^{2/3}}+O\left(|\xi|\right)$, $\lambda_2=\bar{\lambda}_1$ in \eqref{exp_MGW} provides
	\begin{align*}
		m_{2,0}&=-\left(\alpha_{w} ^2+1\right)^2 |\lambda_1|^2+O(|\xi|)=-\left(\alpha_{w} ^2+1\right)^{2/3} |\lambda_1|^2+O(|\xi|),\\
	m_{2,1}&=-2\left(\alpha_{w} ^2+1\right)^2  \Re(\lambda_1)+O(|\xi|)=-\left(\alpha_{w} ^2+1\right)^{4/3} |\lambda_1|^2+O(|\xi|),\\
		m_{3,0}&= -2\left(\alpha_{w} ^2+1\right)^2 |\lambda _1|^2 \Re(\lambda _1)+\frac{1}{2}+O(|\xi|)=-\frac{1}{2}+O(|\xi|),\\
	m_{3,1}&=-\left(\alpha_{w} ^2+1\right)^2 \left(3\Re(\lambda_1)^2-\Im(\lambda_1)^2\right)+O(|\xi|)=O(|\xi|).
	\end{align*}
		To compute a more precise value of $m_{3,1}$, we take into account that $\lambda_{j}=\frac{1+(-1)^ji \sqrt{3}}{2 \left(\alpha ^2+1\right)^{2/3}}+\frac{4 i \alpha |\xi| }{3 \alpha ^2+3}+O\left(|\xi| ^2\right)$ and do not neglect the coefficients of $\xi$.
	\item [$\bullet$] When $\xi\rightarrow+\infty$, the roots behave as $\lambda_{j}^{+}(\xi)=\zeta^2\xi+(-1)^{j}\frac{i\xi^{-\frac{1}{2}}}{2}\zeta+O( |\xi|^{-2})$, where $\zeta=\sqrt{\frac{1+i\alpha_{w}}{\alpha_{w}^{2} +1}}$. This yields the matrix,
	\begin{equation*}
			M_w=\left(
			\begin{array}{cc}
			-2(1+ i \alpha_{w} ) \xi ^2+O\left(\xi^{-1/2}\right)& -2(1+ 2i \alpha_{w}) \left(\alpha_{w} ^2+1\right) \xi +O\left(\xi^{-1/2}\right)\\
			2\left(7+2 i \alpha_w -\dfrac{8(1+i\alpha_w)}{\alpha_w^2 +1}\right)\xi ^3+O\left(1\right)&-2(1-  8 \alpha_w^2 +7i\alpha_w )\xi ^2+O\left(\xi^{1/2}\right)
			\end{array}
			\right).
			\end{equation*}
			Let us illustrate the proof of the above result. We see that
	\begin{align*}
	m_{2,0}&=-\left(\alpha_{w} ^2+1\right) \left(\left(\alpha_{w} ^2+1\right) \zeta^4+1\right)\xi^2+O(\xi^{-1/2})\\
	&=-\left(\alpha_{w} ^2+1\right)\left(\dfrac{1+2i\alpha_{w}-\alpha_w^2}{1+\alpha_{w}^2}+1\right)+O(\xi^{-1/2})=-2 (1+i\alpha_{w})+O(\xi^{-1/2}),\\
	m_{2,1}&=-2\left(\alpha_{w} ^2+1\right)  \left(\left(\alpha_{w} ^2+1\right) \zeta^2+ i \alpha_{w} \right)\xi+O(\xi^{-1/2})=-2(\alpha_{w}^2+1)(1+ 2i \alpha_{w}) +O(\xi^{-1/2}),\\
		m_{3,0}
		&=- 2\left( \left(\alpha_w ^2+1\right)^2 \zeta^6_w+2 \left(\alpha_w ^2+1\right)  i \alpha_w \zeta^4_w+ i \alpha_w\right)\xi^3+O(1)\\
		&=-2\left(\frac{(1+i\alpha_{w})^2}{1+\alpha_w^2}(1+3i\alpha_{w})+i\alpha_{w} \right) \xi ^3+O(\xi^{3/2})=2\left(7+2 i \alpha_w -\dfrac{8(1+i\alpha_w)}{\alpha_w^2 +1}\right)\xi ^3+O\left(1\right),\\
	m_{3,1}&=\left(5 \alpha_w ^2+1-\left( 3\left(\alpha_w ^2+1\right)^2 \zeta_w^4+8 i \alpha_w \left(\alpha_w ^2+1\right)  \zeta^2 \right)\right)\xi^2+O(\xi^{1/2})\\
	&=  \left(5 \alpha_w ^2+1-\left( 3(1+\alpha_{w}i)^2+8 i \alpha_w (1+\alpha_{w}i) \right)\right)\xi ^2+O(\xi^{1/2})\\
	&=-2(1-  8 \alpha_w^2 +7i\alpha_w )\xi ^2+O(\xi^{1/2}).
	\end{align*}
	The matrix values when $\xi\rightarrow-\infty$ can be computed in the same manner.
	\end{itemize}\end{proof}
	
	\begin{proof}[Proof of Lemma~\ref{beh_DtN_e}]
	In the eastern boundary layer domain, the matrix of Fourier multipliers $M_e=(n_{i,j})_{2\leq i\leq 3,0\leq j\leq 1}\in M_2(\mathbb{C})$ is formed by the elements 
\begin{align*}
n_{2,0}&=\left(\alpha_e ^2+1\right) \left(\left(\alpha_e ^2+1\right) \mu _1 \mu _2+\xi ^2\right),\\
n_{2,1}&=\left(\alpha_e ^2+1\right) \left(\left(\alpha_e ^2+1\right) \mu _1+\left(\alpha_e ^2+1\right) \mu _2-2 i \alpha_e  \xi \right),\numberthis\label{exp_MGE}\\
n_{3,0}&=-\frac{1}{2} \left(2 \left(\alpha_e ^2+1\right)^2 \mu _1 \mu _2 \left(\mu _1+\mu _2\right)+4 i \alpha_e  \xi ^3-1\right),\\
n_{3,1}&=-\left(\left(\alpha_e ^2+1\right)^2 \left(\mu _1^2+\mu _2 \mu _1+\mu _2^2\right)+\left(3 \alpha_e ^2-1\right) \xi ^2\right),
\end{align*}
where $\mu_i$, $i=1,2$ are the complex roots of positive real part of the equation
	\begin{equation*}
    P_e(\mu,\xi)=-\mu+(\mu^2+(-\alpha_{e}\mu+i\xi)^2)^2.
\end{equation*}
	\begin{itemize}
	    \item [$\bullet$] When $|\xi|\gg 1$, the eigenvalues behave like (see Lemma \ref{lemma:asymp_behavior}) $$\mu_1=|\xi| ^4+O\left(|\xi| ^5\right),\qquad\mu_2=\displaystyle\frac{1}{\left(\alpha ^2+1\right)^{2/3}}+O\left(|\xi|\right).$$ We obtain immediately that
	\begin{align*}
n_{2,0}&=\left(\alpha_{e} ^2+1\right)\xi ^2+O(\xi^4),\quad n_{2,1}=\left(\alpha_{e} ^2+1\right)^{4/3}+O(\xi),\\
n_{3,0}&=\frac{1}{2}+O(\xi),\quad
n_{3,1}=-\left(\alpha_{e} ^2+1\right)^2\left(\alpha ^2+1\right)^{-4/3}+O(\xi)=-\left(\alpha ^2+1\right)^{2/3}+O(\xi).
\end{align*}
	\item [$\bullet$] Similarly to the proof of Lemma \ref{beh_DtN}, we only discuss the case when $\xi\rightarrow+\infty$, since the behavior at $-\infty$ results from applying the same reasoning. 
	
	The high frequency expansion of the eigenvalues are of the form   $$\mu_{j}(\xi)=\zeta_e^2\xi+\frac{(-1)^{j}}{2}\xi^{-\frac{1}{2}}\zeta_e+O(\xi^{-3/2}),\qquad \zeta_e=\sqrt{\frac{1-i\alpha_{e}}{\alpha_{e}^{2} +1}}.$$ We have
	\begin{equation*}
		M_{e}=\left(
		\begin{array}{cc}
		2(1- i \alpha_e ) \xi ^2+O\left(\xi^{-1/2}\right)& 2 (1-2 i \alpha_e ) \left(\alpha_e ^2+1\right)\xi +O\left(\xi^{-1/2}\right)\\
		2 \left( 3-2 i \alpha_e-\dfrac{4 (1-i \alpha_e )}{\alpha_e ^2+1}\right)\xi ^3+O\left(1\right)&-2 (1-3i\alpha_e)\xi ^2+O\left(\xi^{-1/2}\right)
		\end{array}
		\right).
		\end{equation*}
Indeed, substituting the formulae of the eigenvalues at high frequencies in \eqref{exp_MGE} gives
\begin{align*}
n_{2,0}&=\left(\alpha_{e} ^2+1\right) \left(\left(\alpha_{e} ^2+1\right) \zeta^4_e+1\right)\xi^2+O(\xi^{-1/2})=2(1-i\alpha_e)+O(\xi^{-1/2}),\\
n_{2,1}&=-2\left(\alpha_{e} ^2+1\right) \left(\left(\alpha_{e} ^2+1\right)\zeta^2_e- i \alpha_{e}   \right)\xi+O(\xi^{-1/2})=2 (1-2 i \alpha_e ) \left(\alpha_e ^2+1\right)+O(\xi^{-1/2}),\\
n_{3,0}&=- 2\left( \left(\alpha_e ^2+1\right)^2 \zeta_e^6+ i \alpha_e  \right)\xi ^3+O(\xi^{3/2})=-2\left(\dfrac{2 i \alpha ^3-3 \alpha ^2-2 i \alpha +1}{1+\alpha_e^2}\right)\xi^3+O(1)\\
&=2 \left( 3-2 i \alpha_e-\dfrac{4 (1-i \alpha_e )}{\alpha_e ^2+1}\right)\xi^3+O(1),\\
n_{3,1}&=-\left(3\left(\alpha_e ^2+1\right)^2 \zeta^4+3 \alpha_e ^2-1 \right)\xi ^2+O(\xi^{-1/2})=-\left(3\left(1-i\alpha_e \right)^2 +3 \alpha_e ^2-1 \right)\xi ^2+O(\xi^{-1/2})\\
&=-2(1-3i\alpha_e)\xi ^2+O(\xi^{-1/2}).
\end{align*}
	\end{itemize}\end{proof}
	\hyphenation{aniso-tropic}

\printbibliography

@misc{Sougadinis, place={Paris}, title={Lectures at the Collège de France}, author={Souganidis, P. E.}, year={2009}}

@book{Dijkstra2005,
	title={Nonlinear physical oceanography: a dynamical systems approach to the large scale ocean circulation and El Nino},
	author={Dijkstra, Henk A},
	volume={28},
	year={2005},
	publisher={Springer Science \& Business Media}
}

@article{dalibard2011effective,
	title={Effective boundary condition at a rough surface starting from a slip condition},
	author={Dalibard, Anne-Laure and G{\'e}rard-Varet, David},
	journal={Journal of Differential Equations},
	volume={251},
	number={12},
	pages={3450--3487},
	year={2011},
	publisher={Elsevier}
}

@InProceedings{Alazard2016,
  author       = {Alazard, Thomas and Burq, Nicolas and Zuily, Claude},
  title        = {Cauchy theory for the gravity water waves system with non-localized initial data},
  booktitle    = {Annales de l'Institut Henri Poincare (C) Non Linear Analysis},
  year         = {2016},
  volume       = {33},
  number       = {2},
  pages        = {337--395},
  organization = {Elsevier},
}

@Article{Basson2008,
  author    = {Basson, Arnaud and G{\'e}rard-Varet, David},
  title     = {Wall laws for fluid flows at a boundary with random roughness},
  journal   = {Communications on pure and applied mathematics},
  year      = {2008},
  volume    = {61},
  number    = {7},
  pages     = {941--987},
  publisher = {Wiley Online Library},
}

@Article{Bresch2005,
  author    = {Bresch, Didier and G{\'e}rard-Varet, David},
  title     = {Roughness-induced effects on the quasi-geostrophic model},
  journal   = {Communications in mathematical physics},
  year      = {2005},
  volume    = {253},
  number    = {1},
  pages     = {81--119},
  publisher = {Springer},
}

@book{Dalibard2009,
	title={Mathematical study of degenerate boundary layers: A large scale ocean circulation problem},
	author={Dalibard, Anne-Laure and Saint-Raymond, Laure},
	volume={253},
	number={1206},
	year={2018},
	publisher={American Mathematical Society}
}

@Article{Dalibard2017,
  author    = {Dalibard, Anne-Laure and G{\'e}rard-Varet, David},
  title     = {Nonlinear boundary layers for rotating fluids},
  journal   = {Analysis \& PDE},
  year      = {2017},
  volume    = {10},
  number    = {1},
  pages     = {1--42},
  publisher = {Mathematical Sciences Publishers},
}

@Article{Dalibard2014,
  author    = {Dalibard, Anne-Laure and Prange, Christophe},
  title     = {Well-posedness of the Stokes--Coriolis system in the half-space over a rough surface},
  journal   = {Analysis \& PDE},
  year      = {2014},
  volume    = {7},
  number    = {6},
  pages     = {1253--1315},
  publisher = {Mathematical Sciences Publishers},
}

@Article{Desjardins1999,
  author    = {Desjardins, Beno\^{i}t and Grenier, Emmanuel},
  title     = {On the homogeneous model of wind-driven ocean circulation},
  journal   = {SIAM Journal on Applied Mathematics},
  year      = {1999},
  volume    = {60},
  number    = {1},
  pages     = {43--60},
  publisher = {SIAM},
}

@Article{Gerard-Varet2010,
  author    = {G{\'e}rard-Varet, David and Masmoudi, Nader},
  title     = {Relevance of the slip condition for fluid flows near an irregular boundary},
  journal   = {Communications in Mathematical Physics},
  year      = {2010},
  volume    = {295},
  number    = {1},
  pages     = {99--137},
  publisher = {Springer},
}

@Book{Pedlosky2013,
  title     = {Geophysical fluid dynamics},
  publisher = {Springer Science \& Business Media},
  year      = {2013},
  author    = {Pedlosky, Joseph},
}

@Article{Colin1996,
  author    = {Colin, T},
  title     = {Remarks on a homogeneous model of ocean circulation},
  journal   = {Asymptotic Analysis},
  year      = {1996},
  volume    = {12},
  number    = {2},
  pages     = {153--168},
  publisher = {IOS Press},
}

@Article{Gerard-Varet2003,
  author    = {G{\'e}rard-Varet, David},
  title     = {Highly rotating fluids in rough domains},
  journal   = {Journal de Math{\'e}matiques Pures et Appliqu{\'e}es},
  year      = {2003},
  volume    = {82},
  number    = {11},
  pages     = {1453--1498},
  publisher = {Elsevier},
}

@Article{Jaeger2001,
  author    = {J{\"a}ger, Willi and Mikeli{\'c}, Andro},
  title     = {On the roughness-induced effective boundary conditions for an incompressible viscous flow},
  journal   = {Journal of Differential Equations},
  year      = {2001},
  volume    = {170},
  number    = {1},
  pages     = {96--122},
  publisher = {Elsevier},
}

@Article{Stommel1948,
  author    = {Stommel, Henry},
  title     = {The westward intensification of wind-driven ocean currents},
  journal   = {Eos, Transactions American Geophysical Union},
  year      = {1948},
  volume    = {29},
  number    = {2},
  pages     = {202--206},
  publisher = {Wiley Online Library},
}

@Article{Munk1950,
  author  = {Munk, Walter H},
  title   = {On the wind-driven ocean circulation},
  journal = {Journal of meteorology},
  year    = {1950},
  volume  = {7},
  number  = {2},
  pages   = {80--93},
}

@Article{Masmoudi2000,
  author    = {Masmoudi, Nader},
  title     = {Ekman layers of rotating fluids: the case of general initial data},
  journal   = {Communications on Pure and Applied Mathematics},
  year      = {2000},
  volume    = {53},
  number    = {4},
  pages     = {432--483},
  publisher = {Wiley Online Library},
}

@Article{Kato1975b,
  author    = {Kato, Tosio},
  title     = {The Cauchy problem for quasi-linear symmetric hyperbolic systems},
  journal   = {Archive for Rational Mechanics and Analysis},
  year      = {1975},
  volume    = {58},
  number    = {3},
  pages     = {181--205},
  publisher = {Springer},
}

@Article{Yang2003,
  author    = {Yang, Jiayan},
  title     = {On the importance of resolving the western boundary layer in wind-driven ocean general circulation models},
  journal   = {Ocean Modelling},
  year      = {2003},
  volume    = {5},
  number    = {4},
  pages     = {357--379},
  publisher = {Elsevier},
}

@Article{Lemarie-Rieusset1999,
  author  = {Lemarié-Rieusset, Pierre Gilles},
  title   = {Solutions faibles d'énergie infinie pour les équations de Navier-Stokes dans $\mathbb{R}^{3}$.},
  journal = {Comptes Rendus de l'Académie des Sciences-Series I-Mathematics},
  year    = {1999},
  volume  = {328},
  number  = {12},
  pages   = {1133-1138},
}

@Book{Lemarie-Rieusset2002,
  title     = {Recent developments in the Navier-Stokes problem},
  publisher = {CRC Press},
  year      = {2002},
  author    = {Lemarié-Rieusset, Pierre Gilles},
}

@Article{Ladyzenskaja1980,
  author  = {O. A. Lady\v{z}enskaja and V. A. Solonnikov},
  title   = {Determination of solutions of boundary value problems for stationary Stokes and Navier–Stokes equations having an unbounded Dirichlet integral},
  journal = {Zap. Nauchn. Sem. Leningrad. Otdel. Mat. Inst. Steklov. (LOMI)},
  year    = {1980},
  volume  = {96},
  pages   = {117–160},
  note    = {In Russian; translated in J. Soviet Math. 21:5 (1983), 728–851. MR},
}

@Article{Necasa,
  author        = {Ne{\v{c}}as, Jindrich},
  title         = {Les M{\'e}thodes Directes en Th{\'e}orie des {\'E}quations Elliptiques. Academia, Prague, 1967},
  journal       = {MR0227584},
  __markedentry = {[g.lopezruiz:]},
}

@book{katok1997introduction,
  title={Introduction to the modern theory of dynamical systems},
  author={Katok, Anatole and Hasselblatt, Boris},
  volume={54},
  year={1997},
  publisher={Cambridge university press}
}
\end{document}